\documentclass[a4paper,11pt]{amsart}
\usepackage{amssymb, amsmath, amsthm,stmaryrd,bm}
\usepackage{mdwlist}

\def\res{\hbox{ {\vrule height .3cm}{\leaders\hrule\hskip.3cm}}\hskip5.0\mu}

\newcommand\beqn{\begin{equation}}
\newcommand\eeqn{\end{equation}}
\newcommand\beqny{\begin{eqnarray}}
\newcommand\eeqny{\end{eqnarray}}
\newcommand\beqnyn{\begin{eqnarray*}}
\newcommand\eeqnyn{\end{eqnarray*}}

\usepackage{amsmath,amsfonts,amssymb,amsthm}
\renewcommand{\theequation}{\arabic{section}.\arabic{equation}}
\newtheorem{theorem}{Theorem}[section]
\newtheorem{lemma}[theorem]{Lemma}

\newtheorem{corollary}[theorem]{Corollary}
\newtheorem{definition}[theorem]{Definition}
\newtheorem{remark}[theorem]{Remark}

\def\b{\beta}          
\def\g{\gamma}           
         
\def\e{\epsilon}

\def\t{\tau}

\def\r{\rho} 

\def\s{\sigma}

\numberwithin{equation}{section}

\newcommand{\op}[1]{\operatorname{\text{\rm #1}}}

\begin{document}

\setlength\parskip{5pt}
\title[Fine properties of branch point singularities]{Fine properties of branch point singularities: \\ Dirichlet energy minimizing multi-valued functions}
\author{Brian Krummel \& Neshan Wickramasekera} 

\begin{abstract}
In the early 1980's Almgren developed a theory of
Dirichlet energy minimizing multi-valued functions, proving that the Hausdorff
dimension of the singular set (including branch points) of such a
function is at most $(n-2),$ where $n$ is the dimension of its domain. Almgren 
used this result in an essential way
to show that the same upper bound holds for the dimension of the singular
set of an area minimizing $n$-dimensional rectifiable current of arbitrary
codimension. In either case, the dimension bound is sharp.
Building on Almgren's work, we develop estimates  to study the asymptotic behaviour of a multi-valued Dirichlet energy minimizer on approach to its singular set. 
Our estimates imply that a Dirichlet energy minimizer at ${\mathcal H}^{n-2}$ a.e.\ point of its singular set has a unique set of homogeneous multi-valued cylindrical tangent functions (blow-ups) to which the minimizer, modulo a set of single-valued harmonic functions, decays exponentially fast upon rescaling. A corollary is that the singular set  is countably $(n-2)$-rectifiable.  Our work is inspired by the work of L.\ Simon (\cite{Sim93}) on the analysis of singularities of minimal submanifolds in multiplicity 1 classes, and uses some new estimates and strategies together with techniques from \cite{Wic14} to overcome additional difficulties arising from  higher multiplicity and low regularity of the minimizers in the presence of branch points. The results described here were announced in \cite{KrumWic1} where the special case of two-valued Dirichlet minimizing functions was treated.
\end{abstract}

\maketitle
\tableofcontents

\section{Introduction}

Our purpose here is to develop estimates to study the asymptotic behavior, on approach to branch points, of a $q$-valued ($q \geq 2$) locally Dirichlet energy minimizing function mapping a domain $\Omega \subset {\mathbb R}^{n}$ into the space ${\mathcal A}_{q}({\mathbb R}^{m})= \{\sum_{j=1}^{q} \llbracket a_{j} \rrbracket\, : \, a_{j} \in {\mathbb R}^{m}\}$, the space of unordered $q$-tuples of points $a_{1}, \ldots, a_{q}$  in ${\mathbb R}^{m}$ (identified with Dirac masses 
$\llbracket a_{j} \rrbracket$ at $a_{j}$, $j=1,  \ldots, q$). The results contained in the present article were announced in our earlier paper \cite{KrumWic1} in which, among other things, the special case $q=2$ of the work here was treated.

In the early 1980's Almgren, in the first part of his monumental work published posthumously as \cite{Almgren}, introduced the (non-linear) Sobolev space $W^{1, 2}(\Omega; {\mathcal A}_{q}({\mathbb R}^{m}))$ of $q$-valued functions, and studied regularity properties and the singular behavior of locally Dirichlet energy minimizing functions in $W^{1, 2}_{\rm loc}(\Omega; {\mathcal A}_{q}({\mathbb R}^{m}))$. This study provided what may be considered the ``linear theory'' necessary for his work, carried out in later parts of \cite{Almgren}, on the branch point singularities of a locally area minimizing integer multiplicity rectifiable current of codimension $\geq 2$, i.e.\   interior singularities at which the support of the current is not immersed and one tangent cone to the current is a plane of some integer multiplicity $\geq 2$. Indeed, one of the key results of Almgren's seminal work is that a certain linearization (or blowing-up) procedure performed on an $n$-dimensional area minimizer in ${\mathbb R}^{n+m}$ (or more generally in an $n+m$ dimensional Riemannian manifold) at a branch point with area density an integer $q \geq 2$ and a tangent plane $P$ (of multiplicity $q$) produces a $q$-valued locally Dirichlet energy minimizing function on  $P$, with its value at each point of $P$ equal to an  unordered $q$-tuple of points in the orthogonal complement of $P$. Thus, assuming without loss of generality that $P = \{0\} \times {\mathbb R}^{n}$, this linearization procedure leads to ${\mathcal A}_{q}({\mathbb R}^{m})$-valued locally Dirichlet energy minimizing functions on ${\mathbb R}^{n}$. 

Almgren's  work on locally energy minimizing functions  $u  \in W^{1, 2}_{\rm loc}(\Omega; {\mathcal A}_{q}({\mathbb R}^{m}))$ established two key results: the first is that $u$ is locally H\"older continuous on $\Omega$ with a uniform H\"older exponent depending only on $n, m, q$, and the second is that locally away from a relatively closed set ${\mathcal B}_{u} \subset \Omega$ of Hausdorff dimension $\leq n-2$  the values of $u$ are given by those of $q$ (single-valued) harmonic functions each taking values in ${\mathbb R}^{m}.$ Easy examples show that $u$ need not be Lipschitz, and that this dimension bound on ${\mathcal B}_{u}$ is the best possible general estimate. We shall call a point in ${\mathcal B}_{u}$ a branch point singularity of $u$; thus $Z \in {\mathcal B}_{u}$ if there is no neighborhood about $Z$ in which the values of $u$ are given by the values of $q$ harmonic functions. Using this linear theory and a much more delicate refinement of the blowing-up procedure mentioned above (in which the height of an area minimizing current relative to a carefully constructed ``centre manifold'' rather than relative to a planar tangent cone is blown up), Almgren then proved the fundamental result that the Hausdorff dimension of the interior branch set (in fact the Hausdorff dimension of the entire interior singular set) of a locally are minimizing integer multiplicity rectifiable current of dimension $n$ and codimension $\geq 2$ is at most $n-2.$  It had long been  known prior to the work \cite{Almgren} that when its codimension is 1, an area minimizer does not admit branch point singularities and that its singular set in fact has Hausdorff dimension at most $n-7.$ (See the recent work of De~Lellis and Spadaro (\cite{DeLSpa11}, \cite{DeLSpa-I}, \cite{DeLSpa-II}, \cite{DeLSpa-III}) for a more accessible, streamlined account of Almgren's theory, including connections of parts of it with PDE and more recent developments in metric geometry.)

Our work builds on Almgren's linear theory, and addresses the general question of obtaining what may be considered a first order  asymptotic expansion for a locally Dirichlet energy minimizing function $u \, : \, \Omega \to {\mathcal A}_{q}({\mathbb R}^{m})$ near a branch point. Low regularity of $u$ (recall that $u$ is no more than H\"older continuous in general) and the fact that branch points are non-isolated (except in dimension $n= 2$) add considerably to the difficulty of this task. Inspired by the seminal work of L.~Simon (\cite{Sim93}), we develop 
a number of new uniform integral estimates for $u$ for this purpose. These estimates, when combined with Almgren's theory, imply various results concerning the asymptotic behaviour of $u$ near generic points along ${\mathcal B}_{u},$ results on the structure of ${\mathcal B}_{u}$ as well as refinements of these results near special points in ${\mathcal B}_{u}$. For instance, letting ${\mathcal B}_{u, q}$ denote the set of branch points of multiplicity $q$ (i.e.\ points $Z \in {\mathcal B}_{u}$ such that $u(Z) = q\llbracket a \rrbracket$ for some $a \in {\mathbb R}^{m}$),  we obtain (in Theorem~A below) that for ${\mathcal H}^{n-2}$ a.e.\ $Z \in {\mathcal B}_{u, q}$, there is a unique ``tangent function'' $\varphi^{(Z)} \, : \, {\mathbb R}^{n} \to {\mathcal A}_{q}({\mathbb R}^{m})$ which is cylindrical (i.e.\ translation invariant along a subspace of ${\mathbb R}^{n}$ of dimension $n-2$) and homogeneous of some positive degree $\alpha$ (having the form 
$\alpha =  r_{Z}/q_{Z}$ for relatively prime positive integers $r_{Z}, q_{Z}$  with $q_{Z} \leq q$) such that $u,$ modulo a single-valued harmonic function $h$ (equal at every point $X$ to the average of the values of $u(X)$), asymptotically decays in $L^{2}$  to $\varphi^{(Z)}$ upon rescaling about $Z$; in view of continuity of $u$, this result immediately implies (in the corollary following Theorem~A) a similar asymptotic description for $u$ near ${\mathcal H}^{n-2}$ a.e.\ point along the entire branch set ${\mathcal B}_{u};$ concerning the structure of ${\mathcal B}_{u}$, we deduce (in Theorem~B) that ${\mathcal B}_{u}$ is countably $(n-2)$ rectifiable; and in case $q=2$, near a point $Z \in 
{\mathcal B}_{u} = {\mathcal B}_{u, 2}$ where there is a tangent function $\varphi^{(Z)}$ whose degree of homogeneity is $1/2$ (the lowest degree of homogeneity that can occur in case $q=2$), we obtain (in Theorem~C) that ${\mathcal B}_{u}$ is in fact an $n-2$ dimensional embedded $C^{1, \alpha}$ submanifold and that 
$u$ has a unique cylindrical tangent function $\varphi^{(Y)}$ with degree of homogeneity $1/2$ at every point $Y$ in ${\mathcal B}_{u}$ near $Z$.

In \cite{KrumWic1} we considered the special case $q=2$ of the present work as well as a certain class, related to stable minimal hypersurfaces, of non-minimizing 2-valued critical points of Dirichlet energy.  Except for the proof of Theorem~C which is contained in \cite{KrumWic1}, the present article however is self-contained. Although we follow the same broad strategy to establish the main decay estimates both in the case $q=2$ and for general $q$ considered here, there are a number of key steps that are either unnecessary in the case $q=2$, or require considerable additional effort or an entirely different approach to carry out for general $q$. For this reason,  the reader may wish to consult \cite{KrumWic1} first. 
The main reason for the substantial additional complexity  in the case $q \geq 3$ is that in this case, we generally have that $\overline{{\mathcal B}_{u} \setminus {\mathcal B}_{u, q}} \cap {\mathcal B}_{u, q} \neq \emptyset$, and consequently, unlike when $q=2$ in which case ${\mathcal B}_{u} = {\mathcal B}_{u, 2}$, it is no longer true that locally away from ${\mathcal B}_{u, q}$ the values of $u$ are given by $q$ single-valued harmonic functions.  

In particular, for this reason, it has been necessary here  to develop a new strategy, in place of the PDE theoretic arguments of \cite{Sim93} that were adapted in \cite{KrumWic1}, to classify ``homogeneous blow-ups''---i.e.\  homogeneous functions $w \, : \, {\mathbb R}^{n} \to {\mathcal A}_{q}({\mathbb R}^{m})$ produced by blowing up (certain) sequences of locally energy minimizing $q$-valued functions $u^{(\nu)}$ relative to sequences of homogeneous cylindrical $q$-valued functions $\varphi^{(\nu)}$, $\nu=1, 2, 3, \ldots,$ whenever both sequences $(u^{(\nu)})$ and $(\varphi^{(\nu)})$ converge to a fixed non-zero cylindrical homogeneous local energy minimizer $\varphi^{(0)} \, : \, {\mathbb R}^{n} \to {\mathcal A}_{q}({\mathbb R}^{m}) $. This classification (carried out in  Section~\ref{sec:homogblowup_sec} below) amounts to, in the language of \cite{Sim93} (or of \cite{AllardAlmgren}), verification of ``integrability of homogeneous Jacobi fields''  and is at the heart of the main decay estimates we prove. These blow-ups do not in general inherit the same local  energy minimizing property that $u^{(\nu)}$ are assumed to have, although they are locally energy minimizing away from the axis of $\varphi^{(0)}$. Nonetheless, we here establish two key facts concerning such a homogeneous blow-up $w$: (i) monotonicity of the Almgren frequency function, with base point on the axis of $\varphi^{(0)}$, associated with ${\widetilde w} = w - D_{x}\varphi^{(0)} \cdot \lambda(y)$ (Lemma~\ref{blowup-freq}) where $\lambda(y)$ is a certain 
${\mathbb R}^{2}$-valued linear function of $y$ and (ii) a  mean value property (identity \eqref{homogrep1_eqn4} and Lemma~\ref{homogrep1_mvp_lemma}) for the squared $y$-gradient of the Fourier coefficients, with respect to the $x$-variables, of the average of $w$. Here $(x,y)$ denotes coordinates in ${\mathbb R}^{n}$ with $y$ denoting the variables along the axis of $\varphi^{(0)}$. The classification of homogeneous blow-ups is accomplished using a combination of geometric and PDE theoretic arguments that rely on these two facts. 

The first of these facts (frequency monotonicity) is established with the help of a new energy comparison estimate (Lemma~\ref{competitor_lemma}) for local energy minimizers close to $\varphi^{(0)}$ and a uniform height estimate (inequality \eqref{homogrep2_eqn2}) for ${\widetilde w}(X)$ in terms of the distance of $X$ from the axis of $\varphi^{(0)}.$ The latter is based on several other estimates derived from a variant of the frequency monotonicity identity for $u^{(\nu)}$ (Lemma~\ref{keyest_identity}, referred to in the literature on free boundary problems as the Weiss monotonicity formula). The second fact (the mean value property) follows from a first variation estimate (Lemma~\ref{fourierest_lemma}, used also in \cite{KrumWic1} in a different way) for minimizers. The energy estimate of Lemma~\ref{competitor_lemma} implies energy stationarity of ${\widetilde w}$ with respect to a restricted class of deformations, namely, radial deformations in the domain (``squeeze deformations'' in the terminology of \cite{Almgren}) centered at any point on the axis of $\varphi^{(0)}$. 
As a rule of thumb, this stationarity condition is the more subtle of the two first variation identities behind monotonicity of frequency and the present context is no exception to this; in fact, interestingly, in the present context it holds only for radial deformations and not, in general, for arbitrary domain deformations. See the example discussed in Remark~\ref{example}. The other ingredient needed for monotonicity of frequency of ${\widetilde w}$ is 
the energy stationarity with respect to certain range variations (called ``squash deformations'' in \cite{Almgren}). In the present context however it is established non-variationally, based on the uniform height estimate (inequality \eqref{homogrep2_eqn2}) for ${\widetilde w}$ mentioned above.

A more comprehensive non-technical outline of the proofs of all of our main results (including the statements of the key estimates needed), comparing and contrasting our arguments to those of \cite{Sim93}, is given in Section~\ref{outline} below. 

\subsection{The work of Almgren}\label{Alm}

Let $u$ be a $q$-valued locally Dirichlet energy minimizing function on a domain $\Omega \subset {\mathbb R}^{n}$ taking values in ${\mathcal A}_{q}({\mathbb R}^{m})$ equipped with its usual metric. Recall that the branch set ${\mathcal B}_{u}$ of $u$ is defined by requiring that a point $Z \in \Omega \setminus {\mathcal B}_{u}$ if there is $\sigma >0$ and $q$ single-valued harmonic functions $u_{j} \, : \, B_{\sigma}(Z) \to {\mathbb R}^{m}$, $j =1, \ldots, q$,  such that everywhere on $B_{\sigma}(Z)$, the $q$ values of $u$ are  given by the values of these $q$ harmonic functions, i.e.\ $u(X) = \sum_{j=1}^{q} \llbracket u_{j}(X)\rrbracket$ for every $X \in B_{\sigma}(Z).$  In  [Alm83], Almgren established an existence theory  for $q$-valued Dirichlet energy minimizing functions (showing the existence of a minimizer in $W^{1,2}$ with given $q$-valued boundary data), and obtained sharp interior regularity estimates for them, including the sharp upper bound on the Hausdorff dimension of their singular sets.
Almgren's regularity theory says that $u$ is locally uniformly H\"older continuous, and moreover, away from a closed set $\Sigma_{u}  \subset \Omega$ of Hausdorff dimension $\leq  n- 2$, it is regular in the sense that near every point in $\Omega \setminus \Sigma_{u}$, the values of $u$ are locally given by those of $q$ single-valued harmonic functions, no two of which have a common value unless they are identical. The set $\Sigma_{u}$ is the singular set of $u$, which of course contains the branch set ${\mathcal B}_{u}$. 

Almgren's proof of the Hausdorff dimension bound on $\Sigma_{u}$ follows broadly the strategy for bounding the Hausdorff dimension of the singular set of a minimal submanifold in a multiplicity 1 class. This strategy in the case of minimal submanifolds is based on the existence of non-trivial, singular tangent cones at every singular point of the minimal submanifold, which is a consequence of the standard monotonicity formula and a lower bound on the volume density. In the setting of multi-valued Dirichlet energy minimizers, Almgren discovered and employed a fundamental monotonicity formula, namely, an expression for the $\rho$-derivative of  the ``frequency function'' 
\begin{equation}
N_{u, Z}(\rho) = \frac{\rho\int_{B_{\rho}(Z)} |Du|^{2}}{\int_{\partial\, B_{\rho}(Z)} |u|^{2}} \tag{$\star$}
\end{equation}
 associated with the energy minimizer $u$ and a given base point $Z \in \Omega$ (see Section~\ref{sec:frequency_sec} below); this expression---the analogue of the minimal surface monotonicity formula---shows that $N_{u, Z}(\cdot)$ is monotone non-decreasing and hence in particular that the limit ${\mathcal N}_{u}(Z) = \lim_{\rho \to 0} \, 
N_{u, z}(\rho)$ exists for each $Z \in \Omega$. Moreover, assuming without loss of generality that $u$ is average-free (by subtracting the single-valued harmonic average $h(X) = q^{-1}\sum_{j=1}^{q}u_{j}(X)$ from each of the values of $u(X)$ to obtain a $q$-valued function $u_{f}$ with $\Sigma_{u_{f}} = \Sigma_{u}$ and which is still energy minimizing), if 
$u(Z) = q \llbracket 0 \rrbracket$, then any rescaling sequence $u(Z + \rho_{j}X)/\|u(Z + \rho_{j}(\cdot))\|_{L^{2}(B_{1}(0))}$ with $\rho_{j} \to 0^{+}$ has a subsequence that converges to a non-zero tangent function $\varphi^{(Z)}$ (also referred to as a blow-up in the literature)  that is locally Dirichlet energy minimizing and homogeneous of degree $\alpha = {\mathcal N}_{\varphi^{(Z)}}(0) = {\mathcal N}_{u}(Z) >0$. It is not known whether $\varphi^{(Z)}$ is unique i.e.\ whether $\varphi^{(Z)}$ is independent of the sequence $\{\rho_{j}\}$ (except in case $n=2$, see \cite{DeLSpa11}  and Section~\ref{other} below, and now at generic points as shown here) but by energy minimality of $\varphi^{(Z)}$ the dimension of the subspace along which $\varphi^{(Z)}$ is translation invariant is at most $n-2.$ As shown by Almgren, this latter fact leads to the Hausdorff dimension bound ${\rm dim}_{\mathcal H} \, (\{Z \in \Omega \, : \, u(Z) = q\llbracket 0 \rrbracket\}) \leq n-2$, which in view of continuity of $u$ gives ${\rm dim}_{\mathcal H} \, \Sigma_{u} \leq n-2$ by  induction on $q$. 
 
In the second part of \cite{Almgren}, Almgren used this Hausdorff dimension estimate on $\Sigma_{u}$  to establish 
that the Hausdorff dimension of the interior singular set of an  area minimizing $n$-dimensional rectifiable current is at most $n- 2$. As illustrated by complex algebraic varieties such as $\{(z, w) \, : \, z^{2} = w^{3}\} \cap {\mathbb C} \times {\mathbb C}$, the value $n-2$ is the sharp Hausdorff dimension bound on the singular set of both energy minimizing multi-valued functions and area minimizing rectifiable currents of codimension $\geq 2$. In bounding the Hausdorff dimension of the singular set in either case, what is non-trivial is to estimate the size of the branch set. 
\subsection{The present work}
Our work, broadly speaking,  is aimed at studying the asymptotic behavior of $u$ on approach to its branch points, including the question of uniqueness of tangent functions along the branch set ${\mathcal B}_{u}$. We establish a number of new a priori estimates for $u$ with this purpose in mind. A main result  (Theorem~A below) implied by these estimates gives for 
${\mathcal H}^{n-2}$ a.\ e.\ point $Z \in {\mathcal B}_{u, q} = {\mathcal B}_{u} \cap \{u(Z) = q\llbracket a \rrbracket \;\; \mbox{for some $a \in {\mathbb R}^{m}$}\}$ what may be considered a ``first order'' asymptotic expansion of $u$ near $Z$, including uniqueness of its tangent function $\varphi^{(Z)}$  at $Z$ and an exponential rate of convergence of $u$ to $\varphi^{(Z)}$ in $L^{2}$ upon rescaling:

\noindent \textbf{Theorem A.} \emph{
Let $q \geq 2$ and let $u : \Omega \rightarrow \mathcal{A}_q(\mathbb{R}^m)$ be a locally $W^{1,2}$ locally Dirichlet energy minimizing function on an open subset $\Omega$ of $\mathbb{R}^n$. Let $\mathcal{B}_{u,q}$ denotes the set of points $X \in \mathcal{B}_u$ such that $u(X) = q \llbracket u_1(X) \rrbracket$ for some $u_1(X) \in \mathbb{R}^m$. Then for $\mathcal{H}^{n-2}$-a.e. point $Z \in \mathcal{B}_{u,q}$ there exist  a number $\rho_Z > 0$; an integer $N_{Z} \geq 1$; relatively prime positive integers $k_{Z}$, $q_{Z}$ with $N_{Z}q_{Z} \leq q$; a unique non-zero, homogeneous, average-free, Dirichlet energy minimizing function $\varphi^{(Z)} : \mathbb{R}^n \rightarrow \mathcal{A}_q(\mathbb{R}^m)$ given by  
\begin{equation*} 
	\varphi^{(Z)}(X) = \sum_{j=1}^q \llbracket \varphi^{(Z)}_j(X) \rrbracket 
	= (q - N_Z q_Z) \llbracket 0 \rrbracket + \sum_{l=1}^{N_Z} \op{Re}(c^{(Z)}_l (x_1+ix_2)^{k_Z/q_Z}) 
\end{equation*}
for some $c^{(Z)}_l \in \mathbb{C}^m \setminus \{0\}$, where $X = (x_1, x_2, \ldots, x_n);$ and an orthogonal rotation $Q_Z$ of $\mathbb{R}^n$ such that for each $X \in B_{\rho_Z}(0)$,
\begin{equation*}
	u(Z + X) = \sum_{j=1}^q \llbracket h(Z +X) + \varphi^{(Z)}_j(Q_{Z} X) + \epsilon^{(Z)}_j(X) \rrbracket 
\end{equation*} 
where $h \, : \, \Omega  \to {\mathbb R}^{m}$ is the (single-valued) harmonic function (independent of $Z$) equal to the average of the values of $u$, i.e.\ $h(X) = q^{-1}\sum_{j=1}^{q} u_{j}(X)$; and 
$\e_{j}^{(Z)}(X) \in {\mathbb R}^{m}$, $\sum_{j=1}^q \epsilon^{(Z)}_j(X) = 0$ for each $X \in B_{\r_{Z}}(0)$ and 
\begin{equation*}
	\sigma^{-n} \int_{B_{\sigma}(0)} \sum_{j=1}^q |\epsilon^{(Z)}_j(X)|^2 dX \leq C_Z \sigma^{2\frac{k_Z}{q_{Z}}+2\mu_Z}
\end{equation*}
for all $\sigma \in (0,\rho_Z)$ and some constants $\mu_Z, C_Z > 0$ independent of $\sigma$.
In case $n=2$, ${\mathcal B}_{u, q}$ consists of isolated points and the above conclusions (with $Q_{Z}$ equal to the identity) hold for every $Z \in {\mathcal B}_{u, q}.$}

It follows from the work of Micallef and White~\cite[Theorem 3.2]{MicWhi94} that the constants $c_{l}^{(Z)}$ in the conclusion of Theorem~A satisfy $c^{(Z)}_l \cdot c^{(Z)}_l = 0$ whenever $q_Z \neq 1.$ (See also the discussion in Remark~\ref{example} below). 

If $u_{f} \, : \, \Omega \to {\mathcal A}_{q}({\mathbb R}^{m})$ is defined by $u_{f}(X) = \sum_{j=1}^{q} \llbracket u_{j}(X) - h(X)\rrbracket$, then  $u_{f}$ is Dirichlet energy minimizing (\cite[Theorem 2.6]{Almgren}), and at any point $Z$ where the conclusions of Theorem~A hold, the integers $k_{Z}$, $q_{Z}$ in the theorem are determined by ${\mathcal N}_{u_{f}}(Z) = k_{Z}/q_{Z}$. 

Since $u$ is continuous (\cite[Theorem 2.13]{Almgren}), corresponding to each $Y \in \Omega \setminus {\mathcal B}_{u, q}$ there are a $\t_{Y} >0$, an integer $M_{Y} \geq 2$ and positive integers $p_{Y}^{(1)}, \ldots, p_{Y}^{(M_{Y})}$ with $\sum_{\ell=1}^{M_{Y}} p_{Y}^{(l)} = q$ 
such that $\left. u\right|_{B_{\t_{Y}}(Y)} = \sum_{\ell=1}^{M_{Y}} u_{\ell}$ where $u_{\ell} \, : \, B_{\t_{Y}}(Y) \to {\mathcal A}_{p_{Y}^{(\ell)}}({\mathbb R}^{m})$ is Dirichlet energy minimizing for each $\ell \in \{1, \ldots, M_{Y}\}$. Thus, Theorem~A and induction on $q$ immediately imply the following assertion concerning the entire branch set ${\mathcal B}_{u}$:    

\noindent \textbf{Corollary.} \emph{Let $q \geq 2$ and let $u : \Omega \rightarrow \mathcal{A}_q(\mathbb{R}^m)$ be a locally $W^{1,2}$ locally Dirichlet energy minimizing function on an open subset $\Omega$ of $\mathbb{R}^n$.  
For $\mathcal{H}^{n-2}$-a.e. point $Z \in \mathcal{B}_{u}$ there exist a number $\rho_{Z} >0$; an integer $M_{Z} \geq 1$;  positive integers $p_{Z}^{(1)}, \ldots, p_{Z}^{(M_{Z})}$ not all equal to 1 and with $\sum_{\ell=1}^{M_{Z}} p_{Z}^{(\ell)} = q$; positive integers 
$N_{Z}^{(1)}, \ldots, N_{Z}^{(M_{Z})}$;  relatively prime pairs of positive integers $(k_{Z}^{(1)}, q_{Z}^{(1)}), \ldots, (k_{Z}^{(M_{Z})}, q_{Z}^{(M_{Z})})$ with $N_{Z}^{(\ell)}q_{Z}^{(\ell)} \leq p_{Z}^{(\ell)}$ for each $\ell \in \{1, \ldots, M_{Z}\}$ and orthogonal rotations 
$Q_Z^{(1)}, \ldots, Q_{Z}^{(M_{Z})}$ of $\mathbb{R}^n$ such that for each $X \in B_{\rho_Z}(0)$,
\begin{equation*}
	u(Z + X) = \sum_{\ell=1}^{M_{Z}}\sum_{j=1}^{p_{Z}^{(\ell)}} \llbracket h_{Z}^{(\ell)}(Z +X) + \varphi^{(Z, \ell)}_{j}(Q_{Z}^{(\ell)}X) + \epsilon^{(Z, \ell)}_{j}(X) \rrbracket 
\end{equation*} 
where for each $\ell \in \{1, \ldots, M_{Z}\}$: 
\begin{itemize}
\item[(i)] $h_{Z}^{(\ell)} \, : \, B_{\rho_{Z}}(Z)  \to {\mathbb R}^{m}$ is a (single-valued) harmonic function 
\item[(ii)] either $p_{Z}^{\ell} = 1$ and $\varphi_{1}^{(Z, \ell)} \equiv 0,$ or $p_{Z}^{\ell} \geq 2$ and $\varphi^{(Z, \ell)} : \mathbb{R}^n \rightarrow \mathcal{A}_{p_{Z}^{(\ell)}}(\mathbb{R}^m)$ is a non-zero, homogeneous, average-free, Dirichlet energy minimizing function given by  
\begin{equation*} 
	\varphi^{(Z, \ell)}(X) = \sum_{j=1}^{p_{Z}^{(\ell)}} \llbracket \varphi^{(Z, \ell)}_j(X) \rrbracket 
	= (p_{Z}^{(\ell)} - N_Z^{(\ell)} q_Z^{(\ell)}) \llbracket 0 \rrbracket + \sum_{r=1}^{N_Z^{(\ell)}} \op{Re}(c^{(Z, \ell)}_r (x_1+ix_2)^{k_Z^{(\ell)}/q_Z^{(\ell)}}) 
\end{equation*}
for some $c^{(Z, \ell)}_r \in \mathbb{C}^m \setminus \{0\}$, where $X = (x_1, x_2, \ldots, x_n),$ and 
\item[(iii)] 
$\sum_{j=1}^{p_{Z}^{(\ell)}} \epsilon^{(Z)}_j(X) = 0$ at each point $X \in B_{\r_{Z}}(0)$ and 
\begin{equation*}
	\sigma^{-n} \int_{B_{\sigma}(0)} \sum_{j=1}^{p_{Z}^{(\ell)}} |\epsilon^{(Z, \ell)}_j(X)|^2 dX \leq C_Z \sigma^{2\frac{k_Z^{(\ell)}}{q_{Z}^{(\ell)}}+2\mu_{Z}}
\end{equation*}
for all $\sigma \in (0,\rho_Z)$ and some constants $\mu_Z, C_Z > 0$ independent of $\sigma$. 
\end{itemize}
In case $n=2$, the above conclusions (with $Q_{Z}^{(\ell)}$ equal to the identity for each $\ell$) hold for every $Z \in {\mathcal B}_{u}.$}


Our estimates leading to Theorem~A give also the following structure result for the branch set:  

\noindent \textbf{Theorem B.} \emph{
Let $u$ be a $q$-valued locally $W^{1,2}$ Dirichlet energy minimizing function on an open subset $\Omega$ of $\mathbb{R}^n$.  
For each closed ball $B \subset \Omega$, either $\mathcal{B}_u \cap B = \emptyset$ or ${\mathcal H}^{n-2} \, ({\mathcal B}_{u} \cap B) >0$ and $\mathcal{B}_u \cap B$ is a countably $(n-2)$-rectifiable set. In fact if we let $S_{0} = \emptyset$ and for each $k = 1,2,\ldots,q-1$, $S_k$ be the set of points $X \in \mathcal{B}_u$ such that $u(X) = \sum_{j=1}^k q_j \llbracket u_j(X) \rrbracket$ for precisely $k$ distinct values $u_j(X) \in \mathbb{R}^m$ and some positive integers $q_j$ with $\sum_{j=1}^k q_j = q$ (so $S_{1} = {\mathcal B}_{u, q}$), then the following hold:
\begin{itemize}
\item[(i)] $\mathcal{B}_u = S_1 \cup S_2 \cup \cdots \cup S_{q-1}$; 
\item [(ii)] for each $k \in \{1, \ldots, q\}$, $\bigcup_{0 \leq l \leq  k-1} S_l$ is a relatively closed subset of $\Omega$  and for each closed ball $B \subset \Omega \setminus \bigcup_{0 \leq l \leq  k-1} S_l$, $S_k \cap B = \cup_{j=1}^{N_{k}} L_{j}$ for some finite number $N_{k}$ and pairwise disjoint, locally compact sets $L_{1}, \ldots, L_{N_{k}}$ with $L_{j}$ 
locally $(n-2)$-rectifiable (having finite $(n-2)$-dimensional Hausdorff measure locally in a neighborhood of each point of $L_{j}$) for each $j \in \{1, \ldots, N_{k}\}$.
\end{itemize}}

In case $q=2$, much more is true near any point $Z \in {\mathcal B}_{u} = {\mathcal B}_{u, 2}$ where the expansion of $u(Z + (\cdot))$ as in Theorem~A is valid with $k_{Z} = 1$ (and $q_{Z} = 2$). The specific result we obtain in this case is the following, which was proven in \cite{KrumWic1} but we include its statement here for completeness:

\noindent 
{\bf Theorem C.} \emph{Let $u : \Omega \rightarrow \mathcal{A}_2(\mathbb{R}^m)$ be a (two-valued) locally $W^{1, 2}$, locally Dirichlet energy minimizing function. If $Z_{0} \in {\mathcal B}_{u} = {\mathcal B}_{u, 2}$ and the asymptotic expansion of $u(Z_{0}+ (\cdot))$ as in Theorem~A holds with $k_{Z_{0}} = 1$, then the asymptotic expansion as in Theorem~A is valid for \emph{each} $Z \in {\mathcal B}_{u} \cap B_{\r_{Z_{0}}/2}(Z_{0})$ with $k_{Z} = 1$, $C_{Z} = C_{Z_{0}}$, $\mu_{Z} = \mu_{Z_{0}}$ and $\r_{Z} = \r_{Z_{0}}/4$ (notation as in Theorem~A). Furthermore, letting $\e^{Z}$ be as in Theorem~A, we have in this case that
$$\sup_{B_{\s}(0)} |\e^{Z}|^{2} \leq C_{0} \, \s^{1 + \frac{\mu_{Z_{0}}}{n}}$$
for each $Z \in {\mathcal B}_{u} \cap B_{\r_{Z_{0}}/2}(Z_{0})$ and $\s \in (0, \r_{Z_{0}}/4)$, 
where  $C_{0}$ is independent of $\s$ and $Z$; we also have  that  
${\mathcal B}_{u} \cap B_{\r_{Z_{0}}/2}(Z_{0})$ is an $(n-2)$-dimensional $C^{1, \alpha}$ submanifold for some $\alpha = \alpha_{Z_{0}} \in (0, 1)$.}  

\emph{In fact, the asymptotic expansion as in Theorem~A  is valid at any point $Z_{0} \in {\mathcal B}_{u}$ at which one tangent function of $u_{f} = u - h$, after composing with an orthogonal rotation of ${\mathbb R}^{n}$,  has the form $\{\pm {\rm Re} \, \left(c(x_{1} + ix_{2})^{1/2}\right)\}$ with $c \in {\mathbb C}^{m} \setminus \{0\}$ a constant, or equivalently,  at any point $Z_{0} \in {\mathcal B}_{u}$ at which the frequency ${\mathcal N}_{u_{f}}(Z_{0})$ of $u_{f}$ is $1/2$}.

\subsection{Other related work on multi-valued Dirichlet minimizing functions}\label{other}
In case $n=2$, the conclusions of Theorem~A and Corollary above have previously been established by De Lellis and Spadaro (\cite{DeLSpa11}) by a different argument. 


In very recent work, De~Lellis, Marchese, Spadaro and Valtorta (\cite{DMSV}) have shown, by an entirely different method, Theorem~B with an important additional local uniform $(n-2)$-dimensional Minkowski content estimate on $S_{1} = {\mathcal B}_{u, q}$ and hence in particular finiteness of the $(n-2)$-dimensional Hausdorff measure of ${\mathcal B}_{u, q} \cap K$ for every compact $K \subset \Omega$;  note that this is a stronger conclusion than that given by Theorem~B(ii) above for $k=1$. They show this by employing a far reaching general method developed recently by A.~Naber and D.~Valtorta to study singular sets in various settings (\cite{NV16}, \cite{NV}). The Naber--Valtorta method does not require, nor does it seem to give in general, knowledge of the asymptotic behavior of the objects near their singularities or uniqueness of their tangents, but it gives under remarkably general hypotheses rectifiability and a  local uniform estimate on the Minkowski content (of the dimension relevant to the problem) of  the singular set.

It follows from the work of the first author (\cite{Krum}) that under the hypotheses of Theorem~C, ${\mathcal B}_{u} \cap B_{\r_{Z_{0}}/2}(Z_{0})$ is in fact real analytic. 

The  special case $q=2$ of the present work is contained in our previous work \cite{KrumWic1}. 
  



\section{An outline of the method and a guide to the article}\label{outline}  We adapt the method developed by L.\ Simon in his seminal work \cite{Sim93} on the asymptotic behaviour of minimal submanifolds near singularities. In \cite{Sim93},  a ``multiplicity 1 class'' ${\mathcal M}$ of $n$-dimensional minimal submanifolds $M$  is considered, with each $M \in {\mathcal M}$ 
assumed stationary for the $n$-dimensional area functional (i.e.\ the $n$-dimensional Hausdorff measure) in some open subset $U_{M} \subset {\mathbb R}^{n+m}$, and with the class ${\mathcal M}$ assumed to be closed under ambient rigid motions, homotheties and taking measure-theoretic (i.e.\ varifold) limits of its elements (see \cite{Sim93}, 1.3(a),(b) for the precise meaning of this).
In particular, the occurrence of tangent planes (or more general tangent cones) of multiplicity $> 1$, and hence also branch point singularities,  in elements of ${\mathcal M}$ is ruled out a priori. This is the key difference between \cite{Sim93} and our setting here. Considerable additional difficulties arise in our setting because of the presence of multiplicity $> 1$, branch points and low regularity resulting from the occurrence of branch points. Apart from these issues, difficulties arise also because of the differences in the variational structure of the two problems.  These difficulties are overcome by developing additional estimates, techniques and  replacements for PDE arguments of \cite{Sim93}, which include:  (i) two new estimates for $q$-valued Dirichlet minimizers based on  first variation and energy comparison considerations (Lemma~\ref{fourierest_lemma} and Lemma~\ref{competitor_lemma}); (ii) a new approach, necessitated by the presence of branch points of varying multiplicity, to the classification of homogeneous ($q$-valued) blow-ups $w$ of sequences of $q$-valued energy minimizers converging to a non-zero cylindrical $q$-valued function $\varphi^{(0)}$ (Lemma~\ref{homogrep_lemma})---an approach based on the two variational estimates referred to in (i) together with establishing monotonicity of the frequency function associated with $w - L$ (which we need, and show, only when $w$ is homogeneous), where $L$ is a certain function linear in the variable along the axis of $\varphi^{(0)}$; (iii) versions of various general techniques from \cite{Wic14}, used at a number of places to handle other issues arising from higher multiplicity.  

Below we shall outline the main points of the overall method, keeping the discussion as non-technical as possible and highlighting only the key estimates. We intend this discussion also as a guide to the article, and shall provide along the way references to the sections of the article where various assertions are proved. We start with a very brief description of the main ingredient of the method in its original form found in \cite{Sim93}. A reader unfamiliar with \cite{Sim93} may wish to consult our earlier paper \cite{KrumWic1} which considers the special case $q=2$ of the present work. In that case, while there are some difficulties arising from low regularity (i.e.\ the possibility of a tangent function to a minimizer having degree of homogeneity $< 1$) and from the differences in the variational structure of the problem, many of the issues that arise from higher multiplicity and branching in the general case are absent, and the argument is closer to that of \cite{Sim93}. 


The key step in Simon's work is an excess improvement lemma (\cite[Lemma 1]{Sim93}). To describe this lemma, let ${\mathbf C}^{(0)}  = {\mathbf C}^{(0)}_{0} \times L_{0}$ be a given ``cylindrical'' cone in ${\mathcal M}$, i.e.\ a cone with $L_{0}$ a subspace of dimension $k$ equal to the maximal dimension of the subspace of translation invariance for any singular cone in ${\mathcal M}$ (so that, importantly, the cross-section ${\mathbf C}^{(0)}_{0}$ has only an isolated singularity). Let us assume that $k  = n-2$, which is the case relevant to the present work (though \cite{Sim93} allows arbitrary $k \in \{0, 1, \ldots, n-1\}$ with an additional integrability hypothesis on the cylindrical cones in ${\mathcal M}$ in case $k \leq n-3$; in case $k \in \{n-1, n-2\}$ this integrability hypothesis is automatically satisfied since in this case the cross section of any cylindrical cone in ${\mathcal M}$ is either a union of rays in ${\mathbb R}^{m+1}$ (if $k= n-1$) or a union of two-dimensional planes in ${\mathbb R}^{m+2}$ (if $k=n-2$) meeting only at the origin).   \cite[Lemma 1]{Sim93} says that whenever ${\mathbf C}$, a stationary cylindrical cone, and $M \in {\mathcal M}$ are  sufficiently close (depending on ${\mathbf C}^{(0)}$) in $L^{2}$-distance and in mass  to ${\mathbf C}^{(0)}$ at scale 1 and $M$ has ``enough'' singularities with density at least the density of ${\mathbf C}$ at the origin (in a certain precise sense), there is a stationary cylindrical cone ${\mathbf C}^{(1)}$  and a fixed smaller scale $\theta \in (0, 1/2)$ such that  the $L^{2}$ distance (height excess) of $M$ to ${\mathbf C}^{(1)}$ at scale $\theta$, namely $E_{\theta}({\mathbf C}^{(1)}) = \sqrt{\theta^{-n-2}\int_{M\cap B_{\theta}} {\rm dist}^{2} \, (X,  {\mathbf C}^{(1)}) \,d{\mathcal H}^{n}(X)},$  is at most half of $E_{1}({\mathbf C})$, the $L^{2}$ distance of $M$ to ${\mathbf C}$ at scale 1. 

The proof of this lemma in \cite{Sim93} proceeds by considering a sequence $M_{j} \in {\mathcal M}$ and a sequence of stationary cylindrical cones ${\mathbf C}_{j},$ with both sequences converging to ${\mathbf C}^{(0)}$. The main steps of the proof are as follows: 
\begin{itemize}
\item [(i)] The multiplicity 1 hypothesis on ${\mathcal M}$ is used to apply Allard's regularity theorem to $M_{j}$ in the region outside a fixed but arbitrarily small neighborhood of the axis of ${\mathbf C}^{(0)}$ whenever the height-excess $E_{j} = \sqrt{\int_{M_{j} \cap B_{1}} {\rm dist}^{2} \, (X,  {\mathbf C}_{j}) d{\mathcal H}^{n}}$  is sufficiently small (i.e.\ $j$ is sufficiently large). This gives the (vector-valued) height of $M_{j}$ off ${\mathbf C}_{j}$ in this region as a function $u_{j}$ over ${\mathbf C}^{(0)}$  (taking values in $\left({\rm reg} \, {\mathbf C}^{(0)}\right)^{\perp}$) satisfying the minimal surface system and hence estimates on its derivatives. These estimates allow $u_{j}$ to be blown-up by $E_{j}$ producing a smooth function (a blow-up) $v \, : \, {\rm reg} \, {\mathbf C}^{(0)} \cap B_{1} \to \left({\rm reg} \, {\mathbf C}^{(0)}\right)^{\perp}$ with $E_{j}^{-1}u_{j} \to v$ (after passing to a subsequence) smoothly away from the axis of ${\mathbf C}^{(0)}$ and with $v$ harmonic over ${\rm reg} \, {\mathbf C}^{(0)} \cap B_{1}$. Note the key fact used here that $M_{j}$ for sufficiently large $j$ have no singularities away from any fixed  small neighborhood of the axis of ${\mathbf C}^{(0)}$, which follows from Allard's theorem in view of the multiplicity 1 hypothesis. 

\item [(ii)] Making use of the assumption that $M_{j}$ has enough singularities with density at least the density of ${\mathbf C}_{j}$ (which must concentrate near the axis of ${\mathbf C}_{j}$), the monotonicity formula is used to establish a number of fundamental integral estimates for $M_{j}$ including an estimate showing that $E_{j}$ does not concentrate near the axis of ${\mathbf C}_{j}$. 

\item [(iii)] The estimates in step (ii) imply that   
$v \in L^{2}_{\rm loc}\left({\mathbf C}^{(0)} \cap B_{1}\right)$, the convergence $E_{j}^{-1}u_{j} \to v$ is in $L^{2}_{\rm loc}\left({\mathbf C}^{(0)} \cap B_{1}\right)$ and hence that $v$ inherits the integral estimates corresponding to those in step (ii). These estimates (for $v$), Fourier analysis and PDE arguments are used in \cite{Sim93} to show that the homogeneous degree 1 blow-ups are cylindrical, and that a general blow-up $v$ upon rescaling at the origin decays in $L^{2}$ to a unique homogeneous degree 1 blow-up at a fixed exponential rate. 

\item [(iv)] In view of the  excess non-concentration estimate of step (ii), the decay estimate in step (iii) for the blow-ups leads directly to the desired excess improvement conclusion for $M_{j}$ for all sufficiently large $j$.   
\end{itemize}

Here we prove an analog of this lemma, Lemma~\ref{main_lemma} below, for the class (corresponding to ${\mathcal M}$) of average-free locally Dirichlet energy 
minimizing functions $u \, : \, B_{1}^{n}(0) \to {\mathcal A}_{q}({\mathbb R}^{m})$. The ``cylindrical cones'' in our context are the 
$q$-valued homogeneous functions on ${\mathbb R}^{n}$ of the specific form 
$$\varphi(X) = (q - Nq_{0}) \llbracket 0 \rrbracket + \sum_{j=1}^{N} \varphi_{j}(QX)$$ for some positive integers $N$, $q_{0}$ with $Nq_{0} \leq q$, an orthogonal transformation $Q \, : \, {\mathbb R}^{n} \to {\mathbb R}^{n}$ and for $\varphi_{j} \, : \, {\mathbb R}^{n} \to {\mathcal A}_{q_{0}}({\mathbb R}^{m})$ of the form $\varphi_j(X) = \op{Re}(c_j (x_1+ix_2)^{k/q_{0}})$
where $X = (x_{1}, x_{2}, y)$, $i = \sqrt{-1}$, $c_j \in \mathbb{C}^m \setminus \{0\}$ and $k$ is a positive integer relatively prime to $q_{0}$ (independent of $j$). (See Definitions~\ref{varphi0_defn}--\ref{Phi_defn} below.) Let us refer, in the rest of this discussion, to such a $\varphi$ as a \emph{cylindrical $q$-valued function}. 

In  Lemma~\ref{main_lemma}, we fix a non-zero locally Dirichlet energy minimizing cylindrical $q$-valued function 
$$\varphi^{(0)} = \sum_{j=1}^{J} m_{j}\varphi_{j}^{(0)}$$ 
where $\varphi_{j}^{(0)} = \op{Re}(c_j^{(0)} (x_1+ix_2)^{k/q_{0}})$, $c_{j}^{(0)} \in {\mathbb C}^{m}$ are distinct, and we allow the possibility that $\varphi^{(0)}_{j} \equiv 0$ for precisely one value of $j$. Let the degree of homogeneity of $\varphi^{(0)}$ be $\alpha$, so that $\alpha = k_{0}/q_{0}$ for relatively prime positive integers  $k_{0}, q_{0}$ with $q_{0} \leq q.$  In the context of Lemma~\ref{main_lemma}, the singular set of an average-free Dirichlet energy minimizing $q$-valued function $u$ is taken to be its zero set ${\mathcal B}_{u, q} = \{Z \in B_{1} \, : \, u(Z) = q\llbracket 0 \rrbracket\}$ (although of course ultimately the full singular set of $u$ is $\Sigma_{u}$ as defined in section~\ref{Alm} above). The frequency ${\mathcal N}_{u}(Z)$ plays the role of minimal surface density, and the improvement-of-excess assertion of the lemma is proved subject the the assumption that there are enough points $Z \in {\mathcal B}_{u, q}$ with ${\mathcal N}_{u}(Z) \geq \alpha.$ 


To prove  Lemma~\ref{main_lemma}, we consider a sequence of average-free, Dirichlet energy minimizing $q$-valued functions $u^{(\nu)}$ on $B^{n}_{1}(0)$, $\nu = 1, 2, 3, \ldots,$ and a sequence of average-free, homogeneous degree $\alpha$, $q$-valued cylindrical functions $\varphi^{(\nu)}$ on ${\mathbb R}^{n}$ (not assumed to be Dirichlet energy minimizing), with both sequences converging in $L^{2}(B_{1})$ to $\varphi^{(0)}$.  Let us outline the proof of Lemma~\ref{main_lemma} in four steps that correspond to steps (i)-(iv) above describing the work \cite{Sim93}. 

\noindent
{\bf Step 1.} We parameterize (in Section~\ref{sec:graphical_sec}) $u^{(\nu)}$ via multi-valued functions $v^{(\nu)}_{j, k}$ over the graphs of the ``components'' $\varphi^{(\nu)}_{j, k}$ (see Definition~\ref{varphi_defn}) of $\varphi^{(\nu)}$ away from a fixed arbitrarily small neighborhood of the axis $\{0\} \times {\mathbb R}^{n-2}$ of $\varphi^{(0)}$. This is analogous to step (i) above in  the approach of \cite{Sim93}. However, in the case $q \geq 3$, unlike in \cite{Sim93} or in \cite{KrumWic1}, the smallness of the $L^{2}$ excess 
$$E_{\nu} = \sqrt{\int_{B^{n}_{1}(0)} {\mathcal G}\, (u^{(\nu)}, \varphi^{(\nu)})^{2}}$$ (notation as in Section~\ref{sec:preliminaries_sec}) is not sufficient to guarantee a good parameterization of $u^{(\nu)}$ of this type.  This issue arises from the presence of higher multiplicity and branching away from the axis $\{0\} \times {\mathbb R}^{n-2}$ of $\varphi^{(0)}$, and it is overcome following an idea from~\cite{Wic14} which goes as follows:  we show (in Corollary \ref{graphical_cor} below) that such a parameterization is possible if we impose an additional condition, namely, Hypothesis~$(\star\star)$ in Section~\ref{sec:graphical_sec}, which says that $u^{(\nu)}$ is substantially closer in $L^{2}$ to $\varphi^{(\nu)}$ than it is to any other homogeneous degree $\alpha$, cylindrical $q$-valued function $\varphi$ with fewer components. Having to make this assumption has the  consequence that excess improvement of $u^{(\nu)}$ at a smaller scale can (eventually) only be established subject to $u^{(\nu)},$ $\varphi^{(\nu)}$ satisfying Hypothesis~$(\star\star)$, and indeed this is what we do initially (in Lemma~\ref{premain_lemma}).  However Hypothesis~$(\star\star)$ is undesirable for the purpose of iteration of this lemma since it is difficult to verify at a smaller scale. This issue is resolved by using Lemma~\ref{premain_lemma} itself to remove, in Lemma~\ref{main_lemma}, Hypothesis~$(\star\star)$  at the expense of weakening the conclusion to allow excess improvement to be guaranteed at one of a fixed, finite number of scales (as opposed to a single scale). Allowing such multiple scales is no obstacle for iteration of the lemma to reach  uniqueness and asymptotic decay conclusions.  

We note that in case $q \geq 3$, due to the presence of branch points of $u^{(\nu)}$ of multiplicity $< q$, the functions $v^{(\nu)}_{j,k}$ parameterizing $u^{(\nu)}$ necessarily have to be multi-valued. In particular (unlike in \cite{Sim93} or  in \cite{KrumWic1}), $v^{(\nu)}_{j, k}$ do not solve a PDE away from the axis of $\varphi^{(0)}$. It is however not difficult to see that the $v^{(\nu)}_{j, k}$, and hence also $E_{\nu}^{-1}v^{(\nu)}_{j, k}$, are locally energy minimizing (as functions on a domain in 
${\mathbb R}^{n}$) away from the axis of $\varphi^{(0)}$. This property makes Almgren's continuity and compactness results applicable to the  
$v^{(\nu)}_{j, k}$ to produce (passing to a subsequence) a blow-up 
$$w_{j, k} = \lim_{\nu \to \infty} \, E_{\nu}^{-1} v^{(\nu)}_{j, k},$$ 
where the convergence is uniform with energies of $E_{\nu}^{-1}v_{j, k}^{(\nu)}$ converging to the energy of $w_{j,k}$ locally away from the axis of $\varphi^{(0)}$ and away from $\partial \, B_{1}^{n}(0)$. 




\noindent
{\bf Step 2.} This step comprises the integral estimates established in Sections~\ref{sec:L2estimates_sec1} and ~\ref{sec:L2estimates_sec2}, and corresponds to step (ii) above in the approach of \cite{Sim93}. Among these estimates are the new results Lemma~\ref{fourierest_lemma} and Lemma~\ref{competitor_lemma} which are consequences of the first variation identity \eqref{freqidentity2} and a comparison function construction respectively. We shall discuss their role below (in Step 3). 
The other main results in this step, namely Theorem~\ref{keyest_thm}, Lemmas~\ref{keyest_cor}-\ref{nonconest_lemma} and Corollary~\ref{nonconest_cor} all establish estimates analogues to those in \cite{Sim93}. Their proofs however require additional ideas, and in particular rely  
also on general techniques developed in \cite{Wic14} to handle the difficulties arising from the presence of higher multiplicity. 

Theorem~\ref{keyest_thm} gives the basic estimate. It says that if $\varphi^{(0)}$ is as above with its degree of homogeneity $\alpha$, $u$ is a $q$-valued average-free Dirichlet energy minimizing function on $B_1(0)$ with ${\mathcal N}_{u}(0) \geq \alpha$ and $\varphi$ is a homogeneous degree $\alpha$ cylindrical $q$-valued function, and if $u$ and $\varphi$ are both sufficiently close to $\varphi^{(0)}$ in $L^2(B_1(0))$ (depending on $\varphi^{(0)}$),  then   
\begin{equation}\label{basic}
	\int_{B_{1/2}(0)} R^{2-n}\left| \frac{\partial (u/R^{\alpha})}{\partial R} \right|^2 
	\leq C \int_{B_1(0)} \mathcal{G}(u,\varphi)^2, 
\end{equation}
where $R(X) = |X|$ for $X \in {\mathbb R}^{n}$ and $C = C(n,m,\varphi^{(0)}) \in (0,\infty)$ is a constant. This is a consequence of the variational identities \eqref{freqidentity1} and \eqref{freqidentity2} satisfied by $u$ and $\varphi^{(0)}$. In particular, it uses Lemma~\ref{keyest_identity} giving the identity 
\begin{equation*}
	\frac{d}{d\rho} \left( \rho^{-2\alpha} (D_{u,Y}(\rho) - \alpha H_{u,Y}(\rho)) \right) 
	= 2\rho^{2-n} \int_{\partial B_{\rho}(Y)} \left| \frac{\partial (u/R_Y^{\alpha})}{\partial R_Y} \right|^2 
\end{equation*}
for all $Y \in B_{1}(0)$ and $\rho \in (0,\op{dist}(Y, \partial \, B_{1}(0))$, where $R_Y(X) = |X-Y|$ for $X \in \mathbb{R}^n$, $D_{u,Y}(\rho) = \rho^{2-n} \int_{B_{\rho}(Y)} |Du|^2$, and $H_{u,Y}(\rho) = \rho^{1-n} \int_{\partial B_{\rho}(Y)} |u|^2$. This is a variant of Almgren's frequency function monotonicity formula \eqref{freq_derivative} and is derived from the identities \eqref{freqidentity1} and \eqref{freqidentity2}.  (Although this identity appears not to have been used in the study of branch points prior to our work \cite{KrumWic1}, we have learnt since then that it has previously been used---for purposes unlike ours---in the literature on free boundary problems where it is known as the Weiss monotonicity formula and was first used by G.~Weiss (\cite{Wei}).)

The above basic estimate \eqref{basic} implies all of the estimates in Section~\ref{sec:L2estimates_sec2}. In particular, Corollary~\ref{nonconest_cor} is a key consequence of it, which says that the excess $E_{\nu}$ does not concentrate near the axis of $\varphi^{(0)}$ whenever there is no large gap in the set $\{Z \, : \, {\mathcal N}_{u}(Z) \geq \alpha\} \cap B_{1}(0)$. Also among the consequences of \eqref{basic} is Lemma~\ref{nonconest_lemma} giving the estimate 
\begin{equation}\label{basic_cor}
|\xi|^{2} + \int_{B_{1/2}(0) \cap \{r > \tau\}} \frac{\mathcal{G}(u(X), \varphi(X) - D_x \varphi(X) \cdot \xi)^2}{|X-Z|^{n+2\alpha-\sigma}} 
	\leq C \int_{B_1(0)} \mathcal{G}(u,\varphi)^2,
\end{equation}	
valid  for any  $\t \in (0, 1/2)$, $\sigma \in (0, 1/2q)$ and some fixed constant $C$ (depending on $n, m, \varphi^{(0)}, \sigma$) whenever $u$, $\varphi$ are sufficiently close to $\varphi^{(0)}$ (depending on $\t$, $\varphi^{(0)}$) and 
$Z = (\xi, \eta) \in {\mathbb R}^{2} \times {\mathbb R}^{n-2}$ is such that ${\mathcal N}_{u}(Z) \geq \alpha$. This estimate plays an important role in our proof of the classification theorem for homogeneous blow-ups (Theorem~\ref{homogrep_lemma}), discussed in Step 3 below. 

\noindent
{\bf Step 3.} This step comprises Sections~\ref{sec:blowupclass_sec}, \ref{sec:homogblowup_sec} and \ref{sec:asymptotics_sec}, and corresponds to step (iii) above in the approach of \cite{Sim93}. The main result in this step is Theorem~\ref{blowupdecay_lemma} which says that a blow-up $w = (w_{j, k})$ (as described in Step 1 above) corresponding to a sequence of energy minimizers $u^{(\nu)}$ having the property that  
\begin{equation}\label{no-large-gaps}
\{Z \, : \, {\mathcal N}_{u^{(\nu)}}(Z) \geq \alpha\} \cap B_{1}(0) \to \{0\} \times {\mathbb R}^{n-2} \cap B_{1}(0)
\end{equation} 
in Hausdorff distance as $\nu \to \infty$ must decay exponentially in $L^{2}$, upon rescaling about the origin, to a unique homogeneous degree $\alpha$ cylindrical $q$-valued function. The proof of this result is based on the uniform integral estimates for $w$ derived in Section~\ref{sec:blowupclass_sec} (by passing to the limit in the estimates in Sections~\ref{sec:L2estimates_sec1} and \ref{sec:L2estimates_sec2}), and classification of blow-ups that are homogeneous of degree $\alpha$ (Theorem~\ref{homogrep_lemma}).

In~\cite{Sim93}, the corresponding classification result (for homogeneous degree 1 blow-ups) uses a PDE argument based on the fact that the (single-valued) blow-ups 
satisfy the Jacobi field equation on ${\rm reg} \, {\mathbf C}^{(0)}$ (which in the case when ${\mathbf C}^{(0)}$ has an $(n-2)$-dimensional axis is just  Laplace's equation on the half-planes that make up ${\mathbf C}^{(0)}$). Here we need a new approach (at least for the case $q \geq 3$) since $u^{(\nu)}$ and hence also $w$ in general have branch points (of multiplicity $< q$) away from the axis of $\varphi^{(0)}$. We proceed as follows: 

First we note the direct consequence of the estimate \eqref{basic_cor} above giving that corresponding to any blow-up $w = (w_{j, k})$ arising from a sequence of minimizers $u^{(\nu)}$ satisfying \eqref{no-large-gaps}, there is a bounded function $\lambda: \, B_{1/2}(0) \cap \{0\} \times {\mathbb R}^{n-2} \to {\mathbb R}^{2}$ such that for $\sigma \in (0, 1/q)$ and each $(j, k)$, 
\begin{equation}\label{basic_cor_blowup}  
\sup_{(0, y) \in \{0\} \times {\mathbb R}^{n-2} \cap B_{1/2}(0)}\,	\int_{B_{1/2}(0)}  \frac{|w_{j,k}(X) - D_x \varphi^{(0)}_j(X) \cdot \lambda(y)|^2}{|X-(0,y)|^{n+2\alpha-\sigma}} dX \leq C 
\end{equation}
where $C = C(n,m,q,\alpha,\varphi^{(0)},\sigma) \in (0,\infty)$. Also, we use our  first variation result Lemma~\ref{fourierest_lemma}  (taken with $u = u^{(\nu)}$) to establish, still for an arbitrary blow-up $w = (w_{j,k})$ corresponding to $(u^{(\nu)})$ satisfying \eqref{no-large-gaps}, a certain mean value property (Lemma~\ref{homogrep1_mvp_lemma}) for the squared $y$-gradient of each of the first two Fourier coefficients 
(corresponding to $D_{x_{1}} \varphi^{(0)}$ and $D_{x_{2}} \varphi^{(0)}$)
of the (single valued) average of the values of $w$. This mean value property and a PDE argument is then used to show that whenever $w$ is homogeneous of degree $\alpha$, the function $\lambda$ in the estimate \eqref{basic_cor_blowup} is linear in $y$.  Consequently,  if we let ${\widetilde w} = (w_{j,k}(X)  - D_x \varphi^{(0)}_j(X) \cdot \lambda(y))$ where $X = (x,y)$, then $\widetilde{w}$ is homogeneous of degree $\alpha$ and is,  up to a multiplicative constant, the blow-up relative to $\varphi^{(\nu)}$ of the sequence $u^{(\nu)}(R_{\nu}X)$ for a suitable sequence of rotations $(R_{\nu})$ of ${\mathbb R}^{n}.$ Moreover,  it follows from \eqref{basic_cor_blowup} that  
\begin{equation}\label{L2decay}
\int_{B_{\rho}((0,y))} |{\widetilde w}|^{2} \leq C\rho^{n + 2\alpha - \sigma}
\end{equation} 
for each $(0, y) \in \{0\} \times {\mathbb R}^{n-2} \cap B_{1/4}(0)$ and $\rho \in (0, 1/4)$ and hence, by virtue of local energy minimality of $\widetilde{w}$ away from $\{0\} \times {\mathbb R}^{n-2}$, also that $\int_{B_{1}(0)} |D\widetilde{w}|^{2} < \infty.$ 

In the second step of the classification, we use the above $L^{2}$ decay estimate \eqref{L2decay} and our energy comparison result Lemma~\ref{competitor_lemma} (taken with $u = u^{(\nu)} \circ R_{\nu}$) to show that for each point $(0, y) \in \{0\} \times {\mathbb R}^{n-2} \cap B_{1/4}(0)$ and each function $\zeta \in C^{1}_{c}(B_{1/4}(0); {\mathbb R})$, 
\begin{equation*} \label{homogrep2_eqn4}
	\int_{B_{1/2}(0)} |D\widetilde{w}|^2 
	\leq \int_{B_{1/2}(0)} |D(\widetilde{w}(X+t \zeta(X) (X-(0,y))))|^2 
\end{equation*}
for all $t \in {\mathbb R}$ with $|t|$ small; in particular $\widetilde{w}$ is energy critical for \emph{radial} domain deformations of the type $X \mapsto X + t\zeta(X)(X - (0, y))$. (Note that $\widetilde{w}$ need not be locally energy minimizing in ${\mathbb R}^{n}$ although it is so in ${\mathbb R}^{n} \setminus \{0\} \times {\mathbb R}^{n-2}$; in fact, interestingly, even energy criticality of $\widetilde{w}$ need not hold for more general, non-radial  deformations of the type $X \mapsto X + t(\zeta^{1}(X), \ldots, \zeta^{n}(X))$,  $\zeta^{j} \in C^{1}_{c}(B_{1/4}(0)).$ See the example discussed in the remark at the end of Section~\ref{sec:homogblowup_sec}.) This allows us to establish that the Almgren frequency function 
\begin{equation*}
	\rho \mapsto N_{\widetilde{w},(0,y)}(\rho) = \frac{\rho^{2-n} \int_{B_{\rho}(0,y)} |D\widetilde{w}|^2}{\rho^{1-n} \int_{\partial B_{\rho}(0,y)} |\widetilde{w}|^2}
\end{equation*}
associated with ${\widetilde w}$ and each point $(0, y) \in \mathbb{R}^{n-2} \cap B_{1/4}(0)$ is monotonically increasing, which implies in particular that $\left( \frac{\sigma}{\rho} \right)^{2 N_{{\widetilde w}, (0, y)}(\rho)} \rho^{-n} \int_{B_{\rho}(0, y)} |\widetilde{w}|^2 \leq  \sigma^{-n} \int_{B_{\sigma}(0, y)} |{\widetilde w}|^2$
for each point $(0, y) \in \{0\} \times {\mathbb R}^{n-2} \cap B_{1/4}(0)$ and each $\sigma, \rho$ with $0 < \sigma \leq \rho \leq 1/2$. Hence by \eqref{basic_cor_blowup} we have that ${\mathcal N}_{\widetilde{w}}(0, y) = \lim_{\rho \to 0^{+}} N_{\widetilde{w}, (0, y)}(\rho) \geq \alpha$ for each point $(0, y) \in \{0\} \times {\mathbb R}^{n-2} \cap B_{1/4}(0).$ Standard arguments using again the monotonicity formula for $N_{{\widetilde w}, (0, y)} (\rho)$ then imply that $\widetilde{w}$ is translation invariant along $\{0\} \times {\mathbb R}^{n-2}$, which leads to the desired classification result for $w$. 

\noindent
{\bf Step 4.} A preliminary excess decay conclusion (Lemma~\ref{premain_lemma}) for $u^{(\nu)}$ for sufficiently large $\nu$, subject to Hypothesis~$(\star\star)$ (of section~\ref{sec:graphical_sec}) on $u^{(\nu)}$ and  $\varphi^{(\nu)}$ among other things, now follows directly from the excess non-concentration estimate of Corollary~\ref{nonconest_cor} and the decay estimate of Theorem~\ref{blowupdecay_lemma} for the blow-ups. As mentioned in Step 1, the final excess decay lemma (Lemma~\ref{main_lemma}) is readily obtained following the corresponding argument in \cite{Wic14} by using Lemma~\ref{premain_lemma} itself to weaken both its hypotheses and conclusions, namely to remove Hypothesis~$(\star\star)$ at the expense of allowing multiple scales  at one of which the excess improvement is asserted. 

Given Lemma~\ref{main_lemma}, to prove the main theorems (including Theorems A and B above)  we proceed as follows: By the work of Almgren, for every average-free, Dirichlet energy minimizing $q$-valued function $u$ and $\mathcal{H}^{n-2}$-a.e. $Z \in \mathcal{B}_{u, q}$, there exists a sequence of radii $\rho_j \downarrow 0$ such that  
\begin{equation*}
	\frac{u(Z + \rho_j X)}{\rho_{j}^{-n/2} \|u\|_{L^2(B_{\rho_j}(Z))}} \rightarrow \varphi^{(0)}
\end{equation*}
locally in $L^{2}$  for some nonzero, average-free, cylindrical, homogeneous degree $\alpha$, Dirichlet energy minimizing $q$-valued function $\varphi^{(0)}$, where $\alpha  = {\mathcal N}_{u}(Z)> 0$.  Thus after scaling enough, $u$ is close to $\varphi^{(0)}$ in $L^2(B_1(0))$. The main theorems then follow in Section~\ref{sec:finale_sec} by iteratively applying Lemma~\ref{main_lemma} considering at each stage the two alternatives it gives, in a manner completely analogous to~\cite{Sim93}.

\section{Preliminaries}\label{sec:preliminaries_sec}

\subsection{Multi-valued functions}  For each integer $q \geq 1$, $\mathcal{A}_q(\mathbb{R}^m)$ denotes the space of unordered $q$-tuples $\{a_{1}, \ldots, a_{q}\}$ of points $a_{j}\in {\mathbb R}^{m}$ (with repetition allowed); more precisely, identifying each $a_{j}$ with the Dirac mass $\llbracket a_{j}\rrbracket$ at $a_{j}$, 
$${\mathcal A}_{q}({\mathbb R}^{m}) = \left\{\sum_{j=1}^q \llbracket a_j \rrbracket \, : \, a_{1}, \ldots, a_{j} \in {\mathbb R}^{m}\right\}.$$ 
We equip $\mathcal{A}_q(\mathbb{R}^m)$ with the metric defined by 
\begin{equation*}
	\mathcal{G}(a,b) = \inf_{\sigma} \left( \sum_{j=1}^q |a_j - b_{\sigma(j)}|^2 \right)^{1/2}
\end{equation*}
for $a = \sum_{j=1}^q \llbracket a_j \rrbracket, b = \sum_{j=1}^q \llbracket b_j \rrbracket \in {\mathcal A}_{q}({\mathbb R}^{m})$, where the infimum is taken over all permutations $\sigma$ of $\{1,2,\ldots,q\}$.  We write  
\begin{equation*}
	|a| = \mathcal{G}(a, q \llbracket 0 \rrbracket) = \left( \sum_{j=1}^q |a_j|^2 \right)^{1/2}
\end{equation*}
for $a = \sum_{j=1}^q \llbracket a_j \rrbracket \in {\mathcal A}_{q}({\mathbb R}^{m})$.  We define the \textit{separation} of a point $a \in \mathcal{A}_q(\mathbb{R}^m)$ by ${\rm sep} \, a = +\infty$ if $a = q\llbracket a_{1}\rrbracket$ for some $a_{1} \in {\mathbb R}^{m}$ (which includes the case $q=1$), and  
\begin{equation} \label{separation_defn}
	\op{sep} a = \min_{i \neq j} |a_i - a_j| 
\end{equation}
if  $a = \sum_{j=1}^N q_j \llbracket a_j \rrbracket$ for some $N \geq 2,$ (uniquely defined) distinct $a_j \in \mathbb{R}^m$ and positive integers $q_j$ with $\sum_{j=1}^N q_j = q$. 

If $m_j,$ $q_{j}$ are positive integers and $a^{(j)} = \sum_{i=1}^{q_j} \llbracket a_i^{(j)} \rrbracket \in {\mathcal A}_{q_j}({\mathbb R}^{m})$ for $j = 1,2,\ldots,N$, then 
we define $\sum_{j}^{N} m_{j}a^{(j)}$ to be the point in ${\mathcal A}_{q}({\mathbb R}^{m})$ where $q = \sum_{j=1}^N m_j q_j$, given by 
$$\sum_{j=1}^{N} m_{j}a^{(j)} =\sum_{j=1}^N \sum_{i=1}^{q_j} m_j \llbracket a_i^{(j)} \rrbracket.$$
(Note that there is no canonical way to add a pair of general elements $a = \sum_{i=1}^q \llbracket a_i \rrbracket$ and $b = \sum_{i=1}^q \llbracket b_i \rrbracket$ in ${\mathcal A}_{q}({\mathbb R}^{m})$ to obtain a unique element $a + b$ in ${\mathcal A}_{q}({\mathbb R}^{m})$; an expression such as $\sum_{i=1}^q \llbracket a_i+b_i \rrbracket$ does not yield a well defined sum $a+ b$ since reordering the $a_i$'s or the $b_i$'s in this expression will yield different elements of ${\mathcal A}_{q}({\mathbb R}^{m})$.)  

If $\lambda \in {\mathbb R}$ and $a = \sum_{i=1}^q \llbracket a_i \rrbracket \in {\mathcal A}_q({\mathbb R}^{m})$, we define $\lambda a$ to be the point in ${\mathcal A}_{q}({\mathbb R}^{m})$ given by 
$$\lambda a = \sum_{i=1}^{q} \llbracket \lambda a_i \rrbracket .$$


A $q$-valued function on a set $\Omega \subseteq \mathbb{R}^n$ is a map $u : \Omega \rightarrow \mathcal{A}_q(\mathbb{R}^m)$. For $X \in \Omega$, we shall denote by $u_{j}(X)$ ($1 \leq j \leq q$) the ``$q$ values of $u(X)$'',  so that  $u(X) = \sum_{j=1}^q \llbracket u_j(X) \rrbracket$ for $X \in \Omega$.
The \textit{average} $u_a : \Omega \rightarrow \mathbb{R}^m$ of $u$ is the single-valued function defined by $u_a(X) = \frac{1}{q} \sum_{j=1}^q u_j(X)$ for $X \in \Omega$.  We say that $u$ is \textit{average-free} if $u_a(X) = 0$ for all $X \in \Omega$.  Of course we have $u(X) = \sum_{j=1}^{q} \llbracket u_a(X) + \left(u_f\right)_{j}(X)\rrbracket$ where $u_f : \Omega \rightarrow \mathcal{A}_q(\mathbb{R}^m)$  is defined by $u_f(X) = \sum_{j=1}^q \llbracket u_j(X) - u_a(X) \rrbracket$ for all $X \in \Omega$.  We shall call $u_f$ the \textit{average-free part} of $u$.  Observe that given any two $q$-valued functions $u, v : \Omega \rightarrow \mathcal{A}_q(\mathbb{R}^m)$ with respective averages $u_a, v_a$ and average-free parts $u_f, v_f$, 
\begin{equation} \label{norm_avg}
	\mathcal{G}(u(X),v(X))^2 = q |u_a(X) - v_a(X)|^2 + \mathcal{G}(u_f(X),v_f(X))^2
\end{equation}
for all $X \in \Omega$. 

If $u \, : \, \Omega \to {\mathcal A}_{q}({\mathbb R}^{m})$, then the graph of $u$ is the subset of $\Omega \times {\mathbb R}^{m}$ defined by 
${\rm graph} \, u = \{(X, u_{j}(X)) \, : \, X \in \Omega, \; 1 \leq j \leq q\}.$ 

If $m_j$ are positive integers and $u^{(j)} : \Omega \rightarrow \mathcal{A}_{q_{j}}(\mathbb{R}^m)$ for $j = 1,2,\ldots,N$, we define $$u = \sum_{j=1}^N m_j u^{(j)}$$ to be the function $u : \Omega \rightarrow \mathcal{A}_q(\mathbb{R}^m)$, where $q = \sum_{j=1}^N m_j q_j$, given by $u(X) = \sum_{j=1}^N \sum_{i=1}^{q_j} m_j \llbracket u_i^{(j)}(X) \rrbracket$ for each $X \in \Omega$, where $u^{(j)}(X) = \sum_{i=1}^{q_j} \llbracket u_i^{(j)}(X) \rrbracket$.

If $\lambda \in {\mathbb R}$ and $u \, : \, \Omega \to {\mathcal A}_{q}({\mathbb R}^{m})$, we define $\lambda u$ to be to be the function $\lambda u \, : \, \Omega \to {\mathcal A}_{q}({\mathbb R}^{m})$ given by $\lambda u(X) = \sum_{i=1}^{q} \llbracket \lambda u_i(X) \rrbracket$ for each $X \in \Omega$.

Let $\Omega \subseteq \mathbb{R}^n$ be open.  Since $\mathcal{A}_q(\mathbb{R}^m)$ is a metric space, we can define the space $C^0(\Omega;\mathcal{A}_q(\mathbb{R}^m))$ of continuous $q$-valued functions on $\Omega$ in the usual way. For each $\mu \in (0,1]$, we define $C^{0,\mu}(\Omega;\mathcal{A}_q(\mathbb{R}^m))$ to be the space of functions $u : \Omega \rightarrow \mathcal{A}_q(\mathbb{R}^m)$ such that for every $\Omega' \subset \subset \Omega$, 
\begin{equation*}
	[u]_{\mu,\Omega'} = \sup_{X,Y \in \Omega', \, X \neq Y} \frac{\mathcal{G}(u(X),u(Y))}{|X-Y|^{\mu}} < \infty. 
\end{equation*}
We say a function $u : \Omega \rightarrow \mathcal{A}_q(\mathbb{R}^m)$ is differentiable at $Y \in \Omega$ if there exists an affine function $\ell_Y : \mathbb{R}^n \rightarrow \mathcal{A}_q(\mathbb{R}^m)$, i.e. a $q$-valued function $\ell_Y$ of the form $\ell_Y(X) = \sum_{j=1}^q \llbracket A^Y_j X + b^Y_j \rrbracket$ for some $m \times n$ real-valued constant matrices $A^Y_j$ and constants $b^Y_j \in \mathbb{R}^m$, such that 
\begin{equation*}
	\lim_{X \rightarrow Y} \frac{\mathcal{G}(u(X),\ell_Y(X))}{|X-Y|} = 0. 
\end{equation*}
$\ell_Y$ is unique if it exists.  The derivative of $u$ at $Y$ is $Du(Y) = \sum_{j=1}^q \llbracket A^Y_j \rrbracket$.  We say that $u$ is strongly differentiable at $Y$ if additionally $A^Y_i = A^Y_j$ whenever $b^Y_i = b^Y_j$. 

Given $1 \leq p < \infty$, $L^p(\Omega;\mathcal{A}_q(\mathbb{R}^m))$ denotes the space of Lebesgue measurable functions $u : \Omega \rightarrow \mathcal{A}_q(\mathbb{R}^m)$ such that $\|u\|_{L^p(\Omega)} \equiv \|\mathcal{G}(u,q \llbracket 0 \rrbracket)\|_{L^p(\Omega)} < \infty$.  We equip $L^p(\Omega;\mathcal{A}_q(\mathbb{R}^m))$ with the metric  
\begin{equation*}
	d(u,v) = \left( \int_{\Omega} \mathcal{G}(u(X),v(X))^p \right)^{1/p}
\end{equation*}
for $u,v \in L^p(\Omega;\mathcal{A}_q(\mathbb{R}^m))$. 

The Sobolev space $W^{1,2}(\Omega;\mathcal{A}_q(\mathbb{R}^m))$ of $q$-valued functions is defined in~\cite[Definitions and terminology 2.1]{Almgren} as follows:  Let $N = N(m,q) \geq 1$ be an integer and fix $\bm{\xi} : \mathcal{A}_q(\mathbb{R}^{m}) \rightarrow \mathbb{R}^N$ a bi-Lipschitz embedding such that $\op{Lip}(\bm{\xi}) \leq 1$ and $\op{Lip}(\bm{\xi}^{-1} |_{\mathcal{Q}}) \leq C(m,q)$ where $\mathcal{Q} = \bm{\xi}(\mathcal{A}_q(\mathbb{R}^{m}))$, see~\cite[Theorem 1.2]{Almgren}.  $W^{1,2}(\Omega;\mathcal{A}_q(\mathbb{R}^m))$ is the space of Lebesgue measurable functions $u : \Omega \rightarrow \mathcal{A}_q(\mathbb{R}^m)$ such that $\bm{\xi} \circ u \in W^{1,2}(\Omega;\mathbb{R}^N)$.  (See~\cite[Definition 0.5]{DeLSpa11} for an equivalent characterization of $W^{1,2}(\Omega;\mathcal{A}_q(\mathbb{R}^m))$ in terms of Sobolev functions taking values in a metric space.) 

Every $u \in W^{1,2}(\Omega;\mathcal{A}_q(\mathbb{R}^{m}))$ is approximately strongly differentiable at $\mathcal{L}^n$-a.e.~$Y \in \Omega$ in the sense that there exists a Lebesgue measurable subset $\Omega_Y \subseteq \Omega$ such that $\Omega_Y$ has density one at $Y$ and $u |_{\Omega_Y}$ is strongly differentiable at $Y$.  The derivative of $u$ at $Y$ is $Du(Y) = D(u |_{\Omega_Y})(Y)$.  Whenever $u \in W^{1,2}(\Omega;\mathcal{A}_q(\mathbb{R}^{m}))$, $u \in L^{2}(\Omega;\mathcal{A}_q(\mathbb{R}^{m}))$ and $Du \in L^{2}(\Omega;\mathcal{A}_q(\mathbb{R}^{m \times n}))$.

\subsection{Multi-valued Dirichlet energy minimizing functions} 

\begin{definition} Let $\Omega \subset {\mathbb R}^{n}$ be open. A $q$-valued function $u \in W^{1,2}(\Omega;\mathcal{A}_q(\mathbb{R}^m))$ is said to be locally Dirichlet energy minimizing in $\Omega$ (or more simply locally energy minimizing) if
\begin{equation*}
	\int_{\Omega} |Du|^2 \leq \int_{\Omega} |Dv|^2 
\end{equation*}
for any open ball $B$ with $\overline{B} \subset \Omega$ and any $v \in W^{1,2}(\Omega;\mathcal{A}_q(\mathbb{R}^m))$ with $v = u$ a.e. on $\Omega \setminus B.$ 
\end{definition}

Almgren developed an existence and regularity theory for multivalued energy minimizing functions as an essential ingredient in his fundamental work~\cite{Almgren} on interior regularity of area minimizing rectifiable currents of arbitrary dimension and codimension $\geq 2$. Almgren's theory, contained in~\cite[Chapter 2]{Almgren}, in particular establishes the following results (see~\cite{DeLSpa11} for a nice, concise exposition of Almgren's work on Dirichlet minimizers as well as for an alternative, ``intrinsic'' viewpoint of the theory):

\begin{enumerate}
\item[(i)]~\cite[Theorem 2.2]{Almgren} (c.f.~\cite[Theorem 0.8]{DeLSpa11}), which asserts, for a bounded $C^1$ domain $\Omega \subset {\mathbb R}^{n}$, the existence in the Sobolev space $W^{1,2}(\Omega, {\mathcal A}_{q}({\mathbb R}^{m}))$ of a $q$-valued energy minimizing function with prescribed $q$-valued boundary data.

\item[(ii)]~\cite[Theorem 2.13]{Almgren} (c.f.~\cite[Theorems 0.9 and 3.9]{DeLSpa11}), according to which a $q$-valued locally energy minimizer is (${\mathcal L}^{n}$ a.e.\ equal to) a locally uniformly H\"{o}lder continuous function in the interior of its domain. More specifically, there exists $\mu = \mu(n,m,q) \in (0,1)$ such that if $u \in W^{1,2}(\Omega;\mathcal{A}_q(\mathbb{R}^m))$ is locally energy minimizing, then there exists $\widetilde{u} \, : \, \Omega \to {\mathcal A}_{q}({\mathbb R}^{m})$ with $\widetilde{u}\res\Omega^{\prime} \in C^{0,\mu}(\Omega';\mathcal{A}_q(\mathbb{R}^m))$ for each open set $\Omega' \subset \subset \Omega$ such that $u(X)= \widetilde{u}(X)$ 
for ${\mathcal L}^{n}$ a.e.\  $X \in \Omega.$ Henceforth we shall identify $u$ with $\widetilde{u}$ and simply write $u$ in place of $\widetilde{u}$. 

For a local Dirichlet minimizer $u$ as above, \cite[Theorem 2.13]{Almgren} furthermore establishes the estimate
\begin{equation*}
	\rho^{\mu} [u]_{\mu,B_{\rho/2}(X_0)} \leq C \rho^{1-n/2} \|Du\|_{L^2(B_{\rho}(X_0))}
\end{equation*}
for each ball $B_{\rho}(X_0)$ with $\overline{B_{\rho}(X_0)} \subseteq \Omega$, where $C = C(n,m,q) \in (0,\infty)$.  For such $u$, it then follows from standard interpolation inequalities and \eqref{freqidentity1} below that
\begin{equation} \label{regestimate}
	\sup_{B_{\rho/2}(X_0)} |u| + \rho^{\mu} [u]_{\mu,B_{\rho/2}(X_0)} + \rho^{1-n/2} \|Du\|_{L^2(B_{\rho/2}(X_0))} 
	\leq C \rho^{-n/2} \|u\|_{L^2(B_{\rho}(X_0))}
\end{equation}
for each ball $B_{\rho}(X_0)$ with $\overline{B_{\rho}(X_0)} \subseteq \Omega$, where $C = C(n,m,q) \in (0,\infty)$.

\item[(iii)]~\cite[Theorem 2.6]{Almgren} (c.f.~\cite[Lemma 3.23]{DeLSpa11}), which implies that if $u \in W^{1,2}(\Omega;\mathcal{A}_q(\mathbb{R}^m))$ is locally energy minimizing, then the average-free part $u_f$ of $u$ is locally energy minimizing. 
\end{enumerate}

Let  $u \in W^{1,2}(\Omega;\mathcal{A}_q(\mathbb{R}^m))$ be locally energy minimizing. We consider several types of (interior) singularities of $u$.   

First, following Almgren~\cite[Theorem 2.14]{Almgren}, we define the singular set of $u$, denoted $\Sigma_{u},$ by $\Sigma_{u} = \Omega \setminus {\rm reg} \, u$ where ${\rm reg} \, u$, the regular set of $u$, is the set of points $Y \in \Omega$ with the property that there is a number $\rho \in (0,\op{dist}(Y,\partial \Omega))$ and single-valued harmonic functions $u_j : B_{\rho}(Y) \rightarrow \mathbb{R}^m$, $j = 1,2,\ldots,q,$ such that 
\begin{enumerate}
\item[(a)] $u(X) = \sum_{j=1}^q \llbracket u_j(X) \rrbracket$ for $X \in B_{\rho}(Y)$ and 
\item[(b)] for $i \neq j$, either $u_i(X)  = u_j(X)$ for each $X \in B_{\rho}(Y)$ or $u_i(X) \neq u_j(X)$ for each $X \in B_{\rho}(Y);$  
\end{enumerate}
such $u_j$ are clearly unique up to relabeling if they exist.

We define $\Sigma_{u, q}$ to be the set of points $Y \in \Sigma_{u}$ such that $u(Y) = q \llbracket u_1(Y) \rrbracket$ for some $u_1(Y) \in \mathbb{R}^m$.  

Of course $\Sigma_{u} = \Sigma_{u_f}$ and $\Sigma_{u, q}= \Sigma_{u_{f}, q}  = \{ X \in \Sigma_{u_{f}} : u_f(X) = q \llbracket 0 \rrbracket \}.$ 

Finally, we define the \emph{branch set} $\mathcal{B}_u$ of $u$ to be the set of points $Y \in \Omega$ such that there is no $\rho \in (0,\op{dist}(Y,\partial \Omega))$ for which $u(X) = \sum_{j=1}^q \llbracket u_j(X) \rrbracket$ for some single-valued harmonic functions $u_j  \, : \, B_{\rho}(Y) \to \mathbb{R}^m$, $j = 1,2,\ldots,q$ and all $X \in B_{\rho}(Y)$.  

Clearly $\Sigma_{u}$ is a closed subset of $\Omega$, and $\Sigma_{u, q}$ and ${\mathcal B}_{u}$ are closed subsets of $\Sigma_{u}.$ 

Since $u$ is energy minimizing, it is not difficult to see that 
${\rm dim}_{\mathcal H} \, (\Sigma_{u} \setminus {\mathcal B}_{u}) \leq n-2$. Almgren's theory established, in addition to (i)-(iii) above,  the sharp Hausdorff dimension bound on ${\mathcal B}_{u}$: 

\begin{enumerate}
\item[(iv)]~\cite[Theorem 2.14]{Almgren} (c.f.~\cite[Theorem 0.11]{DeLSpa11}): If $u \in W^{1,2}(\Omega;\mathcal{A}_q(\mathbb{R}^m))$ is energy minimizing, then $\op{dim}_{\mathcal{H}}(\Sigma_{u}) \leq n-2$.
\end{enumerate}

The main question we address here is what can be said about the asymptotic behaviour of $u$ on approach to ${\mathcal B}_{u}$. Rather than focusing directly on $\mathcal{B}_u$, it is more convenient to study the entire singular set $\Sigma_{u}$ since $\Sigma_{u}$ satisfies the following ``persistence of singularities'' property in relation to convergent sequences of energy minimizers; namely, if $u_k,u \in W^{1,2}(\Omega;\mathcal{A}_q(\mathbb{R}^m))$ are energy minimizing functions such that $u_k \rightarrow u$ locally in $L^2$ (or equivalently, in view of \eqref{regestimate} above, uniformly on compact subsets of $\Omega$), then for each connected compact set $K \subset \Omega$ either $\Sigma_{u_k,q} \cap K = \emptyset$ for sufficiently large $k$ or $\Sigma_{u,q} \cap K \neq \emptyset$ or there exists an open set $U$ with $K \subset U \subseteq \Omega$ and harmonic function $h : U \rightarrow \mathbb{R}$ such that $u(X) = q \llbracket h(X) \rrbracket$ for all $X \in U$.  
This property does not hold if we replace $\Sigma_{u_k,q}$ and $\Sigma_{u,q}$ with $\mathcal{B}_{u_k,q} = \mathcal{B}_{u_k} \cap \Sigma_{u_k,q}$ and $\mathcal{B}_{u,q} = \mathcal{B}_u \cap \Sigma_{u,q}$ respectively.  (Consider for example $u_k, u : \mathbb{R}^2 \rightarrow \mathcal{A}_q(\mathbb{C}) \cong \mathcal{A}_q(\mathbb{R}^2)$ defined by $u_{k}(x_1,x_2) = ((x_1+ix_2)^q - 1/k)^{1/q}$ and $u(x_1,x_2) = \sum_{j=0}^{q-1} \llbracket e^{i 2\pi j/q} (x_1+ix_2) \rrbracket$ for $(x_1,x_2) \in \mathbb{R}^2$.)

\section{Monotonicity of frequency and its preliminary consequences} \label{sec:frequency_sec}

Let $u \in W^{1,2}(\Omega;\mathcal{A}_q(\mathbb{R}^m))$ be energy minimizing, $Y \in \Omega$, and $0 < \rho < \op{dist}(Y,\Omega)$.  Following \cite{Almgren}, we define the frequency function associated with $u$ and $Y$ by
\begin{equation} \label{freqency_definition1}
	N_{u,Y}(\rho) = \frac{D_{u,Y}(\rho)}{H_{u,Y}(\rho)}
\end{equation}
whenever $H_{u,Y}(\rho) \neq 0$, where 
\begin{equation} \label{freqency_definition2}
	D_{u,Y}(\rho) = \rho^{2-n} \int_{B_{\rho}(Y)} |Du|^2 \text{ and } H_{u,Y}(\rho) = \rho^{1-n} \int_{\partial B_{\rho}(Y)} |u|^2. 
\end{equation}

A fundamental discovery of Almgren (\hspace{1sp}\cite[Theorem 2.6]{Almgren}) is that if $u \in W^{1,2}(\Omega; {\mathcal A}_{q}({\mathbb R}^{m}))$  is a stationary point of the energy functional with respect to two types of deformations, namely the ``squash'' and ``squeeze'' deformations (see~\cite[Sections 2.4 and 2.5]{Almgren}), then for any point $Y \in \Omega$, $N_{u,Y}(\rho)$ is a monotone nondecreasing function of $\rho$ on any open interval $I$ on which it is defined.  Almgren's argument, briefly, is as follows: Stationarity of $u$ with respect to  squash and squeeze deformations respectively lead to the identities
\begin{gather}
	\int_{\Omega} |Du|^2 \zeta = -\int_{\Omega} u^{\kappa}_l D_i u^{\kappa}_l D_i \zeta, \label{freqidentity1} \\
	\int_{\Omega} \sum_{i,j=1}^n \left( \tfrac{1}{2} |Du|^2 \delta_{ij} - D_i u^{\kappa}_l D_j u^{\kappa}_l \right) D_i \zeta^j = 0, \label{freqidentity2}
\end{gather}
for all $\zeta,\zeta^1,\zeta^2,\ldots,\zeta^n \in C^1_c(\Omega;\mathbb{R})$, where $u(X) = \sum_{l=1}^q \llbracket u_l(X) \rrbracket,$ 
$u^{\kappa}_l$ denotes the $\kappa$-th coordinate of $u_l$, and we use the convention of summing over repeated indices.  By \eqref{freqidentity1} and \eqref{freqidentity2} with $\zeta^j(X) = \zeta(X) \,(X-Y)$ where $\zeta$ approximates the characteristic function of $B_{\rho}(Y)$ 
\begin{gather}
	\int_{B_{\rho}(Y)} |Du|^2 = \int_{\partial B_{\rho}(Y)} u^{\kappa}_l D_R u^{\kappa}_l, \label{freqidentity3} \\
	\frac{d}{d\rho} \left( \rho^{2-n} \int_{B_{\rho}(Y)} |Du|^2 \right) = 2\rho^{2-n} \int_{\partial B_{\rho}(Y)} |D_R u|^2, \label{freqidentity4} 
\end{gather}
for a.e. $\rho \in (0,\op{dist}(Y,\partial \Omega))$, where $R = R(X) = |X-Y|$
Thus by direct computation
\begin{equation}\label{freq_derivative}
	N_{u,Y}'(\rho) = \frac{2\rho^{1-2n}}{H_{u,Y}(\rho)^2} \left( \left( \int_{\partial B_{\rho}(Y)} |u|^2 \right) 
		\left( \int_{\partial B_{\rho}(Y)} R^2 |D_R u|^2 \right) - \left( \int_{\partial B_{\rho}(Y)} R u^{\kappa} D_R u^{\kappa} \right)^2 \right) 
\end{equation}
Thus $N'_{u,Y}(\rho) \geq 0$ for a.e. $\rho \in (0,\op{dist}(Y,\partial \Omega))$ by the Cauchy-Schwartz inequality.  Hence for any $Y \in \Omega$ and any $s$ and $t$ with $0 < s < t < \op{dist}(Y,\Omega)$, $N_{u,Y}(\rho)$ is monotonically increasing for $\rho \in (s,t)$ whenever $H_{u,Y}(\rho) \neq 0$ for all $\rho \in (s,t)$.  (See also Propositions 3.1 and 3.2 and Theorem 3.15 of~\cite{DeLSpa11} for details.)  Thus if for some $t \in (0,\op{dist}(Y,\Omega))$, $H_{u,Y}(\rho) \neq 0$ for all $\rho \in (0,t)$, then the limit
\begin{equation*}
	\mathcal{N}_u(Y) = \lim_{\rho \downarrow 0} N_{u,Y}(\rho)
\end{equation*}
exists. We call $\mathcal{N}_u(Y)$ the frequency of $u$ at $Y$ whenever it exists. 

The monotonicity of $N_{u,Y}$ has the following standard consequences.

\begin{lemma} \label{consequence_lemma}
Let $\Omega \subset \mathbb{R}^n$ be a connected open set and $u \in W^{1,2}(\Omega;\mathcal{A}_q(\mathbb{R}^m))$ be a non-zero energy minimizing function.  Then: 
\begin{enumerate} 
\item[(a)] $H_{u,Y}(\rho) \neq 0$ for each $Y \in \Omega$ and $\rho \in (0,\op{dist}(Y,\partial \Omega))$, $\mathcal{N}_u(Y)$ exists for each $Y \in \Sigma_u$, and
\begin{equation} \label{consequence_eqn1}
	\left( \frac{\sigma}{\rho} \right)^{2N_{u,Y}(\rho)} H_{u,Y}(\rho) \leq H_{u,Y}(\sigma) 
	\leq \left( \frac{\sigma}{\rho} \right)^{2\mathcal{N}_u(Y)} H_{u,Y}(\rho)
\end{equation}
for $0 < \sigma < \rho < \op{dist}(Y,\partial \Omega)$.  
\item[(b)] For each $Y \in \Omega$ and $0 < \sigma < \rho < \op{dist}(Y,\partial \Omega)$, 
\begin{equation} \label{consequence_eqn2} 
	\left( \frac{\sigma}{\rho} \right)^{2N_{u,Y}(\rho)} \rho^{-n} \int_{B_{\rho}(Y)} |u|^2 
	\leq \sigma^{-n} \int_{B_{\sigma}(Y)} |u|^2
	\leq \left( \frac{\sigma}{\rho} \right)^{2\mathcal{N}_u(Y)} \rho^{-n} \int_{B_{\rho}(Y)} |u|^2.
\end{equation}
\item[(c)] ${\mathcal N}$ is upper semi-continuous in the sense that if $u_k,u \in W^{1,2}(\Omega;\mathcal{A}_q(\mathbb{R}^m))$ are energy minimizing functions such that $u_k \rightarrow u$ in $L^2(\Omega';\mathcal{A}_q(\mathbb{R}^m))$ for all $\Omega' \subset \subset \Omega$, and if $Y_k, Y \in \Omega$ such that $Y_k \rightarrow Y$, then $\mathcal{N}_u(Y) \geq \limsup_{k \rightarrow \infty} \mathcal{N}_{u_k}(Y_k)$.
\item[(d)] If $u$ is average-free and $Y \in \Sigma_{u,q}$, then $\mathcal{N}_u(Y) \geq \mu$ for some $\mu = \mu(n,m,q) \in (0,1)$.
\item[(e)] For every $X$ such that $0 < \op{dist}(X,\Sigma_{u,q}) < \op{dist}(X,\partial \Omega)/4$, 
\begin{align} \label{consequence_eqn3}
	&|u(X)| + d(X)^{\mu} [u]_{\mu,B_{d(X)/2}(X)} + d(X)^{1-n/2} \|Du\|_{L^2(B_{d(X)/2}(X))} \\ 
	&\hspace{10mm} \leq 4^{\mathcal{N}_u(Y)} C d(X)^{\mathcal{N}_u(Y)} \rho^{-n/2-\mathcal{N}_u(Y)} \|u\|_{L^2(\Omega)} \nonumber
\end{align}
where $d(X) = \op{dist}(X,\Sigma_{u,q})$, $Y \in \Sigma_{u, q}$ with $|X - Y| = d(X)$, $\rho = \op{dist}(X,\partial \Omega)$, $\mu = \mu(n,m,q) \in (0,1)$, and $C = C(n,m,q) \in (0,\infty)$.
\end{enumerate}
\end{lemma}
\begin{proof} 
(a) follows from the monotonicity of $N_{u,Y}$ and the fact that $N_{u,Y}(\rho) = \rho H'_{u,Y}(\rho)/2H_{u,Y}(\rho)$.  (b) follows from (a) via integration.  (c) follows from monotonicity of $N_{u,Y}$ and the continuity of Dirichlet energy under uniform limits (see Lemma~\ref{compactness_lemma} in the appendix).  (d) follows from (b) and the fact that there exists $\mu = \mu(n,m,q) \in (0,1)$ such that $u \in C^{0,\mu}(\Omega';\mathcal{A}_q(\mathbb{R}^m))$ for every Dirichlet minimizing $q$-valued function $u \in W^{1,2}(\Omega;\mathcal{A}_q(\mathbb{R}^m))$ and $\Omega' \subset \subset \Omega$.  (e) is a consequence of \eqref{regestimate} and (b):
\begin{align*}
	&|u(X)| + d(X)^{\mu} [u]_{\mu,B_{d(X)/4}(X)} + d(X)^{1-n/2} \|Du\|_{L^2(B_{d(X)/4}(X))} 
	\\&\hspace{10mm} \leq C d(X)^{-n/2} \|u\|_{L^2(B_{d(X)/2}(X))} 
	\leq C d(X)^{-n/2} \|u\|_{L^2(B_{2d(X)}(Y))} 
	\\&\hspace{10mm} \leq C \left( \frac{4d(X)}{\rho} \right)^{\mathcal{N}_u(Y)} \rho^{-n/2} \|u\|_{L^2(B_{\rho/2}(Y))} 
	\leq C \left( \frac{4d(X)}{\rho} \right)^{\mathcal{N}_u(Y)} \rho^{-n/2} \|u\|_{L^2(\Omega)} 
\end{align*}
where $Y \in \Sigma_{u,q}$ such that $d(X) = |X-Y|$ and $C = C(n,m,q) \in (0,\infty)$.  
\end{proof}

Now fix a non-zero, average-free, energy minimizing function $u \in W^{1,2}(\Omega;\mathcal{A}_q(\mathbb{R}^m))$ and $Y \in \Sigma_{u,q}$.  Let
\begin{equation*}
	u_{Y,\rho}(X) = \frac{u(Y+\rho X)}{\rho^{-n/2} \|u\|_{L^2(B_{\rho}(Y))}}
\end{equation*}
for $0 < \rho < \op{dist}(Y,\partial \Omega)$. If $\rho_k$ is a sequence of numbers with $\rho_k \rightarrow 0^+$, it follows from \eqref{regestimate} and \eqref{consequence_eqn2} that there exists an average-free function $\varphi \in W^{1,2}_{\rm loc}(\mathbb{R}^n;\mathcal{A}_q(\mathbb{R}^m)) \cap C^{0,\mu}_{\rm loc}(\mathbb{R}^n;\mathcal{A}_q(\mathbb{R}^m))$ such that after passing to a subsequence, $u_{Y,\rho_k} \rightarrow \varphi$ uniformly on $B_{\sigma}(0)$ for every $\sigma > 0$.  We say $\varphi$ is a \textit{blow up of $u$ at $Y$}.  It follows from \eqref{consequence_eqn2} that $\varphi$ is non-zero. By~\cite[Theorem 2.13]{Almgren}, $\varphi$ is energy minimizing on $\mathbb{R}^n$, $N_{\varphi,0}(\rho) = \mathcal{N}_{\varphi}(0) = \mathcal{N}_u(Y)$ for each $\rho > 0$ and $\varphi$ is homogeneous of degree $\mathcal{N}_u(Y)$, i.e. $\varphi(\lambda X) = \sum_{j=1}^q \llbracket \lambda^{\mathcal{N}_u(Y)} \varphi_j(X) \rrbracket$ for every $X \in \mathbb{R}^n$ and $\lambda > 0$, where we write $\varphi(X) = \sum_{j=1}^q \llbracket \varphi_j(X) \rrbracket$.


Suppose $\varphi \in W^{1,2}(\mathbb{R}^n;\mathcal{A}_q(\mathbb{R}^m))$ is any homogeneous degree $\alpha$ ($\geq \mu$), average-free, energy minimizing function.  For each $Y \in \mathbb{R}^n$, it follows from Lemma~\ref{consequence_lemma}(c) (applied with $u_k = u = \varphi$ and $Y_k = t_k Y$ for some sequence $t_k \rightarrow 0^+$) that $\mathcal{N}_{\varphi}(Y) \leq \mathcal{N}_{\varphi}(0)$.  Let
\begin{equation*}
	S(\varphi) = \{ Y \in \mathbb{R}^n : \mathcal{N}_{\varphi}(Y) = \mathcal{N}_{\varphi}(0) \}. 
\end{equation*}
It follows from~\cite[Theorem 2.14]{Almgren} that $S(\varphi)$ is a linear subspace in $\mathbb{R}^n$ and that $\varphi(X) = \varphi(X + Y)$ for all $Y \in S(\varphi)$.  Since the only average-free, energy minimizing functions in $W_{\rm loc}^{1,2}(\mathbb{R};\mathcal{A}_q(\mathbb{R}^m))$ are constant functions, $\dim S(\varphi) \leq n-2$;  if ${\rm dim} \, S(\varphi) = n-2$, we say that $\varphi$ is \emph{cylindrical}.

Let $u \in W^{1,2}(\Omega;\mathcal{A}_q(\mathbb{R}^m))$ be a nonzero, average-free, energy minimizing function.  For $j = 0,1,2,\ldots,n-2$, define $\Sigma^{(j)}_{u,q}$ to be the set of points $Y \in \Sigma_{u,q}$ such that $\dim S(\varphi) \leq j$ for every blow up $\varphi$ of $u$ at $Y$.  Observe that
\begin{equation*}
	\Sigma_{u,q} = \Sigma_{u,q}^{(n-2)} \supseteq \Sigma_{u,q}^{(n-3)} \supseteq \cdots \supseteq \Sigma_{u,q}^{(1)} \supseteq \Sigma_{u,q}^{(0)}. 
\end{equation*}

\begin{lemma} \label{alm_fed_lemma}
Let $u \in W^{1,2}(\Omega;\mathcal{A}_q(\mathbb{R}^m))$ be a nonzero, average-free, energy minimizing function.  For each $j = 1,2,\ldots,n$, $\Sigma_{u,q}^{(j)}$ has Hausdorff dimension at most $j$.  For $\alpha > 0$, $\{ Y \in \Sigma_{u,q}^{(0)} : \mathcal{N}_u(Y) = \alpha \}$ is discrete.
\end{lemma}
\begin{proof}
The well-known argument, due to Almgren (\cite[Theorem 2.26]{Almgren}) in the context of stationary integral varifolds, is based on upper semi-continuity of frequency and the fact that $\mathcal{N}_{\varphi}(0) = \mathcal{N}_u(Y)$ for any blow-up $\varphi$ of $u$ at $Y$.  See the proof of Lemma 1 in Section 3.4 of~\cite{Sim96} for
a nice, concise presentation of the argument in the context of energy minimizing maps.
\end{proof}

We shall later need the following result,  which is the analogue of ~\cite[Lemma~2.4]{Sim93}. 

\begin{lemma} \label{lemma2_4}
Let $K >0$.  There are functions $\delta : (0,1) \rightarrow (0,1)$ and $R : (0,1) \rightarrow (3,\infty)$ depending on $n$, $m$, $q$, and $K$ such that if $\alpha \in (0, K]$, $\epsilon \in (0, 1)$ and $u \in W^{1,2}(\Omega;\mathcal{A}_q(\mathbb{R}^m))$ is an average-free energy minimizing function with $B_{R(\varepsilon)}(0) \subseteq \Omega$, $0 \in \Sigma_{u,q}$, 
\begin{equation*}
	N_{u,0}(R(\varepsilon)) - \alpha < \delta(\varepsilon),
\end{equation*}
and if $Y \in \Sigma_{u,q} \cap \ B_1(0)$ and $\mathcal{N}_u(Y) \geq \alpha$, then the following hold:
\begin{enumerate}
\item[(i)] $0 \leq N_{u,Y}(\rho) - \alpha < \varepsilon^2$ for $\rho \in (0,R(\varepsilon)-3)$. 
\item[(ii)] For every $\rho \in (0,1]$, there is a energy minimizing average-free function $\varphi \in W^{1,2}(\mathbb{R}^n;\mathcal{A}_q(\mathbb{R}^m))$ (depending on $\rho$) that is homogeneous (of degree $\mathcal{N}_{\varphi}(0)$) such that $|\mathcal{N}_{\varphi}(0) - \alpha| < \varepsilon^2$ and
\begin{equation*}
	\int_{B_1(0)} \mathcal{G}(u_{Y,\rho},\varphi)^2 < \varepsilon^2. 
\end{equation*}
\item[(iii)] For every $\rho \in (0,1]$, either there is a non-zero energy minimizing average-free function $\varphi \in W^{1,2}(\mathbb{R}^n;\mathcal{A}_q(\mathbb{R}^m))$ (depending on $\rho$) such that $\varphi$ is homogeneous (of degree $\mathcal{N}_{\varphi}(0)$), $|\mathcal{N}_{\varphi}(0) - \alpha| < \varepsilon^2$, ${\rm dim} \, S(\varphi) = n-2$, and 
\begin{equation*}
	\int_{B_1(0)} \mathcal{G}(u_{Y,\rho},\varphi)^2 < \varepsilon^2 
\end{equation*}
or there is an $(n-3)$-dimensional subspace $L$ of $\mathbb{R}^n$ such that 
\begin{equation*}
	\{ X \in \Sigma_{u_{Y,\rho}} \cap \overline{B_1(0)}: \mathcal{N}_{u_{Y,\rho}}(X) \geq \alpha \} \subset \{ X \in \mathbb{R}^n : \op{dist}(X,L) < \varepsilon \}. 
\end{equation*}
\end{enumerate}
\end{lemma}
\begin{proof}
To prove (i), first observe that by the monotonicity of $N_{u,Y}$, $N_{u,Y}(\rho) \geq \alpha$ for all $\rho \in (0,R(\varepsilon)-1)$.  Clearly
\begin{equation} \label{lemma2_4_eqn1}
	D_{u,Y}(\rho) \leq \left( 1 + \frac{|Y|}{\rho} \right)^{n-2} D_{u,0}(\rho+|Y|) 
\end{equation}
for all $\rho \in (0,R(\varepsilon)-1)$.  By integration by parts and using $\int_{\partial B_{\sigma}(Z)} u^{\kappa}_l D_{R_{Z}} u^{\kappa}_l = \int_{B_{\sigma}(Z)} |Du|^2$, 
\begin{align}\label{identity}
	H_{u,Z}(\rho) &= \rho^{1-n} \int_{\partial B_{\rho}(Z)} |u|^2 
	\\&= \rho^{-n} \int_{\partial B_{\rho}(Z)} |u|^2 (X-Z) \cdot \frac{(X-Z)}{R_{Z}}\nonumber
	\\&= \rho^{-n} \int_{B_{\rho}(Z)} (n |u|^2 + 2R_{Z} u^{\kappa}_l D_{R_{Z}} u^{\kappa}_l)\nonumber
	\\&= n \rho^{-n} \int_{B_{\rho}(Z)} |u|^2 + 2 \rho^{-n} \int_0^{\rho} \sigma \int_{B_{\sigma}(Z)} |Du|^2 \nonumber
\end{align}
for all $Z \in B_1(0)$ and $\rho \in (0,R(\varepsilon)-|Z|)$.  Since for $\r > |Y|,$ 
\begin{eqnarray*}
&&\int_{0}^{\r} \s \int_{B_{\s}(Y)} |Du|^{2} \geq \int_{|Y|}^{\r} \s \int_{B_{\s - |Y|}(0)} |Du|^{2} = 
\int_{0}^{\r - |Y|}(\s + |Y|)\int_{B_{\s}(0)} |Du|^{2} \nonumber\\ 
&&\hspace{4in}\geq \int_{0}^{\r - |Y|} \s \int_{B_{\s}(0)}|Du|^{2},
\end{eqnarray*}
it follows from \eqref{identity} taken with $Z = Y$ that 
\begin{eqnarray*}
H_{u, Y}(\rho) &\geq& n\rho^{-n}\int_{B_{\rho - |Y|}(0)}|u|^{2} + 2 \rho^{-n}\int_{0}^{\r - |Y|} \s \int_{B_{\s}(0)}|Du|^{2}\\
&=& \left( 1 - \frac{|Y|}{\rho} \right)^n H_{u,0}(\rho-|Y|) 
\end{eqnarray*}
where in the last equality we have used again \eqref{identity}, with $\rho - |Y|$ in place of $\rho$ and  $Z = 0$. This and \eqref{consequence_eqn1} now imply that 
\begin{equation} \label{lemma2_4_eqn2}
	H_{u,Y}(\rho) 
	\geq \left( 1 - \frac{|Y|}{\rho} \right)^n \left( \frac{\rho-|Y|}{\rho+|Y|} \right)^{2\alpha+2\delta(\varepsilon)} H_{u,0}(\rho+|Y|) 
\end{equation}
for all $\rho \in (1,R(\varepsilon)-1)$.  By \eqref{lemma2_4_eqn1} and \eqref{lemma2_4_eqn2},
\begin{align*}
	N_{u,Y}(\rho) &\leq \left( 1 + \frac{|Y|}{\rho} \right)^{n-2} \left( 1 - \frac{|Y|}{\rho} \right)^{-n} 
		\left( \frac{\rho+|Y|}{\rho-|Y|} \right)^{2\alpha+2\delta(\varepsilon)} N_{u,0}(\rho+|Y|)
	\\&\leq \left( 1 + \frac{1}{\rho} \right)^{n-2} \left( 1 - \frac{1}{\rho} \right)^{-n} \left( \frac{\rho+1}{\rho-1} \right)^{2K+2} N_{u,0}(\rho+|Y|)
\end{align*}
for all $\rho \in (1,R(\varepsilon)-1)$.  Hence $N_{u,Y}(\rho) - \alpha < \varepsilon^2$ for all $\rho \in (R(\varepsilon)-2,R(\varepsilon)-1)$ provided $R(\varepsilon)$ is sufficiently large and $\delta(\varepsilon)$ sufficiently small.  By the monotonicity of $N_{u,Y}$, $N_{u,Y}(\rho) - \alpha < \varepsilon^2$ for all $\rho \in (0,R(\varepsilon)-1)$.

Given (i), the claim in (ii) is an easy consequence of the monotonicity of frequency function and the compactness property of Dirichlet minimizing $q$-valued functions.  The proof of (iii) is similar to the proof of Lemma 2.4(iii) of~\cite{Sim93}.
\end{proof}

\section{Statements of the main results}
\begin{definition}\label{varphi0_defn}
Here and subsequently, we shall fix a non-zero, homogeneous,  average-free, cylindrical locally energy minimizing function $\varphi^{(0)} \in W^{1, 2}_{\rm loc} \, ({\mathbb R}^{n}; {\mathcal A}_{q}({\mathbb R}^{m}))$ with $S(\varphi^{(0)}) = \{0\} \times {\mathbb R}^{n-2}.$ We note, as can easily be verified, that $\varphi^{(0)}$ has degree of homogeneity $\alpha = {\mathcal N}_{\varphi^{(0)}}(0) = r_{0} /q_{0}$ for relatively prime positive integers $r_{0}$, $q_{0}$ with $q_{0} \leq q$, and that 
\begin{equation}\label{cylindrical-1}
	\varphi^{(0)} = \sum_{j=1}^J m_j \varphi^{(0)}_j 
\end{equation}
for some $J \geq 1$, where $\varphi^{(0)}_j$ are distinct multi-valued functions, $m_{j}$---the \emph{multiplicity} of $\varphi^{(0)}_{j}$---is a positive integer $\leq q$ for each $j$, and either $\varphi^{(0)}_j : \mathbb{R}^n \rightarrow \mathbb{R}^m$ is the (single-valued) zero function $\varphi^{(0)}_j(\cdot) = \llbracket 0 \rrbracket$ or $\varphi^{(0)}_j : \mathbb{R}^n \rightarrow \mathcal{A}_{q_0}(\mathbb{R}^m)$ is the ($q_0$-valued) function given by 
\begin{equation}\label{cylindrical-2}
	\varphi^{(0)}_j(X) = \op{Re}(c^{(0)}_j (x_1+ix_2)^{\alpha})
\end{equation}
for some $c^{(0)}_j \in \mathbb{C}^m \setminus \{0\}$, where $i = \sqrt{-1}$ and $X = (x_{1}, x_{2}, y)$ with $y \in {\mathbb R}^{n-2}$. 
\end{definition}

It follows from \eqref{cylindrical-1} that $q_{0} \sum_{j=1}^{J}m_{j} = q$ in case $\varphi_{j}^{(0)} \not\equiv 0$ for each $j$, or $m_{j_{1}} + q_{0}\sum_{j=1, \, j \neq j_{1}}^{J}m_{j} = q$ in case   $\varphi_{j_{1}}^{(0)} = 0$ for some (unique) $j_{1}$; in particular, either $q_{0}J \leq q$ or $1 + q_{0}(J-1) \leq q$, and in either case 
$$1 \leq J \leq \lceil q/q_{0} \rceil.$$ Note also that since $\varphi^{(0)}$ is assumed to be locally energy minimizing, $\Sigma_{\varphi^{(0)}} = \Sigma_{\varphi^{(0)},q} = \{0\} \times \mathbb{R}^{n-2}$ and in particular ${\rm graph} \, \varphi_{j}^{(0)} \cap ({\mathbb R}^{n} \setminus \{0\} \times {\mathbb R}^{n-2}) \times {\mathbb R}^{m}$, $1 \leq j \leq J,$ are pairwise disjoint, embedded submanifolds of $(\mathbb{R}^n \setminus \{0\} \times \mathbb{R}^{n-2}) \times {\mathbb R}^{m}.$ Moreover, $\varphi_{j}^{(0)}$ is locally energy minimizing in ${\mathbb R}^{n}$ for each $j$.

Given a nonzero, average-free, energy minimizing function $u \in W^{1,2}(\Omega;\mathcal{A}_q(\mathbb{R}^m))$ and a point $Y \in \Sigma_u \setminus \Sigma_{u,q},$ by continuity of $u$ we can find $\rho > 0$, $k \in \{2, \ldots, q\}$ and positive integers $q_{1}, \ldots, q_{k}$ with $\sum_{j=1}^{k} q_{j} = q$ such that $u \res B_{\rho}(Y) = \sum_{j=1}^{k} u_{j}$ where for each $j$, $u_{j}  \in W^{1, 2}(B_{\rho}(Y); {\mathcal A}_{q_{j}}({\mathbb R}^{m}))$ is energy minimizing. This and the fact that $\Sigma_u^{(n-3)}$ has Hausdorff dimension $\leq n-3$ allow us to reduce the proofs of Theorems A and B to an analysis  of $u$ near points of $\Sigma_{u,q}^{(n-2)} \setminus \Sigma_{u,q}^{(n-3)}.$ 
On the other hand, for every point $Y \in \Sigma_{u,q}^{(n-2)} \setminus \Sigma_{u,q}^{(n-3)}$ there exists at least one blow up $\varphi$ of $u$ at $Y$ with ${\rm dim} \, S(\varphi) = n-2$. Since $\varphi$ is locally energy minimizing, after composing with an orthogonal change of coordinates on $\mathbb{R}^n$ taking $S(\varphi)$ to $\{0\}\times {\mathbb R}^{n-2}$, $\varphi$ takes the form of $\varphi^{(0)}$ as described above. Thus in order to understand the behavior of $u$ near a point $Y \in \Sigma_{u,q}^{(n-2)} \setminus \Sigma_{u,q}^{(n-3)}$, we are free to assume that for some $\rho >0$, $u_{Y, \rho}$  is close in $B_{1}(0)$ to a cylindrical minimizer of the form $\varphi^{(0)}$ described above.




\begin{definition} \label{varphi_defn}
Let $\varphi^{(0)}$ and the associated functions $\varphi_{1}^{(0)}, \ldots, \varphi_{J}^{(0)}$ and the numbers $$\alpha, q_{0}, J, m_{1}, \ldots, m_{J}$$ be as in Definition~\ref{varphi0_defn}.  Let $p_0 = J-1$ if $\varphi^{(0)}_{j_1} \equiv 0$ for some $j_1$ and $p_0 = J$ otherwise.  For $\varepsilon > 0$ such that 
\begin{eqnarray} \label{varphi_welldefined}
&4 \varepsilon^{2} < \min \left\{\displaystyle{\min_{1 \leq j, k \leq J,  \; j \neq k, \; \varphi_{j}^{(0)} \neq 0, \; \varphi^{(0)}_{k} \neq 0} \, \int_{B_{1}(0)} \mathcal{G}(\varphi^{(0)}_j(X),\varphi^{(0)}_k(X))^2 dX,} \right. \\
&\hspace{3.5in}	\left.\displaystyle{\min_{1 \leq j \leq J, \; \varphi^{(0)}_{j} \neq 0} \int_{B_{1}(0)} |\varphi_{j}^{(0)}(X)|^{2} \, dX}\right\}\nonumber
\end{eqnarray}
and for $p \in \{p_0, p_0+1, \ldots, \lceil q/q_0 \rceil \}$, we define $\Phi_{\varepsilon,p}(\varphi^{(0)})$ to be the set of functions $\varphi : \mathbb{R}^n \rightarrow \mathcal{A}_q(\mathbb{R}^m)$ satisfying the following requirements:
\begin{itemize}
\item[(a)] $\varphi$ is of the form 
\begin{equation*}
	\varphi = \sum_{j=1}^J \sum_{k=1}^{p_j} m_{j,k}  \varphi_{j,k} , 
\end{equation*}
where: 
\begin{itemize}
\item[(i)] for each $j \in \{1, \ldots, J\}$ and $k \in \{1, \ldots, p_{j}\}$,  either $\varphi_{j,k} : \mathbb{R}^n \rightarrow \mathbb{R}^m$ is the (single-valued) zero function $\varphi_{j,k}(\cdot) = \llbracket 0 \rrbracket$ or $\varphi_{j,k} : \mathbb{R}^n \rightarrow \mathcal{A}_{q_0}(\mathbb{R}^m)$ is the function given by 
\begin{equation*}
	\varphi_{j,k}(x_1,x_2,\ldots,x_n) = \op{Re}(c_{j,k} (x_1+ix_2)^{\alpha})
\end{equation*}
for some $c_{j,k} \in \mathbb{C}^m \setminus \{0\}$;
\item[(ii)] $\varphi_{j,k} \neq \varphi_{j',k'}$ whenever $(j,k) \neq (j^{\prime}, k^{\prime})$;
\item [(iii)] $p_j$ are positive integers such that $\sum_{j=1}^J p_j = p+1$ if $\varphi_{j_1,k_1}(\cdot) = 0$ for some $j_1,k_1$ and $\sum_{j=1}^J p_j = p$ otherwise;
\item[(iv)] $m_{j,k}$---the multiplicity of $\varphi_{j, k}$---is a positive integer. 
\end{itemize}
\item[(b)] if $\varphi_{j}^{(0)} \neq 0$ for each $j \in \{1, \ldots, J\}$, then $\varphi_{j, k} \neq 0$ for each $j \in \{1, \ldots, J\}$ and $k \in \{1, \ldots, p_{j}\}$, and 
\begin{equation*}
	\sum_{j=1}^J \sum_{k=1}^{p_j} \int_{B_1(0)} \mathcal{G}(\varphi_{j,k}(X),\varphi^{(0)}_j(X))^2 dX \leq \varepsilon^2, or 
\end{equation*} 
if $\varphi_{j_{1}}^{(0)} = 0$ for some (unique) $j_{1} \in \{1, \ldots, J\}$, then $\varphi_{j, k} \neq 0$ for each $j \in \{1, \ldots, J\} \setminus \{j_{1}\}$ and $k \in \{1, \ldots, p_{j}\}$, and 
\begin{equation*}
	\sum_{k = 1}^{p_{j_{1}}}\int_{B_{1}(0)} |\varphi_{j_{1}, k}(X)|^{2} dX +  \sum_{j=1, j \neq j_{1}}^J \sum_{k=1}^{p_j} \int_{B_1(0)} \mathcal{G}(\varphi_{j,k}(X),\varphi^{(0)}_j(X))^2 dX \leq \varepsilon^2. 
\end{equation*} 
\end{itemize}
We shall call the multivalued functions $\varphi_{j,k}$ the \textit{components of $\varphi$}.
 \end{definition}

\noindent
\begin{remark}{\rm (1) We do not assume that the functions $\varphi \in \Phi_{\varepsilon, p}(\varphi^{(0)})$ are locally energy minimizing, nor that $\varphi$ are average free although in case $q_{0} \geq 2$ the latter is automatically true by the specific form of $\varphi$.

\noindent
(2) If $\varphi \in \Phi_{\varepsilon,p}(\varphi^{(0)})$ then $\varphi$ has precisely $p$ nonzero components.  Since $\varepsilon = \varepsilon(\varphi^{(0)})$ is chosen sufficiently small to satisfy \eqref{varphi_welldefined}, it follows from (b) that for each $j \in \{1, \ldots, J\}$, precisely $p_{j}$ of the components of $\varphi$, labeled $\varphi_{j, 1}, \ldots, \varphi_{j, p_{j}}$, are near $\varphi_{j}^{(0)}$.

\noindent
(3) It follows from (b) that if $\varphi_{j_{1}, k_{1}} = 0$ for some $j_{1} \in \{1, \ldots, J\}$ and 
$k_{1} \in \{1, \ldots, p_{j_{1}}\}$, then $\varphi^{(0)}_{j_{1}} = 0$.

\noindent 
(4) $\varphi^{(0)} \in \Phi_{\varepsilon,p_0}(\varphi^{(0)})$.  
} \end{remark}

\begin{definition}\label{Phi-p_defn}
For $\varepsilon > 0$ and $p \in \{p_0,p_0+1,\ldots, \lceil q/q_0 \rceil \}$, let $\widetilde{\Phi}_{\varepsilon,p}(\varphi^{(0)})$ denote the set of all $q$-valued functions $\varphi(e^A X)$, $X \in \mathbb{R}^n,$ where $\varphi \in \Phi_{\varepsilon,p}(\varphi^{(0)})$ and $A$ is an $n \times n$ skew-symmetric matrix $A = (A_{ij})$ with $A_{ij} = 0$ if $i,j \leq 2$, $A_{ij} = 0$ if $i,j \geq 3$, and $|A| < \varepsilon$.  
\end{definition}

\begin{definition}\label{Phi_defn}
 Let $\Phi_{\varepsilon}(\varphi^{(0)}) = \bigcup_{p=p_0}^{\lceil q/q_0 \rceil} \Phi_{\varepsilon,p}(\varphi^{(0)})$ and  $\widetilde{\Phi}_{\varepsilon}(\varphi^{(0)}) = \bigcup_{p=p_0}^{\lceil q/q_0 \rceil} \widetilde{\Phi}_{\varepsilon,p}(\varphi^{(0)})$. 
\end{definition}

We now state the main lemma of this paper. This is analogous to \cite[Lemma 1]{Sim93}, except for the fact that in its conclusion (ii) we assert improvement of excess at one of a fixed  number of scales (as opposed to a single scale). The proof of this lemma will be given in Section \ref{sec:main_lemma_sec}, which involves a combination of ideas from \cite{Sim93}, \cite{Wic14} and some that are new here (see the discussion in Section~\ref{outline}). 

\begin{lemma} \label{main_lemma}
Let $\varphi^{(0)} : \mathbb{R}^n \rightarrow \mathcal{A}_q(\mathbb{R}^m)$ be as above.  Given $\vartheta_j \in (0,1/8)$ for $j = 1,2,\ldots, \lceil q/q_0 \rceil - p_0 + 1$ with $\vartheta_j < \vartheta_{j-1}/8$ for $j > 1$, there are $\delta_0 \in (0,1/4)$ and $\varepsilon_0 \in (0,1)$ depending only on $n$, $m$, $q$, $\alpha$, $\varphi^{(0)}$, and  $\vartheta_1,\ldots,\vartheta_{\lceil q/q_0 \rceil-p_0+1}$ such that if $u \in W^{1,2}(B_1(0);\mathcal{A}_q(\mathbb{R}^m))$ is an average-free, energy minimizing function with 
\begin{equation*}
	E \equiv \left(\int_{B_1(0)} \mathcal{G}(u,\varphi)^2\right)^{1/2} < \varepsilon_{0}
\end{equation*}
for some $\varphi \in \widetilde{\Phi}_{\varepsilon_{0}}(\varphi^{(0)})$, then either 
\begin{enumerate}
\item[(i)] $B_{\delta_0}(0,y_0) \cap \{ X \in B_{1/2}(0) \cap \Sigma_{u,q} : \mathcal{N}_u(X) \geq \alpha \} = \emptyset$ for some $y_0 \in B^{n-2}_{1/2}(0)$ or 

\item[(ii)] there is $j \in \{1,2,\ldots,\lceil q/q_0 \rceil-p_0+1\}$ and $\widetilde{\varphi} \in \widetilde{\Phi}_{\gamma \varepsilon_0}(\varphi^{(0)})$ such that 
\begin{equation*}
	\vartheta_j^{-n-2\alpha} \int_{B_{\vartheta_{j}}(0)} \mathcal{G}(u,\widetilde{\varphi})^2 
	\leq C_j \vartheta_j^{2\mu} \int_{B_1(0)} \mathcal{G}(u,\varphi)^2, 
\end{equation*}
where $\gamma \in [1,\infty)$ is a constant depending only on $n$, $m$, $q$, $\alpha$, $\varphi^{(0)}$ and $\vartheta_1, \ldots,\vartheta_{\lceil q/q_0 \rceil - p_0 + 1}$; $\mu \in (0,1)$  is a constant depending only on $n$, $m$, $q$, $\alpha$ and $\varphi^{(0)};$ $C_{1}, \ldots, C_{\lceil q/q_0 \rceil - p_0 + 1}$ are constants such that $C_{1}$ depends only on $n$, $m$, $q$, $\alpha$, $\varphi^{(0)},$ and for $j \geq 2$, $C_j$ depends only on $n$, $m$, $q$, $\alpha$, $\varphi^{(0)}$ and $\vartheta_1,\ldots,\vartheta_{j-1}$ (in particular, for $j \geq 2$, $C_{j}$ is independent of $\vartheta_{j}, \ldots, \vartheta_{\lceil q/q_0 \rceil - p_0 + 1}$.) 
\end{enumerate}
\end{lemma}

By iteratively applying Lemma~\ref{main_lemma} in a manner completely analogous to the corresponding argument in \cite{Sim93}, we obtain the following result, from which Theorems A and B of the introduction will readily follow (see Section \ref{sec:finale_sec}).

\begin{theorem} \label{main_thm}
Let $\Omega$ be an open set in $\mathbb{R}^n$ and let $u \in W^{1,2}(\Omega;\mathcal{A}_q(\mathbb{R}^m))$ be a nonzero, average-free, energy minimizing function.  Then $\Sigma_{u,q}$ is countably $(n-2)$-rectifiable.  Moreover, if for $\alpha >0$ we let 
\begin{equation*}
	\Sigma_{u,q,\alpha} = \{ X \in \Sigma_{u,q} : \mathcal{N}_u(X) = \alpha \text{ and $u$ has a cylindrical blow-up at $X$} \}, 
\end{equation*} 
then: 
\begin{enumerate}
\item[(a)] For any compact set $K \subset \Omega$, $K \cap \Sigma_{u,q,\alpha} \neq \emptyset$ for only finitely many $\alpha > 0$.  
\item[(b)] For every $\alpha > 0$ and $\mathcal{H}^{n-2}$-a.e. $Z \in \Sigma_{u,q,\alpha}$, there exists a unique non-zero, average-free, cylindrical, homogeneous degree $\alpha$ energy minimizing function $\varphi^{(Z)} : \mathbb{R}^n \rightarrow \mathcal{A}_q(\mathbb{R}^m)$ and a number $\rho_Z > 0$ such that 
\begin{equation*}
	\rho^{-n} \int_{B_{\rho}(0)} \mathcal{G}(u(Z+X),\varphi^{(Z)}(X))^2 dX \leq C_Z \rho^{2\alpha+2\mu_Z}
\end{equation*}
for all $\rho \in (0,\rho_Z]$ and some constants $\mu_Z \in (0,1)$ and $C_Z \in (0,\infty)$ depending on $n$, $m$, $q$, $\alpha$, $u$, and $Z$.
\item[(c)] For every $\alpha > 0$ there is an open set $V_{\alpha} \supset \Sigma_{u,q,\alpha}$ such that $V_{\alpha} \cap \{ X \in \Sigma_{u,q} : \mathcal{N}_u(X) \geq \alpha \}$ has locally finite $\mathcal{H}^{n-2}$-measure; i.e.\ for each $Y \in V_{\alpha} \cap \{ X \in \Sigma_{u,q} : \mathcal{N}_u(X) \geq \alpha \}$, there is $\rho>0$ such that 
${\mathcal H}^{n-2} (B_{\rho}(Y) \cap \{ X \in \Sigma_{u,q} : \mathcal{N}_u(X) \geq \alpha \}) < \infty$. In particular, for each $\alpha >0$,  $\Sigma_{u, q, \alpha}$  has locally finite $\mathcal{H}^{n-2}$-measure.
\end{enumerate}
\end{theorem}

\section{A graphical representation and a blow-up procedure} \label{sec:graphical_sec}

Let $\varphi^{(0)}$ be as in Definition~\ref{varphi0_defn} and let $p \in \{p_0, \ldots, \lceil q/q_0 \rceil \}$. Let $\varphi \in \Phi_{\varepsilon, p}(\varphi^{(0)})$ for suffciently small $\varepsilon > 0$ depending on $\varphi^{(0)}$ and $u \in W^{1,2}(B_1(0);\mathcal{A}_q(\mathbb{R}^m))$ be an average-free energy minimizing function close to $\varphi$ in $L^2.$ In case $p > p_0$, assume also that $u$ is significantly closer in $L^{2}$ to $\varphi$ than it is to any other homogeneous cylindrical function having fewer nonzero components than $\varphi$ (in the sense of \eqref{graphical_eqn2} or \eqref{graphicalcor_eqn1} below for suitably small $\overline{\b}$, $\b_{0}$).  In 
Lemma~\ref{graphical_lemma} and  Corollary \ref{graphical_cor} below we will show, subject to these hypotheses, that  away from the singular axis $\{0\} \times \mathbb{R}^{n-2}$ of $\varphi$, we can parameterize ${\rm graph} \, u$ by an appropriate multivalued functions $v_{j,k}$ defined on ${\rm graph} \,  \varphi$ and satisfying certain estimates. These results will allow us to produce ``blow-ups'' of a sequence $(u_{j})$ of energy minimizers relative to a sequence of homogeneous functions $(\varphi_{j})$ as in definition~\ref{varphi_defn} whenever both sequences converge in $L^{2}$ to $\varphi^{(0)}$. We shall discuss this blow-up procedure also in this section.  

\subsection{A class of multi-valued functions on graphs of homogeneous cylindrical functions}\label{functions-on-graphs} 

Let $\Omega \subset \mathbb{R}^n \setminus \{0\} \times {\mathbb R}^{n-2}$ and let $\varphi \in \Phi_{\e, p}(\varphi^{(0)})$ be as in Definition~\ref{varphi_defn}. Note that for each component $\varphi_{j,k}$ of $\varphi$, ${\rm graph} \, \varphi_{j,k} |_{{\mathbb R}^{n} \setminus \{0\} \times {\mathbb R}^{n-2}}$ is an embedded submanifold of $(\mathbb{R}^n \setminus \{0\} \times \mathbb{R}^{n-2}) \times {\mathbb R}^{m},$ and recall that the component $\varphi_{j,k}$ has multiplicity $m_{j,k}$.  For various choices of $\Omega \subset {\mathbb R}^{n}\setminus \{0\}\times {\mathbb R}^{n-2}$, we shall be interested in the class of functions ${\mathcal F}_{\varphi, \Omega}$ which we define as $\mathcal{F}_{\varphi,\Omega} = \prod_{j=1}^J \prod_{k=1}^{p_j} C^0({\rm graph} \, \varphi_{j,k}|_{\Omega};\mathcal{A}_{m_{j,k}}(\mathbb{R}^m))$ so that an element $v = (v_{j,k}) \in \mathcal{F}_{\varphi,\Omega}$ consists of continuous functions $v_{j,k} : \op{graph} \varphi_{j,k} |_{\Omega} \rightarrow \mathcal{A}_{m_{j,k}}(\mathbb{R}^m)$ for each $j \in \{1, \ldots, J\}$ and $k \in \{1, \ldots, p_{j}\}$. 

Note that corresponding to any ball $B \subset \Omega$, there are single-valued harmonic functions $\varphi_{j,k,l} : B \rightarrow \mathbb{R}^m$ such that for each $X \in B$, 
\begin{equation} \label{varphi_localized}
	\varphi_{j,k}(X) = \sum_{l=1}^{q_{j, k}} \llbracket \varphi_{j,k,l}(X) \rrbracket
\end{equation}
where  
\begin{align}\label{varphi_values}
&q_{j, k} = q_{0} \;\; \text{if $\varphi_{j,k}$ is non-zero and}\\
& q_{j, k} = 1 \;\; \text{(with $\varphi_{j, k, l}(X) = \varphi_{j, k, 1}(X) = 0 \in {\mathbb R}^{m}$) if $\varphi_{j, k}$ is the zero function} \nonumber.
\end{align}
Hence if $v = (v_{j,k}) \in {\mathcal F}_{\varphi, \Omega}$ then for each fixed $j \in \{1, \ldots, J\}$, $k \in \{1, \ldots, p_{j}\}$, the function $v_{j, k}$ defines an ${\mathcal A}_{q_{j,k}m_{j, k}}({\mathbb R}^{m})$-valued function $\overline{v}_{j, k}$ on $\Omega$ via 
\begin{equation}\label{vbar}
\overline{v}_{j, k}(X) = \sum_{l=1}^{q_{j, k}} v_{j, k, l}(X) 
\end{equation}
for $X \in \Omega$, where, after choosing a ball $B \subset \Omega$ with $X \in B$ and $\varphi_{j, k, l}$ as in \eqref{varphi_localized},
\begin{equation}\label{v_localized}
v_{j, k,l}(X) = v_{j, k}(X, \varphi_{j, k, l}(X))
\end{equation}
for each $l \in \{1, \ldots, q_{j,k}\}$; thus $v_{j, k, l} \, : \, B \to {\mathcal A}_{m_{j, k}}({\mathbb R}^{m})$ and hence  
\begin{equation} \label{v_localized_1}
v_{j, k, l}(X) = \sum_{h=1}^{m_{j,k}} \llbracket v_{j,k,l,h}(X) \rrbracket
\end{equation}
for some $v_{j, k, l, h}(X) \in {\mathbb R}^{m}$, $1 \leq h \leq m_{j, k}$. 

We associate to a given function $v = (v_{j,k}) \in {\mathcal F}_{\varphi, \Omega}$ the function $u : \Omega \rightarrow \mathcal{A}_q(\mathbb{R}^m)$ defined, using the above notation, by   
\begin{equation} \label{phiplusv}
	u(X) = \sum_{j=1}^J \sum_{k=1}^{p_j} \sum_{l=1}^{ q_{j, k}}\sum_{h=1}^{m_{j,k}} \llbracket \varphi_{j,k,l}(X) + v_{j,k,l,h}(X) \rrbracket 
\end{equation}
for $X \in \Omega$. (Strictly speaking, per the discussion above, the definitions of $\overline{v}_{j, k}(X)$ and $u(X)$ require choosing a  ball $B \subset \Omega$ with $X \in B$, but it is clear that $\overline{v}_{j, k}(X)$, $u(X)$ are independent of the choice of $B$). 

\begin{definition} \label{compwisemin_defn}
Let $\Omega \subset \mathbb R^{n} \setminus \{0\} \times {\mathbb R}^{n-2}$ be open.  For each $j \in \{1,2,\ldots,J\}$ and $k \in \{1,2,\ldots,p_j\}$, let $m_{j,k}$ and $q_{j,k}$ be positive integers and $\varphi_{j,k} : \mathbb{R}^n \rightarrow \mathcal{A}_{q_{j,k}}(\mathbb{R}^m)$ be functions such that $\sum_{j=1}^J m_{j,k} q_{j,k} = q$ and either $q_{j,k} = 1$ and $\varphi_{j,k}(\cdot) = 0$ or $q_{j,k} = q_0$ and $\varphi_{j,k}(X) = \op{Re}(c_{j,k}(x_1+ix_2)^{\alpha})$ for some $c_{j,k} \in \mathbb{C}^m \setminus \{0\}$.  We say that $v = (v_{j,k})_{1 \leq j \leq J, 1 \leq k \leq p_{j}},$ where $v_{j, k} \, : \, {\rm graph} \, \varphi_{j,k}|_{\Omega} \to \mathcal{A}_{m_{j,k}}(\mathbb{R}^m),$ is \textit{component-wise minimizing} in $\Omega$ if for each ball $B \subset \subset \Omega$, each function 
$v_{j,k,l} \, : \, B \to {\mathcal A}_{m_{j, k}}({\mathbb R}^{m})$ as in \eqref{v_localized} is in $W^{1, 2}(B; {\mathcal A}_{m_{j, k}}({\mathbb R}^{m}))$ and is Dirichlet energy minimizing in $B.$
\end{definition}

\begin{remark} {\rm 
If $m_{j, k}$, $\varphi_{j, k}$ are as in Definition~\ref{compwisemin_defn}, we do not require that the function $\varphi = \sum_{j=1}^{J} \sum_{k=1}^{p_{j}} m_{j, k} \varphi_{j, k}$ be close to $\varphi^{(0)}$ (i.e.\  that $\varphi   \in \Phi_{\varepsilon}(\varphi^{(0)})$ for $\varepsilon > 0$ small).  This freedom will allow us to show that the blow-ups constructed below are component-wise minimizing. 
} \end{remark}

\subsection{Graphical representation of Dirichlet energy minimizers with small excess}\label{grahpcail} 

We can now state the main results of this section,  Lemma~\ref{graphical_lemma} and Corollary~\ref{graphical_cor} below.  In Lemma~\ref{graphical_lemma}  and subsequently we shall use the following notation: 
For every $\gamma \in (0,1)$, $\zeta \in \mathbb{R}^{n-2}$, $\rho > 0$, and $\kappa \in (0,1]$, we let 
\begin{equation} \label{annuli_defn}
	A_{\rho,\kappa}(\zeta) = \left\{ (x,y) \in \mathbb{R}^2 \times \mathbb{R}^{n-2} : (|x| - \rho)^2 + |y - \zeta|^2 < \kappa^2 (1-\gamma)^2 \rho^2/16 \right\} .
\end{equation}

\begin{lemma} \label{graphical_lemma} 
Let $\varphi^{(0)}$ be as in Definition~\ref{varphi0_defn} and let $p \in \{p_0, p_0+1,\ldots, \lceil q/q_0 \rceil \}$.  Given $\gamma,\kappa \in (0,1)$, there exist $\overline{\varepsilon}, \overline{\beta} \in (0,1)$ depending only on $n$, $m$, $q$, $p$, $\alpha$, $\varphi^{(0)}$, $\gamma$, and $\kappa$ such that the following holds true.  Suppose that $u \in W^{1,2}(A_{1,1}(0);\mathcal{A}_q(\mathbb{R}^m))$ is an average-free energy minimizing function, $\varphi \in \Phi_{\overline{\varepsilon},p}(\varphi^{(0)})$ is as in Definition~\ref{varphi_defn}, 
\begin{equation} \label{graphical_eqn1}
	\int_{A_{1,1}(0)} \mathcal{G}(u(X),\varphi^{(0)}(X))^2 dX < \overline{\varepsilon}^2 
\end{equation}
and that either \begin{itemize} 
\item[(i)] $p = p_0$ or 
\item[(ii)] $p > p_0$ and 
\begin{equation} \label{graphical_eqn2}
	\int_{A_{1,1}(0)} \mathcal{G}(u,\varphi)^2 
	\leq \overline{\beta} \inf_{\varphi' \in \bigcup_{p'=p_0}^{p-1} \Phi_{c \overline{\varepsilon},p'}(\varphi^{(0)})} \int_{A_{1,1}(0)} \mathcal{G}(u,\varphi')^2, 
\end{equation}
\end{itemize}
where $c = c(n,\gamma) = 3 \left( \int_{A_{1,1}(0)} |(x_1,x_2)|^{2\alpha} \big/ \int_{B_1(0)} |(x_1,x_2)|^{2\alpha} \right)^{1/2}$.  
Then:
\begin{itemize}
\item[(a)] when $p > p_0$, 
\begin{equation} \label{graphical_eqn3}
	\inf_{X \in S^1 \times \mathbb{R}^{n-2}} \op{sep} \varphi(X) \geq C \inf_{\varphi' \in \bigcup_{p'=p_0}^{p-1} \Phi_{c \overline{\varepsilon},p'}(\varphi^{(0)})} \int_{A_{1,1}(0)} \mathcal{G}(u,\varphi')^2
\end{equation}
where the separation $\op{sep}$ is defined by \eqref{separation_defn} and $C = C(n,m,q,p,\alpha,\varphi^{(0)},\gamma,\kappa) \in (0,\infty)$ is a constant;
\item[(b)]  for each $j \in \{1, \ldots, J\}$ and $k \in \{1, \ldots, p_{j}\}$, there exists a unique function $$v_{j,k} : \op{graph} \varphi_{j,k} |_{A_{1,\kappa}(0)} \rightarrow \mathcal{A}_{m_{j,k}}(\mathbb{R}^m)$$ such that $u$ is given by \eqref{phiplusv} in $\Omega = A_{1,\kappa}(0)$, $v= (v_{j,k})$ is component-wise minimizing in $A_{1, \kappa}(0)$, and 
\begin{equation} \label{graphical_eqn4}
	\sup_{A_{1,\kappa}(0)} |\overline{v}_{j,k}|^2 +  [\overline{v}_{j, k}]^{2}_{\mu,A_{1,\kappa}(0)} + \int_{A_{1,\kappa}(0)} |D \overline{v}_{j,k}|^2 \leq C \int_{A_{1,1}(0)} \mathcal{G}(u,\varphi)^2, 
\end{equation}
where $\overline{v}_{j,k}$ are as in \eqref{vbar}, $\mu  = \mu(n, m, q) \in (0, 1)$ and $C = C(n,m,q,p,\alpha,\varphi^{(0)},\gamma,\kappa) \in (0,\infty)$ is a constant.  
\end{itemize}
\end{lemma} 

\begin{remark} {\rm 
In Lemma~\ref{graphical_lemma}, $c$ is chosen so that $\varphi' \in \Phi_{c \overline{\varepsilon}}(\varphi^{(0)})$ if and only if $\int_{A_{1,1}(0)} \mathcal{G}(\varphi',\varphi^{(0)})^2 \leq (3\overline{\varepsilon})^2$. 
} \end{remark} 

To prove the excess decay lemma, Lemma~\ref{main_lemma}, we need a variant of Lemma~\ref{graphical_lemma}, Corollary~\ref{graphical_cor} below.  Let $\varphi^{(0)}$ be as in Definition~\ref{varphi0_defn} and let $p \in \{p_0, p_0+1, \ldots, \lceil q/q_0 \rceil \}$. In Corollary~\ref{graphical_cor} and in a number of other results in subsequent sections, we shall make the first or both of the following two hypotheses with a choice of appropriately small constants 
$\varepsilon_{0}$ and $\b_{0}$ (that depend on $\varphi^{(0)}$). 

\noindent
{\bf Hypothesis~$(\star)$:}
$u \in W^{1,2}(B_1(0);\mathcal{A}_q(\mathbb{R}^m))$ is an average-free locally energy minimizing function with 
 \begin{equation*}
	\int_{B_1(0)} \mathcal{G}(u(X),\varphi^{(0)}(X))^2 dX < \varepsilon_0^2. 
\end{equation*}

\noindent
{\bf Hypothesis~$(\star\star)$:} $u \in W^{1,2}(B_1(0);\mathcal{A}_q(\mathbb{R}^m))$ is an average-free locally energy minimizing function, $\varphi \in \Phi_{\varepsilon_0,p}(\varphi^{(0)})$ and either 
\begin{itemize}
\item[(i)] $p = p_0$ or 
\item[(ii)] $p > p_0$ and 
\begin{equation} \label{graphicalcor_eqn1}
	\int_{B_1(0)} \mathcal{G}(u,\varphi)^2 
	\leq \beta_0 \inf_{\varphi' \in \bigcup_{p'=p_0}^{p-1} \Phi_{3\varepsilon_0,p'}(\varphi^{(0)})} \int_{B_1(0)} \mathcal{G}(u,\varphi')^2. 
\end{equation}
\end{itemize}

\begin{corollary} \label{graphical_cor}  
Let $\varphi^{(0)}$ be as in Definition~\ref{varphi0_defn} and let $0 < \tau < \gamma < 1$.  There exist $\varepsilon_0, \beta_0 \in (0,1)$ depending only on $n$, $m$, $q$, $\alpha$, $\varphi^{(0)}$, $\gamma$, and $\tau$ such that if $u$, $\varphi$ satisfy 
Hypothesis~$(\star)$ and Hypothesis~$(\star\star)$, then:
\begin{itemize}
\item[(a)] when $p > p_0$, 
\begin{equation} \label{graphicalcor_eqn2}
	\inf_{X \in S^1 \times \mathbb{R}^{n-2}} \op{sep} \varphi(X) \geq C \inf_{\varphi' \in \bigcup_{p'=p_0}^{p-1} \Phi_{3\varepsilon_0,p'}(\varphi^{(0)})} \int_{B_1(0)} \mathcal{G}(u,\varphi')^2
\end{equation}
where the separation $\op{sep}$ is defined by \eqref{separation_defn} and $C = C(n,m,q,\alpha,\varphi^{(0)}) \in (0,\infty)$ is a constant independent of $\tau$; 
\item[(b)] for each $j \in \{1, \ldots, J\}$ and $k \in \{1, \ldots, p_{j}\}$, there exists a unique function 
$$v_{j,k} : \op{graph} \varphi_{j,k} |_{B_{\gamma}(0) \cap \{r > \tau\}} \rightarrow \mathcal{A}_{m_{j,k}}(\mathbb{R}^m)$$ 
such that $u$ is given by \eqref{phiplusv} in $B_{\gamma}(0) \cap \{r > \tau\}$, $v= (v_{j,k})$ is component-wise minimizing, and 
\begin{equation} \label{graphicalcor_eqn3}
	\tau^n \sup_{B_{\gamma}(0) \cap \{r > \tau\}} |\overline{v}_{j,k}|^2 + \tau^{n+2\mu} [\overline{v}_{j,k}]_{\mu,B_{\gamma}(0) \cap \{r > \tau\}}^2 
		+ \int_{B_{\gamma}(0) \cap \{r > \tau\}} r^2 |D \overline{v}_{j,k}|^2 
	\leq C \int_{B_1(0)} \mathcal{G}(u,\varphi)^2, 
\end{equation}
where $r = r(X) = |(x_1,x_2)|$ for $X \in B_1(0)$, $\overline{v}_{j,k}$ are as in \eqref{vbar}, and $C = C(n,m,q,\alpha,\varphi^{(0)},\gamma) \in (0,\infty)$ is a constant independent of $\tau$.  
\end{itemize}
\end{corollary}

\begin{remark} \label{graphical_rmk} {\rm (a) Let $\Omega$ be an open subset of $\mathbb{R}^n$, $\varepsilon > 0$ be small, $p \in \{p_0+1, \ldots, \lceil q/q_0 \rceil \}$, $\varphi^{(0)}$ be as in Definition~\ref{varphi0_defn}, and $u \in L^2(\Omega;\mathcal{A}_q(\mathbb{R}^m))$.  By the compactness of $\Phi_{\varepsilon}(\varphi^{(0)})$, there exists $\phi \in \bigcup_{p'=p_0}^{p-1} \Phi_{\varepsilon,p'}(\varphi^{(0)})$ such that 
\begin{equation*}
	\int_{\Omega} \mathcal{G}(u,\phi)^2 
	= \inf_{\varphi' \in \bigcup_{p'=p_0}^{p-1} \Phi_{\varepsilon,p'}(\varphi^{(0)})} \int_{\Omega} \mathcal{G}(u,\varphi')^2. 
\end{equation*}
It follows that 
\begin{equation*}
	\left( \int_{\Omega} \mathcal{G}(\phi,\varphi^{(0)})^2 \right)^{1/2} 
	\leq \left( \int_{\Omega} \mathcal{G}(u,\phi)^2 \right)^{1/2}  + \left( \int_{\Omega} \mathcal{G}(u,\varphi^{(0)})^2 \right)^{1/2}  
	\leq 2 \left( \int_{\Omega} \mathcal{G}(u,\varphi^{(0)})^2 \right)^{1/2}. 
\end{equation*}
Thus, assuming Hypothesis~$(\star)$, \eqref{graphicalcor_eqn1} is equivalent to 
\begin{equation*}
	\int_{B_1(0)} \mathcal{G}(u,\varphi)^2 
	\leq \beta_0 \inf_{\varphi' \in \bigcup_{p'=p_0}^{p-1} \Phi_{\varepsilon,p'}(\varphi^{(0)})} \int_{B_1(0)} \mathcal{G}(u,\varphi')^2, 
\end{equation*}
for any $\varepsilon  > 2\varepsilon_0$ such that \eqref{varphi_welldefined} holds true. A similar assertion of course holds concerning  \eqref{graphical_eqn2}. 

\noindent (b) Let the hypotheses be as in Lemma~\ref{graphical_lemma}, and let $c_{j, k}$ correspond to $\varphi_{j, k}$ as in the Definition~\ref{varphi_defn}. It is clear that provided $\overline{\beta}$ is sufficiently small, \eqref{graphical_eqn2} implies that for each $j \in \{1, \ldots, J\}$,
\begin{equation} \label{graphical_eqn5}
	|c_{j,k} - c_{j,k'}|^2 \geq C(n,m,q,\alpha) \inf_{\varphi' \in \bigcup_{p'=p_0}^{p-1} \Phi_{c \overline{\varepsilon},p'}(\varphi^{(0)})} 
		\int_{A_{1,1}(0)} \mathcal{G}(u,\varphi')^2
\end{equation}
whenever $k \neq k'$ and $\varphi_{j,k}, \varphi_{j,k'}$ are both non-zero, and 
\begin{equation} \label{graphical_eqn6}
	|c_{j,k}|^2 \geq C(n,m,q,\alpha) \inf_{\varphi' \in \bigcup_{p'=p_0}^{p-1} \Phi_{c \overline{\varepsilon},p'}(\varphi^{(0)})} 
		\int_{A_{1,1}(0)} \mathcal{G}(u,\varphi')^2
\end{equation}
whenever $\varphi_{j,k}$ is non-zero. 
In principle, it is still possible that the graphs of different components of $\varphi$ might intersect or be close to intersecting. The estimate \eqref{graphical_eqn3} (and similarly \eqref{graphicalcor_eqn2}) is a quantitative statement ruling this out in the case that $\varphi$ is sufficiently close to a $q$-valued energy minimizing function $u$, i.e.\ when \eqref{graphical_eqn1} and \eqref{graphical_eqn2} hold true. 
} \end{remark}

We shall give the proofs of Lemma~\ref{graphical_lemma} and Corollary~\ref{graphical_cor} at the end of this section.  The proof of Lemma~\ref{graphical_lemma} will proceed by induction on $p$ and will make use of a blow-up procedure that inductively uses the lemma itself.  We shall next discuss this blow-up procedure. This procedure will again play a direct role in the proof of the excess decay lemma, Lemma~\ref{main_lemma}.

\subsection{Blow-ups of Dirichlet energy minimizers relative to cylindrical functions} \label{blow-up-procedure} 
Let $\Omega = A_{1,1}(0)$ or $\Omega = B_1(0)$.  Let $p \in \{p_0, p_0+1,\ldots, \lceil q/q_0 \rceil \}.$  Suppose that $\varepsilon_{\nu} \downarrow 0$ as $\nu \rightarrow \infty$ and that $0 < \beta_{\nu} \leq \overline{\beta}$ if $\Omega = A_{1,1}(0)$, where $\overline{\beta}$ is as in Lemma~\ref{graphical_lemma}, or $\beta_{\nu} \downarrow 0$ as $\nu \rightarrow \infty$ if $\Omega = B_1(0)$.  For $\nu = 1,2,3,\ldots$, let $u^{(\nu)} \in W^{1,2}(\Omega;\mathcal{A}_q(\mathbb{R}^m))$ be an average-free energy minimizing function and $\varphi^{(\nu)} \in \Phi_{\varepsilon_{\nu},p}(\varphi^{(0)})$ such that 
\begin{equation*}
	\int_{\Omega} \mathcal{G}(u^{(\nu)}(X),\varphi^{(0)}(X))^2 dX < \varepsilon_{\nu}^2 
\end{equation*}
and that either 
\begin{itemize}
\item[(i)] $p = p_0$ or 
\item[(ii)] $p > p_0$ and 
\begin{equation} \label{fineblowups_eqn1}
	\int_{\Omega} \mathcal{G}(u^{(\nu)},\varphi^{(\nu)})^2 
	\leq \beta_{\nu} \inf_{\varphi' \in \bigcup_{p'=p_0}^{p-1} \Phi_{c\varepsilon_\nu,p'}(\varphi^{(0)})} \int_{\Omega} \mathcal{G}(u^{(\nu)},\varphi')^2 ,
\end{equation}
where $c$ is as in Lemma~\ref{graphical_lemma} if $\Omega = A_{1,1}(0)$ and $c = 3$ if $\Omega = B_1(0)$.  
\end{itemize}
Choosing notation consistent with Definition~\ref{varphi_defn}, write 
\begin{equation*}
	\varphi^{(\nu)} = \sum_{j=1}^J \sum_{k=1}^{p_j} m_{j,k} \varphi^{(\nu)}_{j,k} 
\end{equation*} 
on $\mathbb{R}^n$, where $\varphi^{(\nu)}_{j,k}$ denote the components of $\varphi^{(\nu)}$ that are close to $\varphi^{(0)}_j$ and have multiplicity $m_{j,k}$.  After passing to a subsequence we may assume that $p_j$ and $m_{j,k}$ are independent of $\nu$.  By applying Lemma~\ref{graphical_lemma} if $\Omega = A_{1,1}(0)$ or Corollary \ref{graphical_cor} if $\Omega = B_1(0)$, there exist: 
\begin{itemize}
\item[(i)] open sets $\Omega_{\nu} \subset\subset \Omega \setminus \{0\} \times \mathbb{R}^{n-2}$ such that 
\begin{equation*}
	\Omega_{\nu} \subseteq \Omega_{\nu+1} \text{ for all } \nu, \hspace{10mm} \Omega = \bigcup_{\nu=1}^{\infty} \Omega_{\nu};
\end{equation*}
\item[(ii)] $m_{j,k}$-valued functions $v^{(\nu)}_{j,k} : \op{graph} \varphi^{(\nu)}_{j,k} |_{\Omega_{\nu}} \rightarrow \mathcal{A}_{m_{j,k}}(\mathbb{R}^m)$ such that \eqref{phiplusv} holds true with $\Omega_{\nu}$, $u^{(\nu)}$, $\varphi^{(\nu)}$, $v^{(\nu)}_{j,k}$ in place of $\Omega$, $u$, $\varphi$, $v_{j,k}$;
\item[(iii)] $v^{(\nu)} = (v^{(\nu)}_{j,k})$ is component-wise minimizing in $\Omega_{\nu}$, and 
\begin{equation} \label{fineblowups_eqn2}
	\sup_{\Omega'} |\overline{v}^{(\nu)}_{j,k}|^2 + [\overline{v}^{(\nu)}_{j,k}]_{\mu,\Omega'}^2 + \int_{\Omega'} |D\overline{v}^{(\nu)}_{j,k}|^2 
	\leq C \int_{\Omega} \mathcal{G}(u^{(\nu)},\varphi^{(\nu)})^2 
\end{equation}
whenever $\Omega' \subset\subset  \Omega_{\nu}$, where $\overline{v}^{(\nu)}_{j,k,l}$ is as in \eqref{vbar} with $\varphi^{(\nu)}$, $v^{(\nu)}$ in place of $\varphi$, $v$ and $C = C(n,m,q,\varphi^{(0)},\Omega',\Omega) \in (0,\infty)$ a constant independent of $\nu$.  
\end{itemize}
For each $\nu$, let 
\begin{equation*}
	E_{\nu} = \left( \int_{\Omega} \mathcal{G}(u^{(\nu)},\varphi^{(\nu)})^2 \right)^{1/2}.
\end{equation*}
By passing to a subsequence of $(\nu)$, we may assume that for each given $j \in \{1, \ldots, J\}$ and $k \in \{1, \ldots, p_{j}\}$ one of the following three possibilities holds: 
\begin{enumerate}
\item[(a)] $\varphi^{(0)}_j$ is not identically zero, 
\item[(b)] $\varphi^{(0)}_j \equiv 0$ and $\varphi^{(\nu)}_{j,k} \equiv 0$ for all $\nu = 1,2,3,\ldots$, and 
\item[(c)] $\varphi^{(0)}_j \equiv 0$ but $\varphi^{(\nu)}_{j,k}$ is not identically zero for all $\nu = 1,2,3,\ldots$. 
\end{enumerate}

We shall construct cylindrical functions $\varphi_{j, k}^{(\infty)}$ and functions (blow-ups) 
$$w_{j,k}: \op{graph} \varphi^{(\infty)}_{j,k} |_{\Omega \setminus \{0\} \times {\mathbb R}^{n-2}} \rightarrow \mathcal{A}_{m_{j,k}}(\mathbb{R}^m)$$
by considering these three cases as follows: 

\noindent \textbf{Case (a):} Let $X = (re^{i\theta},y)$ denote cylindrical coordinates on $\mathbb{R}^n$, where $r > 0$, $\theta \in \mathbb{R}$, and $y \in \mathbb{R}^{n-2}$.  Let $\varphi^{(\nu)}_{j,k}(re^{i\theta},y) = \op{Re}(c^{(\nu)}_{j,k} r^{\alpha} e^{i\alpha \theta})$ for $c^{(\nu)}_{j,k} \in \mathbb{C}^m$ with $|c^{(\nu)}_{j,k} - c^{(0)}_j| \leq C(n,q,\alpha) \varepsilon_{\nu}$ (where $c^{(0)}_j$ is as in Definition~\ref{varphi0_defn}).  
Observe that $\op{Re}(c^{(\nu)}_{j,k} r^{\alpha} e^{i\alpha \theta})$ and $\op{Re}(c^{(0)}_j r^{\alpha} e^{i\alpha \theta})$ are well-defined single-valued functions of $r > 0$, $\theta \in \mathbb{R}$, and $y \in \mathbb{R}^{n-2}$ which are $2\pi q_0$-periodic as functions of $\theta$.  In particular, for each $r > 0$, $\theta \in \mathbb{R}$, and $y \in \mathbb{R}^{n-2}$, $(re^{i\theta},y, \op{Re}(c^{(0)}_j r^{\alpha} e^{i\alpha \theta}))$ and $(re^{i\theta},y, \op{Re}(c^{(\nu)}_{j,k} r^{\alpha} e^{i\alpha \theta}))$ are well-defined points on ${\rm graph}\,\varphi^{(0)}$ and ${\rm graph}\,\varphi^{(\nu)}$ respectively.
Define $w^{(\nu)}_{j,k} : \op{graph} \varphi^{(0)}_j |_{\Omega_{\nu}} \rightarrow \mathcal{A}_{m_{j,k}}(\mathbb{R}^m)$ by 
\begin{equation*}
	w^{(\nu)}_{j,k}(re^{i\theta},y, \op{Re}(c^{(0)}_j r^{\alpha} e^{i\alpha \theta})) 
	= v^{(\nu)}_{j,k}(re^{i\theta},y, \op{Re}(c^{(\nu)}_{j,k} r^{\alpha} e^{i\alpha \theta}))/E_{\nu}
\end{equation*} 
for every $r > 0$, $\theta \in \mathbb{R}$, and $y \in \mathbb{R}^{n-2}$ such that $X = (re^{i\theta},y) \in \Omega_{\nu}$.  
By \eqref{fineblowups_eqn2} and the compactness of multivalued energy minimizing functions (see Lemma~\ref{compactness_lemma}), after passing to a subsequence, there exists a function $w_{j,k} : \op{graph} \varphi^{(0)}_j |_{\Omega \setminus \{0\} \times {\mathbb R}^{n-2}} \rightarrow \mathcal{A}_{m_{j,k}}(\mathbb{R}^m)$ such that 
\begin{equation*}
	w^{(\nu)}_{j,k} \to w_{j,k} \text{ uniformly on } {\rm graph}\,\varphi^{(0)} |_{\Omega'}
\end{equation*}
for each $\Omega' \subset\subset \Omega \setminus \{0\} \times \mathbb{R}^{n-2}$ and $(w_{j,k})$ is component-wise minimizing in $\Omega \setminus \{0\} \times {\mathbb R}^{n-2}$.

\noindent \textbf{Case (b):}  Define $w^{(\nu)}_{j,k} : \op{graph} \varphi^{(0)}_j |_{\Omega_{\nu}} \rightarrow \mathcal{A}_{m_{j,k}}(\mathbb{R}^m)$ by $w^{(\nu)}_{j,k}(X,0) = v^{(\nu)}_{j,k}(X,0)/E_{\nu}$ for all $X \in \Omega_{\nu}$.  By \eqref{fineblowups_eqn2} and the compactness of multivalued energy minimizing functions, after passing to a subsequence, there exists a function $w_{j,k} : \op{graph} \varphi^{(0)}_j |_{\Omega \setminus \{0\} \times {\mathbb R}^{n-2}} \rightarrow \mathcal{A}_{m_{j,k}}(\mathbb{R}^m)$ such that $w^{(\nu)}(\cdot,0) = v^{(\nu)}_{j,k}(\cdot,0)/E_{\nu} \to w_{j,k}(\cdot,0)$ uniformly on each compact subset of $\Omega \setminus \{0\} \times \mathbb{R}^{n-2}$ and $(w_{j,k})$ is component-wise minimizing in $\Omega \setminus \{0\} \times {\mathbb R}^{n-2}$.

\noindent \textbf{Case (c):}  This case is more complicated than the cases (a) and (b).  The difficulty is that since $\varphi_{j,k}^{(\nu)}$ is a nonzero branched $q_0$-valued function and $\varphi_j^{(0)}$ is the single-valued zero function, there is no good way to pair the values of $\varphi_{j,k}^{(\nu)}$ and $\varphi_j^{(0)}$ unless $q_0 = 1$.  
(We might, for example, attempt to define $w^{(\nu)}_{j,k} : {\rm graph}\,\varphi^{(0)}_j |_{\Omega_{\nu}} \rightarrow \mathcal{A}_{m_{j,k} q_0}(\mathbb{R}^m)$ by 
\begin{equation*}
	w^{(\nu)}_{j,k}(re^{i\theta},y,0) = \sum_{l=1}^{q_0} \sum_{h=1}^{m_{j,k}} \llbracket v^{(\nu)}_{j,k,l,h}(X)/E_{\nu} \rrbracket , 
\end{equation*}
where $v^{(\nu)}_{j,k}$ is as in \eqref{v_localized_1}, 
and let $w^{(\nu)}_{j,k}(\cdot,0) \to w_{j,k}(\cdot,0)$ uniformly on compact subsets of $\Omega \setminus \{0\} \times \mathbb{R}^{n-2}$.  But then we cannot make sense, as we shall need to, of a $q$-valued function $\widetilde{u}^{(\nu)}$ close to $u^{(\nu)}$ that takes (roughly) the form   
\begin{equation*}
	\widetilde{u}^{(\nu)}(X) = \sum_{j=1}^J \sum_{k=1}^{p_j} \llbracket \varphi^{(\nu)}_{j,k}(X) + E_{\nu} w_{j,k}(X,\varphi^{(0)}_j(X)) \rrbracket
\end{equation*}
since 
$\varphi^{(0)}_j$ is the single-valued zero function whereas $\varphi^{(\nu)}_{j,k}$ is a nonzero branched $q_0$-valued function for all $\nu$ and thus there is no canonical way to pair the values of $\varphi^{(\nu)}_{j,k}(X)$ and $w_{j,k}(X,0)$.) We proceed as follows:

Suppose that $\varphi^{(0)}_j \equiv 0$ and that $\varphi^{(\nu)}_{j,k}$ is not identically zero for all $\nu$.  Let $\varphi^{(\nu)}_{j,k}(re^{i\theta},y) = \op{Re}(c^{(\nu)}_{j,k} r^{\alpha} e^{i\alpha \theta})$ for $c^{(\nu)}_{j,k} \in \mathbb{C} \setminus \{0\}$.  After passing to a subsequence, let $c^{(\nu)}_{j,k}/|c^{(\nu)}_{j,k}| \rightarrow c^{(\infty)}_{j,k}$ and let $\varphi^{(\infty)}_{j,k}(re^{i\theta},y) = \op{Re}(c^{(\infty)}_{j,k} r^{\alpha} e^{i\alpha \theta})$.  Note that ${\rm graph}\, \varphi^{(\infty)}_{j,k} |_{\mathbb{R}^n \setminus \{0\} \times \mathbb{R}^{n-2}}$ is an immersed submanifold of $(\mathbb{R}^n \setminus \{0\} \times \mathbb{R}^{n-2}) \times \mathbb{R}^m$.  
Define $w^{(\nu)}_{j,k} : \op{graph} \varphi^{(\infty)}_{j,k} |_{\Omega_{\nu}} \rightarrow \mathcal{A}_{m_{j,k}}(\mathbb{R}^m)$ by 
\begin{equation*}
	w^{(\nu)}_{j,k}(re^{i\theta},y, \op{Re}(c^{(\infty)}_j r^{\alpha} e^{i\alpha \theta})) 
	= v^{(\nu)}_{j,k}(re^{i\theta},y, \op{Re}(c^{(\nu)}_{j,k} r^{\alpha} e^{i\alpha \theta}))/E_{\nu}
\end{equation*} 
for all $r > 0$, $\theta \in \mathbb{R}$, and $y \in \mathbb{R}^{n-2}$ such that $X = (re^{i\theta},y) \in \Omega_{\nu}$.  As before, after passing to a subsequence, there exists a function $w_{j,k} : \op{graph} \varphi^{(\infty)}_{j,k} |_{\Omega \setminus \{0\} \times {\mathbb R}^{n-2}} \rightarrow \mathcal{A}_{m_{j,k}}(\mathbb{R}^m)$ such that 
\begin{equation*}
	w^{(\nu)}_{j,k} \to w_{j,k} \text{ uniformly on } {\rm graph}\,\varphi^{(\infty)} |_{\Omega'}
\end{equation*}
for each $\Omega' \subset\subset \Omega \setminus \{0\} \times \mathbb{R}^{n-2}$ and $(w_{j,k})$ is component-wise minimizing in $\Omega \setminus \{0\} \times {\mathbb R}^{n-2}$. 

We will say that $w = (w_{j,k})$ is \textit{a blow-up of $u^{(\nu)}$ relative to $\varphi^{(\nu)}$ by the excess $E_{\nu}$}.  

In case $\varphi^{(0)}_j$ is not identically zero for some $j \in \{1, \ldots, J\},$ we let $\varphi^{(\infty)}_{j,k} = \varphi^{(0)}_j$ for each $k \in \{1, \ldots, p_{j}\}$. Similarly, in case $\varphi^{(0)}_j \equiv 0$ and $\varphi^{(\nu)}_{j,k} \equiv 0$ for some $j \in \{1, \ldots, J\}$ and $k \in \{1, \dots, p_{j}\}$, we let $\varphi^{(\infty)}_{j,k} \equiv 0.$ Thus in all three cases (a), (b) and (c) above, $w_{j,k}$ is a function on $\op{graph} \varphi^{(\infty)}_{j,k} |_{\Omega \setminus \{0\} \times {\mathbb R}^{n-2}}$ for all $j \in \{1, \ldots, J\}$ and $k \in \{1, \ldots, p_{j}\}$. 

\begin{remark} \label{homog blowup rmk} {\rm Suppose instead that $\Omega = \mathbb{R}^n$ and for $\varepsilon_{\nu} \downarrow 0$ and $\beta_{\nu} > 0$ sufficiently small we have $u^{(\nu)} \in \Phi_{\varepsilon_{\nu}}(\varphi^{(0)})$ (so $u^{(\nu)}$ is not necessarily energy minimizing but is cylindrical and homogeneous of degree $\alpha$) and $\varphi^{(\nu)} \in \Phi_{\varepsilon_{\nu}, p}(\varphi^{(0)})$ such that either (i) $p = p_0$ or (ii) $p > p_0$ and \eqref{fineblowups_eqn1} holds true.  Then by Lemma~\ref{cylindrical norm lemma} of the appendix, \eqref{fineblowups_eqn1}, and Remark~\ref{graphical_rmk}(b), each component of $u^{(\nu)}$ is uniformly $E_{\nu}$-close to a unique component of $\varphi^{(\nu)}$ in $B_1(0)$, where $E_{\nu} = \left( \int_{B_1(0)} \mathcal{G}(u^{(\nu)},\varphi^{(\nu)})^2 \right)^{1/2}$.  Thus we can use the above procedure to produce a blow-up $w = (w_{j,k})$ of $u^{(\nu)}$ relative to $\varphi^{(\nu)}$ in $\mathbb{R}^n$.  The functions $w_{j,k}(X,\varphi^{(\infty)}_{j,k}(X))$ will be cylindrical, homogeneous degree $\alpha$ and ${\mathcal A}_{m_{j, k}}({\mathbb R}^{m})$-valued, but $w$ will not necessarily be component-wise minimizing. 
} \end{remark}

\subsection{Proofs of Lemma~\ref{graphical_lemma} and Corollary~\ref{graphical_cor}}

\begin{proof}[Proof of Lemma~\ref{graphical_lemma}] 
We will prove Lemma~\ref{graphical_lemma} by induction on $p$.  Observe that in the case $p = p_0$, if $0 < \kappa < 1$ and $\overline{\varepsilon}$ is sufficiently small, it readily follows from \eqref{graphical_eqn1} and the estimate \eqref{regestimate} (which implies that a sequence of locally energy minimizing functions converging in $L^{2}$ is converging uniformly in the interior) that there exist unique functions $v_{j,k}$ satisfying \eqref{phiplusv} in $\Omega = A_{1,\kappa'}(0)$, where $0 < \kappa < \kappa' < 1$.  Let us check that $v = (v_{j,k})$ is component-wise minimizing in $A_{1,\kappa'}(0).$  The estimate \eqref{graphical_eqn4} on $A_{1,\kappa}(0)$ will then follow from the estimate \eqref{regestimate}.  Let $B$ be an arbitrary ball in $A_{1,\kappa'}(0).$  For each $X \in B$, let $\varphi_{j,k}(X) = \sum_{l=1}^{q_{j,k}} \llbracket \varphi_{j,k,l}(X) \rrbracket$ for smooth harmonic functions $\varphi_{j,k,l} : B \rightarrow \mathbb{R}^m$ and integers $q_{j, k} \in \{1, q_{0}\}$ as in \eqref{varphi_localized} and \eqref{varphi_values},  and let $v_{j,k,l}(X) = v_{j,k}(X,\varphi_{j,k,l}(X)) = \sum_{h=1}^{m_{j,k}} \llbracket v_{j,k,l,h}(X) \rrbracket$ 
as in \eqref{v_localized} and \eqref{v_localized_1} and notice that $v_{j,k,l} \in W^{1,2}(B;\mathcal{A}_{m_{j,k}}(\mathbb{R}^m)).$

For each $j,k,l$ let $\widetilde{v}_{j,k,l} \in W^{1,2}(B;\mathcal{A}_{m_{j,k}}(\mathbb{R}^m))$ such that $\widetilde{v}_{j,k,l}(X) = v_{j,k,l}(X)$ in a neighborhood of $\partial B$.  For each $X \in B$, express $\widetilde{v}_{j,k,l}(X) = \sum_{h=1}^{m_{j,k}} \llbracket \widetilde{v}_{j,k,l,h}(X) \rrbracket$ for $v_{j, k, l, h}(X) \in {\mathbb R}^{m}$.  We define a competitor $\widetilde{u} \in W^{1,2}(B;\mathcal{A}_q(\mathbb{R}^m))$ for $u$ by 
\begin{equation*}
	\widetilde{u}(X) = \sum_{j=1}^J \sum_{k=1}^{p_j} \sum_{l=1}^{ q_{j, k}}\sum_{h=1}^{m_{j,k}} \llbracket \varphi_{j,k,l}(X) + \widetilde{v}_{j,k,l,h}(X) \rrbracket 
\end{equation*}
for every $X \in B$.  Since $u$ is energy minimizing in $\Omega$, 
\begin{align} \label{DirMinToCWMin}
	&\int_B \sum_{j=1}^J \sum_{k=1}^{p_j} \sum_{l=1}^{q_{j,k}} (m_{j,k} |D\varphi_{j,k,l}|^2 + 2m_{j,k} D\varphi_{j,k,l} \cdot Dv_{j,k,l;a} + |Dv_{j,k,l}|^2) 
	= \int_B |Du|^2 \\&\hspace{10mm} \leq \int_B |D\widetilde{u}|^2 = 
	\int_B \sum_{j=1}^J \sum_{k=1}^{p_j} \sum_{l=1}^{q_{j,k}} (m_{j,k} |D\varphi_{j,k,l}|^2 + 2m_{j,k} D\varphi_{j,k,l} \cdot D\widetilde{v}_{j,k,l;a} 
		+ |D\widetilde{v}_{j,k,l}|^2), \nonumber 
\end{align}
where $v_{j,k,l;a}, \widetilde{v}_{j,k,l;a} : B \rightarrow \mathbb{R}^m$ are the single-valued functions $v_{j,k,l;a}(X) = \frac{1}{m_{j,k}} \sum_{h=1}^{m_{j,k}} v_{j,k,l,h}(X)$ and $\widetilde{v}_{j,k,l;a}(X) = \frac{1}{m_{j,k}} \sum_{h=1}^{m_{j,k}} \widetilde{v}_{j,k,l,h}(X)$.  Since $\varphi_{j,k,l}$ is a single-valued harmonic function and $\widetilde{v}_{j,k,l;a} = v_{j,k,l;a}$ near $\partial B$, 
\begin{equation*}
	\int_B D\varphi_{j,k,l} \cdot Dv_{j,k,l;a} = \sum_{k=1}^q \int_B D\varphi_{j,k,l} \cdot D\widetilde{v}_{j,k,l;a} 
\end{equation*}
for all $k$ and thus \eqref{DirMinToCWMin} implies that 
\begin{equation*}
	\int_B \sum_{j=1}^J \sum_{k=1}^{p_j} \sum_{l=1}^{q_{j,k}} |Dv_{j,k,l}|^2 
	\leq \int_B \sum_{j=1}^J \sum_{k=1}^{p_j} \sum_{l=1}^{q_{j,k}} |D\widetilde{v}_{j,k,l}|^2. 
\end{equation*}
Since each $\widetilde{v}_{j,k,l}$ is arbitrary, we conclude that each $v_{j,k,l}$ is energy minimizing in $B.$

Now let $p > p_0$ and assume by induction that the following holds true:
\begin{enumerate}
\item[(A1)] There exists $\widetilde{\varepsilon}, \widetilde{\beta} \in (0,1)$ depending only on $n$, $m$, $q$, $\varphi^{(0)}$, $p$, $\gamma$, and $\kappa$ such that if $p_0 \leq \widetilde{p} \leq p-1$, $\widetilde{u} \in W^{1,2}(A_{1,1}(0));\mathcal{A}_q(\mathbb{R}^m))$ is an average-free, energy minimizing function, $\widetilde{\varphi} \in \Phi_{\widetilde{\varepsilon},\widetilde{p}}(\varphi^{(0)})$ is such that \eqref{graphical_eqn1} holds  with $\widetilde{\varepsilon}$, $\widetilde{u}$ in place of $\overline{\varepsilon}$, $u$ and if either (i) $\widetilde{p} = p_0$ or (ii) $\widetilde{p} > p_0$ and \eqref{graphical_eqn2} holds  with $\widetilde{\beta}$, $\widetilde{\varepsilon}$, $\widetilde{p}$, $\widetilde{u}$, $\widetilde{\varphi}$ in place of $\overline\beta$, $\overline{\varepsilon}$, $p$, $u$, $\varphi$, then the conclusion of Lemma~\ref{graphical_lemma} holds with $\widetilde{\varepsilon}$, $\widetilde{p}$, $\widetilde{u}$, $\widetilde{\varphi}$ in place of $\overline{\varepsilon}$, $p$, $u$, $\varphi$.
\end{enumerate}
Let the hypotheses be as in Lemma~\ref{graphical_lemma}, and  select $s_{1} \in \{p_0,\ldots,p-1\}$ and $\psi^{(1)} \in \Phi_{c \overline{\varepsilon},s_{1}}(\varphi^{(0)})$ such that 
\begin{equation} \label{graphical_eqn7}
	\int_{A_{1,1}(0)} \mathcal{G}(u,\psi^{(1)})^2 
	< 2 \inf_{\varphi' \in \bigcup_{p'=p_0}^{p-1} \Phi_{c \overline{\varepsilon},p'}(\varphi^{(0)})} \int_{A_{1,1}(0)} \mathcal{G}(u,\varphi')^2. 
\end{equation}
Consider first the case that either (i) $s_{1} = p_0$ or (ii) $s_{1} > p_0$ and 
\begin{equation} \label{graphical_eqn8}
	\int_{A_{1,1}(0)} \mathcal{G}(u,\psi^{(1)})^2 
	\leq \widetilde{\beta} \inf_{\varphi' \in \bigcup_{p'=p_0}^{s_{1}-1} \Phi_{c^2 \overline{\varepsilon},p'}(\varphi^{(0)})} \int_{A_{1,1}(0)} \mathcal{G}(u,\varphi')^2. 
\end{equation}
We claim that in this case the conclusion of Lemma~\ref{graphical_lemma} holds.  To see this, let $\varepsilon_{\nu} \downarrow 0$ and $\beta_{\nu} \downarrow 0$ and for $\nu = 1,2,3,\ldots$ let $u^{(\nu)} \in W^{1,2}(B_1(0);\mathcal{A}_q(\mathbb{R}^m))$ be an average-free energy minimizing $q$-valued function, $\varphi^{(\nu)} \in \Phi_{\varepsilon_{\nu},p}(\varphi^{(0)})$, and $\phi^{(\nu)} \in \Phi_{\varepsilon_{\nu},s_{1}}(\varphi^{(0)})$ such that \eqref{graphical_eqn1}, \eqref{graphical_eqn2}, \eqref{graphical_eqn7}, and  \eqref{graphical_eqn8} hold true with $\varepsilon_{\nu}$, $\beta_{\nu}$, $u^{(\nu)}$, $\varphi^{(\nu)}$, and $\phi^{(\nu)}$ in place of $\overline{\varepsilon}$, $\overline{\beta}$, $u$, $\varphi$, and $\psi^{(1)}$.  We want to show that the conclusion of Lemma~\ref{graphical_lemma} holds with $u^{(\nu)}$ and $\varphi^{(\nu)}$ in place of $u$ and $\varphi$.  
Notice that if $s_1 > p_0$ then by the triangle inequality and \eqref{graphical_eqn2} 
\begin{equation*}
	\inf_{\varphi' \in \bigcup_{p'=p_0}^{s_{1}-1} \Phi_{c \overline{\varepsilon},p'}(\varphi^{(0)})} \int_{A_{1,1}(0)} \mathcal{G}(u^{(\nu)},\varphi')^2
	\leq 4 \inf_{\varphi' \in \bigcup_{p'=p_0}^{s_{1}-1} \Phi_{c \overline{\varepsilon},p'}(\varphi^{(0)})} \int_{A_{1,1}(0)} \mathcal{G}(\varphi^{(\nu)},\varphi')^2 
\end{equation*}
for $\nu$ sufficiently large and thus by applying the triangle inequality again using \eqref{graphical_eqn2} and \eqref{graphical_eqn8} 
\begin{align} \label{graphical_eqn9}
	\int_{A_{1,1}(0)} \mathcal{G}(\varphi^{(\nu)},\phi^{(\nu)})^2 
	&\leq 2 \int_{A_{1,1}(0)} \mathcal{G}(u^{(\nu)},\varphi^{(\nu)})^2 + 2 \int_{A_{1,1}(0)} \mathcal{G}(u^{(\nu)},\phi^{(\nu)})^2 
	\\&\leq 4 \int_{A_{1,1}(0)} \mathcal{G}(u^{(\nu)},\phi^{(\nu)})^2 \nonumber 
	\\&\leq 16 \widetilde{\beta} \inf_{\varphi' \in \bigcup_{p'=p_0}^{s_{1}-1} \Phi_{c^2 \overline{\varepsilon},p'}(\varphi^{(0)})} 
		\int_{A_{1,1}(0)} \mathcal{G}(\varphi^{(\nu)},\varphi')^2 \nonumber 
\end{align}
Now by (A1), \eqref{graphical_eqn1}, and \eqref{graphical_eqn8}, we can blow up $u^{(\nu)}$ relative to $\phi^{(\nu)}$.  By Remark~\ref{homog blowup rmk} and \eqref{graphical_eqn9}, we can blow up $\varphi^{(\nu)}$ relative to $\phi^{(\nu)}$.  By \eqref{graphical_eqn2}, $u^{(\nu)}$ and $\varphi^{(\nu)}$ blow up to the same blow-up $w = (w_{j,k})$, which is component-wise minimizing, homogeneous degree $\alpha$, and translation invariant along $\{0\} \times \mathbb{R}^{n-2}$.   

We want to use $w$ to construct a function $\widetilde{\varphi}^{(\nu)} \in \Phi_{2c\varepsilon_\nu}(\varphi^{(0)})$.  Let $E_{\nu} = \left( \int_{A_{1,1}(0)} \mathcal{G}(u^{(\nu)},\phi^{(\nu)})^2 \right)^{1/2}.$  Let $\phi^{(\nu)} = \sum_{j=1}^J \sum_{k=1}^{p_j} m_{j,k} \phi^{(\nu)}_{j,k}$ where $\phi^{(\nu)}_{j,k}$ are distinct components of $\phi^{(\nu)}$ with multiplicity $m_{j,k}$.  If $\varphi^{(0)}_j$ is not identically zero, let $\phi^{(\nu)}_{j,k}(X) = \op{Re}(c^{(\nu)}_{j,k} (x_1+ix_2)^{\alpha})$ for $c^{(\nu)}_{j,k} \in \mathbb{C}^m \setminus \{0\}$ with $|c^{(\nu)}_{j,k} - c^{(0)}_j| \leq C(n,m,q,\alpha)\, \varepsilon_{\nu}$ and let $c^{(\infty)}_{j,k} = c^{(0)}_j$.  If $\varphi^{(0)}_j$ is identically zero and $\phi^{(\nu)}_{j,k}$ is nonzero, then $\phi^{(\nu)}_{j,k}(X) = \op{Re}(c^{(\nu)}_{j,k} (x_1+ix_2)^{\alpha})$ for $c^{(\nu)}_{j,k} \in \mathbb{C}^m \setminus \{0\}$ and we let $c^{(\infty)}_{j,k} = \lim_{\nu \rightarrow \infty} c^{(\nu)}_{j,k}/|c^{(\nu)}_{j,k}|$ (as in the blow-up construction above).  For each $P \in {\rm graph}\,\phi^{(\infty)}_{j,k} |_{A_{1,1}(0)}$, let $w_{j,k}(P) = \sum_{h=1}^{m_{j,k}} \llbracket w_{j,k,h}(P) \rrbracket$ for some $w_{j,k,h}(P) \in \mathbb{R}^m$.  Define $\widetilde{\varphi}^{(\nu)} \in \Phi_{2c\varepsilon_\nu}(\varphi^{(0)})$ by 
\begin{align*}
	\widetilde{\varphi}^{(\nu)}(re^{i\theta},y) 
		=& \sum_{(j,k)\,:\,\phi^{(\nu)}_{j,k} \equiv 0} \,\sum_{h=1}^{m_{j,k}} \llbracket E_{\nu} w_{j,k,h}(re^{i\theta},y,0) \rrbracket
		\\&+ \sum_{(j,k)\,:\,\phi^{(\nu)}_{j,k} \not\equiv 0} \,\sum_{l=1}^{q_0} \sum_{h=1}^{m_{j,k}} 
			\llbracket \op{Re}(c^{(\nu)}_{j,k} r^{\alpha} e^{i\alpha(\theta+2\pi l)}) 
			+ E_{\nu} w_{j,k,h}(re^{i\theta},y,\op{Re}(c^{(\infty)}_{j,k} r^{\alpha} e^{i\alpha(\theta+2\pi l)})) \rrbracket 
\end{align*}
for each $(re^{i\theta},y) \in \mathbb{R}^m$.  
Observe that $w = (w_{j,k})$ is a blow-up of both $\varphi^{(\nu)}$ and $\widetilde{\varphi}^{(\nu)}$ with respect to $\phi^{(\nu)}$.  Furthermore, by Remark~\ref{graphical_rmk}(b), \eqref{graphical_eqn2}, \eqref{graphical_eqn7} and \eqref{graphical_eqn8}, if $\phi^{(\nu)}_{j,k}$ and $\phi^{(\nu)}_{j',k'}$ are nonzero components of $\phi^{(\nu)}$ then $|c^{(\nu)}_{j,k} - c^{(\nu)}_{j',k'}| \geq C(m,n,q,\alpha) \widetilde{\beta}^{-1} E_{\nu}$ whenever $c^{(\nu)}_{j,k} \neq c^{(\nu)}_{j',k'}$ and $|c^{(\nu)}_{j,k}| \geq C(m,n,q,\alpha) \widetilde{\beta}^{-1} E_{\nu}$.  Similarly by Remark~\ref{graphical_rmk}(b) and \eqref{graphical_eqn2} the distinct nonzero components of $\varphi^{(\nu)}$ are $L^{2}(B_{1}(0);\mathcal{A}_{q_0}(\mathbb{R}^{m}))$-distance $\geq C(m,n,q,\alpha) E_{\nu}$ apart and the nonzero components $\varphi^{(\nu)}$ have $L^{2}(B_{1}(0);\mathcal{A}_{q_0}(\mathbb{R}^{m}))$-norm $\geq C(m,n,q,\alpha) E_{\nu}$.  It follows that for large $\nu$, $\varphi^{(\nu)}$ and $\widetilde{\varphi}^{(\nu)}$ both have precisely $p$ nonzero components. 
Moreover, for large $\nu$, we can pair up the components of $\varphi^{(\nu)}$ and $\widetilde{\varphi}^{(\nu)}$ by expressing $\varphi^{(\nu)}$ and $\widetilde{\varphi}^{(\nu)}$ as 
\begin{equation*} 
	\varphi^{(\nu)} = \sum_{j=1}^J \sum_{k=1}^{\widehat{p}_j} \widehat{m}_{j,k} \varphi^{(\nu)}_{j,k} , \quad 
	\widetilde{\varphi}^{(\nu)} = \sum_{j=1}^J \sum_{k=1}^{\widehat{p}_j} \widehat{m}_{j,k} \widetilde{\varphi}^{(\nu)}_{j,k} 
\end{equation*}
where $\widehat{p}_j$ and $\widehat{m}_{j,k}$ are positive integers independent of $\nu$ satisfying $\sum_{k=1}^{\widehat{p}_j} \widehat{m}_{j,k} = m_j$, $\varphi^{(\nu)}_{j,k}$ are distinct components of $\varphi^{(\nu)}$ close to $\varphi^{(0)}_j$, $\widetilde{\varphi}^{(\nu)}_{j,k}$ are distinct components of $\widetilde{\varphi}^{(\nu)}$ close to $\varphi^{(0)}_j$, and for each $j,k$ 
\begin{equation} \label{graphical_eqn10}
	\varphi^{(\nu)}_{j,k} \equiv 0 \text{ if and only if } \widetilde{\varphi}^{(\nu)}_{j,k} \equiv 0, \hspace{10mm}
	\lim_{\nu \rightarrow \infty} \frac{1}{E_{\nu}} \sup_{A_{1,1}(0)} \mathcal{G}(\varphi^{(\nu)}_{j,k}, \widetilde{\varphi}^{(\nu)}_{j,k}) = 0.  
\end{equation}

We claim that there exists a constant $C = C(m,n,q,\alpha,\gamma) > 0$ such that 
\begin{equation} \label{graphical_eqn11}
	\inf_{X \in S^1 \times \{0\}} \op{sep} w_{j,k,l}(X) \geq C > 0. 
\end{equation}
By \eqref{graphical_eqn9} and \eqref{graphical_eqn10}, 
\begin{equation} \label{graphical_eqn12}
	\int_{A_{1,1}(0)} \mathcal{G}(\widetilde{\varphi}^{(\nu)},\phi^{(\nu)})^2 
	\leq 32 \widetilde{\beta} \inf_{\varphi' \in \bigcup_{p'=p_0}^{s_{1}-1} \Phi_{c \overline{\varepsilon},p'}(\varphi^{(0)})} \int_{A_{1,1}(0)} \mathcal{G}(\varphi^{(\nu)},\varphi')^2
\end{equation}
for $\nu$ sufficiently large.  Fix any ball $B = B_{(1-\gamma)/4}(X_0)$ with $X_0 \in S^1 \times \{0\}$.  Let $\phi^{(\nu)}_{j,k}(X) = \sum_{l=1}^{q_{j,k}} \llbracket \phi^{(\nu)}_{j,k,l}(X) \rrbracket$ for each $X \in B$ and some harmonic functions $\phi^{(\nu)}_{j,k,l}, \phi^{(\infty)}_{j,k,l} : B \rightarrow \mathbb{R}^m$ (like in~\eqref{varphi_localized}) and let $w_{j,k,l}(X) = w_{j,k}(X,\phi^{(\infty)}_{j,k,l}(X))$ for each $X \in B$  (like in~\eqref{v_localized}).  On $B$, $w_{j,k,l}$ is a locally Dirichlet energy minimizing and is given by 
\begin{equation*}
	w_{j,k,l}(x_1,x_2,y) = \sum_{h=1}^{m_{j,k}} \llbracket \op{Re}(a_{j,k,l,h} (x_1+ix_2)^{\alpha}) \rrbracket 
\end{equation*}
for some $a_{j,k,l,h} \in \mathbb{C}^m$.  Moreover, by \eqref{graphical_eqn12} we can apply Remark~\ref{graphical_rmk}(b) to $\widetilde{\varphi}^{(\nu)}_{j,k,l}$ to obtain 
\begin{equation*}
	|a_{j,k,l,h} - a_{j,k,l,h'}| \geq C(m,n,q,\alpha,\gamma) > 0. 
\end{equation*}
whenever $a_{j,k,l,h} \neq a_{j,k,l,h'}$.  Thus it suffices to prove the following general claim: let $B = B_{(1-\gamma)/4}(1,0,0)$ and suppose $f \in W^{1,2}(B;\mathcal{A}_Q(\mathbb{R}^m))$ is a locally Dirichlet energy minimizing function on $B$ given by 
\begin{equation*}
	f(x_1,x_2,y) = \sum_{h=1}^Q \llbracket \op{Re}(a_h (x_1+ix_2)^{\alpha}) \rrbracket 
\end{equation*}
for some $a_h \in \mathbb{C}^m$ satisfying,  for some constant $\Lambda \in [1,\infty)$, $\|f\|_{L^2(B)} \leq \Lambda$ and $|a_h - a_{h'}| \geq 1/\Lambda$ whenever $a_h \neq a_{h'}$. Then there exists a constant $C = C(m,n,Q,\alpha,\gamma,\Lambda) > 0$ such that 
\begin{equation} \label{graphical_eqn13}
	\op{sep} f(1,0,0) \geq C > 0. 
\end{equation}
Suppose to the contrary that there exists a constant $\Lambda \in [1,\infty)$ and a sequence of locally Dirichlet energy minimizing functions $f^{(\nu)} \in W^{1,2}(B;\mathcal{A}_Q(\mathbb{R}^m))$ such that 
\begin{equation*}
	f^{(\nu)}(x_1,x_2,y) = \sum_{h=1}^Q \llbracket \op{Re}(a_h^{(\nu)} (x_1+ix_2)^{\alpha}) \rrbracket 
\end{equation*}
for some $a_h^{(\nu)} \in \mathbb{C}^m$, $\|f^{(\nu)}\|_{L^2(B)} \leq \Lambda$, $|a^{(\nu)}_h - a^{(\nu)}_{h'}| \geq 1/\Lambda$ whenever $a_h \neq a_{h'}$, and 
\begin{equation} \label{graphical_eqn14}
	\lim_{\nu \rightarrow \infty} \op{sep} f^{(\nu)}(1,0,0) =0.
\end{equation}
\eqref{graphical_eqn14} implies that we can reorder the constants $a_h^{(\nu)}$ so that $a_1^{(\nu)} \neq a_2^{(\nu)}$ and 
\begin{equation*}
	\lim_{\nu \rightarrow \infty} |\op{Re}(a_1^{(\nu)} - a_2^{(\nu)})| 
	= \lim_{\nu \rightarrow \infty} \left. |\op{Re}(a_1^{(\nu)} (x_1+ix_2)^{\alpha}) - \op{Re}(a_2^{(\nu)} (x_1+ix_2)^{\alpha})| \right|_{X = (1,0,0)} 
	= 0. 	
\end{equation*}
Since $a_1^{(\nu)} \neq a_2^{(\nu)}$, we in fact have $|a^{(\nu)}_1 - a^{(\nu)}_2| \geq 1/\Lambda$.  After passing to a subsequence, $a^{(\nu)}_h \rightarrow a_h$ for each $h = 1,2,\ldots,Q$ and $f^{(\nu)} \rightarrow f$ uniformly in $B$ for a locally Dirichlet energy minimizing function $f : B \rightarrow \mathcal{A}_Q(\mathbb{R}^m)$ given by 
\begin{equation*}
	f(x_1,x_2,y) = \sum_{h=1}^Q \llbracket \op{Re}(a_h (x_1+ix_2)^{\alpha}) \rrbracket .
\end{equation*}
Moreover, $|a_1 - a_2| \geq 1/\Lambda$ but $\op{Re}(a_1) = \op{Re}(a_2)$, so $(1,0,0)$ must be a singular point of $f$, contradicting the fact that $f$ is locally Dirichlet energy minimizing.  Therefore \eqref{graphical_eqn13} holds true, which (by taking $f = w_{j,k,l}$) implies \eqref{graphical_eqn11}. 

To show conclusion (a), let $B \subset A_{1,1}(0)$ be an open ball.  
By (A1), $\inf_{X \in B} \op{sep} \phi^{(\nu)}(X) \geq C \beta_{\nu}^{-1} E_{\nu}$ for $\nu$ sufficiently large, where $C = C(n,m,q,\alpha,\varphi^{(0)},\gamma,\kappa) \in (0,\infty)$ is a constant.  It follows by the construction of $\widetilde{\varphi}^{(\nu)}$ and \eqref{graphical_eqn11} that 
\begin{align*}
	&\inf_{X \in B} \op{sep} \widetilde{\varphi}^{(\nu)}(X) 
	\\&\hspace{10mm} \geq \min\left\{ \inf_{X \in B} \op{sep} \phi^{(\nu)}(X) - 2E_{\nu} \max_{j,k,l} \sup_{X \in A_{1,1}(0)} |w_{j,k,l}(X)|, 
		E_{\nu} \min_{j,k,l} \inf_{X \in B} \op{sep} w_{j,k,l}(X) \right\}
	\\&\hspace{10mm} \geq C E_{\nu} 
\end{align*}
for some constant $C = C(n,m,q,\alpha,\varphi^{(0)},\gamma,\kappa) \in (0,\infty)$.  Hence by \eqref{graphical_eqn10}, 
\begin{equation*}
	\inf_{X \in B} \op{sep} \varphi^{(\nu)}(X) \geq \frac{1}{2} C E_{\nu} 
\end{equation*}
for $\nu$ sufficiently large.  Since $B$ is arbitrary, we have shown that conclusion (a) holds true. 

To show conclusion (b), observe that by conclusion (a) and $\lim_{\nu \rightarrow \infty} E_{\nu}^{-1} \sup_{A_{1,3/4}} \mathcal{G}(u^{(\nu)}, \varphi^{(\nu)}) = 0$, there exists unique functions $v^{(\nu)}_{j,k}$ satisfying \eqref{phiplusv} with $A_{1,\kappa'}(0)$, $u^{(\nu)}$, $\varphi^{(\nu)}$, and $v^{(\nu)}_{j,k}$ in place of $\Omega$, $u$, $\varphi$, and $v_{j,k}$, where $0 < \kappa < \kappa' < 1$.  By using the argument from before, we can show that $v^{(\nu)} = (v^{(\nu)}_{j,k})$ is component-wise minimizing in $A_{1,\kappa'}(0)$.  It then follows from the estimate \eqref{regestimate} that \eqref{graphical_eqn4} holds true with $v^{(\nu)}_{j,k}$ in place of $v_{j,k}$. 

If instead $s_1 > p_0$ and 
\begin{equation*}
	\int_{A_{1,1}(0)} \mathcal{G}(u,\psi^{(1)})^2 
	> \widetilde{\beta} \inf_{\varphi' \in \bigcup_{p'=p_0}^{s_1-1} \Phi_{c^2 \overline{\varepsilon},p'}(\varphi^{(0)})} \int_{A_{1,1}(0)} \mathcal{G}(u,\varphi')^2, 
\end{equation*}
we can select $s_2  \in \{p_0,\ldots,s_{1}-1\}$ and $\psi^{(2)} \in \Phi_{c^2 \overline{\varepsilon},s_2}(\varphi^{(0)})$ such that 
\begin{equation*}
	\int_{A_{1,1}(0)} \mathcal{G}(u,\psi^{(2)})^2 
	< 2 \inf_{\varphi' \in \bigcup_{p'=p_0}^{s_1-1} \Phi_{c^2 \overline{\varepsilon},p'}(\varphi^{(0)})} \int_{A_{1,1}(0)} \mathcal{G}(u,\varphi')^2 
\end{equation*}
and repeat the above argument.  It is clear that at most $p-p_0$ repetitions of the argument are necessary to reach the conclusion of Lemma~\ref{graphical_lemma}. 
\end{proof}

\begin{proof}[Proof of Corollary \ref{graphical_cor}] 
Let $(\xi,\zeta) \in B_{\gamma}(0) \cap \{r > \tau\}$ and $\rho = |\xi|$.  Since $A_{\rho,1}(\zeta) \subset B_1(0)$, we have by Hypothesis~$(\star)$ that
\begin{equation*}
	\int_{A_{\rho,1}(\zeta)} \mathcal{G}(u,\varphi^{(0)})^2 
	\leq \int_{B_1(0)} \mathcal{G}(u,\varphi^{(0)})^2 
	< \varepsilon_0^2.
\end{equation*}
When $p > p_0$, by the triangle inequality, Hypothesis~$(\star\star)$ implies that 
\begin{equation} \label{graphicalrmk_eqn1}
	\int_{B_1(0)} \mathcal{G}(u,\varphi)^2 
	\leq 4 \beta_0 \inf_{\varphi' \in \bigcup_{p'=p_0}^{p-1} \Phi_{3\varepsilon_0,p'}(\varphi^{(0)})} \int_{B_1(0)} \mathcal{G}(\varphi,\varphi')^2
\end{equation}
provided $\beta_0 \leq 1/4$.  Hence by again applying the triangle inequality,  
\begin{align*}
	\int_{A_{\rho,1}(\zeta)} \mathcal{G}(u,\varphi)^2 
	&\leq \int_{B_1(0)} \mathcal{G}(u,\varphi)^2 
	\leq 4 \beta_0 \inf_{\varphi' \in \bigcup_{p'=p_0}^{p-1} \Phi_{3\varepsilon_0,p'}(\varphi^{(0)})} \int_{B_1(0)} \mathcal{G}(\varphi,\varphi')^2 
	\\&\leq C \beta_0 \rho^{-n-2\alpha} \inf_{\varphi' \in \bigcup_{p'=p_0}^{p-1} \Phi_{3\varepsilon_0,p'}(\varphi^{(0)})} 
		\int_{A_{\rho,1}(\zeta)} \mathcal{G}(\varphi,\varphi')^2 
	\\&\leq 2C \beta_0 \rho^{-n-2\alpha} \int_{A_{\rho,1}(\zeta)} \mathcal{G}(u,\varphi)^2 + 2C \beta_0 \rho^{-n-2\alpha} 
		\inf_{\varphi' \in \bigcup_{p'=p_0}^{p-1} \Phi_{3\varepsilon_0,p'}(\varphi^{(0)})} \int_{A_{\rho,1}(\zeta)} \mathcal{G}(u,\varphi')^2 
\end{align*}
for some constant $C = C(n,\alpha,\gamma) \in (0,\infty)$ and thus 
\begin{equation} \label{graphicalcor_eqn5}
	\int_{A_{\rho,1}(\zeta)} \mathcal{G}(u,\varphi)^2 
	\leq 4C \beta_0 \rho^{-n-2\alpha} \inf_{\varphi' \in \bigcup_{p'=p_0}^{p-1} \Phi_{3\varepsilon_0,p'}(\varphi^{(0)})} 
		\int_{A_{\rho,1}(\zeta)} \mathcal{G}(u,\varphi')^2 
\end{equation}
provided $2C \beta_0 \rho^{-n-2\alpha} \leq 1/2$.  Therefore, noting that $\rho > \tau$, provided 
\begin{equation*}
	\varepsilon_0 \leq \min\left\{ \tau^{n/2+\alpha} \overline{\varepsilon}, \frac{c\overline{\varepsilon}}{3} \right\} , \quad 
	\beta_0 \leq \min\left\{\frac{1}{4}, \frac{\tau^{n+2\alpha} \overline{\beta}}{4C} \right\}
\end{equation*}
for $\overline{\varepsilon}$, $\overline{\beta}$, and $c$ as in Lemma~\ref{graphical_lemma}, we can apply Lemma~\ref{graphical_lemma} with $\rho^{-\alpha} u(\rho x, \zeta + \rho y)$ in place of $u$ to conclude that  there exist unique functions $v^{(\zeta,\rho)}_{j,k} : {\rm graph}\, \varphi_{j,k} |_{A_{\rho,1/2}(\zeta)} \rightarrow \mathcal{A}_{m_{j,k}}(\mathbb{R}^m)$ such that \eqref{phiplusv} holds true with $A_{\rho,1/2}(\zeta)$ and $v^{(\zeta,\rho)}$ in place of $\Omega$ and $v$ and 
\begin{equation} \label{graphicalcor_eqn4}
	\sup_{A_{\rho,1/2}(\zeta)} |\overline{v}^{(\zeta,\rho)}_{j,k}|^2 + \rho^{2\mu} [\overline{v}^{(\zeta,\rho)}_{j,k}]_{\mu,A_{\rho,1/2}(\zeta)}^2 
		+ \rho^{2-n} \int_{A_{\rho,1/2}(\zeta)} |D \overline{v}^{(\zeta,\rho)}_{j,k}|^2 
	\leq C \rho^{-n} \int_{A_{\rho,1}(\zeta)} \mathcal{G}(u,\varphi)^2, 
\end{equation}
where $\overline{v}^{(\zeta,\rho)}_{j,k}$ are as in \eqref{vbar} with $v^{(\zeta,\rho)}_{j,k}$ in place of $v_{j,k}$ and $C = C(n,m,q,p,\varphi^{(0)},\gamma) \in (0,\infty)$ is a constant.  We obtain functions $v_{j,k}$ satisfying \eqref{phiplusv} in $\Omega = B_1(0) \cap \{r > \tau\}$ by letting $v_{j,k} = v_{j,k}^{(\zeta,\rho)}$ on ${\rm graph}\, \varphi_{j,k} |_{A_{\rho,1/2}(\zeta)}$ and noting that the functions $v_{j,k}$ are well-defined by the uniqueness of $v^{(\zeta,\rho)}_{j,k}$.  The function $v = (v_{j,k})$ is component-wise minimizing by the argument in the proof of Lemma~\ref{graphical_lemma}.  The estimate \eqref{graphicalcor_eqn3} obviously follows from \eqref{graphicalcor_eqn4} and a covering argument. 

Finally, to see \eqref{graphicalcor_eqn2} when $p > p_0$, we take $\gamma = 1/2$ in Lemma~\ref{graphical_lemma}.  By applying Lemma~\ref{graphical_lemma} with $4^{-\alpha} u(X/4)$ in place of $u,$ 
\begin{equation} \label{graphicalcor_eqn6}
	\inf_{X \in S^1 \times \mathbb{R}^{n-2}} \op{sep} \varphi(X) 
		\geq C \inf_{\varphi' \in \bigcup_{p'=p_0}^{p-1} \Phi_{c \overline{\varepsilon},p'}(\varphi^{(0)})} \int_{A_{1/4,1}(0)} \mathcal{G}(u,\varphi')^2
\end{equation}
for some constant $C = C(m,n,q,\alpha,\varphi^{(0)}) \in (0,\infty)$.  By the triangle inequality and \eqref{graphicalcor_eqn5}
\begin{equation*}
	\inf_{\varphi' \in \bigcup_{p'=p_0}^{p-1} \Phi_{c \overline{\varepsilon},p'}(\varphi^{(0)})} \int_{A_{1/4,1}(0)} \mathcal{G}(u,\varphi')^2
	\geq \frac{1}{4} \inf_{\varphi' \in \bigcup_{p'=p_0}^{p-1} \Phi_{c \overline{\varepsilon},p'}(\varphi^{(0)})} \int_{A_{1/4,1}(0)} \mathcal{G}(\varphi,\varphi')^2
\end{equation*} 
provided $4C \beta_0 (1/4)^{-n-2\alpha} \leq 1$, and also by the triangle inequality and Hypothesis~$(\star\star)$ 
\begin{equation*}
	\inf_{\varphi' \in \bigcup_{p'=p_0}^{p-1} \Phi_{3 \varepsilon_0,p'}(\varphi^{(0)})} \int_{B_{1}(0)} \mathcal{G}(u,\varphi')^2
	\leq 4 \inf_{\varphi' \in \bigcup_{p'=p_0}^{p-1} \Phi_{3 \varepsilon_0,p'}(\varphi^{(0)})} \int_{B_{1}(0)} \mathcal{G}(\varphi,\varphi')^2 
\end{equation*} 
provided $\beta_0 \leq 1/4$.  Hence by the homogeneity of $\varphi$ and $\varphi'$ and Remark~\ref{graphical_rmk}(a)
\begin{equation} \label{graphicalcor_eqn7}
	\inf_{\varphi' \in \bigcup_{p'=p_0}^{p-1} \Phi_{c \overline{\varepsilon},p'}(\varphi^{(0)})} \int_{A_{1/4,1}(0)} \mathcal{G}(u,\varphi')^2
	\geq C \inf_{\varphi' \in \bigcup_{p'=p_0}^{p-1} \Phi_{3 \overline{\varepsilon},p'}(\varphi^{(0)})} \int_{B_{1}(0)} \mathcal{G}(u,\varphi')^2
\end{equation} 
for some constant $C = C(n,\alpha) \in (0,\infty)$ (provided $\varepsilon_0 \leq c \overline{\varepsilon}/3$).  Combining \eqref{graphicalcor_eqn6} and \eqref{graphicalcor_eqn7} yields \eqref{graphicalcor_eqn2}.
\end{proof}

\section{A priori estimates: Part I} \label{sec:L2estimates_sec1}

Let $\varphi^{(0)}$ be the homogeneous degree $\alpha$ cylindrical function as in Definition~\ref{varphi0_defn}.  In this section and the next we establish several key integral estimates for average free locally energy minimizing functions $u \in W^{1,2}(B_1(0);\mathcal{A}_q(\mathbb{R}^m))$ that are close to $\varphi^{(0)}$ in $L^2(B_1(0);\mathcal{A}_q(\mathbb{R}^m))$.   A number of these estimates are inspired by the results in \cite{Sim93}. These estimates will play a fundamental role in the proof of the main excess decay estimate for energy minimizers, Lemma~\ref{main_lemma}. 

The first result in this section is Theorem~\ref{keyest_thm}. Its role in Lemma~\ref{main_lemma} is two fold: first, its direct consequence for the blow-ups (produced as described in Section~\ref{blow-up-procedure}) is a key ingredient in the proof of our asymptotic decay estimate (Theorem~\ref{blowupdecay_lemma}) for the blow-ups, which plays an essential role in the proof of Lemma~\ref{main_lemma}.  Secondly, it will be used in Section~\ref{sec:L2estimates_sec2} to obtain various further estimates that will in particular rule out concentration near the set $B_{1/2}(0) \cap \{Z \, : \, {\mathcal N}_{u}(Z) \geq \alpha\}$ of $\int_{B_{1}(0)} {\mathcal G}(u, \varphi)^{2},$ the excess of $u$ relative to a cylindrical, homogeneous degree $\alpha$ function $\varphi$ close to $\varphi_{0}$. This non-concentration-of-excess implies that the convergence of the blow-up sequences is in $L^{2}_{\rm loc} (B_{1})$, and it is also of fundamental importance to obtaining excess improvement, namely option (ii) of the conclusion of Lemma~\ref{main_lemma}, subject to the assumption that option (i) of its conclusion fails.  The proof of Theorem~\ref{keyest_thm} is based on the variational identities \eqref{freqidentity3} and \eqref{freqidentity4} and in particular on a variant of the frequency function monotonicity formula, Lemma~\ref{keyest_identity} below. 

The other two results in this section, Lemma~\ref{fourierest_lemma} (giving an identity implied by energy stationarity of $u$) and Lemma~\ref{competitor_lemma} (giving an energy comparison estimate for $u$ implied by the energy minimizing property of $u$), will be needed for the classification of homogeneous degree $\alpha$ blow-ups (Lemma~\ref{homogrep_lemma}). This classification in the language of \cite{Sim93} (or \cite{AllardAlmgren}) provides  ``integrability of homogeneous degree $\alpha$ Jacobi fields,'' which is the reason behind exponential decay of $u$ to a unique cylindrical  function at any point where option (i) of the conclusion of Lemma~\ref{main_lemma} fails at all scales. Our proof of Lemma~\ref{homogrep_lemma} is based on establishing monotonicity of the frequency function $\rho \mapsto N_{w, Z}(\rho)$ associated with a homogeneous degree $\alpha$ blow-up $w$ for any $Z \in \{0\} \times {\mathbb R}^{n-2} =$ the axis of $\varphi^{(0)}$. This monotonicity requires stationarity of $w$ with respect to deformations of the domain variables that are radial from the point $Z$ (identity~\eqref{homogrep2_freqid2}), a fact that we deduce from Lemma~\ref{competitor_lemma}. It is interesting to note that this stationarity fails for more general, non-radial domain deformations (see the example discussed in Remark~\ref{example}).   

We will denote a general point $X \in \mathbb{R}^n$ as $X = (x,y)$, where $x \in \mathbb{R}^2$ and $y \in \mathbb{R}^{n-2}$, and let $x = re^{i\theta}$ for $r > 0$ and $\theta \in \mathbb{R}$.  Recall from Section \ref{sec:frequency_sec} that 
\begin{equation*}
	D_{u,Y}(\rho) = \rho^{2-n} \int_{B_{\rho}(Y)} |Du|^2, \quad 
	H_{u,Y}(\rho) = \rho^{1-n} \int_{\partial B_{\rho}(Y)} |u|^2, \quad 
	N_{u,Y}(\rho) = \frac{D_{u,Y}(\rho)}{H_{u,Y}(\rho)}. 
\end{equation*}
Note that since $u \in W^{1,2}(B_1(0);\mathcal{A}_q(\mathbb{R}^m))$, for each $Y \in B_1(0)$, $H_{u,Y}$ is $W^{1,1}$ and $D_{u,Y}$ is absolutely continuous.  Moreover, since $H'_{u,Y}(\rho) = 2\rho^{-1} D_{u,Y}(\rho)$ for a.e.~$\rho \in (0,1-|Y|)$, $H_{u,Y}$ is $C^1$.  

\begin{lemma} \label{keyest_identity} 
Let $\alpha \in \mathbb{R}$.  If $u \in W^{1,2}(B_1(0);\mathcal{A}_q(\mathbb{R}^m))$ is an average-free, energy minimizing $q$-valued function, then for each $Y \in B_1(0)$, 
\begin{equation*}
	\frac{d}{d\rho} \left( \rho^{-2\alpha} (D_{u,Y}(\rho) - \alpha H_{u,Y}(\rho)) \right) 
	= 2\rho^{2-n} \int_{\partial B_{\rho}(Y)} \left| \frac{\partial (u/R^{\alpha})}{\partial R} \right|^2 
\end{equation*}
for a.e.~$ \rho \in (0,1-|Y|)$, where $R(X) = |X - Y|$.
\end{lemma}
\begin{proof}
Compute directly using the identities \eqref{freqidentity3} and \eqref{freqidentity4}.
\end{proof} 

\begin{theorem} \label{keyest_thm}
Let $\varphi^{(0)}$ be as in Definition~\ref{varphi0_defn}.  Given $\gamma \in (0,1)$, there exists $\varepsilon_0 \in (0,1)$ depending only on $n$, $m$, $q$, $\alpha$, $\varphi^{(0)}$, and $\gamma$ such that if $\varphi \in \Phi_{\varepsilon_0}(\varphi^{(0)})$ and $u \in W^{1,2}(B_1(0);\mathcal{A}_q(\mathbb{R}^m))$ is an average-free, energy minimizing $q$-valued function with  $0 \in \Sigma_{u,q}$ and $\mathcal{N}_u(0) \geq \alpha$
then: 
\begin{align*}
	&(a) \hspace{5mm} \int_{B_{\gamma}(0)} R^{2-n} \left| \frac{\partial (u/R^{\alpha})}{\partial R} \right|^2 
	\leq C \int_{B_1(0)} \mathcal{G}(u,\varphi)^2, \\
	&(b) \hspace{5mm} \int_{B_{\gamma}(0)} |D_y u|^2 \leq C \int_{B_1(0)} \mathcal{G}(u,\varphi)^2 
\end{align*}
for some constant $C = C(n,m,q,\alpha,\varphi^{(0)},\gamma) \in (0,\infty)$, where $R = |X|$.  
\end{theorem}

\begin{proof} 
By Lemma~\ref{keyest_identity}, 
\begin{align} \label{keyest_eqn1}
	&\frac{d}{d\rho} \left( \rho^{n-2} (D_{u,0}(\rho) - \alpha H_{u,0}(\rho)) \right) 
	= \frac{d}{d\rho} \left( \rho^{n-2+2\alpha} \cdot \rho^{-2\alpha} (D_{u,0}(\rho) - \alpha H_{u,0}(\rho)) \right) \\
	&\hspace{10mm} = (n-2+2\alpha) \rho^{n-3} (D_{u,0}(\rho) - \alpha H_{u,0}(\rho)) 
		+ 2 \rho^{2\alpha} \int_{\partial B_{\rho}(0)} \left| \frac{\partial (u/R^{\alpha})}{\partial R} \right|^2 \nonumber 
\end{align}
for a.e.~$ \rho \in (0,1)$.  Again by Lemma~\ref{keyest_identity} and the fact that $\mathcal{N}_u(0) \geq \alpha$, 
\begin{equation*}
	\rho^{-2\alpha} (D_{u,0}(\rho) - \alpha H_{u,0}(\rho)) 
	\geq 2 \int_{B_{\rho}(0)} R^{2-n} \left| \frac{\partial (u/R^{\alpha})}{\partial R} \right|^2
\end{equation*}
for all $ \rho \in (0,1)$.  Thus \eqref{keyest_eqn1} gives us 
\begin{equation} \label{keyest_eqn2}
	\frac{d}{d\rho} \left( \rho^{n-2} (D_{u,0}(\rho) - \alpha H_{u,0}(\rho)) \right) 
	\geq 2 \frac{d}{d\rho} \left( \rho^{n-2+2\alpha} \int_{B_{\rho}(0)} R^{2-n} \left| \frac{\partial (u/R^{\alpha})}{\partial R} \right|^2 
		\right) 
\end{equation}
for a.e.~$ \rho \in (0,1)$. 

Let $\psi : [0,\infty) \rightarrow \mathbb{R}$ be a smooth function with $\psi(t) = 1$ for $t \in [0,\gamma]$, $\psi(t) = 0$ for $t \geq (1 + \gamma)/2$, and $0 \leq \psi'(t) \leq 3/(1 - \gamma)$ for $t \in [0,\infty)$.  Multiplying both sides of \eqref{keyest_eqn2} by $\psi(\rho)^2$ and integrating yields
\begin{align} \label{keyest_eqn3}
	&\int_0^1 \frac{d}{d\rho} \left( \rho^{n-2} (D_{u,0}(\rho) - \alpha H_{u,0}(\rho)) \right) \psi(\rho)^2 d\rho \\
	&\hspace{10mm} \geq 2 \int_0^1 \frac{d}{d\rho} \left( \rho^{n-2+2\alpha} \int_{B_{\rho}(0)} 
		R^{2-n} \left| \frac{\partial (u/R^{\alpha})}{\partial R} \right|^2 \right) \psi(\rho)^2 d\rho \nonumber \\ 
	&\hspace{10mm} = -4 \int_0^1 \rho^{n-2+2\alpha} \int_{B_{\rho}(0)} 
		R^{2-n} \left| \frac{\partial (u/R^{\alpha})}{\partial R} \right|^2 \psi(\rho) \psi'(\rho) d\rho \nonumber \\ 
	&\hspace{10mm} = -4 \int_{\gamma}^{(1+\gamma)/2} \rho^{n-2+2\alpha} \int_{B_{\rho}(0)} 
		R^{2-n} \left| \frac{\partial (u/R^{\alpha})}{\partial R} \right|^2 \psi(\rho) \psi'(\rho) d\rho \nonumber \\ 
	&\hspace{10mm} \geq -4 \gamma^{n-2+2\alpha} \int_{B_{\gamma}(0)} R^{2-n} \left| \frac{\partial (u/R^{\alpha})}{\partial R} \right|^2 	\int_{\gamma}^{(1+\gamma)/2} \psi(\rho) \psi'(\rho) d\rho \nonumber \\ 
	&\hspace{10mm} = 2 \gamma^{n-2+2\alpha} \int_{B_{\gamma}(0)} R^{2-n} \left| \frac{\partial (u/R^{\alpha})}{\partial R} \right|^2 . \nonumber
\end{align}
By the coarea formula, 
\begin{equation} \label{keyest_eqn4}
	\int_0^1 \frac{d}{d\rho} \left( \rho^{n-2} D_{u,0}(\rho) \right) \psi(\rho)^2 d\rho 
	= \int_0^1 \int_{\partial B_{\rho}(0)} |Du|^2 \psi(\rho)^2 d\rho 
	= \int |Du|^2 \psi(R)^2 
\end{equation}
and by integration by parts and the coarea formula again
\begin{align} \label{keyest_eqn5}
	\int_0^1 \frac{d}{d\rho} \left( \rho^{n-2} H_{u,0}(\rho) \right) \psi(\rho)^2 d\rho 
	&= -2 \int_0^1 \rho^{n-2} H_{u,0}(\rho) \psi(\rho) \psi'(\rho) d\rho \\ 
	&= -2 \int_0^1 \int_{\partial B_{\rho}(0)} \rho^{-1} |u|^2 \psi(\rho) \psi'(\rho) d\rho \nonumber \\ 
	&= -2 \int R^{-1} |u|^2 \psi(R) \psi'(R). \nonumber 
\end{align}
By \eqref{keyest_eqn3}, \eqref{keyest_eqn4}, and \eqref{keyest_eqn5}, 
\begin{equation} \label{keyest_eqn6}
	\int (|Du|^2 \psi(R)^2 + 2\alpha R^{-1} |u|^2 \psi(R) \psi'(R)) 
	\geq C \int_{B_{\gamma}(0)} R^{2-n} \left| \frac{\partial (u/R^{\alpha})}{\partial R} \right|^2
\end{equation} 
for $C = C(n,\alpha,\gamma) \in (0,\infty)$. 

Now let $(\zeta_1,\ldots,\zeta_n) = \psi(R)^2 (x_1,x_2,0,\ldots,0)$ in \eqref{freqidentity2} to obtain 
\begin{equation} \label{keyest_eqn7}
	\int \left( |Du|^2 - |D_x u|^2 \right) \psi(R)^2 
	= -2 \int \left( \tfrac{1}{2} r^2 |Du|^2 - r^2 |D_r u|^2 - r D_r u_l^{\kappa} (y \cdot D_y u_l^{\kappa}) \right) R^{-1} \psi(R) \psi'(R)
\end{equation}
where $r = |x|$, $u(X) = \sum_{l=1}^q \llbracket u_l(X) \rrbracket$ with $u_l$ locally defined and differentiable in $B_1(0) \setminus {\mathcal B}_{u}$, $u_l^{\kappa}$ denotes the $\kappa$-th coordinate of $u_l$, and we use the convention of summing over repeated indices.  Let $\zeta = \psi(R)^2$ in \eqref{freqidentity1} to obtain
\begin{equation} \label{keyest_eqn8}
	\int |Du|^2 \psi(R)^2 = -2 \int \left( r u_l^{\kappa} D_r u_l^{\kappa} + u_l^{\kappa} (y \cdot D_y u_l^{\kappa}) \right) R^{-1} \psi(R) \psi'(R).
\end{equation}
By multiplying \eqref{keyest_eqn8} by $\alpha$ and adding it to \eqref{keyest_eqn7}, and adding also $\int 2\alpha^2 R^{-1} |u|^2 \psi(R) \psi'(R)$ to both
sides, we obtain 
\begin{align*}
	&\int \left( (\alpha |Du|^2 + |D_y u|^2) \psi(R)^2 + 2\alpha^2 R^{-1} |u|^2 \psi(R) \psi'(R) \right) \nonumber \\
	&\hspace{10mm} = -2 \int \left( \tfrac{1}{2} r^2 |Du|^2 - r D_r u_l^{\kappa} (r D_r u_l^{\kappa} - \alpha u_l^{\kappa}) - \alpha^2 |u|^2 - (y \cdot D_y u_l^{\kappa}) (r D_r u_l^{\kappa} - \alpha u_l^{\kappa}) \right) R^{-1} \psi(R) \psi'(R). 
\end{align*}
By Cauchy's inequality, 
\begin{align*} 
	&\int \left( \left( \alpha |Du|^2 + \tfrac{1}{2} |D_y u|^2 \right) \psi(R)^2 + 2\alpha^2 R^{-1} |u|^2 \psi(R) \psi'(R) \right) \\
	&\hspace{10mm} \leq -2 \int \left( \tfrac{1}{2} r^2 |Du|^2 - r D_r u_l^{\kappa} (r D_r u_l^{\kappa} - \alpha u_l^{\kappa}) - \alpha^2 |u|^2 \right) R^{-1} \psi(R) \psi'(R) \nonumber \\
	&\hspace{20mm} + 2 \int |r D_r u - \alpha u|^2 \psi'(R)^2. \nonumber 
\end{align*} 
Hence by~\eqref{keyest_eqn6} and the definition of $\psi$, 
\begin{align} \label{keyest_eqn9}
	&C \int_{B_{\gamma}(0)} R^{2-n} \left| \frac{\partial (u/R^{\alpha})}{\partial R} \right|^2 + \frac{1}{2} \int_{B_{\gamma}(0)} |D_y u|^2 \\ 
	&\hspace{10mm} \leq -2 \int \left( \tfrac{1}{2} r^2 |Du|^2 - r D_r u_l^{\kappa} (r D_r u_l^{\kappa} - \alpha u_l^{\kappa}) - \alpha^2 |u|^2 \right) R^{-1} \psi(R) \psi'(R) \nonumber \\
	&\hspace{20mm} + 2 \int |r D_r u - \alpha u|^2 \psi'(R)^2 \nonumber 
\end{align}
for $C = C(n,\alpha,\gamma) \in (0,\infty)$.

Now let, for $\rho, \kappa \in (0, 1)$ and $\zeta \in {\mathbb R}^{n-2}$,  $A_{\rho, \kappa}(\zeta)$ be the annulus defined by $A_{\rho, \kappa}(\zeta) = \{(re^{i\theta}, y) \, : \, 0 \leq \theta < 2\pi, (r, y) \in B^{n-1}_{\frac{1}{4}\kappa(1 - \gamma)\rho}(\rho, \zeta)\}$ and note that  
$A_{\rho, \kappa}(\zeta) \cap A_{\rho^{\prime}, \kappa}(\zeta^{\prime}) \neq \emptyset   \iff B^{n-1}_{\frac{1}{4}\kappa(1 - \gamma)\rho}(\rho, \zeta) \cap B^{n-1}_{\frac{1}{4}\kappa(1 - \gamma)\rho^{\prime}}(\rho^{\prime}, \zeta^{\prime}) \neq \emptyset \nonumber
\iff |(\rho, \zeta) - (\rho^{\prime}, \zeta^{\prime})| < \frac{1}{4}\kappa(1 - \gamma)(\rho + \rho^{\prime})$. 
 By applying the Besicovitch covering theorem to the collection of closed balls $\{\overline{B^{n-1}_{(1-\gamma)\rho/16}(\rho,\zeta)} \, : \, \rho> 0, \zeta \in {\mathbb R}^{n-2}, \rho^{2} + |\zeta|^{2} < \frac{(3 + \gamma)^{2}}{16}\},$ we find countable 
 collections $\mathcal{I}_1, \mathcal{I}_2, \ldots, \mathcal{I}_N$, where $N \leq C(n) < \infty$, of points $(\rho,\zeta)$ with $\rho > 0$,  $\zeta \in \mathbb{R}^{n-2}$ and $\rho^2 + |\zeta|^2 \leq (3+\gamma)^2/16$ such that 
 $\{ \overline{A_{\rho,1/4}(\zeta)} : (\rho,\zeta) \in \mathcal{I}_j \}$ is a collection of pairwise disjoint annuli for each $j = 1,2,\ldots,N$ and $B_{(3+\gamma)/4}(0) \setminus \{0\} \times {\mathbb R}^{n-2} \subset \bigcup_{(\rho,\zeta) \in \mathcal{I}} \overline{A_{\rho,1/4}(\zeta)}$ where $\mathcal{I} = \mathcal{I}_1 \cup \mathcal{I}_2 \cup \cdots \cup \mathcal{I}_{N}.$  Observe that if $A_{\rho,1}(\zeta) \cap A_{\rho',1}(\zeta') \neq \emptyset$ then $\tfrac{3+\gamma}{5-\gamma} \rho \leq \rho' \leq \tfrac{5-\gamma}{3+\gamma} \rho$ and $|(\rho, \zeta) - (\rho^{\prime}, \zeta^{\prime})| < c\rho -\frac{1}{16}(1-\gamma)\rho^{\prime}$ where $c = \frac{1-\gamma}{16} \left(4 + \frac{5(5 - \gamma)}{3+\gamma}\right),$ whence 
 $B^{n-1}_{\frac{(1-\gamma)(3+\gamma)\rho}{16(5-\gamma)}}(\rho^{\prime}, \zeta^{\prime}) \subset B^{n-1}_{(1 - \gamma)\rho^{\prime}/16}(\rho^{\prime}, \zeta^{\prime}) \subset B^{n-1}_{c\rho}(\rho, \zeta).$ Since   
 the balls $B^{n-1}_{(1 - \gamma)\rho^{\prime}/16}(\rho^{\prime}, \zeta^{\prime})$ for $(\rho^{\prime}, \zeta^{\prime}) \in {\mathcal I}_{j}$ are pairwise disjoint, it follows  that for each $j$ and each $(\rho,\zeta) \in \mathcal{I}_{j}$, 
\begin{equation}
{\rm card} \, \{(\rho^{\prime}, \zeta^{\prime}) \in {\mathcal I}_{j} \, : \, A_{\rho,1}(\zeta) \cap A_{\rho',1}(\zeta') \neq \emptyset\} \leq M
\end{equation}
for some constant $M = M(n, \gamma),$ and consequently, for each $j$ there exists an integer $m_{j} \leq M+1$ and disjoint sets ${\mathcal I}_{j, k} \subset {\mathcal I}_{j}$ ($1 \leq k \leq m_{j}$) such that ${\mathcal I}_{j} = \cup_{k=1}^{m_{j}} {\mathcal I}_{j, k}$ and 
$\{A_{\rho, 1}(\zeta) \, : \, (\rho, \zeta) \in {\mathcal I}_{j, k}\}$ is pairwise disjoint for each $k \in \{1, \ldots, m_{j}\}$. Let $\{\widetilde\chi_{(\rho, \zeta)}\}_{(\rho, \zeta) \in {\mathcal I}}$ be a smooth partition of unity subordinate to the collection of balls $\{B^{n-1}_{(1 - \gamma)\rho/16}(\rho, \zeta) \, : \, (\rho, \zeta) \in {\mathcal I}\}$ (with, in particular,  ${\rm spt} \, \widetilde\chi_{(\rho, \zeta)} \subset B^{n-1}_{(1-\gamma)\rho/16}(\rho, \zeta)$). For $r >0$, $0 \leq \theta < 2\pi$ and $y \in {\mathbb R}^{n-2}$, 
let $\chi_{(\rho, \zeta)}(re^{i\theta}, y) = \widetilde\chi_{(\rho, \zeta)}(r, y)$ so we have that ${\rm spt} \, \chi_{(\rho, \zeta)} \subset A_{\rho, 1/4}(\zeta)$  and $\sum_{(\rho, \zeta) \in {\mathcal I}} \chi_{(\rho, \zeta)} \equiv 1$ on 
$\cup_{(\rho, \zeta) \in {\mathcal I}} A_{\rho, 1/4}(\zeta).$ We claim that 
\begin{align} \label{keyest_eqn10}
	&-\int \left( \tfrac{1}{2} r^2 |Du|^2 - r D_r u_l^{\kappa} (r D_r u_l^{\kappa} - \alpha u_l^{\kappa}) - \alpha^2 |u|^2 \right) R^{-1} \psi(R) \psi'(R) \chi_{(\rho, \zeta)}
	\\&\hspace{10mm} + \int |r D_r u - \alpha u|^2 \psi'(R)^2 \chi_{(\rho, \zeta)}
	\leq C \int_{A_{\rho,1}(\zeta)} \mathcal{G}(u,\varphi)^2  \nonumber 
\end{align}
for each $(\rho, \zeta) \in {\mathcal I}$ and some constant $C = C(n,m,q,\alpha,\varphi^{(0)},\gamma) \in (0,\infty)$. 
Then by summing \eqref{keyest_eqn10} over $(\rho, \zeta) \in {\mathcal I} = \cup_{j=1}^{N} \cup_{k=1}^{m{j}}{\mathcal I}_{j, k}$ and keeping in mind that 
${\rm spt} \, \psi \subset B_{(1+\gamma)/2}(0)$, we deduce that  
\begin{align} \label{keyest_eqn11}
	&- \int \left( \tfrac{1}{2} r^2 |Du|^2 - r D_r u_l^{\kappa} (r D_r u_l^{\kappa} - \alpha u_l^{\kappa}) - \alpha^2 |u|^2 \right) R^{-1} \psi(R) \psi'(R) 
	\\&\hspace{10mm} + \int |r D_r u - \alpha u|^2 \psi'(R)^2 
	\leq C \int_{B_{1}(0)} \mathcal{G}(u,\varphi)^2 \nonumber 
\end{align}
for some constant $C = C(n,m,q,\alpha,\varphi^{(0)},\gamma) \in (0,\infty)$.  Combining \eqref{keyest_eqn9} and \eqref{keyest_eqn11} yields the conclusion of Theorem~\ref{keyest_thm}.

To prove \eqref{keyest_eqn10}, fix $(\rho,\zeta) \in {\mathcal I}$  and let $\chi = \chi_{(\rho, \zeta)}$. Let $\overline{\varepsilon}$ and $\overline{\beta}$ be as in Lemma~\ref{graphical_lemma}.  Let $p$ be the integer such that $\varphi \in \Phi_{\varepsilon_0,p}(\varphi^{(0)})$.  If 
\begin{equation*}
	\rho^{-n-2\alpha} \int_{A_{\rho,1}(\zeta)} \mathcal{G}(u,\varphi)^2 
	\geq \frac{1}{16} \left( \frac{\overline{\beta}}{2} \right)^q \overline{\varepsilon}^2, 
\end{equation*}
by the $W^{1,2}$ estimates on $u$, 
\begin{align*}
	\int_{A_{\rho,1/2}(\zeta)} (|u|^2 + r^2 |Du|^2) 
	&\leq C \int_{A_{\rho,1}(\zeta)} |u|^2 
	\\&\leq 2C \left( \int_{A_{\rho,1}(\zeta)} |\varphi|^2 + \int_{A_{\rho,1}(\zeta)} \mathcal{G}(u,\varphi)^2 \right) 
	\\&\leq C \rho^{n+2\alpha} + C \int_{A_{\rho,1}(\zeta)} \mathcal{G}(u,\varphi)^2
	\\&\leq C \int_{A_{\rho,1}(\zeta)} \mathcal{G}(u,\varphi)^2
\end{align*}
for $C = C(n,m,q,\alpha,\varphi^{(0)},\gamma) \in (0,\infty)$ and \eqref{keyest_eqn10} follows. 

Suppose instead  that
\begin{equation} \label{keyest_eqn12}
	\rho^{-n-2\alpha} \int_{A_{\rho,1}(\zeta)} \mathcal{G}(u,\varphi)^2 
	< \frac{1}{16} \left( \frac{\overline{\beta}}{2} \right)^q \overline{\varepsilon}^2 
\end{equation}
and that either (i) $p = p_0$ or (ii) $p > p_0$ and 
\begin{equation*}
	\int_{A_{\rho,1}(\zeta)} \mathcal{G}(u,\varphi)^2 
	\leq \overline{\beta} \inf_{\varphi' \in \bigcup_{p'=p_0}^{p-1} \Phi_{c \overline{\varepsilon},p'}(\varphi^{(0)})} 
		\int_{A_{\rho,1}(\zeta)} \mathcal{G}(u,\varphi')^2, 
\end{equation*}
where $c$ is as in the statement of Lemma~\ref{graphical_lemma}.  Note that provided $\varepsilon_0 \leq \overline{\varepsilon}/4$, \eqref{keyest_eqn12} implies that $\rho^{-n-2\alpha} \int_{A_{\rho,1}(\zeta)} \mathcal{G}(u,\varphi^{(0)})^2 < \overline{\varepsilon}^2$.  Thus by Lemma~\ref{graphical_lemma} there exists $v_{j,k} : \op{graph} \varphi_{j,k} |_{A_{1,1/2}(0)} \rightarrow \mathcal{A}_{m_{j,k}}(\mathbb{R}^m)$, where $\varphi_{j,k}$ are the components of $\varphi$ with multiplicity $m_{j,k}$ as in Definition~\ref{varphi_defn}, such that $u$ is given by \eqref{phiplusv} on $A_{\rho,1/2}(\zeta)$, $v = (v_{j,k})$ is component-wise minimizing, and 
\begin{equation} \label{keyest_eqn13}
	\sup_{A_{\rho,1/2}(\zeta)} |v_{j,k,l,h}|^2 + \rho^{-n}\int_{A_{\rho,1/2}(\zeta)} r^2 |Dv_{j,k,l,h}|^2 \leq C \rho^{-n} \int_{A_{\rho,1}(\zeta)} \mathcal{G}(u,\varphi)^2 
\end{equation}
for some constant $C = C(n,m,q,\alpha,\varphi^{(0)},\gamma) \in (0,\infty)$.  Here  $v_{j,k,l,h}$ is as in \eqref{v_localized_1} where $\varphi_{j,k,l}$ is as in \eqref{varphi_localized}.  Since $u$ is given by~\eqref{phiplusv} and $v_{j,k,l,h}$ satisfy \eqref{keyest_eqn13}, we have that 
\begin{align} \label{keyest_eqn14}
	&-\int \left( \tfrac{1}{2} r^2 |Du|^2 - r D_r u_l^{\kappa} (r D_r u_l^{\kappa} - \alpha u_l^{\kappa}) - \alpha^2 |u|^2 \right) R^{-1} \psi(R) \psi'(R) \chi
	\\&\hspace{10mm} +  \int |r D_r u - \alpha u|^2 \psi'(R)^2 \chi 
	\leq \int \left(\tfrac{1}{2}r^2 |D\varphi_{j,k,l}|^2  - \alpha^{2} |\varphi_{j, k. l}|^{2}\right)R^{-1} \psi(R) \psi'(R) \chi \nonumber \\
	&\hspace{20mm} + \int\left( r^2 D\varphi_{j,k,l}^{\kappa} \cdot Dv_{j,k,l;a}^{\kappa} - r D_r \varphi_{j,k,l}^{\kappa} (r D_r v_{j,k,l;a}^{\kappa} - \alpha v_{j,k,l;a}^{\kappa}) \right) R^{-1} \psi(R) \psi'(R) \chi \nonumber \\
	&\hspace{30mm} - \int 2\alpha^2 \varphi_{j,k,l}^{\kappa} v_{j,k,l;a}^{\kappa} R^{-1} \psi(R) \psi'(R) \chi
	+ C \int_{A_{\rho,1}(\zeta)} \mathcal{G}(u,\varphi)^2 \nonumber 
\end{align}
where $v_{j,k,l;a} = \frac{1}{m_{j,k}} \sum_{h=1}^{m_{j,k}} v_{j,k,l,h}.$
Since $D_{\theta \theta} \varphi_{j,k} + \alpha^2 \varphi_{j,k} = 0$ in $A_{\rho,1/2}(\zeta)$ and $\chi(re^{i\theta},y)$ is independent of $\theta$, by integration by parts in the $\theta$ variable, 
\begin{align} \label{keyest_eqn15}
	&\int_{A_{\rho,1/4}(\zeta)} \left( \tfrac{1}{2} r^2 |D\varphi_{j,k,l}|^2 - \alpha^2 |\varphi_{j,k,l}|^2 \right) R^{-1} \psi(R) \psi'(R) \chi \\
	&\hspace{10mm} = \frac{1}{2} \int_{A_{\rho,1/4}(\zeta)} \left( |D_{\theta} \varphi_{j,k,l}|^2 - \alpha^2 |\varphi_{j,k,l}|^2 \right) R^{-1} \psi(R) \psi'(R) \chi \nonumber \\
	&\hspace{10mm} = \frac{-1}{2} \int_{A_{\rho,1/4}(\zeta)} \left( \varphi_{j,k,l}^{\kappa} D_{\theta \theta} \varphi_{j,k,l}^{\kappa} + \alpha^2 |\varphi_{j,k,l}|^2 \right) R^{-1} \psi(R) \psi'(R) \chi = 0 \nonumber
\end{align}
for every $j$ and $k$ and 
\begin{align} \label{keyest_eqn16}
	\hspace{15mm}&\hspace{-15mm}  \int_{A_{\rho,1/4}(\zeta)} \left( r^2 D\varphi_{j,k,l}^{\kappa} \cdot Dv_{j,k,l;a}^{\kappa} - r D_r \varphi_{j,k,l}^{\kappa} (r D_r v_{j,k,l;a}^{\kappa} - \alpha v_{j,k,l;a}^{\kappa}) - 2\alpha^2 \varphi_{j,k,l}^{\kappa} v_{j,k,l;a}^{\kappa} \right) R^{-1} \psi(R) \psi'(R) \chi \nonumber \\
	&\hspace{10mm} = \int_{A_{\rho,1/4}(\zeta)} \left( D_{\theta} \varphi_{j,k,l}^{\kappa} D_{\theta} v_{j,k,l;a}^{\kappa} - \alpha^2 \varphi_{j,k,l}^{\kappa} v_{j,k,l;a}^{\kappa} \right) R^{-1} \psi(R) \psi'(R) \chi \nonumber \\
	&\hspace{10mm} = - \int_{A_{\rho,1/4}(\zeta)} \left( D_{\theta \theta} \varphi_{j,k,l}^{\kappa} v_{j,k,l;a}^{\kappa} + \alpha^2 \varphi_{j,k,l}^{\kappa} v_{j,k,l;a}^{\kappa} \right) R^{-1} \psi(R) \psi'(R) \chi = 0 
\end{align}
for every $j$ and $k,$ so  \eqref{keyest_eqn10} follows from \eqref{keyest_eqn14} and \eqref{keyest_eqn16}.  

If instead \eqref{keyest_eqn12} holds true but $p > p_0$ and 
\begin{equation} \label{keyest_eqn18}
	\int_{A_{\rho,1}(\zeta)} \mathcal{G}(u,\varphi)^2 
	> \overline{\beta} \inf_{\varphi' \in \bigcup_{p'=p_0}^{p-1} \Phi_{c \overline{\varepsilon},p'}(\varphi^{(0)})} \int_{A_{\rho,1}(\zeta)} \mathcal{G}(u,\varphi')^2, 
\end{equation}
choose $p_1 \in \{p_0,p_0+1,\ldots,p-1\}$ and $\varphi^{(1)} \in \Phi_{c\overline{\varepsilon},p_1}(\varphi^{(0)})$ such that 
\begin{equation} \label{keyest_eqn19}
	\int_{A_{\rho,1}(\zeta)} \mathcal{G}(u,\varphi^{(1)})^2 
	\leq 2 \inf_{\varphi' \in \bigcup_{p'=p_0}^{p-1} \Phi_{c \overline{\varepsilon},p'}(\varphi^{(0)})} \int_{A_{\rho,1}(\zeta)} \mathcal{G}(u,\varphi')^2. 
\end{equation}
Notice that by \eqref{keyest_eqn12}, \eqref{keyest_eqn18} and \eqref{keyest_eqn19}, 
\begin{equation*}
	\rho^{-n-2\alpha} \int_{A_{\rho,1}(\zeta)} \mathcal{G}(u,\varphi^{(1)})^2 
	\leq \frac{2}{\overline{\beta}} \rho^{-n-2\alpha} \int_{A_{\rho,1}(\zeta)}\mathcal{G}(u,\varphi)^2 < \frac{\overline{\varepsilon}^2}{16}
\end{equation*}
and thus by the triangle inequality $\varphi^{(1)} \in \Phi_{\overline{\varepsilon}/2,p_1}(\varphi^{(0)})$ and $\rho^{-n-2\alpha} \int_{A_{\rho,1}(\zeta)} \mathcal{G}(u,\varphi^{(0)})^2 < \overline{\varepsilon}^2$ provided $\varepsilon_0 \leq \overline{\varepsilon}/4$.  Now if either $p_1 = p_0$ or $p_1 > p_0$ and 
\begin{equation} \label{keyest_eqn17}
	\int_{A_{\rho,1}(\zeta)} \mathcal{G}(u,\varphi^{(1)})^2 
	\leq \overline{\beta} \inf_{\varphi' \in \bigcup_{p'=p_0}^{p_1-1} \Phi_{c \overline{\varepsilon},p'}(\varphi^{(0)})} \int_{A_{\rho,1}(\zeta)} \mathcal{G}(u,\varphi')^2, 
\end{equation}
then we can repeat the argument above with $\varphi^{(1)}$ in place of $\varphi$ to get 
\begin{align*}
	&-2 \int \left( \tfrac{1}{2} r^2 |Du|^2 - r D_r u_l^{\kappa} (r D_r u_l^{\kappa} - \alpha u_l^{\kappa}) - \alpha^2 |u|^2 \right) R^{-1} \psi(R) \psi'(R) \chi 
	+ 2 \int |r D_r u - \alpha u|^2 \psi'(R)^2 \chi 
	\\& \hspace{10mm} \leq C \int_{A_{\rho,1}(\zeta)} \mathcal{G}(u,\varphi^{(1)})^2 
	\leq \frac{2C}{\overline{\beta}} \int_{A_{\rho,1}(\zeta)} \mathcal{G}(u,\varphi)^2 
\end{align*}
for some constant $C = C(n,m,q,\alpha,\varphi^{(0)},\gamma) \in (0,\infty)$.  If $p_1 > p_0$ and \eqref{keyest_eqn17} is false, then we can choose $p_2 \in \{p_0,p_0+1,\ldots,p_1-1\}$ and $\varphi^{(2)} \in \Phi_{\overline{\varphi},p_2}(\varphi^{(0)})$ such that 
\begin{equation*}
	\int_{A_{\rho,1}(\zeta)} \mathcal{G}(u,\varphi^{(2)})^2 
	\leq 2 \inf_{\varphi' \in \bigcup_{p'=p_0}^{p_1-1} \Phi_{c \overline{\varepsilon},p'}(\varphi^{(0)})} \int_{A_{\rho,1}(\zeta)} \mathcal{G}(u,\varphi')^2 
\end{equation*}
and repeat the above argument.  It is clear that at most $p-p_0$ repetitions are needed to reach \eqref{keyest_eqn10}. 
\end{proof}

\begin{lemma} \label{fourierest_lemma} 
Let $\varphi^{(0)}$ be as in Definition~\ref{varphi0_defn}.  Given $\delta \in (0,1/12)$, there exists $\varepsilon_0, \beta_0 \in (0,1)$ depending only on $n$, $m$, $q$, $\alpha$, $\varphi^{(0)}$ and $\delta$ such that the following holds: Suppose that $\varphi \in \Phi_{\varepsilon_0,p}(\varphi^{(0)})$ for some $p \in \{p_0,p_0+1,\ldots, \lceil q/q_0 \rceil \}$ and that $u \in W^{1,2}(B_1(0);\mathcal{A}_q(\mathbb{R}^m))$ is an average-free energy minimizing function satisfying Hypothesis~$(\star)$, Hypothesis~$(\star\star)$ of Section~\ref{sec:graphical_sec} and the condition 
\begin{equation} \label{fourierest_eqn1}
	B_{\delta}(0,y_0) \cap \{ X \in B_{1/2}(0) \cap \Sigma_{u,q} : \mathcal{N}_u(X) \geq \alpha \} \neq \emptyset 
\end{equation} 
for all $y_0 \in B^{n-2}_{1/2}(0)$.  Let $v_{j,k} : \op{graph} \varphi_{j,k} |_{B_{1/2}(0) \cap \{r > \delta\}} \rightarrow \mathcal{A}_{m_{j,k}}(\mathbb{R}^m)$ be as in Corollary \ref{graphical_cor} with $\gamma = 1/2$ and $\tau = \delta$.  Then for each function $\zeta(x,y) = \widetilde{\zeta}(|x|,y)$ where $\widetilde{\zeta} = \widetilde{\zeta}(r,y) \in C^2_c(B^{n-1}_{1/2}(0))$ with $D_r \widetilde{\zeta}(r,y) = 0$ whenever $r \leq \delta$, we have that 
\begin{align*} 
	&\left| \sum_{j=1}^J \sum_{k=1}^{p_j} m_{j,k} \int_{B^{n-2}_{1/2}(0)} \int_{\delta}^{\infty} \int_0^{2\pi} \sum_{l=1}^{q_{j,k}} 
		r^{2\alpha-1} D(r^{2-2\alpha} \, v_{j,k,l;a}^{\kappa} \, D_{\iota} \varphi_{j,k,l}^{\kappa}) \cdot DD_{y_{\nu}} \zeta \,d\theta \,dr \,dy \right| 
	\\&\hspace{10mm} \leq C \left( \delta^{-2} \int_{B_1(0)} \mathcal{G}(u,\varphi)^2 
		+ \delta^{2\alpha} \left( \int_{B_1(0)} \mathcal{G}(u,\varphi)^2 \right)^{1/2} \right) \sup_{B_{1/2}(0)} |DD_{y_{\nu}} \zeta|
\end{align*}
for each $\iota = 1,2$, each $\nu = 1,2,\ldots,n-2$, and some constant $C = C(n,m,q,\alpha,\varphi^{(0)}) \in (0,\infty)$, where $\varphi_{j,k,l}$ is as in \eqref{varphi_localized} and $v_{j,k,l;a} = \frac{1}{m_{j,k}} \sum_{h=1}^{m_{j,k}} v_{j,k,l,h}$ with $v_{j,k,l,h}$ as in \eqref{v_localized_1}.
\end{lemma}
\begin{proof}  
Fix $\iota \in \{1,2\}$ and replace $\zeta^j$ with $\delta_{\iota j} D_{y_{\nu}} \zeta(r,y)$ in \eqref{freqidentity2} and recall that $u$ is given by \eqref{phiplusv} on $B_{1/2}(0) \cap \{r \geq \delta\}$ to obtain  
\begin{align*}
	&\frac{-1}{2} \int_{B_{1/2}(0) \cap \{r \geq \delta\}} (m_{j,k} |D\varphi_{j,k,l}|^2 
		+ 2m_{j,k} D\varphi_{j,k,l}^{\kappa} \cdot Dv_{j,k,l,a}^{\kappa} + |Dv_{j,k,l,h}|^2) D_{\iota} D_{y_{\nu}} \zeta 
	\\&\hspace{10mm} + \int_{B_{1/2}(0) \cap \{r \geq \delta\}} (m_{j,k} D_{\iota} \varphi_{j,k,l}^{\kappa} D\varphi_{j,k,l}^{\kappa} 
		+ m_{j,k} D_{\iota} \varphi_{j,k,l}^{\kappa} Dv_{j,k,l,a}^{\kappa}) \cdot DD_{y_{\nu}} \zeta 
	\\&\hspace{10mm} + \int_{B_{1/2}(0) \cap \{r \geq \delta\}} (m_{j,k} D_{\iota} v_{j,k,l,a}^{\kappa} D\varphi_{j,k,l}^{\kappa} 
		+ D_{\iota} v_{j,k,l,h}^{\kappa} Dv_{j,k,l,h}^{\kappa}) \cdot DD_{y_{\nu}} \zeta 
	\\&\hspace{10mm} = -\int_{B_{1/2}(0) \cap \{r \leq \delta\}} D_{\iota} u_l^{\kappa} D_y u_l^{\kappa} \cdot D_y D_{y_{\nu}} \zeta, 
\end{align*}
where we use the convention of summing over $j,k,l$ and repeated indices.  Observe that, since $\varphi$ is independent of $y$, by integration by parts 
\begin{equation*}
	\int_{B_{1/2}(0) \cap \{r \geq \delta\}} \left( \frac{-1}{2} |D\varphi_{j,k,l}|^2 D_{\iota} D_{y_{\nu}} \zeta 
		+ D_{\iota} \varphi_{j,k,l}^{\kappa} D\varphi_{j,k,l}^{\kappa} \cdot DD_{y_{\nu}} \zeta \right) = 0 
\end{equation*}
for all $j,k,l$.  Thus 
\begin{align} \label{fourierest_eqn2}
	\hspace{20mm} & \hspace{-20mm} m_{j,k} \int_{B_{1/2}(0) \cap \{r \geq \delta\}} 
		(-D\varphi_{j,k,l}^{\kappa} \cdot Dv_{j,k,l,a}^{\kappa} D_{\iota} D_{y_{\nu}} \zeta 
		+ D_{\iota} \varphi_{j,k,l}^{\kappa} Dv_{j,k,l,a}^{\kappa} \cdot DD_{y_{\nu}} \zeta 
		+ D_{\iota} v_{j,k,l,a}^{\kappa} D\varphi_{j,k,l}^{\kappa} \cdot DD_{y_{\nu}} \zeta) \nonumber \\
	&\hspace{-10mm} = -\int_{B_{1/2}(0) \cap \{r \geq \delta\}} \left( \frac{-1}{2} |Dv_{j,k,l,h}|^2 D_{\iota} D_{y_{\nu}} \zeta 
		+ D_{\iota} v_{j,k,l,h}^{\kappa} Dv_{j,k,l,h}^{\kappa} \cdot DD_{y_{\nu}} \zeta \right) \nonumber \\ 
	& - \int_{B_{1/2}(0) \cap \{r \leq \delta\}} D_{\iota} u_l^{\kappa} D_y u_l^{\kappa} \cdot D_y D_{y_{\nu}} \zeta.  
\end{align}
By the estimates on $v$ in Corollary \ref{graphical_cor}, 
\begin{equation*}
	\int_{B_{1/2}(0) \cap \{r \geq \delta\}} r^2 |Dv_{j,k,l,h}|^2 \leq C \int_{B_1(0)} \mathcal{G}(u,\varphi)^2 
\end{equation*}
for some constant $C = C(n,m,q,\alpha,\varphi^{(0)}) \in (0,\infty)$, so 
\begin{equation} \label{fourierest_eqn3}
	\int_{B_{1/2}(0) \cap \{r \geq \delta\}} |Dv_{j,k,l,h}|^2 |DD_{y_{\nu}} \zeta| 
	\leq C \delta^{-2} \int_{B_1(0)} \mathcal{G}(u,\varphi)^2 \sup_{B_{1/2}(0)} |DD_{y_{\nu}} \zeta| 
\end{equation}
for some constant $C = C(n,m,q,\alpha,\varphi^{(0)}) \in (0,\infty)$.  By \eqref{fourierest_eqn1}, for every $y \in B^{n-2}_{1/2}(0)$ there exists a $Z \in B_{\delta}(0,y) \cap B_{1/2}(0)$ such that $\mathcal{N}_u(Z) \geq \alpha$ and thus by \eqref{regestimate} and \eqref{consequence_eqn2},
\begin{equation*}
	\int_{B_{2\delta}(0,y)} |Du|^2 \leq \int_{B_{3\delta}(Z)} |Du|^2 \leq C\delta^{-2} \int_{B_{6\delta}(Z)} |u|^2 \leq C \delta^{n-2+2\alpha} \int_{B_{1/2}(Z)} |u|^2 
	\leq C\delta^{n-2+2\alpha} 
\end{equation*}
for some constant $C = C(n,m,q,\alpha,\varphi^{(0)}) \in (0,\infty)$ whence by a standard covering argument, 
\begin{equation} \label{fourierest_eqn4}
	\int_{B_{1/2}(0) \cap \{r \leq \delta\}} |Du|^2 \leq C\delta^{2\alpha} 
\end{equation}
for some constant $C = C(n,m,q,\alpha,\varphi^{(0)}) \in (0,\infty)$.  By \eqref{fourierest_eqn4} and Theorem~\ref{keyest_thm}(b), 
\begin{align} \label{fourierest_eqn5}
	&\left| \int_{B_{1/2}(0) \cap \{r \leq \delta\}} D_{\iota} u_l^{\kappa} D_y u_l^{\kappa} \cdot D_y D_{y_{\nu}} \zeta \right| \\
	&\hspace{10mm} \leq \left( \int_{B_{1/2}(0) \cap \{r \leq \delta\}} |Du|^2 \right)^{1/2} \left( \int_{B_{1/2}(0)} |D_y u|^2 \right)^{1/2} \sup_{B_{1/2}(0)} |DD_{y_{\nu}} \zeta| \nonumber \\
	&\hspace{10mm} \leq C\delta^{2\alpha} \left( \int_{B_1(0)} \mathcal{G}(u,\varphi)^2 \right)^{1/2} \sup_{B_{1/2}(0)} |DD_{y_{\nu}} \zeta| \nonumber
\end{align}
for some constant $C = C(n,m,q,\alpha,\varphi^{(0)}) \in (0,\infty)$.  Using \eqref{fourierest_eqn3} and \eqref{fourierest_eqn5} to bound the right-hand side of
\eqref{fourierest_eqn2}, we get 
\begin{align} \label{fourierest_eqn6}
	&m_{j,k} \int_{B_{1/2}(0) \cap \{r \geq \delta\}} 
		(-D\varphi_{j,k,l}^{\kappa} \cdot Dv_{j,k,l,a}^{\kappa} D_{\iota} D_{y_{\nu}} \zeta 
		+ D_{\iota} \varphi_{j,k,l}^{\kappa} Dv_{j,k,l,a}^{\kappa} \cdot DD_{y_{\nu}} \zeta) 
	\\&\hspace{10mm} + m_{j,k} \int_{B_{1/2}(0) \cap \{r \geq \delta\}} D_{\iota} v_{j,k,l,a}^{\kappa} D\varphi_{j,k,l}^{\kappa} \cdot DD_{y_{\nu}} \zeta 
		= \mathcal{R} \nonumber 
\end{align}
where 
\begin{equation*}
	|\mathcal{R}| \leq C \left( \delta^{-2} \int_{B_1(0)} \mathcal{G}(u,\varphi)^2 
		+ \delta^{2\alpha} \left( \int_{B_1(0)} \mathcal{G}(u,\varphi)^2 \right)^{1/2} \right) \sup_{B_{1/2}(0)} |DD_{y_{\nu}} \zeta|
\end{equation*}
for some constant $C = C(n,m,q,\alpha,\varphi^{(0)}) \in (0,\infty)$.  Now observe that, using the fact that $\varphi$ is homogeneous degree $\alpha$ and independent of $y$ and $\zeta$ depends only on $r$ and $y$,
\begin{align*}
	&m_{j,k} \int_{B^{n-2}_{1/2}(0)} \int_{\delta}^{\infty} \int_0^{2\pi q_0} 
		r^{2\alpha-1} D(r^{2-2\alpha} v_{j,k,l,a}^{\kappa} D_{\iota} \varphi_{j,k,l}^{\kappa}) \cdot DD_{y_{\nu}} \zeta \,d\theta \,dr \,dy 
	\\&\hspace{10mm} = m_{j,k} \int_{B^{n-2}_{1/2}(0)} \int_{\delta}^{\infty} \int_0^{2\pi q_0} 
		\left( (1-\alpha) v_{j,k,l,a}^{\kappa} D_{\iota} \varphi_{j,k,l}^{\kappa} D_r D_{y_{\nu}} \zeta 
		+ r D_{\iota} \varphi_{j,k,l}^{\kappa} Dv_{j,k,l,a}^{\kappa} \cdot DD_{y_{\nu}} \zeta \right) \,d\theta \,dr \,dy 
\end{align*}
Using \eqref{fourierest_eqn6} to substitute for the integral of $r D_{\iota} \varphi_{j,k,l}^{\kappa} Dv_{j,k,l,a}^{\kappa} \cdot DD_{y_{\nu}} \zeta $, 
\begin{align*}
	&m_{j,k} \int_{B^{n-2}_{1/2}(0)} \int_{\delta}^{\infty} \int_0^{2\pi q_0} 
		r^{2\alpha-1} D(r^{2-2\alpha} v_{j,k,l,a}^{\kappa} D_{\iota} \varphi_{j,k,l}^{\kappa}) \cdot DD_{y_{\nu}} \zeta \,d\theta \,dr \,dy 
	\\&\hspace{10mm} = m_{j,k} \int_{B^{n-2}_{1/2}(0)} \int_{\delta}^{\infty} \int_0^{2\pi q_0} 
		\left( (1-\alpha) v_{j,k,l,a}^{\kappa} D_{\iota} \varphi_{j,k,l}^{\kappa} D_r D_{y_{\nu}} \zeta 
		+ r Dv_{j,k,l,a}^{\kappa} \cdot D\varphi_{j,k,l}^{\kappa} D_{\iota} D_{y_{\nu}} \zeta \right) \,d\theta \,dr \,dy 
	\\&\hspace{20mm} - m_{j,k} \int_{B^{n-2}_{1/2}(0)} \int_{\delta}^{\infty} \int_0^{2\pi q_0} 
		r D_{\iota} v_{j,k,l,a}^{\kappa} D\varphi_{j,k,l}^{\kappa} \cdot DD_{y_{\nu}} \zeta \,d\theta \,dr \,dy + \mathcal{R}. 
\end{align*}
Again since $\varphi$ is homogeneous degree $\alpha$,  locally given by harmonic functions away from $\{r = 0\}$, and independent of $y$ and $\zeta$ depends only on $r$ and $y$, 
\begin{align*}
	&m_{j,k} \int_{B^{n-2}_{1/2}(0)} \int_{\delta}^{\infty} \int_0^{2\pi q_0} 
		r^{2\alpha-1} D(r^{2-2\alpha} v_{j,k,l,a}^{\kappa} D_{\iota} \varphi_{j,k,l}^{\kappa}) \cdot DD_{y_{\nu}} \zeta \,d\theta \,dr \,dy 
	\\&\hspace{10mm} = m_{j,k} \int_{B^{n-2}_{1/2}(0)} \int_{\delta}^{\infty} \int_0^{2\pi q_0} 
		\left( (1-\alpha) v_{j,k,l,a}^{\kappa} D_{\iota} \varphi_{j,k,l}^{\kappa} 
		+ x_{\iota} Dv_{j,k,l,a}^{\kappa} \cdot D\varphi_{j,k,l}^{\kappa} \right) D_r D_{y_{\nu}} \zeta \,d\theta \,dr \,dy 
		\\&\hspace{20mm} - m_{j,k} \int_{B^{n-2}_{1/2}(0)} \int_{\delta}^{\infty} \int_0^{2\pi q_0} 
		\alpha D_{\iota} v_{j,k,l,a}^{\kappa} \varphi_{j,k,l}^{\kappa} D_r D_{y_{\nu}} \zeta \,d\theta \,dr \,dy + \mathcal{R} 
	\\&\hspace{10mm} = m_{j,k} \int_{B^{n-2}_{1/2}(0)} \int_{\delta}^{\infty} \int_0^{2\pi q_0} 
		\left( -\alpha D_{\iota} (v_{j,k,l,a}^{\kappa} \varphi_{j,k,l}^{\kappa}) 
		+ \op{div}(x_{\iota} v_{j,k,l,a}^{\kappa} D\varphi_{j,k,l}^{\kappa}) \right) D_r D_{y_{\nu}} \zeta \,d\theta \,dr \,dy + \mathcal{R} 
	\\&\hspace{10mm} = m_{j,k} \int_{B_{1/2}(0) \cap \{r \geq \delta\}} 
		\left( -\alpha D_{\iota} (v_{j,k,l,a}^{\kappa} \varphi_{j,k,l}^{\kappa}) 
		+ \op{div}(x_{\iota} v_{j,k,l,a}^{\kappa} D\varphi_{j,k,l}^{\kappa}) \right) r^{-1} D_r D_{y_{\nu}} \zeta + \mathcal{R}. 
\end{align*}
By integrating by parts we deduce from the preceding line that
\begin{align*}
	&m_{j,k} \int_{B^{n-2}_{1/2}(0)} \int_{\delta}^{\infty} \int_0^{2\pi q_0} 
		r^{2\alpha-1} D(r^{2-2\alpha} v_{j,k,l,a}^{\kappa} D_{\iota} \varphi_{j,k,l}^{\kappa}) \cdot DD_{y_{\nu}} \zeta 
	\\&\hspace{10mm} = -m_{j,k} \int_{B_{1/2}(0) \cap \{r = \delta\}} 
		\left( -\alpha \frac{x_{\iota}}{r} v_{j,k,l,a}^{\kappa} \varphi_{j,k,l}^{\kappa} 
		+ x_{\iota} v_{j,k,l,a}^{\kappa} D_r \varphi_{j,k,l}^{\kappa} \right) r^{-1} D_r D_{y_{\nu}} \zeta 
	\\&\hspace{20mm} - m_{j,k} \int_{B_{1/2}(0) \cap \{r \geq \delta\}} 
		\left( -\alpha \frac{x_{\iota}}{r} v_{j,k,l,a}^{\kappa} \varphi_{j,k,l}^{\kappa} 
		+ x_{\iota} v_{j,k,l,a}^{\kappa} D_r \varphi_{j,k,l}^{\kappa} \right) D_r (r^{-1} D_r D_{y_{\nu}} \zeta) + \mathcal{R}. 
\end{align*}
Since $\varphi$ is homogeneous degree $\alpha$, $D_r \varphi = \alpha r^{-1} \varphi$ and thus
\begin{equation*}
	m_{j,k} \int_{B^{n-2}_{1/2}(0)} \int_{\delta}^{\infty} \int_0^{2\pi q_0} 
		r^{2\alpha-1} D(r^{2-2\alpha} v_{j,k,l,a}^{\kappa} D_{\iota} \varphi_{j,k,l}^{\kappa}) \cdot DD_{y_{\nu}} \zeta \,d\theta \,dr \,dy = \mathcal{R}, 
\end{equation*}
completing the proof.
\end{proof}

\begin{lemma} \label{competitor_lemma} 
Let $\varphi^{(0)}$ be as in Definition~\ref{varphi0_defn}.  Given $\delta \in (0,1/16)$, there exists $\varepsilon_0, \beta_0 \in (0,1)$ depending only on $n$, $m$, $q$, $\alpha$, $\varphi^{(0)}$, and $\delta$ such that the following holds true.  Suppose that $u$, $\varphi$ satisfy Hypothesis~$(\star)$ and Hypothesis~$(\star\star)$ of Section~\ref{sec:graphical_sec}.   
Let $v_{j,k} : \op{graph} \varphi_{j,k} |_{B_{3/4}(0) \cap \{r > \delta/8\}} \rightarrow \mathcal{A}_{m_{j,k}}(\mathbb{R}^m)$ be as in Corollary \ref{graphical_cor} with $\gamma = 1/8$ and $\tau = \delta/8$ and let $v_{j,k,l}$ be as in \eqref{v_localized}.   Then for all $\zeta_{j,k} \in C^1_c(B_{1/16}(0))$ with $|\zeta_{j,k}| \leq 1/16$ and $|D\zeta_{j,k}| \leq 1$ and for all $Z \in B_{1/16}(0)$ with $\op{dist}(Z,\{0\} \times \mathbb{R}^{n-2}) < \delta/2$, 
\begin{equation} \label{competitor_eqn2}
	\int_{B_{1/16}(0) \cap \{r > \delta\}} |Dv_{j,k,l}|^2 
	\leq \int_{B_{1/16}(0) \cap \{r > \delta\}} |D\widetilde{v}_{j,k,l}|^2 + C \delta^{-2} \int_{B_{1/8}(0) \cap \{\delta/8 < r < 2\delta\}} |v_{j,k,l}|^2
\end{equation}
for some constant $C = C(n,m,q) \in (0,\infty)$, where $\widetilde{v}_{j,k,l}(X) = v_{j,k,l}(X+\zeta_{j,k}(X) (X-Z))$ in $B_{1/4}(0) \cap \{r > \delta/2\}$ and we use the convention of summing over $j,k,l$ and repeated indices.
\end{lemma}
\begin{proof}
Recall from Chapter 2 of~\cite{Almgren} that for every positive integer $q$ there exists a positive integer $N = N(m,q)$, an injective Lipschitz map $\bm{\xi} : \mathcal{A}_q(\mathbb{R}^m) \rightarrow \mathbb{R}^N$ such that $\op{Lip} \bm{\xi} \leq 1$ and $\op{Lip}((\bm{\xi} |_{\mathcal{Q}})^{-1}) \leq C(m,q)$, and a Lipschitz map $\bm{\rho} : \mathbb{R}^N \rightarrow \mathcal{Q}$ such that $\bm{\rho} |_{\mathcal{Q}}$ is the identity map, where $\mathcal{Q} = \bm{\xi}(\mathcal{A}_q(\mathbb{R}^m))$ (as a slight abuse of notation we omit the dependence on $q$, which is obvious from the context).  Let $\chi : [0,\infty) \rightarrow [0,1]$ be a smooth function such that $\chi(r) = 0$ for $r \in [0,\delta/2]$, $\chi(r) = 1$ for $r \geq \delta$, and $|\chi'(r)| \leq 3/\delta$.  

For each $j$ and $k$, define a function $w_{j,k} : \op{graph} \varphi_{j,k} |_{B_{1/16}(0) \cap \{r > \delta/2\}} \rightarrow \mathcal{A}_{m_{j,k}}(\mathbb{R}^m)$ as follows.  Let $\theta_0 \in [0,2\pi)$ arbitrary and let $W' = B_{1/16}(0) \cap \{ (re^{i\theta},y) : r > \delta/2, \,|\theta - \theta_0| < \pi/4 \}$ and $W = B_{1/8}(0) \cap \{ (re^{i\theta},y) : r > \delta/4, \,|\theta - \theta_0| < \pi/2 \}$.  Let $\varphi_{j,k}(X) = \sum_{l=1}^{q_{j,k}} \llbracket \varphi_{j,k,l}(X) \rrbracket$ for all $X \in W$ where $\varphi_{j,k,l} : W \rightarrow \mathbb{R}^m$ are single-valued harmonic functions (as in~\eqref{varphi_localized}).  Let $v_{j,k,l}(X) = v_{j,k}(X,\varphi_{j,k,l}(X))$ for all $X \in W$ (as in~\eqref{v_localized}).  Let $\xi$ be the projection of $Z$ onto $\mathbb{R}^2 \times \{0\}$.  Since $\theta_0$ is arbitrary, in order to define $w_{j,k}$, it suffices to define $w_{j,k,l}(X) = w_{j,k}(X,\varphi_{j,k,l}(X))$ for each $X \in W'$ and $l = 1,2,\ldots,q_{j,k}$.  Notice that for every $X = (x,y) \in W'$, $|\zeta_{j,k}(X)| \,|X-Z| \leq 1/128$ and $|\zeta_{j,k}(X)| \,|x-\xi| \leq 1/16 \cdot (|x| + \delta/2) < |x|/8$ and thus $X + \zeta_{j,k}(X) (X-Z) \in W$.  For each $X = (x,y) \in W' \cap \{r \geq \delta\}$, define  
\begin{equation} \label{competitor_eqn3}
	w_{j,k,l}(X) = w_{j,k}(X,\varphi_{j,k,l}(X)) = v_{j,k,l}(X + \zeta_{j,k}(X)(X-Z)) .
\end{equation}
For each $X = (x,y) \in W' \cap \{\delta/2 < r < \delta\}$, define  
\begin{align} \label{competitor_eqn4}
	w_{j,k,l}(X) &= w_{j,k}(X,\varphi_{j,k,l}(X))  
	\\&= (\bm{\xi}^{-1} \circ \bm{\rho})[(1-\chi(r)) \, \bm{\xi}[v_{j,k,l}(X)] + \chi(r) \, \bm{\xi}[v_{j,k,l}(X + \zeta_{j,k}(X)(X-Z))]] . \nonumber
\end{align}
Define $\widetilde{u}  \in W^{1,2}(B_1(0);\mathcal{A}_q(\mathbb{R}^m))$ by 
\begin{equation*}
	\widetilde{u}(X) = \sum_{j=1}^J \sum_{k=1}^{p_j} \sum_{l=1}^{q_{j,k}} \sum_{h=1}^{m_{j,k}} \left\llbracket \varphi_{j,k,l}(X) + w_{j,k,l,h}(X) \right\rrbracket 
\end{equation*}
for every $X \in B_{1/16}(0) \cap \{r > \delta/2\}$, where $w_{j,k,l}(X) = \sum_{h=1}^{m_{j,k}} \llbracket w_{j,k,l,h}(X) \rrbracket$ for $w_{j,k,l,h}(X) \in \mathbb{R}^m$, and $\widetilde{u} = u$ on $(B_1(0) \setminus B_{1/16}(0)) \cup (B_{1/16}(0) \cap \{r \leq \delta/2\})$. 

Since $u$ is energy minimizing and $u = \widetilde{u}$ in $(B_1(0) \setminus B_{1/16}(0)) \cup (B_{1/16}(0) \cap \{r \leq \delta/2\})$, 
\begin{equation*}
	\int_{B_{1/16}(0) \cap \{r > \delta/2\}} |D\varphi_{j,k,l} + Dv_{j,k,l,h}|^2 
	\leq \int_{B_{1/16}(0) \cap \{r > \delta/2\}} |D\varphi_{j,k,l} + Dw_{j,k,l,h}|^2, 
\end{equation*}
using the convention of summing over $j,k,l,h$.  After some cancellations, 
\begin{align} \label{competitor_eqn5}
	\int_{B_{1/16}(0) \cap \{r > \delta\}} |Dv_{j,k,l}|^2
	\leq {}&{} 2m_{j,k} \int_{B_{1/16}(0) \cap \{r > \delta/2\}} D\varphi_{j,k,l} \cdot (Dw_{j,k,l;a} - Dv_{j,k,l;a}) \\&
		+ \int_{B_{1/16}(0) \cap \{r > \delta/2\}} |Dw_{j,k,l}|^2, \nonumber 
\end{align}
where $w_{j,k,l;a}(X) = \frac{1}{m_{j,k}} \sum_{h=1}^{m_{j,k}} w_{j,k,l,h}(X)$ denotes the average of the values of $w_{j,k,l}(X)$.  Since $\varphi_{j,k,l}$ and $v_{j,k,l;a}$ are (single-valued) harmonic functions defined locally in $B_{1/16}(0) \cap \{r > \delta/2\}$ and $v_{j,k,l;a} = w_{j,k,l;a}$ on $\partial (B_{1/16}(0) \cap \{r > \delta/2\})$, by integration by parts, 
\begin{equation} \label{competitor_eqn6}
	m_{j,k} \int_{B_{1/16}(0) \cap \{r > \delta/2\}}  D\varphi_{j,k,l} \cdot (Dw_{j,k,l;a} - Dv_{j,k,l;a}) = 0. 
\end{equation}
Note that since $|\zeta_{j,k}| \leq 1/16$ and $|D\zeta_{j,k}| \leq 1$, by the inverse function theorem, $X \in B_{1/16}(0) \mapsto X + \zeta_{j,k}(X) (X-Z)$ is invertible.  By \eqref{competitor_eqn4}, $\bm{\xi}, \bm{\xi}^{-1}, \bm{\rho}$ being Lipschitz, and the estimates of \eqref{regestimate}, 
\begin{align} \label{competitor_eqn7}
	\int_{B_{1/16}(0) \cap \{\delta/2 < r < \delta\}} |Dw_{j,k,l}|^2 
	&\leq C \int_{B_{5/64}(0) \cap \{\delta/4 < r < 5\delta/4\}} (r^{-2} |v_{j,k,l}|^2 + |Dv_{j,k,l}|^2) \\
	&\leq C \delta^{-2} \int_{B_{1/8}(0) \cap \{\delta/8 < r < 2\delta\}} |v_{j,k,l}|^2 \nonumber 
\end{align}
for some constant $C = C(n,m,q) \in (0,\infty)$.  By combining \eqref{competitor_eqn3}, \eqref{competitor_eqn5}, \eqref{competitor_eqn6}, and \eqref{competitor_eqn7}, we get \eqref{competitor_eqn2}. 
\end{proof}

\section{A priori estimates: Part II}\label{sec:L2estimates_sec2}
Let $\varphi^{(0)}$ be as in Definition~\ref{varphi0_defn} and recall that the degree of homogeneity of $\varphi^{(0)}$ is $\alpha.$ Let $u$ be an average free Dirichlet energy minimizer  and let  $\varphi \in \Phi_{\varepsilon_0}(\varphi^{(0)})$ for some appropriately small $\epsilon_{0} > 0$. In this section (in Lemma~\ref{keyest_cor}, Lemma~\ref{branchdist_lemma}, Lemma~\ref{nonconest_lemma} and Corollary~\ref{nonconest_cor} below), we draw some important corollaries of Lemma~\ref{keyest_thm}(a), giving in particular an estimate on the distance of the set $\Sigma_{u, q, \alpha}^{+} = B_{1/2} \cap \{Z \, : \, {\mathcal N}_{u}(Z) \geq \alpha\}$ from the axis $\{0\} \times {\mathbb R}^{n-2}$ of $\varphi$  (Lemma~\ref{branchdist_lemma}), and integral estimates implying non-concentration of the excess 
$\int_{B_{1}} {\mathcal G} \, (u, \varphi)^{2}$ near 
$\Sigma_{u, q, \alpha}^{+}$ (Lemma~\ref{keyest_cor} and Lemma~\ref{nonconest_lemma}). All of the results in this section use the corresponding arguments in~\cite{Sim93}, although because of the presence of higher multiplicity the proofs of Lemma~\ref{branchdist_lemma} and Lemma~\ref{nonconest_lemma}  have to proceed via a strategy used in \cite{Wic14}. This strategy involves a preliminary result, Lemma~\ref{branchdist2_lemma}, which gives a weaker bound  on the distance of $\Sigma_{u, q, \alpha}^{+}$ to $\{0\} \times \mathbb{R}^{n-2}$ than does Lemma~\ref{branchdist_lemma}.  The proof of this preliminary result involves a blow-up argument which relies on certain conditional versions of Lemma~\ref{branchdist_lemma} and Lemma~\ref{nonconest_lemma} themselves. In the end, an induction argument (inducting on the number of distinct non-zero components of $\varphi$) will prove both Lemma~\ref{branchdist_lemma} and Lemma~\ref{nonconest_lemma} simultaneously in the required generality.  

\begin{lemma} \label{keyest_cor}
Let $\varphi^{(0)}$ be as in Definition~\ref{varphi0_defn}.  Given $\gamma,\sigma \in (0,1)$, there exists $\varepsilon_0 \in (0,1)$ depending only on $n$, $m$, $q$, $\alpha$, $\varphi^{(0)}$ and $\gamma$ such that if $\varphi \in \Phi_{\varepsilon_0}(\varphi^{(0)})$ and if $u \in W^{1,2}(B_1(0);\mathcal{A}_q(\mathbb{R}^m))$ is an average-free, energy minimizing $q$-valued function with  $0 \in \Sigma_{u,q}$ and 
$\mathcal{N}_u(0) \geq \alpha$
then 
\begin{equation*}
	\int_{B_{\gamma}(0)} R^{-n-2\alpha+\sigma} \mathcal{G}(u,\varphi)^2 \leq C \int_{B_1(0)} \mathcal{G}(u,\varphi)^2 
\end{equation*}
for some constant $C = C(n,m,q,\alpha,\varphi^{(0)},\gamma,\sigma) \in (0,\infty)$, where $R = |X|$.  
\end{lemma}
\begin{proof}
Recall that for any (single-valued) vector field $\zeta = (\zeta^1,\ldots,\zeta^n) \in W^{1,1}_0(\mathbb{R}^n)$, 
\begin{equation*}
	\int_{\mathbb{R}^n} D_i \zeta^i = 0. 
\end{equation*}
Note that since $u,\varphi \in C^0(B_1(0);\mathcal{A}_q(\mathbb{R}^m)) \cap W^{1,2}(B_1(0);\mathcal{A}_q(\mathbb{R}^m))$ and the singular sets of $u$ and $\varphi$ have Hausdorff dimension at most $n-2$, it is easy to check that $\mathcal{G}(u,\varphi)^2 \in W^{1,1}_{\text{loc}}(B_1(0))$.  Taking $\zeta^i = \psi(R)^2 \eta_{\delta}(R) R^{-n+\sigma-2\alpha} \mathcal{G}(u,\varphi)^2 X^i$ in this where $\psi$ is as in the proof of Theorem~\ref{keyest_thm} and for each $\delta > 0$, $\eta_{\delta} \in C^1([0,\infty))$ is a non-decreasing function such that $\eta_{\delta}(t) = 0$ for $t \in [0,\delta/2]$, $\eta_{\delta}(t) = 1$ for $t \in [\delta,\infty)$ and $|D\eta_{\delta}| \leq 3/\delta$ for all $t \in [0,\infty)$, we obtain
\begin{align} \label{keyest_cor_eqn1}
	&\sigma \int \psi(R)^2 \eta_{\delta}(R) R^{-n+\sigma-2\alpha} \mathcal{G}(u,\varphi)^2 
	= -\int \psi(R)^2 \eta_{\delta}(R) R^{1-n+\sigma} D_R (R^{-2\alpha} \mathcal{G}(u,\varphi)^2) \\
	&\hspace{10mm} - 2 \int \psi(R) \psi'(R) \eta_{\delta}(R) R^{1-n+\sigma-2\alpha} \mathcal{G}(u,\varphi)^2
	- \int \psi(R)^2 \eta'_{\delta}(R) R^{1-n+\sigma-2\alpha} \mathcal{G}(u,\varphi)^2. \nonumber
\end{align}
Observe that 
\begin{equation*}
	|D_R (R^{-2\alpha} \mathcal{G}(u,\varphi)^2)| 
	= |D_R (\mathcal{G}(u/R^{\alpha},\varphi/R^{\alpha})^2)| 
	\leq 2 \mathcal{G}(u/R^{\alpha},\varphi/R^{\alpha}) |D_R (u/R^{\alpha})| 
\end{equation*} 
a.e. in $B_1(0)$. Thus using the Cauchy-Schwartz inequality in \eqref{keyest_cor_eqn1} and using Theorem~\ref{keyest_thm}, we obtain (after dropping the last term on the right hand side), 
\begin{align} \label{keyest_cor_eqn2}
	\int \psi(R)^2 \eta_{\delta}(R) R^{-n+\sigma-2\alpha} \mathcal{G}(u,\varphi)^2 
	&\leq \frac{9}{\sigma^2} \int \left( \psi(R)^2 R^{2-n+\sigma} |D_R (u/R^{\alpha})|^2 
		+ \psi'(R)^2 R^{2-n+\sigma-2\alpha} \mathcal{G}(u,\varphi)^2 \right) \nonumber \\
	&\leq C \int \mathcal{G}(u,\varphi)^2 
\end{align}
for some constant $C = C(n,m,q,\alpha,\gamma,\sigma) \in (0,\infty)$.  
Letting $\delta \downarrow 0$ in \eqref{keyest_cor_eqn2} using the monotone convergence theorem gives the desired conclusion.
\end{proof} 

The next two main results, Lemma~\ref{branchdist_lemma} and Lemma~\ref{nonconest_lemma}, concern a point $Z = (\xi,\zeta) \in \Sigma_{u,q} \cap B_{1/2}(0)$ such that $\mathcal{N}_u(Z) \geq \alpha$.  We will first state these results and then prove them both simultaneously by an inductive argument  with the help of a preliminary estimate given in Lemma~\ref{branchdist2_lemma}.

\begin{lemma} \label{branchdist_lemma}
Let $\varphi^{(0)}$ be as in Definition~\ref{varphi0_defn} and $p \in \{{p_0}, {p_0}+1,\ldots, \lceil q/q_0 \rceil \}$.  There exists $\varepsilon_0 \in (0,1)$ depending only on $n$, $m$, $q$, $\alpha$ and $\varphi^{(0)}$ such that if $u$ satisfies Hypothesis~$(\star)$ of Section~\ref{sec:graphical_sec},
$\varphi \in \Phi_{\varepsilon_0,p}(\varphi^{(0)})$ and if $Z \in \Sigma_{u,q} \cap B_{1/2}(0)$ with $\mathcal{N}_u(Z) \geq \alpha$, then 
\begin{align*}
	&(a) \hspace{5mm} \op{dist}^2(Z,\{0\} \times \mathbb{R}^{n-2}) \leq C \int_{B_1(0)} \mathcal{G}(u,\varphi)^2, \\ 
	&(b) \hspace{5mm} \int_{B_1(0)} \mathcal{G}(u(X),\varphi(X-Z))^2 dX \leq C \int_{B_1(0)} \mathcal{G}(u,\varphi)^2 
\end{align*}
for some constant $C = C(n,m,q,\alpha,\varphi^{(0)}) \in (0,\infty)$. 
\end{lemma}

\begin{lemma} \label{nonconest_lemma}
Let $\varphi^{(0)}$ be as in Definition~\ref{varphi0_defn} and let $p \in \{p_0, p_0+1,\ldots, \lceil q/q_0 \rceil \}$.  Given $0 < \tau < \gamma < 1$ and $\sigma \in (0,2/q)$, there exists $\varepsilon_0 \in (0,1)$ depending only on $n$, $m$, $q$, $\alpha$, $\varphi^{(0)}$, $\gamma$, and $\tau$ such that if $u$ satisfies Hypothesis~$(\star)$ of Section~\ref{sec:graphical_sec}, 
$\varphi \in \Phi_{\varepsilon_0,p}(\varphi^{(0)})$, and if $Z \in \Sigma_{u,q} \cap B_{1/2}(0)$ with $\mathcal{N}_u(Z) \geq \alpha$, then 
\begin{equation} \label{nonconest_concl1}
	\int_{B_{\gamma}(0)} R_Z^{2-n} \left| \frac{\partial (u/R_Z^{\alpha})}{\partial R_Z} \right|^2 
	\leq C \int_{B_1(0)} \mathcal{G}(u,\varphi)^2,
\end{equation}
where $R_Z = |X-Z|$ and $C = C(n,m,q,\alpha,\varphi^{(0)}) \in (0,\infty)$ is a constant.  Furthermore, 
\begin{equation} \label{nonconest_concl2}
	\int_{B_{\gamma}(0)} \frac{\mathcal{G}(u,\varphi)^2}{|X-Z|^{n-2+2/q-\sigma}} 
	+ \int_{B_{\gamma}(0) \cap \{r > \tau\}} \frac{\mathcal{G}(u(X), \varphi(X) - D_x \varphi(X) \cdot \xi)^2}{|X-Z|^{n+2\alpha-\sigma}} 
	\leq C \int_{B_1(0)} \mathcal{G}(u,\varphi)^2, 
\end{equation}
where $\xi$ is the projection of $Z$ onto $\mathbb{R}^2 \times \{0\},$ 
$$\varphi(X) - D_x \varphi(X) \cdot \xi = \sum_{j= 1}^{J} \sum_{k=1}^{p_{j}} \sum_{l = 1}^{q_{j, k}} \llbracket \varphi_{j, k, l}(X) - D_{x} \varphi_{j, k, l}(X) \cdot \xi\rrbracket$$
for $X \in \{r(X) > 0\}$ (with $\varphi_{j, k, l}$ and $q_{j, k}$ as in \eqref{varphi_localized} and \eqref{varphi_values})
and $C = C(n,m,q,\alpha,\varphi^{(0)},\sigma) \in (0,\infty)$ is a constant.  In particular, the constants $C$ are independent of $\tau$. 
\end{lemma}

We will prove Lemma~\ref{branchdist_lemma} and Lemma~\ref{nonconest_lemma} by induction on $p$, so let $p \in \{p_{0}+1, \ldots,  \lceil q/q_0 \rceil \}$ and assume that: 
\begin{enumerate}
\item[(A2)] whenever $\widetilde{p} \in \{{p_0}, {p_0}+1, \ldots, p-1\}$, Lemma~\ref{branchdist_lemma} and Lemma~\ref{nonconest_lemma} hold true with $\widetilde{p}$ in place of $p$.  
\end{enumerate}

To prove Lemma~\ref{branchdist_lemma}(a), we need the following preliminary bound on the distance of $Z$ from $\{0\} \times \mathbb{R}^{n-2}$: 

\begin{lemma} \label{branchdist2_lemma}
Let $\varphi^{(0)}$ be as in Definition~\ref{varphi0_defn} and assume that (A2) holds true for some $p \in \{p_0+1, p_0+2, \ldots, \lceil q/q_0 \rceil \}$.  For every $\delta \in (0,1/2)$, there exists $\varepsilon_0, \beta_0 \in (0,1)$ depending only on $n$, $m$, $q$, $p$, $\alpha$, $\varphi^{(0)}$ and $\delta$ such that if $u$, $\varphi$ satisfy Hypothesis~$(\star)$ and Hypothesis~$(\star\star)$ of Section~\ref{sec:graphical_sec}, and if $Z \in \Sigma_{u,q} \cap B_{1/2}(0)$ with $\mathcal{N}_u(Z) \geq \alpha$, then 
\begin{equation} \label{branchdist2_eqn2}
	\op{dist}^2(Z,\{0\} \times \mathbb{R}^{n-2}) 
	\leq \delta \inf_{\varphi' \in \bigcup_{p'=p_0}^{p-1} \Phi_{3\varepsilon_0,p'}(\varphi^{(0)})} \int_{B_1(0)} \mathcal{G}(u,\varphi')^2. 
\end{equation}
\end{lemma}

In order to prove Lemma~\ref{branchdist2_lemma}, we first need to establish the following:

\begin{lemma} \label{branchdist3_lemma}
Let $\varphi^{(0)}$ be as in Definition~\ref{varphi0_defn} and assume that (A2) holds true for some $p \in \{p_0+1, p_0+2, \ldots, \lceil q/q_0 \rceil \}$.  For every $\delta \in (0,1/2)$, there exists $\varepsilon_0, \beta_0, \gamma_0 \in (0,1)$ depending only on $n$, $m$, $q$, $p$, $\alpha$, $\varphi^{(0)}$ and $\delta$ such that if $u$, $\varphi$ satisfy Hypothesis~$(\star)$ and Hypothesis~$(\star\star)$ of Section~\ref{sec:graphical_sec}, 
$Z \in \Sigma_{u,q} \cap B_{1/2}(0)$ and $\mathcal{N}_u(Z) \geq \alpha$ then 
\begin{equation} \label{branchdist3_eqn2-1}
	\op{dist}^2(Z,\{0\} \times \mathbb{R}^{n-2}) 
	\leq \delta \inf_{\varphi' \in  \Phi_{3\varepsilon_0,p_{0}}(\varphi^{(0)})} \int_{B_1(0)} \mathcal{G}(u,\varphi')^2. 
\end{equation}
If additionally  there is $s  \in \{p_{0}+1, \ldots, p-1\}$ with 
\begin{equation} \label{branchdist3_eqn1}
	\inf_{\varphi' \in \bigcup_{p'=p_0}^{s} \Phi_{3\varepsilon_0,p'}(\varphi^{(0)})} \int_{B_1(0)} \mathcal{G}(u,\varphi')^2 
		\leq \gamma_0 \inf_{\varphi' \in \bigcup_{p'=p_0}^{s-1} \Phi_{3\varepsilon_0,p'}(\varphi^{(0)})} \int_{B_1(0)} \mathcal{G}(u,\varphi')^2
\end{equation}
then 
\begin{equation} \label{branchdist3_eqn2}
	\op{dist}^2(Z,\{0\} \times \mathbb{R}^{n-2}) 
	\leq \delta \inf_{\varphi' \in \bigcup_{p'=p_0}^{s} \Phi_{3\varepsilon_0,p'}(\varphi^{(0)})} \int_{B_1(0)} \mathcal{G}(u,\varphi')^2. 
\end{equation}
\end{lemma}

\begin{proof}
We prove Lemma~\ref{branchdist3_lemma} by contradiction.  Fix $\delta > 0$ and without loss of generality fix $s \in \{p_{0},\ldots,p-1\}$.  Suppose $\varepsilon_{\nu} \downarrow 0$, $\beta_{\nu} \downarrow 0$, and $\gamma_{\nu} \downarrow 0$ and for $\nu = 1,2,3,\ldots$, $u^{(\nu)} \in W^{1,2}(B_1(0);\mathcal{A}_q(\mathbb{R}^m))$ is an average-free energy minimizing $q$-valued function, $\varphi^{(\nu)} \in \Phi_{\varepsilon_{\nu},p}(\varphi^{(0)})$ and $Z_{\nu} = (\xi_{\nu},\zeta_{\nu}) \in \Sigma_{u^{(\nu)},q} \cap B_{1/2}(0)$ with ${\mathcal N}_{u^{(\nu)}}(Z_{\nu}) \geq \alpha$,  and that Hypothesis~$(\star)$, Hypothesis~$(\star\star)$, \eqref{branchdist3_eqn1}  hold with 
$\varepsilon_{\nu}$, $\beta_{\nu}$, $\gamma_{\nu}$, $u^{(\nu)}$, $\varphi^{(\nu)}$ in place of $\varepsilon_0$, $\beta_0$, $\gamma_{0}$, $u$, $\varphi$ and yet  
\begin{equation} \label{branchdist3_eqn4}
	|\xi_{\nu}|^2 > \delta \inf_{\varphi' \in \bigcup_{p'=p_0}^{s} \Phi_{3\varepsilon_{\nu},p'}(\varphi^{(0)})} \int_{B_1(0)} \mathcal{G}(u^{(\nu)},\varphi')^2. 
\end{equation}

Select $\phi^{(\nu)} \in \bigcup_{p^{\prime} = p_{0}}^{s} \Phi_{3\varepsilon_{\nu},p^{\prime}}(\varphi^{(0)})$ such that 
\begin{equation} \label{branchdist3_eqn3}
	\int_{B_1(0)} \mathcal{G}(u^{(\nu)},\phi^{(\nu)})^2 
	< 2 \inf_{\varphi' \in \bigcup_{p'=p_0}^{s} \Phi_{3\varepsilon_{\nu},p'}(\varphi^{(0)})} \int_{B_1(0)} \mathcal{G}(u^{(\nu)},\varphi')^2. 
\end{equation}
Note that then by \eqref{branchdist3_eqn1}, $\phi^{(\nu)} \in \Phi_{3\varepsilon_{\nu}, s}(\varphi^{(0)})$. In view of Hypothesis~$(\star)$,  \eqref{branchdist3_eqn1} and \eqref{branchdist3_eqn3}, we can blow up $u^{(\nu)}$  relative to $\phi^{(\nu)}$ by the excess $E_{\nu} = \left( \int_{B_1(0)} \mathcal{G}(u^{(\nu)},\phi^{(\nu)})^2 \right)^{1/2}$ to obtain a blow-up $w = (w_{j,k})$ that is a multi-valued function on the graph of some 
$\phi^{(\infty)} = (\phi^{(\infty)}_{j,k})$ obtained as in the blow-up procedure described in Section~\ref{blow-up-procedure}.
By Hypothesis~$(\star\star)$, we have that $\int_{B_{1}(0)} {\mathcal G}(u^{(\nu)}, \varphi^{(\nu)})^{2} \leq \beta_{\nu} \int_{B_{1}(0)}{\mathcal G}(u^{(\nu)}, \phi^{(\nu)})^{2}$ so $w$ is homogeneous of degree $\alpha$ and translation invariant along $\{0\} \times \mathbb{R}^{n-2}$.   By the assumption (A2) above and Lemma~\ref{branchdist_lemma}(a), after passing to a subsequence, $\xi_{\nu}/E_{\nu}$ converge to some $\lambda \in \mathbb{R}^2$ which satisfies,  by \eqref{branchdist3_eqn3} and \eqref{branchdist3_eqn4}, 
\begin{equation} \label{branchdist3_eqn5}
	|\lambda|^2 \geq \delta/2. 
\end{equation}
Clearly after passing to a subsequence $\zeta_{\nu}$ converge to some $\zeta$ in $\overline{B^{n-2}_{1/2}(0)}$.  By (A2) and Lemma~\ref{nonconest_lemma} with $u^{(\nu)}$ and $\phi^{(\nu)}$ in place of $u$ and $\varphi$, for every $\tau > 0$ and $\nu$ sufficiently large (depending on $\tau$), 
\begin{equation*}
	\int_{B_{3/4}(0) \cap \{r > \tau\}} \frac{\mathcal{G}(u^{(\nu)}(X), \phi^{(\nu)}(X) - D_x \phi^{(\nu)}(X) \cdot \xi_{\nu})^2}{|X-Z_{\nu}|^{n+2\alpha-\sigma}} 
	\leq C \int_{B_1(0)} \mathcal{G}(u^{(\nu)},\phi^{(\nu)})^2, 
\end{equation*}
so by dividing by $E_{\nu}^2$ and letting $\nu \rightarrow \infty$ using the monotone convergence theorem, 
\begin{equation*}
	\sum_{j=1}^J \sum_{k=1}^{p_j} \sum_{l=1}^{q_{j, k}}\sum_{h=1}^{m_{j, k}}\int_{B_{3/4}(0)} \frac{|w_{j,k, l , h}(X) - D_x \varphi^{(0)}_{j, l}(X) \cdot \lambda|^2}{|X-(0,\zeta)|^{n+2\alpha-\sigma}}  \leq C 
\end{equation*}
for some constant $C = C(n,m,q,\alpha,\varphi^{(0)},\sigma) \in (0,\infty)$. Here for $X \in B_{3/4}(0) \setminus \{0\} \times {\mathbb R}^{n-2}$, $w_{j, k, l, h}(X)$ are such that $w_{j, k}(X, \phi^{(\infty)}_{j, k, l}(X)) = \sum_{h=1}^{m_{j, k}} \llbracket w_{j, k, l, h}(X)\rrbracket$ where $\phi_{j, k, l}^{(\infty)}$ ($1 \leq l \leq q_{j, k}$) and $q_{j, k}$ are defined by \eqref{varphi_localized} and \eqref{varphi_values} taken with $\phi^{(\infty)}$ in place of $\varphi$; and also for $j$ such that $\varphi^{(0)}_{j} \neq 0$, $q_{j, k} = q_{0}$ and $\varphi_{j, l}^{(0)}$ ($1 \leq l \leq q_{0}$) are harmonic functions locally defined near $X$ such that 
$\varphi_{j}^{(0)}(X) = \sum_{l=1}^{q_{0}} \llbracket \varphi_{j, l}^{(0)}(X)\rrbracket$; and for $j$ such that $\varphi^{(0)}_{j} = 0$, $\varphi^{(0)}_{j, l} = 0$ for $1 \leq l \leq q_{j, k}$. 

Since $w_{j,k}(X,\phi^{(\infty)}_{j,k,l}(X))$ is homogeneous of degree $\alpha$ and translation invariant along $\{0\} \times \mathbb{R}^{n-2}$, it follows from the preceding estimate that 
\begin{equation} \label{branchdist3_eqn6}
	\sum_{j=1}^J \sum_{k=1}^{p_j} \int_{B_{3/4}(0)} \frac{|D_x \varphi^{(0)}_j(X) \cdot \lambda|^2}{|X-(0,\zeta)|^{n+2\alpha-\sigma}}  \leq C. 
\end{equation}
In view of the homogeneity of $\varphi^{(0)}$ and $L^2$ orthogonality of $D_1 \varphi^{(0)}_j(e^{i\theta},y)$ and $D_2 \varphi^{(0)}_j(e^{i\theta},y)$, \eqref{branchdist3_eqn6} implies that $\lambda = 0$, contradicting \eqref{branchdist3_eqn5}.
\end{proof}

\begin{proof}[Proof of Lemma~\ref{branchdist2_lemma}]
For each $p' \in \{p_0+1, p_0+2, \ldots, p \}$ and $\delta \in (0,1/2)$, let $\varepsilon(p',\delta)$, $\beta(p',\delta)$, and $\gamma(p',\delta)$ denote $\varepsilon_0$, $\beta_0$, and $\gamma_0$ as in Lemma~\ref{branchdist3_lemma} with $p'$ in place of $p$.  Fix $\delta \in (0,1/2)$.  For each $j = 1,2,\ldots,p-p_0$, inductively define 
\begin{gather*}
	\varepsilon^{(1)} = \epsilon(p,\delta), \quad \beta^{(1)} = \beta(p,\delta), \quad \gamma^{(1)} = \gamma(p,\delta) \\
	\varepsilon^{(j)} = \epsilon(p-j+1, \gamma^{(1)} \cdots \gamma^{(j-1)} \delta), \quad 
	\beta^{(j)} = \beta(p-j+1,\gamma^{(1)} \cdots \gamma^{(j-1)} \delta), \\ 
	\gamma^{(j)} = \gamma(p-j+1,\gamma^{(1)} \cdots \gamma^{(j-1)} \delta) .
\end{gather*} 
Then define 
\begin{equation*}
	\varepsilon_0 = \min\{\epsilon^{(1)},\ldots,\epsilon^{(p-p_0)}\}, \quad 
	\beta_0 = \min\{\beta^{(1)},\ldots,\beta^{(p-p_0)}\}.
\end{equation*}
If $p = p_0+1$ or $p \geq p_0+2$ and 
\begin{equation*}
	\inf_{\varphi' \in \bigcup_{p'=p_0}^{p-1} \Phi_{3\varepsilon_0,p'}(\varphi^{(0)})} \int_{B_1(0)} \mathcal{G}(u,\varphi')^2 
		\leq \gamma^{(1)} \inf_{\varphi' \in \bigcup_{p'=p_0}^{p-2} \Phi_{3\varepsilon_0,p'}(\varphi^{(0)})} \int_{B_1(0)} \mathcal{G}(u,\varphi')^2, 
\end{equation*}
then by Lemma~\ref{branchdist3_lemma} we obtain \eqref{branchdist2_eqn2}.  Otherwise, we can find $j_0 \in \{2,3,\ldots,p-p_0\}$ such that 
\begin{equation*}
	\inf_{\varphi' \in \bigcup_{p'=p_0}^{p-j} \Phi_{3\varepsilon_0,p'}(\varphi^{(0)})} \int_{B_1(0)} \mathcal{G}(u,\varphi')^2 
		> \gamma^{(j)} \inf_{\varphi' \in \bigcup_{p'=p_0}^{p-j-1} \Phi_{3\varepsilon_0,p'}(\varphi^{(0)})} \int_{B_1(0)} \mathcal{G}(u,\varphi')^2, 
\end{equation*}
for $j = 1,2,\ldots,j_0-1$ and either $j_0 = p-p_0$ or $j_0 > p-p_0$ and 
\begin{equation*}
	\inf_{\varphi' \in \bigcup_{p'=p_0}^{p-j_0} \Phi_{3\varepsilon_0,p'}(\varphi^{(0)})} \int_{B_1(0)} \mathcal{G}(u,\varphi')^2 
		\leq \gamma^{(j_0)} \inf_{\varphi' \in \bigcup_{p'=p_0}^{p-j_0-1} \Phi_{3\varepsilon_0,p'}(\varphi^{(0)})} \int_{B_1(0)} \mathcal{G}(u,\varphi')^2. 
\end{equation*}
Thus by Lemma~\ref{branchdist3_lemma} we obtain 
\begin{align*}
	\op{dist}^2(Z,\{0\} \times \mathbb{R}^{n-2}) 
	&\leq \gamma^{(1)} \cdots \gamma^{(j_0-1)} \delta \inf_{\varphi' \in \bigcup_{p'=p_0}^{p-j_0} \Phi_{3\varepsilon_0,p'}(\varphi^{(0)})} 
		\int_{B_1(0)} \mathcal{G}(u,\varphi')^2
	\\&\leq \delta \inf_{\varphi' \in \bigcup_{p'=p_0}^{p-1} \Phi_{3\varepsilon_0,p'}(\varphi^{(0)})} \int_{B_1(0)} \mathcal{G}(u,\varphi')^2 . \qedhere
\end{align*}
\end{proof}

 Let $\varepsilon > 0$, $\varphi^{(0)}$ be as in Definition~\ref{varphi0_defn}, $u \in W^{1,2}(B_1(0);\mathcal{A}_q(\mathbb{R}^m))$ be an average-free, energy minimizing function such that 
\begin{equation} \label{translate_hyp}
	\int_{B_1(0)} \mathcal{G}(u,\varphi^{(0)})^2 < \varepsilon^2,  
\end{equation}
and let $\varphi \in \Phi_{\varepsilon, p}(\varphi^{(0)})$ for some $p \in \{p_{0}, \ldots, \lceil q/q_0 \rceil\}$.  We make the following observations which we shall rely on in the proofs of Lemma~\ref{branchdist_lemma} and Lemma~\ref{nonconest_lemma} below:  
\begin{itemize}
\item[(1)] For $X = (x,y)$ with $x \neq 0$, $\varphi$ decomposes into $q$ smooth, homogeneous degree $\alpha$ single-valued functions $\varphi_j$ on $B_{|x|/2}(X)$ so that 
$\varphi(Y) = \sum_{j=1}^{q}\llbracket \varphi_{j}(Y)\rrbracket$ for $Y \in B_{|x|/2}(X)$. Applying Taylor's theorem to $\varphi_{j},$  we then have that for $Z = (\xi, \zeta)$ with $|\xi| < |x|/2$, 
$$\varphi(X -Z) = \sum_{j=1}^{q} \llbracket \varphi_{j}(X) - D_{x}\varphi_{j}(X) \cdot \xi + {\mathcal R}_{j}(x, \xi)\rrbracket$$ where $|{\mathcal R}_{j}(x, \xi)| \leq C |x|^{\alpha - 2}|\xi|^{2}$ with $C = C(\alpha, \sup_{S^{1}} |D^{2}\varphi|)$.  Thus 
$${\mathcal G}(\varphi(X - Z), \varphi(X) - D_{x}\varphi(X) \cdot \xi) \leq C|x|^{\alpha - 2}|\xi|^{2},$$ where $C = C(q,\alpha,\sup_{S^{1}} |D^{2}\varphi|)$ and by definition
$$\varphi(Y) - D_{x}\varphi(Y) \cdot \xi = \sum_{j=1}^{q} \llbracket \varphi_{j}(Y) - D_{x}\varphi_{j}(Y) \cdot \xi \rrbracket$$ for $Y \in B_{|x|/2}(X)$. It follows from this and the triangle inequality that provided $\varepsilon$ is sufficiently small depending only on $n$, $m$, $q$, $\alpha$, and $\varphi^{(0)}$, we have that if $Z = (\xi, \zeta)$ with $|\xi| \leq |x|/2$ then 
\begin{equation} \label{translate_eqn1}
	\mathcal{G}(u(X),\varphi(X-Z)) = \mathcal{G}(u(X),\varphi(X) - D_x \varphi(X) \cdot \xi) + \mathcal{R}, 
\end{equation}
where $|\mathcal{R}| \leq C |x|^{\alpha-2} |\xi|^2$ for some constant $C = C(\varphi^{(0)}) \in (0,\infty)$; 

\item[(2)] Let $\tau \in (0,1/2)$, $X = (x, y)$ with $|x| \geq \tau$ and $Z = (\xi, \zeta)$ with $|\xi| < |x|/2$ and ${\mathcal N}_{u}(Z) \geq \alpha.$ For any $\delta \in (0, 1/2)$, if $\varepsilon = \varepsilon(n, m, q, \alpha, \varphi^{(0)}, \delta)$ is sufficiently small  and additionally if either (i) $p = p_0$ or (ii) $p > p_0$ and 
\begin{equation*}
	\int_{B_1(0)} \mathcal{G}(u,\varphi)^2 
	\leq \beta \inf_{\varphi' \in \bigcup_{p'=p_0}^{p-1} \Phi_{3\varepsilon_0,p'}(\varphi^{(0)})} \int_{B_1(0)} \mathcal{G}(u,\varphi')^2 
\end{equation*}
for $\beta  = \beta(n, m, q, \alpha, \varphi^{(0)}, \delta)$ sufficiently small, then by Corollary~\ref{graphical_cor}(a) and Lemma~\ref{branchdist2_lemma}, we have that 
$$|\xi| \leq C \delta  \, \inf_{S^{1}} {\rm sep} \, \varphi \leq C\delta |x|^{-\alpha} \, {\rm sep} \, \varphi(X) \leq C\delta \t^{-\alpha} \, {\rm sep} \, \varphi(X)$$
where $C = C(n, m, q, \alpha, \varphi^{(0)}),$ whence, for a choice of $\delta = \delta(n, m, q, \alpha, \varphi^{(0)}, \tau)$ sufficiently small, it follows that 
${\mathcal G}(\varphi(X), \varphi(X) - D_{x}\varphi(X) \cdot \xi) = |D_{x}\varphi(X) \cdot \xi|.$ Using this together with the triangle inequality, we deduce from  \eqref{translate_eqn1} the following:

\emph{For any given $\tau \in (0, 1/4)$, there exist $\varepsilon = \varepsilon(n, m, q, \alpha, \varphi^{(0)}, \tau)$ and $\beta = \beta(n, m, q, \alpha, \varphi^{(0)}, \tau)$ such that if \eqref{translate_hyp} holds, $\varphi \in \Phi_{\varepsilon, p}(\varphi^{(0)})$ for some $p \in \{p_{0}, \ldots, \lceil q/q_0 \rceil\}$ and if either (i) $p=p_{0}$ or (ii) $p > p_{0}$ and  $\int_{B_1(0)} \mathcal{G}(u,\varphi)^2 
	\leq \beta \inf_{\varphi' \in \bigcup_{p'=p_0}^{p-1} \Phi_{3\varepsilon_0,p'}(\varphi^{(0)})} \int_{B_1(0)} \mathcal{G}(u,\varphi')^2$, then  
for any $X = (x, y) \in B_{1}(0)$ with $|x| \geq \tau$ and any $Z = (\xi, \zeta) \in B_{1}(0)$ with $|\xi| \leq |x|/2$ and ${\mathcal N}_{u}(Z) \geq \alpha$, we have that  
\begin{equation} \label{translate_eqn2}
	\mathcal{G}(u(X),\varphi(X-Z)) \geq |D_x \varphi(X) \cdot \xi| - \mathcal{G}(u(X),\varphi(X)) - C |x|^{\alpha-2} |\xi|^2 
\end{equation}
where $C = C(\varphi^{(0)}) \in (0,\infty)$.}

\item[(3)] By the triangle inequality and the fundamental theorem of calculus, 
\begin{align} \label{translate_eqn3}
	\left| \mathcal{G}(u(X),\varphi(X-Z)) - \mathcal{G}(u(X),\varphi(X)) \right| 
	&\leq \mathcal{G}(\varphi(X-Z),\varphi(X)) \\
	&\leq \left( \int_0^1 |D\varphi(X-tZ)|^2 dt \right)^{1/2} |\xi| \nonumber 
\end{align}
for a.e. $X = (x,y) \in B_1(0)$;   
also,  by the continuity estimate \eqref{regestimate}, there exists $\tau = \tau(\varepsilon) \in (0,1)$ with $\tau(\varepsilon) \rightarrow 0$ as $\varepsilon \downarrow 0$ such that 
$|\xi| < \tau$ for every  $Z = (\xi,\zeta) \in \Sigma_{u,q} \cap B_{1/2}(0)$. 
Thus using \eqref{translate_eqn3} with $\varphi^{(0)}$ in place of $\varphi$, we deduce that 
\begin{equation*}
	4^{-n-2\alpha} \int_{B_{1/4}(Z)} \mathcal{G}(u(X),\varphi^{(0)}(X-Z))^2 dX 
	\leq C \int_{B_1(0)} \mathcal{G}(u,\varphi^{(0)})^2 + C |\xi|^2 
	\leq C (\varepsilon^2 + \tau^{2}(\varepsilon)) 
\end{equation*}
where $C = C(n,\varphi^{(0)}) \in (0,\infty)$, and hence Theorem~\ref{keyest_thm} and Lemma~\ref{keyest_cor} hold with $4^{\alpha} u(Z+X/4)$ in place of $u$ provided 
$\varepsilon = \varepsilon(n, m, q, \alpha, \varphi^{(0)})$ is sufficiently small and ${\mathcal N}_{u}(Z) \geq \alpha$. 
\end{itemize}

\begin{proof}[Proof of Lemma~\ref{branchdist_lemma}]
First we will prove Lemma~\ref{branchdist_lemma}(a). We may assume, for some $\beta_0 \in (0,1)$ to be determined depending only on $n$, $m$, $q$, $p$, $\alpha$, and $\varphi^{(0)},$ that 
\begin{equation} \label{branchdist_eqn01}
	\int_{B_1(0)} \mathcal{G}(u,\varphi)^2 \leq \beta_0 \inf_{\varphi' \in \bigcup_{p'=p_{0}}^{p-1} \Phi_{3\varepsilon_0,p'}(\varphi^{(0)})} \int_{B_1(0)} \mathcal{G}(u,\varphi')^2
\end{equation}
for if the reverse inequality holds, then we can select $s \in \{p_0,p_0+1,\ldots,p-1\}$ and $\phi \in \Phi_{3\varepsilon_0,s}(\varphi^{(0)})$ such that 
\begin{equation*}
	\int_{B_1(0)} \mathcal{G}(u,\phi)^2 \leq 2 \inf_{\varphi' \in \bigcup_{p'=p_0}^{p-1} \Phi_{\varepsilon_0,p'}(\varphi^{(0)})} \int_{B_1(0)} \mathcal{G}(u,\varphi')^2 
\end{equation*}
and conclude from (A2) that 
\begin{equation*}
	\op{dist}^2(Z,\{0\} \times \mathbb{R}^{n-2}) 
	\leq C \int_{B_1(0)} \mathcal{G}(u,\phi)^2 
	\leq \frac{2C}{\beta_0} \int_{B_1(0)} \mathcal{G}(u,\varphi)^2
\end{equation*}
for some constant $C = C(n,m,q,\alpha,\varphi^{(0)}) \in (0,\infty)$. 

\noindent
\emph{Claim:  There is a constant $\delta_1 = \delta_1(\varphi^{(0)}) > 0$ such that the following holds: for every $\rho \in (0,1/4)$, there is $\varepsilon_0 = \varepsilon_0(\varphi^{(0)},\rho) > 0$ such that if $u$ satisfies Hypothesis~$(\star)$ and if $\varphi \in \Phi_{\varepsilon_0,p}(\varphi^{(0)})$, then for every $a \in \mathbb{R}^2$ and every $Z = (\xi,\zeta) \in \Sigma_{u, q} \cap B_{1/2}(0)$, }
\begin{equation} \label{branchdist_eqn1} 
	\mathcal{L}^n \{ X \in B_{\rho}(Z) : \delta_1 |a| |x|^{\alpha-1} \leq |D_x \varphi(X) \cdot a| \} \geq \delta_1 \rho^n. 
\end{equation}
To see this we argue by contradiction, so suppose the assertion is false; then for any given $\delta_1 > 0$ there is $\rho \in (0,1/4)$ such that with $\varepsilon_{\nu} = 1/\nu$, there exists $\varphi^{(\nu)} \in \Phi_{\varepsilon_{\nu}, p}(\varphi^{(0)})$,  $a_{\nu} \in S^1$, a locally energy minimizing function $u^{(\nu)}$ and a point $Z_{\nu} \in \Sigma_{u^{(\nu)}, q} \cap B_{1/2}(0)$ so that Hypothesis~$(\star)$ holds with $\varepsilon_{\nu}$, $u^{(\nu)}$ in place of $\varepsilon_0$,$u$ and 
\begin{equation*} 
	\mathcal{L}^n \{ X \in B_{\rho}(Z_{\nu}) : \delta_1 |x|^{\alpha-1} \leq |D_x \varphi^{(\nu)}(X) \cdot a_{\nu}| \} < \delta_1 \rho^n. 
\end{equation*}
After passing to a subsequence, $\varphi^{(\nu)} \rightarrow \varphi^{(0)}$ in $C^1$ on compact subsets of $\mathbb{R}^n \setminus \{0\} \times \mathbb{R}^{n-2}$, $Z_{\nu} \rightarrow Z$ for some $Z \in \{0\} \times \mathbb{R}^{n-2} \cap \overline{B_{1/2}(0)}$ (since $u^{(\nu)} \to \varphi^{(0)}$ uniformly on $\overline{B_{1/2}(0)}$), and $a_{\nu} \rightarrow a$ with $a \in S^1,$ whence  
\begin{equation} \label{branchdist_eqn2} 
	\mathcal{L}^n \{ X \in B_{\rho}(Z) : \delta_1 |x|^{\alpha-1} \leq |D_x \varphi^{(0)}(X) \cdot a| \} \leq \delta_1 \rho^n. 
\end{equation}
Thus we have shown that if the claim is false, then for every $\delta_1 > 0$ there are a number $\rho > 0$, a point $Z \in \{0\} \times \mathbb{R}^{n-2}$ 
and a point $a \in S^1$ such that \eqref{branchdist_eqn2} holds, or equivalently (by translating and rescaling),  
\begin{equation*}
	\mathcal{L}^n \{X \in B_{1}(0) : \delta_1 |x|^{\alpha-1} \leq |D_x \varphi^{(0)}(X) \cdot a| \} \leq \delta_1.
\end{equation*}
Using this with $\delta_1 = 1/\nu$, we deduce that for each $\nu=1, 2, 3, \ldots,$ there is a point $a_{\nu} \in S^1$ such that 
\begin{equation*} 
	\mathcal{L}^n \{ X \in B_1(0) : (1/\nu) |x|^{\alpha-1} \leq |D_x \varphi^{(0)}(X) \cdot a_{\nu}| \} < 1/\nu. 
\end{equation*}
After passing to a subsequence, $a_{\nu} \rightarrow a$ where $a \in S^1$ and 
\begin{equation*} 
	D_x \varphi^{(0)}(X) \cdot a = 0 \text{ a.e. on } B_1(0), 
\end{equation*}
but no such $a$ exists in view of the definition of $\varphi^{(0)}$ (Definition~\ref{varphi0_defn}). This contradiction establishes the claim.

Let $Z = (\xi, \zeta) \in \Sigma_{u, q} \cap B_{1/2}(0)$ be such that ${\mathcal N}_{u}(Z) \geq \alpha.$  With $\delta_1$ as in the claim, choose $\kappa = \kappa(n) >0$ such that $\mathcal{L}^n(B^2_{\kappa \delta_1^{1/2} \rho}(0) \times B^{n-2}_{\rho}(0)) < \delta_1 \rho^n/2$.  Let $\rho > 0$ to be chosen.  Assume $|\xi| \leq \rho$ (provided $\varepsilon_{0}$ is sufficiently small depending on $\rho$).  Take $a = \xi$ in \eqref{branchdist_eqn1} and use \eqref{translate_eqn2} with $\tau = \kappa \delta_1^{1/2} \rho$ (which we may do in view of \eqref{branchdist_eqn01} provided $\beta_{0}$ is sufficiently small depending on $\rho$) to deduce that for some set $S \subseteq B_{\rho}(Z) \cap \{ (x,y) : |x| \geq \kappa \delta_1^{1/2} \rho \}$ with $\mathcal{L}^n(S) \geq \delta_1 \rho^n/2$,
\begin{align} \label{branchdist_eqn3} 
	&c \rho^{n+2\alpha-2} |\xi|^2 
	\leq \int_S |x|^{2\alpha-2} |\xi|^2 \leq \delta_1^{-2} \int_{B_{\rho}(Z) \cap \{|x| \geq \kappa \delta_1^{1/2} \rho \}} |D_x \varphi(X) \cdot \xi|^2 \\
	&\hspace{10mm} \leq 3\delta_1^{-2} \int_{B_{\rho}(Z)} \mathcal{G}(u(X),\varphi(X-Z))^2 dX 
			+ 3\delta_1^{-2} \int_{B_{\rho}(Z)} \mathcal{G}(u(X),\varphi(X))^2 dX \nonumber \\
		&\hspace{10mm} + 3C \delta_1^{-2} \int_{B_{\rho}(Z) \cap \{|x| \geq 2|\xi|\}} |x|^{2\alpha-4} |\xi|^4 
			+ \delta_1^{-2} \int_{B_{\rho}(Z) \cap \{|x| \leq 2|\xi|\}} |D_x \varphi(X)|^2 |\xi|^2, \nonumber 
\end{align}
where $C = C(\varphi^{(0)}) \in (0,\infty)$ is a constant and $c = \kappa^{2\alpha-2} \delta_1^{\alpha}/2$ if $\alpha \geq 1$ and $c = 2^{2\alpha-3} \delta_1$ if $\alpha < 1$ (as $\kappa \delta_1^{1/2} \rho \leq |x| \leq |\xi|+\rho \leq 2\rho$ for all $X = (x, y)  \in S$).

We need to bound the terms on the right-hand side of \eqref{branchdist_eqn3}.  For the first term, we note that by Lemma~\ref{keyest_cor} with $4^{\alpha} u(Z+X/4)$ in place of $u$ and $\sigma = 1/2$ and by \eqref{translate_eqn3}, 
\begin{align*}
	&\rho^{-n-2\alpha+1/2} \int_{B_{\rho}(Z)} \mathcal{G}(u(X),\varphi(X-Z))^2 dX 
	\leq C \int_{B_1(0)} \mathcal{G}(u(X),\varphi(X-Z))^2 dX 
	\\&\hspace{10mm} \leq C \int_{B_1(0)} \mathcal{G}(u(X),\varphi(X))^2 dX + C |\xi|^2 \int_{B_1(0)} \int_0^1 |D\varphi(x-t\xi)|^2 dt \, dX
\end{align*}
for $C = C(n,m,q,\alpha,\varphi^{(0)}) \in (0,\infty)$.  Using the change of variable $x' = x-t\xi$, 
\begin{align} \label{branchdist_eqn4} 
	\int_{B_1(0)} \int_0^1 |D\varphi(x-t\xi)|^2 dt \,dX 
	&\leq C \sup_{\partial B^2_1(0) \times \mathbb{R}^{n-2}} |D\varphi|^2 \int_{B^2_1(0)} \int_0^1 |x-t\xi|^{2\alpha-2} dt \, dx \\
	&\leq C \sup_{\partial B^2_1(0) \times \mathbb{R}^{n-2}} |D\varphi|^2 \int_0^1 \int_{B^2_{1+t|\xi|}(0)} |x'|^{2\alpha-2} dx' \, dt \nonumber \\
	&\leq C \sup_{\partial B^2_1(0) \times \mathbb{R}^{n-2}} |D\varphi|^2 \nonumber 
\end{align}
for $C = C(n,\alpha) \in (0,\infty)$.  Hence 
\begin{equation} \label{branchdist_eqn5} 
	\rho^{-n-2\alpha+1/2} \int_{B_{\rho}(Z)} \mathcal{G}(u(X),\varphi(X-Z))^2 dX 
	\leq C \int_{B_1(0)} \mathcal{G}(u(X),\varphi(X))^2 dX + C |\xi|^2
\end{equation}
for some constant $C = C(n,m,q,\alpha,\varphi^{(0)}) \in (0,\infty)$.  For the third term on the right-hand side of \eqref{branchdist_eqn3}, by direct computation considering the cases where $\alpha < 1$, $\alpha = 1$, and $\alpha > 1$ separately, 
\begin{equation} \label{branchdist_eqn6} 
	\int_{B_{\rho}(Z) \cap \{|x| \geq 2|\xi|\}} |x|^{2\alpha-4} |\xi|^4 \leq C \rho^{n-2} |\xi|^4 (|\xi|^{2\alpha-2-1/q} + \rho^{2\alpha-2-1/q})
\end{equation}
for some constant $C = C(n,\alpha,q) \in (0,\infty)$ provided $B_{\rho}(Z) \cap \{|x| \geq 2|\xi|\} \neq \emptyset$.  In fact, in the cases
where $\alpha < 1$ or $\alpha > 1$ we can bound the left-hand side of \eqref{branchdist_eqn6} by $C \rho^{n-2} |\xi|^4 (|\xi|^{2\alpha-2} + \rho^{2\alpha-2})$ and when $\alpha = 1$ we can bound the left-hand side of \eqref{branchdist_eqn6} by $C \rho^{n-2} |\xi|^4 |\log|\xi|| \leq C \rho^{n-2} |\xi|^{4-1/q}$.  For the last term on the right-hand side of \eqref{branchdist_eqn3}, 
\begin{equation} \label{branchdist_eqn7} 
	\int_{B_{\rho}(Z) \cap \{|x| \leq 2|\xi|\}} |D_x \varphi(X)|^2 |\xi|^2 
	\leq C \rho^{n-2} \int_{B^2_{2|\xi|}(0)} |x|^{2\alpha-2} |\xi|^2 dx \leq C \rho^{n-2} |\xi|^{2\alpha+2}
\end{equation}
for $C = C(n,\alpha,\varphi^{(0)}) \in (0,\infty)$.  Therefore, by \eqref{branchdist_eqn3}, \eqref{branchdist_eqn5}, \eqref{branchdist_eqn6}, and \eqref{branchdist_eqn7},   
\begin{equation} \label{branchdist_eqn8} 
	\rho^{n+2\alpha-2} |\xi|^2 \leq C \int_{B_1(0)} \mathcal{G}(u,\varphi)^2 
		+ C (\rho^{3/2} + \rho^{-2\alpha} |\xi|^{2\alpha-1/q} + \rho^{-2-1/q} |\xi|^2 + \rho^{-2\alpha} |\xi|^{2\alpha}) \rho^{n+2\alpha-2} |\xi|^2
\end{equation}
for some constant $C = C(n,m,q,\alpha,\varphi^{(0)}) \in (0,\infty)$.  Since for any given $\tau > 0$ we may choose $\varepsilon_0 = \varepsilon_0(\varphi^{(0)},\tau) \in (0,1)$ sufficiently small to ensure that 
$|\xi| < \tau$, by choosing $\rho = \rho(n, m, q, \alpha, \varphi^{(0)})$ and $\tau = \tau(n, m, q, \alpha, \varphi^{(0)})$ small enough that $\tau < \rho$ and $C(\rho^{3/2} + \rho^{-2\alpha} \tau^{2\alpha-1/q} + \rho^{-2-1/q} \tau^2 + \rho^{-2\alpha} \tau^{2\alpha}) < 1/2$, we conclude from \eqref{branchdist_eqn8} that whenever $\varepsilon_0 = \varepsilon_0(n, m, q, \alpha, \varphi^{(0)})$ is sufficiently small, the hyptheses of the theorem imply that 
\begin{equation} \label{branchdist_eqn9} 
	|\xi|^2 \leq C \int_{B_1(0)} \mathcal{G}(u(X),\varphi(X))^2 dX 
\end{equation}
for some constant $C = C(n,m,q,\alpha,\varphi^{(0)}) \in (0,\infty)$. This is conclusion (a). Conclusion (b), namely the bound
\begin{equation*} 
	\int_{B_1(0)} \mathcal{G}(u(X),\varphi(X-Z))^2 dX \leq C \int_{B_1(0)} \mathcal{G}(u(X),\varphi(X))^2 dX 
\end{equation*}
for some constant $C = C(n,m,q,\alpha,\varphi^{(0)}) \in (0,\infty),$ follows directly from \eqref{translate_eqn3}, \eqref{branchdist_eqn4}  and \eqref{branchdist_eqn9}. 
\end{proof}

\begin{proof}[Proof of Lemma~\ref{nonconest_lemma}]
The estimate \eqref{nonconest_concl1} is an immediate consequence of Theorem~\ref{keyest_thm}(a).  To bound the first term on the left-hand side of \eqref{nonconest_concl2}, first notice that by \eqref{translate_eqn3}, 
\begin{align} \label{nonconest_lemma_eqn1}
	&\int_{B_{\gamma}(0)} \frac{\mathcal{G}(u(X),\varphi(X))^2}{|X-Z|^{n-2+2/q-\sigma}} dX 
	\leq 2 \int_{B_{1/4}(Z)} \frac{\mathcal{G}(u(X),\varphi(X-Z))^2}{|X-Z|^{n-2+2/q-\sigma}} dX \\
		&\hspace{10mm} + 2 \int_{B_{1/4}(Z)} \int_0^1 \frac{|D_x \varphi(x-t\xi,y)|^2 |\xi|^2}{|X-Z|^{n-2+2/q-\sigma}} dt \, dx \, dy
		+ 4^{n-1-\sigma} \int_{B_1(0)} \mathcal{G}(u,\varphi)^2 \nonumber 
\end{align}
We need to bound the terms on the right-hand side of \eqref{nonconest_lemma_eqn1}.  For the first term, observe that by Lemma~\ref{keyest_cor} with $4^{\alpha} u(Z+X/4)$ in place of $u$ and Lemma~\ref{branchdist_lemma}(b), 
\begin{align} \label{nonconest_lemma_eqn2}
	\int_{B_{1/4}(Z)} \frac{\mathcal{G}(u(X),\varphi(X-Z))^2}{|X-Z|^{n+2\alpha-\sigma}} dX
	&\leq C \int_{B_{1/2}(Z)} \mathcal{G}(u(X),\varphi(X-Z))^2 dX \\
	&\leq C \int_{B_1(0)} \mathcal{G}(u(X),\varphi(X))^2 dX\nonumber
\end{align}
for $C = C(n,m,q,\alpha,\varphi^{(0)},\sigma) \in (0,\infty)$.  For the second term on the right-hand side of \eqref{nonconest_lemma_eqn1}, in case $\alpha \geq 1$ we have by Lemma~\ref{branchdist_lemma} that
\begin{align*} 
	\int_{B_{1/4}(Z)} \int_0^1 \frac{|x-t\xi|^{2\alpha-2} |\xi|^2}{|X-Z|^{n-2+2/q-\sigma}} dt \, dx \, dy 
	&\leq \int_{B_{1/4}(Z)} \int_0^1 \frac{|\xi|^2}{|X-Z|^{n-2+2/q-\sigma}} dt \, dx \, dy \\
	&\leq C |\xi|^2 \leq C \int_{B_1(0)} \mathcal{G}(u,\varphi)^2 \nonumber 
\end{align*}
where $C = C(n,m,q,\alpha,\varphi^{(0)},\sigma) \in (0,\infty)$, and in case $\alpha < 1$, we have that 
\begin{align*} 
	&\int_{B_{1/4}(Z)} \int_0^1 \frac{|x-t\xi|^{2\alpha-2} |\xi|^2}{|X-Z|^{n-2+2/q-\sigma}} dt \, dx \, dy \\ 
	&\hspace{10mm} \leq \int_0^1 \int_{B_{1/4}(Z)} \frac{|\xi|^2}{|x-t\xi|^{2-2\alpha} |x-\xi|^{2/q-\sigma/2} |y-\zeta|^{n-2-\sigma/2}} dx \, dy \, dt \nonumber \\ 
	&\hspace{10mm} \leq \int_0^1 \int_{B^{n-2}_{1/4}(\zeta)} \int_{B^2_{1/4}(\xi) \cap \{|x-t\xi| \leq |x-\xi|\}} 
			\frac{|\xi|^2}{|x-t\xi|^{2+2/q-2\alpha-\sigma/2} |y-\zeta|^{n-2-\sigma/2}} dx \, dy \, dt \nonumber \\ 
		&\hspace{20mm} + \int_0^1 \int_{B^{n-2}_{1/4}(\zeta)} \int_{B^2_{1/4}(\xi) \cap \{|x-t\xi| \geq |x-\xi|\}} 
			\frac{|\xi|^2}{|x-\xi|^{2+2/q-2\alpha-\sigma/2} |y-\zeta|^{n-2-\sigma/2}} dx \, dy \, dt \nonumber \\ 
	&\hspace{10mm} \leq C |\xi|^2 \leq C \int_{B_1(0)} \mathcal{G}(u,\varphi)^2 \nonumber 
\end{align*}
where $C = C(n,m,q,\alpha,\varphi^{(0)},\sigma) \in (0,\infty)$.  

To bound the second term on the left-hand side of \eqref{nonconest_concl2}, notice that by \eqref{translate_eqn1}, assuming that $\varepsilon_0$ is small enough that $|\xi| < \tau/2$, 
\begin{align*}
	&\int_{B_{\gamma}(0) \cap \{r > \tau\}} \frac{\mathcal{G}(u(X), \varphi(X) - D_x \varphi(X) \cdot \xi)^2}{|X-Z|^{n+2\alpha-\sigma}} dX 
	\\&\hspace{10mm} \leq 2 \int_{B_{1/4}(Z)}  \frac{\mathcal{G}(u(X),\varphi(X-Z))^2}{|X-Z|^{n+2\alpha-\sigma}} dX 
		+ C \tau^{2\alpha-4} |\xi|^4 \int_{B_{1/4}(Z) \cap \{r > \tau\}} \frac{1}{|X-Z|^{n+2\alpha-\sigma}} dX 
		\\&\hspace{20mm} + C 4^{n+2\alpha-\sigma} \int_{(B_{\gamma}(0) \setminus B_{1/4}(Z)) \cap \{r > \tau\}} (\mathcal{G}(u,\varphi)^2 + |x|^{2\alpha-2} |\xi|^2) 
\end{align*}
for $C = C(\varphi^{(0)}) \in (0,\infty)$, so applying \eqref{nonconest_lemma_eqn2} and Lemma~\ref{branchdist_lemma}, 
\begin{align} \label{nonconest_lemma_eqn5}
	&\int_{B_{\gamma}(0) \cap \{r > \tau\}} \frac{\mathcal{G}(u(X), \varphi(X) - D_x \varphi(X) \cdot \xi)^2}{|X-Z|^{n+2\alpha-\sigma}} dX \\
	&\hspace{10mm} \leq C \int_{B_1(0)} \mathcal{G}(u,\varphi)^2 + C \tau^{2\alpha-4} |\xi|^4 \int_{B_{1/4}(Z) \cap \{r > \tau\}} \frac{1}{|X-Z|^{n+2\alpha-\sigma}} dX \nonumber 
\end{align}
for $C = C(n,m,q,\alpha,\varphi^{(0)},\sigma) \in (0,\infty)$.  By Lemma~\ref{branchdist_lemma}, $|\xi| \leq C \varepsilon_0$ for $C = C(n,m,q,\alpha,\varphi^{(0)}) \in (0,\infty)$.  Take $\varepsilon_0 \leq \tau^2/2C$ so that $|X-Z| \geq \tau/2$ if $X = (x,y)$ with $|x| > \tau$ and $|\xi|^2/\tau^4 < 1$ and thus 
\begin{align} \label{nonconest_lemma_eqn6}
	\tau^{2\alpha-4} |\xi|^4 \int_{B_{1/4}(Z) \cap \{r > \tau\}} \frac{1}{|X-Z|^{n+2\alpha-\sigma}} dX
	&\leq \frac{4^{\alpha} |\xi|^4}{\tau^4} \int_{B_{1/4}(Z)} \frac{1}{|X-Z|^{n-\sigma}} dX \\
	&\leq C |\xi|^2 \leq C \int_{B_1(0)} \mathcal{G}(u,\varphi)^2 \nonumber 
\end{align}
for $C = C(n,m,q,\alpha,\varphi^{(0)},\sigma) \in (0,\infty)$, where the last inequality follows from Lemma~\ref{branchdist_lemma}.  By \eqref{nonconest_lemma_eqn5} and \eqref{nonconest_lemma_eqn6}, 
\begin{equation*}
	\int_{B_{\gamma}(0) \cap \{r > \tau\}} \frac{\mathcal{G}(u(X), \varphi(X) - D_x \varphi(X) \cdot \xi)^2}{|X-Z|^{n+2\alpha-\sigma}} dX 
	\leq C \int_{B_1(0)} \mathcal{G}(u,\varphi)^2 
\end{equation*}
for $C = C(n,m,q,\alpha,\varphi^{(0)},\sigma) \in (0,\infty)$.
\end{proof}

\begin{corollary} \label{nonconest_cor} 
Let $\varphi^{(0)}$ be as in Definition~\ref{varphi0_defn} and let $p \in \{p_0, p_0+1,\ldots, \lceil q/q_0 \rceil \}$.  Given $\delta \in (0,1/2)$, there exists $\varepsilon_0 \in (0,1)$ depending only on $n$, $m$, $q$, $\alpha$, $\varphi^{(0)}$ such that if $u$ satisfies Hypothesis~$(\star)$ of Section~\ref{sec:graphical_sec}, $\varphi \in \Phi_{\varepsilon_{0}, p}(\varphi^{(0)})$ 
and if 
\begin{equation} \label{nonconest_cor_eqn1} 
	B_{\delta}(0,y_0) \cap \{ X \in B_{1/2}(0) \cap \Sigma_{u,q} : \mathcal{N}_u(X) \geq \alpha \} \neq \emptyset 
\end{equation} 
for all $y_0 \in B^{n-2}_{1/2}(0)$,  then 
\begin{equation} \label{nonconest_cor_eqn2} 
	\int_{B_{1/2}(0)} \frac{\mathcal{G}(u,\varphi)^2}{r_{\delta}^{2/q-\sigma}} 
	\leq C \int_{B_1(0)} \mathcal{G}(u,\varphi)^2,  
\end{equation}
for any $\sigma \in (0, 1/q)$, where $r = |x|$, $r_{\delta} = \max\{r,\delta\}$, and $C = C(n,m,q,\alpha,\varphi^{(0)},\sigma) \in (0,\infty)$ is a constant independent of $\delta$. \end{corollary}
\begin{proof} 
Let $\rho \in [\delta,1/16]$ and cover $B_{1/2}(0) \cap \{r \leq \rho\}$ by a finite collection $\{ B_{2\rho}(0,y_j) \}_{j=1,2,\ldots,N}$ of balls with $y_j \in B^{n-2}_{1/2}(0)$ and $N \leq C \rho^{2-n}$ for some constant $C = C(n) \in (0,\infty)$.  By \eqref{nonconest_cor_eqn1}, for each $j = 1,2,\ldots,N$ there exists $Z_j \in B_{\delta}(0,y_j) \cap \Sigma_{u,q}$ with $\mathcal{N}_u(Z_j) \geq \alpha$ and thus by Lemma~\ref{nonconest_lemma}, 
\begin{equation} \label{nonconest_cor_eqn3}
	\rho^{-n+2-2/q+\sigma} \int_{B_{2\rho}(0,y_j)} \mathcal{G}(u,\varphi)^2 \leq C \int_{B_1(0)} \mathcal{G}(u,\varphi)^2 
\end{equation}
for $j = 1,2,\ldots,N$.  Sum \eqref{nonconest_cor_eqn3} over $j = 1,2,\ldots,N$ to get 
\begin{equation} \label{nonconest_cor_eqn4}
	\rho^{-2/q+\sigma} \int_{B_{1/2}(0) \cap \{r \leq \rho\}} \mathcal{G}(u,\varphi)^2 \leq C \int_{B_1(0)} \mathcal{G}(u,\varphi)^2 
\end{equation}
for all $\rho \in [\delta,1/16]$ and some $C = C(n,m,q,\alpha,\varphi^{(0)}) \in (0,\infty)$.  Now replace $\sigma$ with $\sigma/2$ in this inequality, multiply it by $\rho^{\sigma/2-1}$ and integrate over $\rho \in [\delta,1/16]$ to get \eqref{nonconest_cor_eqn2}. 
\end{proof}

\section{The blow-up class} \label{sec:blowupclass_sec} 

In this section we deduce, directly from the estimates in Sections~\ref{sec:L2estimates_sec1} and \ref{sec:L2estimates_sec2}, a number of key estimates for the blow-ups (as in Definition~\ref{blowupclass_defn} below) obtained by the procedure described in Section~\ref{blow-up-procedure}.

\subsection{Definitions and notation}  \label{sec:blowupclass_subsec}  Fix $\varphi^{(0)}$ and let the numbers $\alpha, q_{0}, J, p_{0}, m_{1}, \ldots, m_{J}$ and the functions $\varphi_{1}^{(0)}, \ldots, \varphi_{J}^{(0)}$ be associated with $\varphi^{(0)}$ as in Definition~\ref{varphi0_defn}.   

\begin{definition}\label{parameterspace_defn} Let ${\mathfrak D}$ be the set of ordered collections $\varphi^{(\infty)} = (m_{j,k}, \varphi^{(\infty)}_{j, k})_{1 \leq j \leq J, \,1 \leq k \leq p_{j}}$ consisting of functions $\varphi^{(\infty)}_{j,k}$ and positive integers $m_{j, k}$  where $p_1,\ldots,p_{J}$ are positive integers $($depending on $\varphi^{(\infty)}),$  and for each $j$:
\begin{itemize}
\item[(a)] if $\varphi_{j}^{(0)}$ is non-zero then $\varphi_{j, k}^{(\infty)} = \varphi_{j}^{(0)}$ for each $k \in \{1, \ldots, p_{j}\}$ and $\sum_{k=1}^{p_j} m_{j,k} = m_j$;
\item[(b)] if $\varphi_{j}^{(0)} \equiv 0$ then for each $k \in \{1, \ldots, p_{j}\}$ either $\varphi_{j, k}^{(\infty)} \equiv 0$ or $\varphi^{(\infty)}_{j,k}(X) = \op{Re}(c^{(\infty)}_{j,k} (x_1+ix_2)^{\alpha})$ for some $c^{(\infty)}_{j,k} \in \mathbb{C}^m$ with $|c^{(\infty)}_{j,k}| = 1.$  Moreover, in this case, $\sum_{k=1}^{p_j} m_{j,k} q_{j,k} = m_j,$ where $q_{j,k} = 1$ if $\varphi_{j, k}^{(\infty)} \equiv 0$ and $q_{j,k} = q_0$ if $\varphi_{j, k}^{(\infty)}$ is non-zero.
\end{itemize}
\end{definition}	



\begin{definition} \label{blowupclass_defn}
 The blow-up class $\mathfrak{B}$ consists of all functions $w$  such that $w$ is the blow-up in $B_{1}(0)$, in the sense of Section~\ref{blow-up-procedure}, of a sequence $(u^{(\nu)})_{\nu=1}^{\infty}$ relative to a sequence $(\varphi^{(\nu)})_{\nu = 1}^{\infty}$ by the excesses $E_{\nu} = \left( \int_{B_1(0)} \mathcal{G}(u^{(\nu)},\varphi^{(\nu)})^2 \right)^{1/2}$, where: 
 \begin{itemize}
\item[(a)]  for each $\nu$, $u^{(\nu)} \in W^{1,2}(B_1(0);\mathcal{A}_q(\mathbb{R}^m))$ is an average-free locally energy minimizing $q$-valued function with ${\mathcal N}_{u_{\nu}}(0) \geq \alpha$; 
\end{itemize}
and for some sequences $\varepsilon_{\nu} \downarrow 0$, $\beta_{\nu} \downarrow 0$, $\delta_{\nu} \downarrow 0$, 
\begin{itemize}
\item[(b)] $\varphi^{(\nu)} \in \Phi_{\varepsilon_{\nu}}(\varphi^{(0)})$; 
\item[(c)] for each $y_0 \in B^{n-2}_{1}(0)$ and sufficiently large $\nu$, 
\begin{equation} \label{blowupclass_eqn1}
	B_{\delta_{\nu}}(0,y_0) \cap \{ X \in B_{1}(0) \cap \Sigma_{u^{(\nu)},q} : \mathcal{N}_{u^{(\nu)}}(X) \geq \alpha \} \neq \emptyset;
\end{equation} 
	\item[(d)]  for each $\nu$ 
\begin{equation} \label{blowupclass_eqn2}
	\int_{B_1(0)} \mathcal{G}(u^{(\nu)},\varphi^{(0)})^2 < \varepsilon_{\nu}^2; 
\end{equation}
\item[(e)] for each $\nu$, either (i) $p^{(\nu)} = p_0$ or (ii) $p^{(\nu)} > p_0$ and 
\begin{equation} \label{blowupclass_eqn3}
	\int_{B_1(0)} \mathcal{G}(u^{(\nu)},\varphi^{(\nu)})^2 
	\leq \beta_{\nu} \inf_{\varphi' \in \bigcup_{p'=p_{0}}^{p^{(\nu)}-1} \Phi_{3\varepsilon_{\nu},p'}(\varphi^{(0)})} \int_{B_1(0)} \mathcal{G}(u^{(\nu)},\varphi')^2
\end{equation}
where  $p^{(\nu)}$ is the number of non-zero components of $\varphi^{(\nu)}$ (see Definition~\ref{varphi_defn}).
\end{itemize}
\end{definition}

\begin{remark}\label{blow-up-conditions}
It follows (see the discussion in Section~\ref{blow-up-procedure}) that if $w \in {\mathfrak B}$ and if the functions $\varphi^{(\nu)} \in \Phi_{\varepsilon_{\nu}}(\varphi^{(0)})$ correspond to $w$ as in Definition~\ref{blowupclass_defn}, then there exists $\varphi^{(\infty)} = (m_{j,k}, \varphi^{(\infty)}_{j, k})_{1 \leq j \leq J, \,1 \leq k \leq p_{j}} \in {\mathfrak D}$ such that: 
\begin{enumerate}
\item[(A)]  for each $j \in \{1, \ldots, J\}$ and $k \in \{1, \ldots, p_{j}\}$, $m_{j, k}$ is the multiplicity of the component $\varphi^{(\nu)}_{j,k} \, : \, {\mathbb R}^{n} \to {\mathcal A}_{q_{j, k}}({\mathbb R}^{m})$ of $\varphi^{(\nu)}$  in accordance with Definition~\ref{varphi_defn} (taken with $\varphi^{(\nu)}$ in place of $\varphi$);
 \item[(B)] for each $j \in \{1, \ldots, J\}$ and $k \in \{1, \ldots, p_{j}\}$, precisely one of the following holds: 
		\begin{enumerate}
			\item[(a)] $\varphi^{(0)}_j$ is non-zero, $\varphi^{(\infty)}_{j, k} = \varphi^{(0)}_{j}$, $q_{j, k} = q_{0},$ and $\varphi^{(\nu)}_{j,k} \rightarrow \varphi^{(0)}_j$ uniformly on $B_1(0)$ as $\nu \rightarrow \infty$; 
			\item[(b)] $\varphi^{(0)}_j \equiv 0$, $\varphi^{(\nu)}_{j,k} \equiv 0$  for each $\nu$, $\varphi^{(\infty)}_{j,k} \equiv 0$ and $q_{j,k} = 1$;
			\item[(c)] $\varphi_j^{(0)} \equiv 0$, $\varphi^{(\nu)}_{j,k}(X) = \op{Re}(c^{(\nu)}_{j,k} (x_1+ix_2)^{\alpha})$ for each $\nu$ and some $c^{(\nu)}_{j,k} \in \mathbb{C}^m \setminus \{0\}$, $\varphi^{(\infty)}_{j, k} = \op{Re}(c^{(\infty)}_{j,k} (x_1+ix_2)^{\alpha})$ where $c^{(\infty)}_{j,k} = \lim_{\nu \rightarrow \infty} c^{(\nu)}_{j,k}/|c^{(\nu)}_{j,k}|$, and $q_{j, k} = q_{0}$. 
	\end{enumerate}		
Moreover, if (b) holds true for some $j$ and $k$ then $p^{(\nu)} = \sum_{j=1}^{J} p_{j} - 1$ for each $\nu$; otherwise $p^{(\nu)} = \sum_{j=1}^{J} p_{j}$ for each $\nu$. 

\item[(C)] $w = (w_{j, k})$ where $w_{j,k} : {\rm graph} \, \varphi^{(\infty)}_{j,k} |_{B_1(0) \setminus \{0\} \times {\mathbb R}^{n-2}} \rightarrow \mathcal{A}_{m_{j,k}}(\mathbb{R}^m)$. 
\end{enumerate} 
Thus $\varphi^{(\infty)}$ is determined by the sequence $(\varphi^{(\nu)})$ in the manner described by (A) and (B). 
 \end{remark}
 
 \begin{definition}
 For $\varphi^{(\infty)} \in {\mathfrak D}$, let ${\mathfrak B}(\varphi^{(\infty)})$ denote the set of all $w \in {\mathfrak B}$ associated with $\varphi^{(\infty)}$ as in Remark~\ref{blow-up-conditions}.  
 \end{definition} 
 
We have, of course, that  ${\mathfrak B} = \cup_{\varphi^{(\infty)} \in {\mathfrak D}} {\mathfrak B}(\varphi^{(\infty)}).$

In the following definition, we shall identify a general point $(x_{1}, x_{2}) \in {\mathbb R}^{2}$ with $re^{i\theta}$  where $(r, \theta)$ ($r \geq 0$, $0 \leq \theta < 2\pi$) are the polar coordinates of $(x_{1}, x_{2}).$ 

\begin{definition} \label{homogclass_defn} 
Let ${\mathfrak L} = \cup_{\varphi^{(\infty)} \in {\mathfrak D}} {\mathfrak L}(\varphi^{(\infty)})$ where, for $\varphi^{(\infty)} = (m_{j,k}, \varphi^{(\infty)}_{j, k})_{1 \leq j \leq J, \,1 \leq k \leq p_{j}} \in {\mathfrak D}$, 
$\mathfrak{L}(\varphi^{(\infty)})$ is the set of all functions $\psi = (\psi_{j,k})$ such that $$\psi_{j,k} : {\rm graph} \, \varphi^{(\infty)}_{j,k} |_{B_1(0) \setminus \{0\} \times {\mathbb R}^{n-2}} \rightarrow \mathcal{A}_{m_{j,k}}(\mathbb{R}^m)$$ 
and
\begin{enumerate}
\item[(a)] If $\varphi^{(0)}_j$ is non-zero, then 
\begin{equation} \label{homogclass_eqn1}
	\psi_{j,k}(re^{i\theta},y,\op{Re}(c^{(0)}_j r^{\alpha} e^{i\alpha \theta})) 
	= \sum_{l=1}^{m_{j,k}} \llbracket \op{Re}(a_{j,k,l} r^{\alpha} e^{i\alpha \theta} + \alpha c^{(0)}_j r^{\alpha-1} e^{i(\alpha-1)\theta} (b \cdot y)) \rrbracket 
\end{equation}
for some $a_{j,k,l} \in \mathbb{C}^m$ and some $b \in \mathbb{C}^{n-2}$, with $b$ independent of $j$ and $k$.  (Recall that in this case $\varphi^{(\infty)}_{j,k}(X) = \varphi^{(0)}_j(X) = \op{Re}(c^{(0)}_j (x_1+ix_2)^{\alpha})$ for some $c_{j}^{(0)} \in {\mathbb C}^{m} \setminus \{0\}$ and all 
$X  = (x_{1}, x_{2}, y) \in {\mathbb R}^{n}$);

\item[(b)] If $\varphi^{(0)}_j$ and $\varphi^{(\infty)}_{j,k}$ are both identically zero, then either $\psi_{j,k}(re^{i\theta},y,0) = m_{j,k} \llbracket 0 \rrbracket$ or 
\begin{equation} \label{homogclass_eqn2}
	\psi_{j,k}(re^{i\theta},y,0) 
	= (m_{j,k} - q_0 N_{j,k}) \llbracket 0 \rrbracket 
		+ \sum_{l=1}^{N_{j,k}} \sum_{h=0}^{q_0-1} \llbracket \op{Re}(a_{j,k,l} r^{\alpha} e^{i\alpha (\theta + 2\pi h)}) \rrbracket 
\end{equation}
for some integer $N_{j,k}$ with $1 \leq N_{j,k} \leq m_{j,k}/q_0$ and some $a_{j,k,l} \in \mathbb{C}^m \setminus \{0\}$;

\item[(c)] If $\varphi^{(0)}_j$ is identically zero and $\varphi^{(\infty)}_{j,k}$ is nonzero, then 
\begin{equation} \label{homogclass_eqn3}
	\psi_{j,k}(re^{i\theta},y,\op{Re}(c^{(\infty)}_{j,k} r^{\alpha} e^{i\alpha \theta})) 
	= \sum_{l=1}^{m_{j,k}} \llbracket \op{Re}(a_{j,k,l} r^{\alpha} e^{i\alpha \theta}) \rrbracket 
\end{equation}
for some $a_{j,k,l} \in \mathbb{C}^m$.

\end{enumerate} \end{definition}

Let the notation be as above. We need the following further notation, which is completely analogous to that of Section~\ref{functions-on-graphs}.  For each $\varphi^{(\infty)} = (m_{j,k}, \varphi^{(\infty)}_{j, k}) \in {\mathfrak D}$ and any ball $B \subset B_{1}(0) \setminus \{0\} \times \mathbb{R}^{n-2}$, there are single-valued harmonic functions $\varphi^{(\infty)}_{j,k,l} : B \rightarrow \mathbb{R}^m$ such that on $B$, 
\begin{equation} \label{varphi_infty_localized}
	\varphi^{(\infty)}_{j,k}(X) = \sum_{l=1}^{q_{j, k}} \llbracket \varphi^{(\infty)}_{j,k,l}(X) \rrbracket
\end{equation}
where  
\begin{align}\label{varphi_infty_values}
&q_{j, k} = q_{0} \;\; \text{if $\varphi^{(\infty)}_{j,k}$ is non-zero and}\\
& q_{j, k} = 1 \;\; \text{(with $\varphi^{(\infty)}_{j, k, l}(X) = \varphi^{(\infty)}_{j, k, 1}(X) = 0 \in {\mathbb R}^{m}$) if $\varphi^{(\infty)}_{j, k}$ is the zero function} \nonumber.
\end{align}
We may express $w = (w_{j,k}) \in \mathfrak{B}(\varphi^{(\infty)})$ as 
\begin{equation} \label{w_localized}
	w_{j,k}(X,\varphi^{(\infty)}_{j,k,l}(X)) = w_{j,k,l}(X) = \sum_{h=1}^{m_{j,k}} \llbracket w_{j,k,l,h}(X) \rrbracket 
\end{equation}
for $X \in B_{1}(0) \setminus \{0\} \times {\mathbb R}^{n-2}$ where $\varphi^{(\infty)}$ is as in Remark~\ref{blow-up-conditions} and, by the component-wise energy minimality of $w$ away from $\{0 \} \times {\mathbb R}^{n-2}$, for $\mathcal{L}^n$-a.e.\ point $Y \in B_{1}(0)$ (in fact for $Y$ away from the union of the $(n-2)$-dimensional singular set of $w_{j, k, l}$ and the axis $\{0\} \times {\mathbb R}^{n-2}$) $w_{j, k, l,h}$ can be chosen to be smooth ${\mathbb R}^{m}$-valued  functions defined locally near $Y$.  Similarly, we may express $\psi = (\psi_{j,k}) \in \mathfrak{L}(\varphi^{(\infty)})$ away from $\{0\} \times {\mathbb R}^{n-2}$ as 
\begin{equation} \label{psi_localized}
	\psi_{j,k}(X,\varphi^{(\infty)}_{j,k,l}(X)) = \psi_{j,k,l}(X) = \sum_{h=1}^{m_{j,k}} \llbracket \psi_{j,k,l,h}(X) \rrbracket 
\end{equation}
with $\psi_{j,k,l,h}$ being given by smooth ${\mathbb R}^{m}$-valued functions locally near $\mathcal{L}^n$-a.e.\ point in $B_{1}(0)$. 

Let $w \in {\mathfrak B}$ and let $\varphi^{(\nu)}$, $u^{(\nu)}$ correspond to $w$ as in Definition~\ref{blowupclass_defn}. We shall use the following convention: if in an open ball $B \subset \subset B_{1}(0) \setminus \{0\} \times \mathbb{R}^{n-2}$ we express 
\begin{equation}\label{varphi_nu_localized}
\varphi^{(\nu)}_{j,k}(X) = \sum_{l=1}^{q_{j,k}} \llbracket \varphi^{(\nu)}_{j,k,l}(X) \rrbracket
\end{equation}
 as in \eqref{varphi_localized} taken with $\varphi^{(\nu)}$ in place of $\varphi$ and we express $\varphi^{(\infty)}_{j,k}$ as in \eqref{varphi_infty_localized} for single-valued harmonic functions $\varphi^{(\nu)}_{j,k,l} : B \rightarrow \mathbb{R}^m$ and $\varphi^{(\infty)}_{j,k,l} : B \rightarrow \mathbb{R}^m$, then 
\begin{enumerate}
	\item[(a)] when $\varphi^{(0)}_j$ is nonzero, $\varphi^{(\nu)}_{j,k,l} \rightarrow \varphi^{(\infty)}_{j,k,l}$ uniformly on $B$ as $\nu \rightarrow \infty$; 
	\item[(b)] when $\varphi^{(0)}_j$ and $\varphi^{(\nu)}_{j, k}$ are identically zero, $\varphi^{(\nu)}_{j,k,1}(X) = \varphi^{(\infty)}_{j,k,1}(X) = 0$ 
		on $B$;
	\item[(c)] when $\varphi^{(0)}_j$ is identically zero and $\varphi^{(\nu)}_{j,k}(X) = \op{Re}(c^{(\nu)}_{j,k}(x_1+ix_2)^{\alpha})$ for $c^{(\nu)}_{j,k} \in \mathbb{C}^m \setminus \{0\}$, $\varphi^{(\nu)}_{j,k,l}/|c^{(\nu)}_{j,k}| \rightarrow \varphi^{(\infty)}_{j,k,l}$ uniformly on $B$ as $\nu \rightarrow \infty$. 
\end{enumerate}
Thus if we express $v^{(\nu)}_{j,k}$ is as in Corollary~\ref{graphical_cor} with $u = u^{(\nu)}$ and $\varphi = \varphi^{(\nu)}$ and we let $v^{(\nu)}_{j,k,l}(X) = v^{(\nu)}_{j,k}(X,\varphi^{(\nu)}_{j,k,l}(X))$ as in \eqref{v_localized} and $w_{j,k,l}(X) = w_{j,k}(X,\varphi^{(\infty)}_{j,k,l}(X))$ as in \eqref{w_localized}, then $v^{(\nu)}_{j,k,l}/E_{\nu} \rightarrow w_{j,k,l}$ uniformly on $B$ as $\nu \rightarrow \infty$. 



\begin{definition} \label{homogprojection_defn}
Let $w \in {\mathfrak B}$ and $\rho \in (0, 1/2]$. We say that $\psi \in {\mathfrak L}$ is a projection of $w$ onto ${\mathfrak L}$ in $B_{\rho}(0)$ if 
\begin{itemize}
	\item[(i)] there exists $\varphi^{(\infty)} \in {\mathfrak D}$ such that $w = (w_{j,k}) \in {\mathfrak B}(\varphi^{(\infty)})$ and $\psi = (\psi_{j,k}) \in \mathfrak{L}(\varphi^{(\infty)})$;
	\item[(ii)] $\int_{B_{\rho}(0)}\sum_{j=1}^{J} \sum_{k=1}^{p_{j}}\sum_{l=1}^{q_{j, k}}  |w_{j, k, l}(X)|^{2} < \infty$ and 
	\item[(iii)] 
\begin{align*}
	&\int_{B_{\rho}(0)} \sum_{j=1}^J \sum_{k=1}^{p_j} \sum_{l=1}^{q_{j,k}} \mathcal{G}(w_{j,k,l}(X),\psi_{j,k,l}(X))^2 \,dX 
	\\&\hspace{10mm} = \inf_{\psi' \in \mathfrak{L}(\varphi^{(\infty)})} 
		\int_{B_{\rho}(0)} \sum_{j=1}^J \sum_{k=1}^{p_j} \sum_{l=1}^{q_{j,k}} \mathcal{G}(w_{j,k,l}(X),\psi'_{j,k,l}(X))^2 \,dX,  
\end{align*}
where $w_{j,k,l}$ is as in \eqref{w_localized}, $\psi_{j,k,l}$ is as in \eqref{psi_localized} and $\psi'_{j,k,l}$ is as in \eqref{psi_localized} with $\psi'$ in place of $\psi$.  
\end{itemize} \end{definition}

We shall establish below (in Lemma~\ref{blowup_norm_conv_lemma})  that if $w \in {\mathfrak B}$, then in fact $$\int_{B_{1/2}(0)} \sum_{j=1}^J \sum_{k=1}^{p_j} \sum_{l=1}^{q_{j,k}} |w_{j,k,l}|^2 \leq 1,$$  so the requirement (ii) in the above definition will always be satisfied. It is easy to see that given $w \in {\mathfrak B}$ and $\rho \in (0, 1/2]$, at least one projection $\psi$ of $w$ onto $\mathfrak{L}$ in $B_{\rho}(0)$ exists;  we do not assert that $\psi$ is the unique projection of $w$ onto $\mathfrak{L}$ in $B_{\rho}(0)$.

\subsection{Elements of $\mathfrak{L}$ as blow-ups of cylindrical functions}
In the following lemma and subsequently, we let $r(X) = |(x_{1}, x_{2})|$ for $X  = (x_{1}, x_{2}, y) \in {\mathbb R}^{n}$ where $(x_{1}, x_{2}) \in {\mathbb R}^{2}, y \in {\mathbb R}^{n-2}$. The lemma establishes the elementary fact that each $\psi \in {\mathfrak L}$ arises as the ``blow-up'' of a sequence of  cylindrical functions in $\widetilde{\Phi}_{\varepsilon}(\varphi^{(0)})$ (for any $\varepsilon$ satisfying \eqref{varphi_welldefined}) converging to $\varphi^{(0)}$. 
\begin{lemma} \label{blowup2L_lemma} 
Let $\psi  = (\psi_{j, k}) \in \mathfrak{L}$. For each $\nu =1, 2, 3, \ldots,$ let $\varphi^{(\nu)} \in \Phi_{\varepsilon_{\nu}}(\varphi^{(0)})$ where $\varepsilon_{\nu} \to 0^{+}$ as $\nu \to \infty$. Suppose that $\varphi^{(\nu)}$ satisfy the requirements of Remark~\ref{blow-up-conditions}(A)(B), taken with $\varphi^{(\infty)}  = (m_{j, k}, \varphi^{(\infty)}_{j, k}) \in {\mathfrak D}$  for which $\psi \in {\mathfrak L}(\varphi^{(\infty)})$. Let $(E_{\nu})$ be a sequence of numbers with $E_{\nu} \to 0^{+}$. 
For  each $\nu$ there exists $\widetilde{\varphi}^{(\nu)} \in \widetilde\Phi_{\varepsilon_{\nu} + C E_{\nu}}(\varphi^{(0)})$ where $C = C_{1}(n, q, \alpha, \varphi^{(0)})\|\psi\|_{L^{2}(B_{1}(0))}$ such that: 
\begin{itemize}
\item[(i)] for any ball $B \subset \{X \, : \, |r(X)| \geq 4E_{\nu}|b|\}$ and any $X  \in B$,  
\begin{equation} \label{blowup2L_eqn5}
	\widetilde{\varphi}^{(\nu)}_{j,k}(X) 
	= \sum_{l=1}^{q_{j,k}} \sum_{h=1}^{m_{j,k}} \llbracket \varphi^{(\nu)}_{j,k,l}(X) + E_{\nu} \psi_{j,k,l,h}(X) + \mathcal{R}_{j,k,l,h}(X) \rrbracket
\end{equation}
for ${\mathcal R}_{j, k, l, h}(X) \in {\mathbb R}^{m}$ with 
\begin{equation}\label{homog_error}
	|\mathcal{R}_{j,k,l,h}(X)| \leq C E_{\nu} \varepsilon_{\nu} |b| r(X)^{\alpha-1} + C E_{\nu}^2 |a| |b| r(X)^{\alpha-1} + C E_{\nu}^{2} |b|^{2} r(X)^{\alpha-2}, 
\end{equation}
where $a = (a_{j, k, l})$ and $b$ are as in Definition~\ref{homogclass_defn}, $\psi_{j, k, l, h}$ and $\varphi^{(\nu)}_{j, k, l}$ are as in \eqref{psi_localized} and \eqref{varphi_nu_localized} respectively, and $C = C(n, q, \alpha, \varphi^{(0)})$;
\item[(ii)] For $\tau \in (0, 1/4)$ and $\nu$ sufficiently large, 
\begin{equation}\label{blowup2L_eqn6}
	\int_{B_{1/2}(0) \cap \{r < \tau\}} \mathcal{G}(\widetilde{\varphi}^{(\nu)},\varphi^{(\nu)})^2 
	\leq C \|\psi\|^{2}_{L^{2}(B_{1}(0))}\tau^{2\alpha} E_{\nu}^2
\end{equation}
where $C = C(n, q, \alpha, \varphi^{(0)})$; 
\item[(iii)] \begin{equation} \label{blowup2L_eqn1}
	\int_{B_{1/2}(0)} \sum_{j=1}^J \sum_{k=1}^{p_j} \sum_{l=1}^{q_{j,k}} |\psi_{j,k,l}(X)|^2 \,dX 
		= \lim_{\nu \rightarrow \infty} \frac{1}{E_{\nu}^2} \int_{B_{1/2}(0)} \mathcal{G}(\widetilde{\varphi}^{(\nu)}(X),\varphi^{(\nu)}(X))^2 \,dX, 
\end{equation}
where $\psi_{j,k,l}$ is as in \eqref{psi_localized}. 

Moreover, if $\psi$ is as in Definition~\ref{homogclass_defn} with $b=0$, then $\widetilde{\varphi}^{(\nu)} \in \Phi_{\varepsilon_{\nu} + C E_{\nu}}(\varphi^{(0)})$ where $C = C_{1}(n, q, \alpha)\|\psi\|_{L^{2}(B_{1}(0))}$. 
\end{itemize}
\end{lemma}

\begin{proof}
Let $\psi$ be as in Definition~\ref{homogclass_defn}.  Let  
\begin{equation} \label{B_matrix}
	B = \left(\begin{matrix}
		0 & 0 & \op{Re}(b) \\
		0 & 0 & \op{Im}(b) \\
		-\op{Re}(b^{\rm T}) & -\op{Im}(b^{\rm T}) & 0 
	\end{matrix}\right)
\end{equation}
where $b$ is as in Definition~\ref{homogclass_defn} represented as a row vector and $b^{\rm T}$ is its transpose.  We define 
\begin{equation*}
	\widetilde{\varphi}^{(\nu)}(X) = \sum_{j=1}^J \sum_{k=1}^{p_j} \widetilde{\varphi}^{(\nu)}_{j,k}(X) 
\end{equation*}
for all $X \in \mathbb{R}^n$ where $\widetilde{\varphi}^{(\nu)}_{j,k} : \mathbb{R}^n \rightarrow \mathcal{A}_{m_{j,k} q_{j,k}}(\mathbb{R}^m)$ are homogeneous degree $\alpha$ functions such that  
\begin{enumerate}
\item[(a)]  if $\varphi^{(0)}_j$ is nonzero, then 
\begin{equation} \label{blowup2L_eqn2}
	\widetilde{\varphi}^{(\nu)}_{j,k}(e^{-E_{\nu} B} X) 
	= \sum_{l=1}^{m_{j,k}} \sum_{h=0}^{q_0-1} \llbracket \op{Re}((c^{(\nu)}_{j,k} + E_{\nu} a_{j,k,l}) r^{\alpha} e^{i\alpha (\theta + 2\pi h)}) \rrbracket 
\end{equation}
where $a_{j,k,l}$ are as in Definition~\ref{homogclass_defn}(a) and $\varphi^{(\nu)}_{j,k}(X) = \op{Re}(c^{(\nu)}_{j,k} (x_1+ix_2)^{\alpha})$ for $c^{(\nu)}_{j,k} \in \mathbb{C}^m$ (with $|c^{(\nu)}_{j,k} - c^{(0)}_j| \leq C(n,q,\alpha) \varepsilon_{\nu}$); 

\item[(b)] if $\varphi^{(0)}_j$ and $\varphi^{(\infty)}_{j,k}$ are both identically zero, then $\widetilde{\varphi}^{(\nu)}_{j,k}(e^{-E_{\nu} B} X) = E_{\nu} \psi_{j,k}(X,0)$, where $\psi_{j,k}$ is as in Definition~\ref{homogclass_defn}(b); and 

\item[(c)] if $\varphi^{(0)}_j$ is identically zero and $\varphi^{(\infty)}_{j,k}$ is nonzero, then $\widetilde{\varphi}^{(\nu)}_{j,k}$ is given by \eqref{blowup2L_eqn2} where $a_{j,k,l}$ are as in Definition~\ref{homogclass_defn}(c) and $c^{(\nu)}_{j,k} \in \mathbb{C}^m \setminus \{0\}$ are such that $\varphi^{(\nu)}_{j,k}(X) = \op{Re}(c^{(\nu)}_{j,k} (x_1+ix_2)^{\alpha})$ (so that, per Remark~\ref{blow-up-conditions}(B)(c), 
$c^{(\infty)}_{j,k} = \lim_{\nu\rightarrow\infty} c^{(\nu)}_{j,k}/|c^{(\nu)}_{j,k}|$).
\end{enumerate}

By the pairwise orthogonality of $\cos(\alpha \theta)$, $\sin(\alpha \theta)$, $\cos((\alpha-1)\theta)$, and $\sin((\alpha-1)\theta)$ in $L^2([0,2\pi q_0])$, 
\begin{equation} \label{blowup2L_eqn3}
	\sum_l |a_{j,k,l}|^2 + |b|^2 |c^{(0)}_j|^2 \leq C(n,q,\alpha) \int_{B_{1/2}(0)} \sum_{l=1}^{q_{j,k}} |\psi_{j,k,l}|^2
\end{equation}
for $j \in \{1, \ldots, J\}$ such that $\varphi_{j}^{(0)}$ is non-zero and $k \in \{1, \ldots, p_{j}\}$, and 
\begin{equation} \label{blowup2L_eqn3-1}
	\sum_l |a_{j,k,l}|^2 \leq C(n,q,\alpha) \int_{B_{1/2}(0)}\sum_{l=1}^{q_{j,k}} |\psi_{j,k,l}|^2
\end{equation}
for $j, k$ such that $\varphi^{(0)}_{j} \equiv 0$ and $\psi_{j, k} (\cdot) \not\equiv m_{j, k} \llbracket 0 \rrbracket$.  It readily follows that $\widetilde{\varphi}^{(\nu)} \in \widetilde{\Phi}_{\varepsilon_{\nu} + C E_{\nu}}(\varphi^{(0)})$ where $C = C_{1}(n, q, \alpha, \varphi^{(0)})\|\psi\|_{L^{2}(B_{1}(0))}.$  Moreover $\widetilde{\varphi}^{(\nu)} \in {\Phi}_{\varepsilon_{\nu} + CE_{\nu}}(\varphi^{(0)})$ where $C = C_{1}(n, q, \alpha)\|\psi\|_{L^{2}(B_{1}(0))}$ in case $\psi$ is as in Definition~\ref{homogclass_defn} with $b=0$. 

By Taylor's theorem, for each smooth function $f : B_{\rho}(X_0) \rightarrow \mathbb{R}^m$ we have 
\begin{align} \label{blowup2L_eqn4}
	f(e^{E_{\nu} B} X) ={}& f(X) + \nabla f(X) \cdot E_{\nu} B X + E_{\nu}^2 \int_0^1 (1-t) \,\nabla f(e^{t E_{\nu} B} X) [e^{t E_{\nu} B} B^2 X] \,dt \\&
		+ E_{\nu}^2 \int_0^1 (1-t) \,\nabla^2 f(e^{t E_{\nu} B} X)[e^{t E_{\nu} B} B X,e^{t E_{\nu} B} B X] \,dt \nonumber 
\end{align}
for all $X$ such that $e^{t E_{\nu} B} X \in B_{\rho}(X_0)$ for all $t \in [0,1]$, where $\nabla f(X)[v] = \nabla_v f(X)$ and $\nabla^2 f(X)[v,v] = \nabla_v \nabla_v f(X)$ at a point $X$ in a direction $v$.  Since $\widetilde{\varphi}^{(\nu)}_{j, k} \circ e^{-E_{\nu} B}$ is a regular $q$-valued function away from $\{0\} \times \mathbb{R}^{n-2}$, we can apply \eqref{blowup2L_eqn4} with $\widetilde{\varphi}^{(\nu)}_{j,k} \circ e^{-E_{\nu} B} |_{B_{r(X_0)/2}(X_0)}$ in place of $f(X)$ in order to expand $\widetilde{\varphi}^{(\nu)}$.  In particular, notice that $|r(e^{t E_{\nu} B} X) - r(X)| \leq |e^{t E_{\nu} B} X - X| \leq E_{\nu} |b|$ for all $X \in B_1(0)$ and $t \in [0,1]$.  Hence by \eqref{blowup2L_eqn4} and the definition of $\widetilde{\varphi}^{(\nu)}_{j,k}$, for all $X \in B_1(0)$ with $r(X) \geq 4 E_{\nu} |b|$, 
\begin{equation*} 
	\widetilde{\varphi}^{(\nu)}_{j,k}(X) 
	= \sum_{l=1}^{q_{j,k}} \sum_{h=1}^{m_{j,k}} \llbracket \varphi^{(\nu)}_{j,k,l}(X) + E_{\nu} \psi_{j,k,l,h}(X) + \mathcal{R}_{j,k,l,h}(X) \rrbracket
\end{equation*}
where $\varphi^{(\nu)}_{j,k,l}$ is as in \eqref{varphi_localized} with $\varphi^{(\nu)}$ in place of $\varphi$, $\psi_{j,k,l,h}$ is as in \eqref{psi_localized}, and 
\begin{equation*}
	|\mathcal{R}_{j,k,l,h}(X)| \leq C E_{\nu} \varepsilon_{\nu} |b| r(X)^{\alpha-1} + C E_{\nu}^2 |a| |b| r(X)^{\alpha-1} + C E_{\nu}^2 |b|^2 r(X)^{\alpha-2} 
\end{equation*}
for some constant $C = C(n,q,\alpha,\varphi^{(0)}) \in (0,\infty)$.  Thus, for each $\tau \in (0,1/4)$ and $\nu$ sufficiently large (in particular large enough that $4E_{\nu} |b| < \tau$)
\begin{align*} \label{blowup2L_eqn6}
	\int_{B_{1/2}(0) \cap \{r < \tau\}} \mathcal{G}(\widetilde{\varphi}^{(\nu)},\varphi^{(\nu)})^2 
	&\leq C \|\psi\|^{2}_{L^{2}(B_{1}(0))}\int_{B_{1/2}(0) \cap \{4E_{\nu} |b| < r < \tau\}} E_{\nu}^2 r^{2\alpha-2} 
		+ C \int_{B_{1/2}(0) \cap \{r < 4E_{\nu} |b|\}} r^{2\alpha}  
	\\&\leq C \|\psi\|^{2}_{L^{2}(B_{1}(0))}\tau^{2\alpha} E_{\nu}^2 \nonumber 
\end{align*}
where $C = C(n, q, \alpha, \varphi^{(0)})$. Also directly from the definition of $\psi$,   
\begin{equation} \label{blowup2L_eqn7}
	\int_{B_{1/2}(0) \cap \{r < \tau\}} \sum_{j=1}^J \sum_{k=1}^{p_j} \sum_{l=1}^{q_{j,k}} |\psi_{j,k,l}|^2
	\leq C \int_{B_{1/2}(0) \cap \{r < \tau\}} r^{2\alpha-2} 
	\leq C \tau^{2\alpha} 
\end{equation}
for $\nu$ large enough that $4E_{\nu} < \tau$, where $C = C(n,q,\alpha,\varphi^{(0)},\psi) \in (0,\infty)$.  
By \eqref{blowup2L_eqn5},  \eqref{homog_error}, \eqref{blowup2L_eqn6} and \eqref{blowup2L_eqn7}, we have \eqref{blowup2L_eqn1}.
\end{proof}

\subsection{Main estimates for the blow-ups}
We shall now derive estimates for $w \in \mathfrak{B}$ from the estimates in Sections \ref{sec:L2estimates_sec1} and \ref{sec:L2estimates_sec2}.  These estimates will form the basis of our asymptotic analysis of the blow-ups (carried out in Sections~\ref{sec:homogblowup_sec} and \ref{sec:asymptotics_sec}) which in turn will play a key role in the proof (given in Section~\ref{sec:main_lemma_sec}) of the main excess decay result, Lemma~\ref{main_lemma}. 

\begin{lemma} \label{blowup_norm_conv_lemma}
Let $w = (w_{j,k}) \in \mathfrak{B}$ be the blow-up of $u^{(\nu)}$ relative to $\varphi^{(\nu)}$ and excess $E_{\nu} = \left( \int_{B_1(0)} \mathcal{G}(u^{(\nu)},\varphi^{(\nu)}) \right)^{1/2}$ as in Definition~\ref{blowupclass_defn}, and let $\psi \in {\mathfrak L}$.  Suppose that there exists $\varphi^{(\infty)} = (m_{j, k}, \varphi^{(\infty)}_{j, k}) \in {\mathfrak D}$ such that $w \in {\mathfrak B}(\varphi^{(\infty)})$, $\psi \in {\mathfrak L}(\varphi^{(\infty)})$ and the sequence $(\varphi^{(\nu)})$ is associated with $\varphi^{(\infty)}$ in accordance with Remark~\ref{blow-up-conditions}(A)(B).  For each $\nu = 1, 2, 3, \ldots,$ let  $\widetilde{\varphi}^{(\nu)}$ be the function corresponding to $\varphi^{(\nu)}$, $\psi$, $E_{\nu}$ given by Lemma~\ref{blowup2L_lemma}.  Then for every $\rho \in (0,1/2]$, 
\begin{equation} \label{blowup_norm_conv_eqn1}
	\lim_{\nu \rightarrow \infty} E_{\nu}^{-2} \int_{B_{\rho}(0)} \mathcal{G}(u^{(\nu)},\widetilde{\varphi}^{(\nu)})^2 
	= \int_{B_{\rho}(0)} \sum_{j=1}^J \sum_{k=1}^{p_j} \sum_{l=1}^{q_{j,k}} \mathcal{G}(w_{j,k,l},\psi_{j,k,l})^2, 
\end{equation} 
where $w_{j,k,l}$ is as in \eqref{w_localized} and $\psi_{j,k,l}$ is as in \eqref{psi_localized}.  In particular, 
\begin{equation*}
	\lim_{\nu \rightarrow \infty} E_{\nu}^{-2} \int_{B_{\rho}(0)} \mathcal{G}(u^{(\nu)},\varphi^{(\nu)})^2 
	= \int_{B_{\rho}(0)} \sum_{j=1}^J \sum_{k=1}^{p_j} \sum_{l=1}^{q_{j,k}} |w_{j,k,l}|^2
\end{equation*}
and so
\begin{equation} \label{blowup_norm_conv_eqn2}
\int_{B_{1/2}(0)} \sum_{j=1}^J \sum_{k=1}^{p_j} \sum_{l=1}^{q_{j,k}} |w_{j,k,l}|^2 \leq 1.
\end{equation}

\end{lemma}

\begin{proof}
By Corollary \ref{nonconest_cor}, for every $\tau \in (0,1/4)$ and $\nu$ sufficiently large (depending on $\tau$), 
\begin{equation*} 
	\int_{B_{1/2}(0) \cap \{r < \tau\}} \mathcal{G}(u^{(\nu)},\varphi^{(\nu)})^2 \leq C \tau^{1/q} E_{\nu}^2,  
\end{equation*}
for some constant $C = C(n,m,q,\alpha,\varphi^{(0)}) \in (0,\infty)$.  In view of this, the first conclusion \eqref{blowup_norm_conv_eqn1} follows from the definition of blow-up, \eqref{blowup2L_eqn5} and the estimates \eqref{homog_error} and  
\eqref{blowup2L_eqn6}. The second conclusion follows by setting $\psi = 0$ in the first.
\end{proof}

\begin{lemma} \label{blowup_est_lemma} 
The following hold true for each $w = (w_{j,k}) \in \mathfrak{B}$: 
\begin{enumerate}
\item[(a)] for each $\psi \in \mathfrak{L}(\varphi^{(\infty)})$, where $\varphi^{(\infty)} \in {\mathfrak D}$ is such that $w \in {\mathfrak B}(\varphi^{(\infty)})$,  and each $\rho \in (0,1/2]$ and $\gamma \in (0,1/2]$, 
\begin{align} 
	\label{blowup_est1} &\int_{B_{\gamma \rho}(0)} \sum_{j=1}^J \sum_{k=1}^{p_j} \sum_{l=1}^{q_{j,k}} 
		R^{2-n} \left| \frac{\partial (w_{j,k,l}/R^{\alpha})}{\partial R} \right|^2 
		\leq C \rho^{-n-2\alpha} \int_{B_{\rho}(0)} \sum_{j=1}^J \sum_{k=1}^{p_j} \sum_{l=1}^{q_{j,k}} \mathcal{G}(w_{j,k,l},\psi_{j,k,l})^2, \\
	\label{blowup_est2} &\int_{B_{\gamma \rho}(0)} \sum_{j=1}^J \sum_{k=1}^{p_j} \sum_{l=1}^{q_{j,k}} \frac{\mathcal{G}(w_{j,k,l},\psi_{j,k,l})^2}{r^{2+1/q}} 
		\leq C \rho^{-2-1/q} \int_{B_{\rho}(0)} \sum_{j=1}^J \sum_{k=1}^{p_j} \sum_{l=1}^{q_{j,k}} \mathcal{G}(w_{j,k,l},\psi_{j,k,l})^2, 
\end{align}
where $w_{j,k,l}$ is as in \eqref{w_localized}, $\psi_{j,k,l}$ is as in \eqref{psi_localized}, $R = |X|$, $r = |(x_1,x_2)|$, and $C = C(n,m,q,\alpha,\varphi^{(0)},\gamma) \in (0,\infty)$ is a constant;
\item[(b)] 
\begin{equation} \label{blowup_est3} 
	\int_{B_{1/2}(0)} \sum_{j=1}^J \sum_{k=1}^{p_j} \sum_{l=1}^{q_{j,k}} |D_y w_{j,k,l}|^2 \leq C, 
\end{equation}
where $X = (x,y)$ for $x = (x_1,x_2)$ and $y = (x_3,\ldots,x_n)$ and $C = C(n,m,q,\alpha,\varphi^{(0)}) \in (0,\infty)$ is a constant; 
\item[(c)] there exists a function $\lambda : B^{n-2}_{1/4}(0) \rightarrow \mathbb{R}^2$ such that for  each $\sigma \in (0,1/q)$, 
\begin{equation} \label{blowup_est4}
	\sup_{z \in B^{n-2}_{1/4}(0)} \int_{B_{1/2}(0)} \sum_{j=1}^J \sum_{k=1}^{p_j} \sum_{l=1}^{q_{j,k}} \sum_{h=1}^{m_{j, k}}\frac{|w_{j,k,l, h}(X) - D_x \varphi^{(0)}_{j, l}(X) \cdot \lambda(z)|^2}{
		|X-(0,z)|^{n+2\alpha-\sigma}} dX \leq C 
\end{equation}
and $\|\lambda\|_{C^{1-\sigma/2}(B_{1/2}(0))} \leq C$, where $C = C(n,m,q,\alpha,\varphi^{(0)},\sigma) \in (0,\infty)$.  

Moreover, $\lambda(z)$ is unique subject to the condition  
$$\int_{B_{1/2}(0)} \sum_{j=1}^J \sum_{k=1}^{p_j} \sum_{l=1}^{q_{j,k}} \sum_{h=1}^{m_{j, k}}\frac{|w_{j,k,l,h}(X) - D_x \varphi^{(0)}_{j, l}(X) \cdot \lambda(z)|^2}{
		|X-(0,z)|^{n+2\alpha-\sigma}} dX < \infty \;\;\; \mbox{for some $\sigma \in (0, 1/q)$}.$$ 
Here for $X \in B_{1/2}(0) \setminus \{0\} \times {\mathbb R}^{n-2}$, 
$w_{j, k, l, h}(X) \in {\mathbb R}^{m}$ are as in \eqref{w_localized}; 
for $j$ such that $\varphi^{(0)}_{j} \neq 0$, $q_{j, k} = q_{0}$ and $\varphi^{(0)}_{j, l}$ $(1 \leq l \leq q_{0})$ are ${\mathbb R}^{m}$ valued harmonic functions locally defined near $X$ such that 
$\varphi^{(0)}_{j}(X)  = \sum_{l=1}^{q_{0}} \llbracket \varphi^{(0)}_{j, l}(X)\rrbracket$; and for $j$ such that $\varphi^{(0)}_{j} = 0$, $\varphi^{(0)}_{j, l} = 0$ for $1 \leq l \leq q_{j, k}$.
\end{enumerate}

 \end{lemma}

\begin{proof}
Let $\varphi^{(\infty)} \in {\mathfrak D}$, $w \in {\mathfrak B}(\varphi^{(\infty)})$ and $\psi \in {\mathfrak L}(\varphi^{(\infty)})$, and suppose that $w$ is the blow-up of $u^{(\nu)}$ relative to $\varphi^{(\nu)}$ by the excess $E_{\nu} = \left( \int_{B_1(0)} \mathcal{G}(u^{(\nu)},\varphi^{(\nu)}) \right)^{1/2}$ as in Definition~\ref{blowupclass_defn}.  Let $\widetilde{\varphi}^{(\nu)}$ be the function corresponding to $\psi$,  $\varphi^{(\nu)}$ and $E_{\nu}$ as in Lemma~\ref{blowup2L_lemma}.  

To see \eqref{blowup_est1}, note that $\mathcal{N}_{u^{(\nu)}}(0) \geq \alpha$ and thus by Theorem~\ref{keyest_thm}(a) (applied with the ``scaled and rotated'' functions $\rho^{-\alpha}u^{(\nu)}(\rho e^{-E_{\nu}B}X)$ and $\widetilde{\varphi}^{(\nu)}(e^{-E_{\nu}B}X)$ in place of $u$ and $\varphi$, and then making the change of coordinates  $\rho e^{-E_{\nu}B}X \mapsto X$), 
\begin{equation} \label{blowup_est1_eqn1} 
	\int_{B_{\gamma \rho}(0)} R^{2-n} \left| \frac{\partial (u^{(\nu)}/R^{\alpha})}{\partial R} \right|^2 
	\leq C \rho^{-n-2\alpha} \int_{B_{\gamma \rho}(0)} \mathcal{G}(u^{(\nu)},\widetilde{\varphi}^{(\nu)})^2 
\end{equation}
for some constant $C = C(n,m,q,\alpha,\varphi^{(0)},\gamma) \in (0,\infty)$.  By the homogeneity of $\varphi^{(\nu)}$, for each $\tau >0$ and all sufficiently large $\nu$, 
\begin{equation}\label{blowup_est1_eqn1_1} 
	\left| \frac{\partial (u^{(\nu)}/R^{\alpha})}{\partial R} \right|^2 
		= E_{\nu}^2 \sum_{j=1}^J \sum_{k=1}^{p_j} \sum_{l=1}^{q_{j,k}} \left| \frac{\partial (w^{(\nu)}_{j,k,l}/R^{\alpha})}{\partial R} \right|^2 
\end{equation}
a.e.\ in $B_{\gamma\rho}(0) \cap \{r(X) > \tau\}$, where $w^{(\nu)}_{j,k} : \left.{\rm graph}\,\varphi^{(\infty)}_{j,k}\right|_{B_{3/4}(0) \cap \{r(X) > \tau\}} \rightarrow \mathcal{A}_{m_{j,k}}(\mathbb{R}^m)$ is as in Section~\ref{blow-up-procedure} and we represent $w^{(\nu)}_{j, k, l}(X) = w^{(\nu)}_{j,k}(X,\varphi^{(\infty)}_{j,k,l}(X)) = \sum_{h=1}^{m_{j,k}} \llbracket w^{(\nu)}_{j,k,l,h}(X) \rrbracket$ for $w^{(\nu)}_{j,k,l,h}$ are smooth ${\mathbb R}^{m}$-valued functions defined locally near a.e.~point in $B_{3/4}(0)$ (as in \eqref{w_localized}) and the right hand side of the above is computed with respect to the functions $w^{(\nu)}_{j, k, l, h}$.  By Lemma~\ref{gradient_convergence} in the appendix, 
\begin{equation*} 
	\lim_{\nu \rightarrow \infty} \sum_{j=1}^J \sum_{k=1}^{p_j} \sum_{l=1}^{q_{j,k}} \left| \frac{\partial (w^{(\nu)}_{j,k,l}/R^{\alpha})}{\partial R} \right|^2 
	= \sum_{j=1}^J \sum_{k=1}^{p_j} \sum_{l=1}^{q_{j,k}} \left| \frac{\partial (w_{j,k,l}/R^{\alpha})}{\partial R} \right|^2 
\end{equation*}
pointwise a.e.~in $B_{\gamma \rho}(0)$.  Thus we obtain \eqref{blowup_est1} by using \eqref{blowup_est1_eqn1_1} in \eqref{blowup_est1_eqn1}, dividing both sides of the resulting inequality by $E_{\nu}^2$ and first letting $\nu \rightarrow \infty$ and then letting $\tau \to 0^{+}$, using Fatou's lemma, Lemma~\ref{gradient_convergence} and Lemma~\ref{blowup_norm_conv_lemma}. 

The estimates \eqref{blowup_est2} and \eqref{blowup_est3} similarly follow from Corollary \ref{nonconest_cor} and Theorem~\ref{keyest_thm}(b) respectively, Fatou's lemma, Lemma~\ref{gradient_convergence} and Lemma~\ref{blowup_norm_conv_lemma}. 

To show \eqref{blowup_est4}, let $z \in B^{n-2}_{1/4}(0)$ and let $\tau >0$.  By \eqref{blowupclass_eqn1}, for each $\nu = 1,2,3,\ldots$ there exists $Z_{\nu} \in B_{\delta_{\nu}}(0,z) \cap B_{1/2}(0) \cap \Sigma_{u^{(\nu)},q}$ with $\mathcal{N}_{u^{(\nu)}}(Z_{\nu}) \geq \alpha$.  By Lemma~\ref{nonconest_lemma} with $u^{(\nu)}$ and $\varphi^{(\nu)}$ in place of $u$ and $\varphi$, 
\begin{equation} \label{blowup_est4_eqn1} 
	\int_{B_{1/2}(0) \cap \{r > \tau\}} \frac{\mathcal{G}(u^{(\nu)}(X), \varphi^{(\nu)}(X) - D_x \varphi^{(\nu)}(X) \cdot \xi_{\nu})^2}{
		|X-Z_{\nu}|^{n+2\alpha-\sigma}} dX \leq C E_{\nu}^2 
\end{equation}
for sufficiently large $\nu$, where $\xi_{\nu}$ is the orthogonal projection of $Z_{\nu}$ onto $\mathbb{R}^2 \times \{0\}$ and $C = (n,m,q,\alpha,\varphi^{(0)},\sigma) \in (0,\infty)$ is a constant. 
By Lemma~\ref{branchdist_lemma}(a), $|\xi_{\nu}| \leq C E_{\nu}$ for some constant $C = C(n,m,q,\alpha,\varphi^{(0)}) \in (0,\infty)$ and so after passing to a subsequence $\xi_{\nu}/E_{\nu}$ converges to some $\lambda(z) \in \mathbb{R}^2$.  Thus by dividing both sides of \eqref{blowup_est4_eqn1} by $E_{\nu}^2$ and letting $\nu \rightarrow \infty$ using Fatou's lemma and Lemma~\ref{blowup_norm_conv_lemma}, we obtain \eqref{blowup_est4} for some $\lambda(z) \in \mathbb{R}^2$. 

To see the asserted uniqueness of $\lambda(z)$, fix $z \in B^{n-2}_{1/4}(0)$, let $\lambda_{1}, \lambda_{2} \in {\mathbb R}^{2}$ and suppose that for $i = 1, 2$, 
$$\int_{B_{1/2}(0)} \sum_{j=1}^J \sum_{k=1}^{p_j} \sum_{l=1}^{q_{j,k}} \sum_{h=1}^{m_{j, k}}\frac{|w_{j,k,l, h}(X) - D_x \varphi^{(0)}_{j, l}(X) \cdot \lambda_{i}|^2}{
		|X-(0,z)|^{n+2\alpha-\sigma_{i}}} dX < \infty$$ for some $\sigma_{1}, \sigma_{2}$ with $0 < \sigma_{1} \leq \sigma_{2} <1/q.$
By the triangle inequality, this implies that 
\begin{equation*}
	\int_{B_{1/2}(0)} \sum_{j=1}^J \sum_{k=1}^{p_{j}}\sum_{l =1}^{q_{j, k}} \frac{|D_x \varphi^{(0)}_{j, l}(X) \cdot (\lambda_1 - \lambda_2)|^2}{|X-(0,z)|^{n+2\alpha-\sigma_{2}}} dX < \infty. 
\end{equation*}
In view of the fact that $\varphi^{(0)}$ is homogeneous of degree $\alpha$ and independent of the variable $y \in {\mathbb R}^{n-2}$  and that  $D_1 \varphi^{(0)}_j(e^{i\theta},y)$ and $D_2 \varphi^{(0)}_j(e^{i\theta},y)$ are $L^2$ orthogonal (as functions of $\theta$),  this implies that $\lambda_1 = \lambda_2$.  

Finally, to see that $\lambda \in C^{0,1-\sigma/2}(B_{1/4}(0);\mathbb{R}^2)$, first notice that $\lambda$ is bounded by construction.  Let $z_1,z_2 \in B_{1/4}(0)$ be two distinct points.  Using \eqref{blowup_est4} with $z = z_1,z_2$, the fact that $|X-(0,z_i)| \leq |z_1-z_2|$ for all $X \in B_{\tfrac{|z_1-z_2|}{2}}\big(0,\tfrac{z_1+z_2}{2}\big)$, and the triangle inequality, 
\begin{equation*}
	\int_{B_{\tfrac{|z_1-z_2|}{2}}\big(0,\tfrac{z_1+z_2}{2}\big)} \sum_{j=1}^J \sum_{k=1}^{p_{j}}\sum_{l = 1}^{q_{j, k}}|D_x \varphi^{(0)}_{j, l}(X) \cdot (\lambda(z_1) - \lambda(z_2))|^2 dX 
		\leq C |z_1-z_2|^{n+2\alpha-\sigma}  
\end{equation*}
which in view of the fact that $\varphi^{(0)}$ is homogeneous of degree $\alpha$ and is independent of the variable $y \in {\mathbb R}^{n-2}$ and that $D_1 \varphi^{(0)}_j(e^{i\theta},y)$, $D_2 \varphi^{(0)}_j(e^{i\theta},y)$ are $L^2$ orthogonal implies that
\begin{equation*}
	|\lambda(z_1) - \lambda(z_2)| \leq C |z_1-z_2|^{1-\sigma/2} 
\end{equation*}
for some constant $C = C(n,m,q,\alpha,\varphi^{(0)},\sigma) \in (0,\infty)$. 
\end{proof}

\begin{lemma} \label{blowup_est5_lemma}  
Let $w = (w_{j,k}) \in \mathfrak{B}$.  Let $\zeta(x,y) = \widetilde{\zeta}(|x|,y)$ where $\widetilde{\zeta} = \widetilde{\zeta}(r,y) \in C^2_c(B^{n-1}_{1/2}(0))$ with $D_r \widetilde{\zeta}(r,y) = 0$  in a neighborhood of the axis $\{r=0\}$.  Then using cylindrical coordinates $X = (re^{i\theta},y)$ on $\mathbb{R}^n$, 
\begin{equation} \label{blowup_est5}
	\sum_{j=1}^J \sum_{k=1}^{p_j} m_{j,k} \int_{B^{n-2}_{1/2}(0)} \int_0^{\infty} \int_0^{2\pi} \sum_{l=1}^{q_{j,k}}
		r^{2\alpha-1} D(r^{2-2\alpha} \, w_{j,k,l;a} \cdot D_{\iota} \varphi^{(0)}_{j,l}) \cdot DD_{y_h} \zeta \,d\theta \,dr \,dy = 0
\end{equation}
for all $\iota = 1,2$ and $h = 1,2,\ldots,n-2$, where $\varphi^{(0)}_{j}(X) = \sum_{l=1}^{q_0} \llbracket \varphi^{(0)}_{j,l}(X) \rrbracket$ (as in \eqref{varphi_localized}) if $\varphi^{(0)}_j$ is nonzero and $\varphi^{(0)}_{j,l}(X) = 0$ otherwise and $w_{j,k,l;a}(X) = \frac{1}{m_{j,k}} \sum_{h=1}^{m_{j,k}} w_{j,k,l,h}(X)$ for $w_{j,k,l,h}(X)$ as in \eqref{w_localized}. 
\end{lemma}

\begin{proof}
Let $w$ be the blow-up of $u^{(\nu)}$ relative to $\varphi^{(\nu)}$ and excess $E_{\nu} = \left( \int_{B_1(0)} \mathcal{G}(u^{(\nu)},\varphi^{(\nu)}) \right)^{1/2}$ as in Definition~\ref{blowupclass_defn}.  Let $\delta > 0$ be any positive number such that $D_r \widetilde{\zeta}(r,y) = 0$ for $r \leq \delta$.  Let $v^{(\nu)}_{j,k} : {\rm graph}\, \varphi^{(\nu)}_{j,k} |_{B_{1/2}(0)} \rightarrow \mathcal{A}_{m_{j,k}}(\mathbb{R}^m)$ be as in Corollary \ref{graphical_cor} with $\gamma = 1/2$, $\tau = \delta$, $u = u^{(\nu)}$, and $\varphi = \varphi^{(\nu)}$.  Let $v^{(\nu)}_{j,k,l;a}(X) = \frac{1}{m_{j,k}} \sum_{h=1}^{m_{j,k}} v^{(\nu)}_{j,k,l,h}(X)$ where $v_{j,k,l,h}(X)$ is as in \eqref{v_localized_1} with $\varphi^{(\nu)}_{j,k,l}$ and $v^{(\nu)}_{j,k}$ in place of $\varphi_{j,k,l}$ and $v_{j,k}$.  By Lemma~\ref{fourierest_lemma} with $u^{(\nu)}$ and $\varphi^{(\nu)}$ in place of $u$ and $\varphi$, 
\begin{align*}
	&\left| \sum_{j=1}^J \sum_{k=1}^{p_j} m_{j,k} \int_{B^{n-2}_{1/2}(0)} \int_{\delta}^{\infty} \int_0^{2\pi} \sum_{l=1}^{q_{j,k}} 
		r^{2\alpha-1} D(r^{2-2\alpha} \, v^{(\nu)}_{j,k,l;a} \cdot D_{\iota} \varphi^{(\nu)}_{j,k,l}) \cdot DD_{y_h} \zeta \,d\theta \,dr \,dy \right| 
		\\&\hspace{10mm} \leq C (\delta^{-2} E_{\nu}^2 + \delta^{2\alpha} E_{\nu}) 
\end{align*}
for all $\iota = 1,2$ and $h = 1,2,\ldots,n-2$, where 
$C = C(n,m,q,\alpha,\varphi^{(0)},\zeta) \in (0,\infty)$ is a constant.  By the compactness of (single-valued) harmonic functions, $v^{(\nu)}_{j,k,l;a}/E_{\nu} \rightarrow w_{j,k,l;a}$ in $C^1(B_{1/2}(0) \cap \{r \geq \delta\};\mathbb{R}^m)$ (using the notation and conventions from Subsection \ref{sec:blowupclass_subsec}).  Thus by dividing by $E_{\nu}$ and letting $\nu \rightarrow \infty$, 
\begin{align} \label{blowup_est5_eqn1}
	&\left| \sum_{j=1}^J \sum_{k=1}^{p_j} m_{j,k} \int_{B^{n-2}_{1/2}(0)} \int_{\delta}^{\infty} \int_0^{2\pi} \sum_{l=1}^{q_{j,k}} 
		r^{2\alpha-1} D(r^{2-2\alpha} \, w_{j,k,l; a} \cdot D_{\iota} \varphi^{(0)}_{j,l}) \cdot DD_{y_h} \zeta \,d\theta \,dr \,dy \right| 
		\\&\hspace{10mm} \leq C \delta^{2\alpha} \nonumber 
\end{align}
for all $\iota = 1,2$ and $h = 1,2,\ldots,n-2$.  By using the dominated convergence theorem together with the fact that $D_x \zeta = 0$ in a neighborhood of $\{r=0\}$, $D_{y}w_{j, k, l; a} \in L^{2}(B_{1/2}(0); {\mathbb R})$ (by \ref{blowup_est3}) and $D\varphi^{(0)} \in L^2(B_{1/2}(0);\mathcal{A}_q(\mathbb{R}^m)),$ we can let $\delta \downarrow 0$ in \eqref{blowup_est5_eqn1} to obtain \eqref{blowup_est5}. 
\end{proof}

\begin{lemma} \label{blowup_est6_lemma}  
For each $\delta \in (0,1/4)$, $w = (w_{j,k}) \in \mathfrak{B}$, $z \in B^{n-2}_{1/16}(0)$, and $\zeta_{j,k} \in C^1_c(B_{1/16}(0))$ with $|\zeta_{j,k}| \leq 1/16$ and $|D\zeta_{j,k}| \leq 1$, 
\begin{align} \label{blowup_est6}
	\int_{B_{1/16}(0) \cap \{r > \delta\}} \sum_{j=1}^J \sum_{k=1}^{p_j} \sum_{l=1}^{q_{j,k}} |Dw_{j,k,l}|^2 
	\leq {}& \int_{B_{1/16}(0) \cap \{r > \delta\}} \sum_{j=1}^J \sum_{k=1}^{p_j} \sum_{l=1}^{q_{j,k}} |D\widetilde{w}_{j,k,l}|^2 \\&
		+ C \delta^{-2} \int_{B_{1/8}(0) \cap \{\delta/8 < r < 2\delta\}} \sum_{j=1}^J \sum_{k=1}^{p_j} \sum_{l=1}^{q_{j,k}} |w_{j,k,l}|^2, \nonumber 
\end{align}
where $w_{j,k,l}$ is as in \eqref{w_localized}, $\widetilde{w}_{j,k,l}(X) = w_{j,k,l}(X+\zeta_{j,k}(X) (X-(0,z)))$, and $C = C(n,m,q,\alpha,\varphi^{(0)}) \in (0,\infty)$ (independent of $\delta$). 
 \end{lemma}
 
\begin{proof}
Let $w$ be the blow-up of $u^{(\nu)}$ relative to $\varphi^{(\nu)}$ and excess $E_{\nu} = \left( \int_{B_1(0)} \mathcal{G}(u^{(\nu)},\varphi^{(\nu)}) \right)^{1/2}$ as in Definition~\ref{blowupclass_defn}.  Let $v^{(\nu)}_{j,k} : {\rm graph}\, \varphi^{(\nu)}_{j,k} |_{B_{1/2}(0)} \rightarrow \mathcal{A}_{m_{j,k}}(\mathbb{R}^m)$ be as in Corollary \ref{graphical_cor} with $\gamma = 1/8$, $\tau = \delta/8$, $u = u^{(\nu)}$, and $\varphi = \varphi^{(\nu)}$.  Let $v_{j,k,l}(X)$ is as in \eqref{v_localized} with $\varphi^{(\nu)}$ and $v^{(\nu)}_{j,k}$ in place of $\varphi$ and $v_{j,k}$.  By Lemma~\ref{competitor_lemma} with $u^{(\nu)}$ and $\varphi^{(\nu)}$ in place of $u$ and $\varphi$, 
\begin{align} \label{blowup_est6_eqn1}
	\int_{B_{1/16}(0) \cap \{r > \delta\}} \sum_{j=1}^J \sum_{k=1}^{p_j} \sum_{l=1}^{q_{j,k}} |Dv^{(\nu)}_{j,k,l}|^2 
	\leq {}& \int_{B_{1/16}(0) \cap \{r > \delta\}} \sum_{j=1}^J \sum_{k=1}^{p_j} \sum_{l=1}^{q_{j,k}} |D\widetilde{v}^{(\nu)}_{j,k,l}|^2 \\&
		+ C \delta^{-2} \int_{B_{1/8}(0) \cap \{\delta/8 < r < 2\delta\}} \sum_{j=1}^J \sum_{k=1}^{p_j} \sum_{l=1}^{q_{j,k}} |v^{(\nu)}_{j,k,l}|^2, \nonumber 
\end{align}
where $\widetilde{v}^{(\nu)}_{j,k,l}(X) = v^{(\nu)}_{j,k,l}(X+\zeta_{j,k}(X) (X-(0,z)))$ and $C = C(n,m,q,\alpha,\varphi^{(0)}) \in (0,\infty)$.  By the continuity of Dirichlet energy under uniform limits of Dirichlet energy minimizers (see Lemma~\ref{compactness_lemma} of the appendix), 
\begin{equation} \label{blowup_est6_eqn2}
	\lim_{\nu \rightarrow \infty} \frac{1}{E_{\nu}^2} \sum_{j=1}^J \sum_{k=1}^{p_j} \sum_{l=1}^{q_{j,k}} |Dv^{(\nu)}_{j,k,l}|^2 
		= \sum_{j=1}^J \sum_{k=1}^{p_j} \sum_{l=1}^{q_{j,k}} |Dw_{j,k,l}|^2
\end{equation}
strongly in $L^1(B_{3/32}(0) \cap \{r > \delta/2\})$.  Moreover, by Lemma~\ref{gradient_convergence} of the appendix and the chain rule, 
\begin{equation} \label{blowup_est6_eqn3}
	\lim_{\nu \rightarrow \infty} \frac{1}{E_{\nu}^2} \sum_{j=1}^J \sum_{k=1}^{p_j} \sum_{l=1}^{q_{j,k}} |D\widetilde{v}^{(\nu)}_{j,k,l}|^2 
		= \sum_{j=1}^J \sum_{k=1}^{p_j} \sum_{l=1}^{q_{j,k}} |D\widetilde{w}_{j,k,l}|^2 
\end{equation}
pointwise a.e.~in $B_{1/16}(0) \cap \{r > \delta\}$.  Hence, by dividing both sides of \eqref{blowup_est6_eqn1} by $E_{\nu}^2$ and let $\nu \rightarrow \infty$ using \eqref{blowup_est6_eqn2}, \eqref{blowup_est6_eqn3}, and the dominated convergence theorem, we obtain \eqref{blowup_est6}. 
\end{proof}

\section{Classification of homogeneous blow-ups}\label{sec:homogblowup_sec} 

The main result of this section is the following:

\begin{theorem} \label{homogrep_lemma} 
Let $w  = (w_{j, k}) \in \mathfrak{B}$ be homogeneous of degree $\alpha$ in the sense that $$X \mapsto \sum_{l=1}^{q_{j,k}} \sum_{h=1}^{m_{j,k}} \llbracket w_{j,k,l,h}(X) \rrbracket$$ is an ${\mathcal A}_{m_{j, k}q_{j, k}}({\mathbb R}^{m})$-valued homogeneous degree $\alpha$ function of $X \in {\mathbb R}^{n} \setminus \{0\} \times {\mathbb R}^{n-2}$ for every $j$ and $k$, where $w_{j,k,l,h}$ is as in \eqref{w_localized}.  
Then $w \in \mathfrak{L}$.
\end{theorem}

An analogous result was proven in~\cite{Sim93} using Fourier analysis and PDE techniques, a method we adapted in \cite{KrumWic1} to prove Theorem~\ref{homogrep_lemma} in the special case $q=2$, taking advantage of the fact that  in that case the blow-ups $w \in {\mathfrak B}$ are, away from the axis of $\varphi^{(0)}$, single valued functions over ${\rm graph} \, \varphi^{(0)}$.  Since this latter fact is no longer true in the present setting (of general $q$), the proof requires a new approach and considerably more effort. 

 Recall that for each $z \in B^{n-2}_{1/4}(0)$, $w$ satisfies \eqref{blowup_est4} for a unique $\lambda(z) \in \mathbb{R}^2$. The proof of Theorem~\ref{homogrep_lemma} involves the following main steps:
\begin{enumerate}
        \item[1.] Prove Theorem~\ref{homogrep_lemma} in the special case that \eqref{blowup_est4} holds true with $\lambda(\cdot) = 0$. 
	\item[2.] Show in the general case that  $\lambda(z)$ is a linear function of $z \in \mathbb{R}^{n-2}$, and thus by composing each $u^{(\nu)}$ with a suitable rotation where $(u^{(\nu)})$ is a sequence of locally energy minimizing functions corresponding to $w$ per Definition~\ref{blowupclass_defn}, the proof reduces to the special case $\lambda(\cdot) = 0$. 
	\end{enumerate}

For Step 1, we use the energy comparison estimate of Lemma~\ref{blowup_est6_lemma} to establish monotonicity of the frequency function associated with $w$ at base points on $\{0\} \times \mathbb{R}^{n-2},$ and deduce from \eqref{blowup_est4} 
(with $\lambda(z) = 0$) that $w$ has frequency $\geq \alpha$ at each point of $\{0\} \times \mathbb{R}^{n-2}$ and hence that $w$ is independent of $y \in \mathbb{R}^{n-2}$.  From this the conclusion  that $w \in \mathfrak{L}$ will readily follow.  For Step 2, we first use Lemma~\ref{blowup_est5_lemma} to establish an $L^{2}$ mean value property for the $y$-derivatives of certain Fourier coefficients of the average of $w$, and then use this mean value property and a PDE argument. 
\subsection{Case $\lambda = 0:$ frequency monotonicity}


Suppose $w \in \mathfrak{B}$  and that \eqref{blowup_est4} holds true with $\lambda(z) = 0$, that is 
\begin{equation} \label{homogrep2_eqn1}
	\int_{B_{1/2}(0)} \sum_{j=1}^J \sum_{k=1}^{p_j} \sum_{l=1}^{q_{j,k}} 
		\frac{|w_{j,k,l}(X)|^2}{|X-(0,z)|^{n+2\alpha-\sigma}} dX \leq C 
\end{equation}
for all $z \in B^{n-2}_{1/4}(0)$ and $\sigma \in (0,1/q)$ and some constant $C = C(n,m,q,\alpha,\varphi^{(0)},\sigma) \in (0,\infty)$.  By \eqref{homogrep2_eqn1} and \eqref{regestimate}, for every $\sigma \in (0,1/q)$,  
\begin{equation} \label{homogrep2_eqn2}
	\sum_{j=1}^J \sum_{k=1}^{p_j} \sum_{l=1}^{q_{j,k}} |w_{j,k,l}(X)| + \left( |x|^{2-n} \int_{B_{|x|/2}(X)} \sum_{j=1}^J \sum_{k=1}^{p_j} \sum_{l=1}^{q_{j,k}} 
		|Dw_{j,k,l}|^2 \right)^{1/2} \leq C |x|^{\alpha-\sigma/2}
\end{equation}
for all $X = (x,y) \in B_{1/8}(0) \setminus \{0\} \times {\mathbb R}^{n-2}$ and $\sigma \in (0,1/q)$ where $C = C(n,m,q,\alpha,\varphi^{(0)},\sigma) \in (0,\infty)$. It follows from this and a covering argument that 
\begin{equation}\label{homogrep2_eqn2_1}
\int_{B_{1/16}(0)} \sum_{j=1}^{J}\sum_{k=1}^{p_{j}}\sum_{l=1}^{q_{j, k}}|Dw_{j, k, l}|^{2} < C, 
\end{equation}
where $C = C(n, m, q, \alpha, \varphi^{(0)}).$ 
\begin{lemma}\label{blowup-freq}
Suppose $w = (w_{j,k}) \in \mathfrak{B}$ is such that \eqref{homogrep2_eqn1} holds true.  Let $\overline{w} : B_{1/2}(0) \rightarrow \mathcal{A}_{q}(\mathbb{R}^m)$ be defined by 
\begin{equation*}
	\overline{w}(X) = \sum_{j=1}^{J} \sum_{k=1}^{p_j} \sum_{l=1}^{q_{j,k}} \sum_{h=1}^{m_{j,k}} \llbracket w_{j,k,l,h}(X) \rrbracket
\end{equation*}
for all $X \in B_{1/2}(0)$, where $w_{j,k,l,h}$ are as in \eqref{w_localized}.  Then either $w_{j,k} \equiv 0$ in $B_{1/2}$ for all $j$ and $k,$ or for each $z \in B^{n-2}_{1/16}(0)$, 
\begin{equation} \label{homogrep2_freq}
	N_{\overline{w}, \, (0,z)}(\rho) = \frac{\rho^{2-n} \int_{B_{\rho}(0,z)} |D\overline{w}|^2}{\rho^{1-n} \int_{\partial B_{\rho}(0,z)} |\overline{w}|^2}
\end{equation}
is monotone nondecreasing as a function of $\rho \in (0,1/16)$.  Moreover, in the latter case $N_{\overline{w}, \, (0,z)}(\rho) = \alpha$ for all $\rho \in (0,1/16)$ and some $\alpha \geq 0$ if and only if $\overline{w}((0,z) + X)$ is homogeneous of degree $\alpha$ as a function of $X = (x,y) \in B_{1/16}(0)$. 
\end{lemma}
\begin{proof} 
Recall from Section \ref{sec:frequency_sec} that it suffices to show that $\overline{w}$ satisfies \eqref{freqidentity1} and \eqref{freqidentity2} for an appropriate class of test functions, i.e. 
\begin{gather}
	\int_{\mathbb{R}^n} |D\overline{w}|^2 \zeta = -\int_{\mathbb{R}^n} \overline{w}^{\kappa}_l D_i \overline{w}^{\kappa}_l D_i \zeta, 
		\label{homogrep2_freqid1} \\
	\int_{\mathbb{R}^n} \left( \tfrac{1}{2} |D\overline{w}^{\kappa}_l|^2 e_i - D_i \overline{w}^{\kappa}_l D\overline{w}^{\kappa}_l \right) 
		\cdot D_i (\zeta(X) \, (X-(0,z))) = 0, \label{homogrep2_freqid2}
\end{gather}
for all $\zeta \in C^1_c(B_{1/16}(0))$ and $z \in \mathbb{R}^{n-2}$, where $e_1,e_2,\ldots,e_n$ denotes the standard basis for $\mathbb{R}^n$ and we use the convention of summing over repeated induces. 

Since $w$ is component-wise minimizing, \eqref{homogrep2_freqid1} holds true whenever $\zeta \in C^1_c(B_{1/16}(0) \setminus \{0\} \times \mathbb{R}^{n-2})$.  For each $\delta \in (0,1/32)$, let $\widetilde{\chi}_{\delta} : [0,\infty) \rightarrow \mathbb{R}$ be a smooth function such that $0 \leq \widetilde{\chi}_{\delta}(r) \leq 1$, $\widetilde{\chi}_{\delta}(r) = 0$ for all $r \in [0,\delta/2]$, $\widetilde{\chi}_{\delta}(r) = 1$ for all $r \geq \delta$, and $|\widetilde{\chi}'_{\delta}(r)| \leq 3/\delta$. Define $\chi_{\delta} : \mathbb{R}^n \rightarrow \mathbb{R}$ by $\chi_{\delta}(x,y) = \widetilde{\chi}_{\delta}(|x|)$.  Let $\zeta \in C^1_c(B_{1/16}(0);\mathbb{R})$ be arbitrary and replace $\zeta$ in \eqref{homogrep2_freqid1} by $\chi_{\delta} \zeta$ to get, using \eqref{homogrep2_eqn2} and \eqref{homogrep2_eqn2_1}, that
\begin{align} \label{homogrep2_eqn3}
	\left| \int_{B_{1/16}(0)} (|D_i \overline{w}^{\kappa}_l|^2 \zeta + \overline{w}^{\kappa}_l D_i \overline{w}^{\kappa}_l D_i \zeta) \chi_{\delta} \right|
	&\leq \int_{B_{1/16}(0)} |\overline{w}| |D\overline{w}| |D\chi_{\delta}| |\zeta| 
	\\&\leq C \delta^{2\alpha-\sigma} \sup_{B_{1/16}(0)} |\zeta| \nonumber 
\end{align}
for every $\sigma \in (0,1/2q)$ and some constant $C = C(n,m,q,\alpha,\varphi^{(0)},\sigma) \in (0,\infty)$.  
In view of \eqref{homogrep2_eqn2_1} we can let $\delta \downarrow 0$ in \eqref{homogrep2_eqn3} to obtain \eqref{homogrep2_freqid1}.  Note that this argument does not work to prove \eqref{homogrep2_freqid2}, so we instead argue as follows. 

Let $\zeta \in C^1_c(B_{1/16}(0);\mathbb{R})$, $z \in B^{n-2}_{1/16}(0)$, and $\tau > 0$ be arbitrary.  Recall that $w$ satisfies \eqref{blowup_est6} and thus (by taking in \eqref{blowup_est6} $\zeta_{j, k} = \zeta$ for every $j$ and $k$) we have that for every $t \in (-\epsilon,\epsilon)$, where $\epsilon > 0$ depends on $\|\zeta\|_{C^1(B_{1/16}(0))}$,  
\begin{equation*}
	\int_{B_{1/16}(0) \cap \{r > \delta\}} |D\overline{w}|^2 
	\leq \int_{B_{1/16}(0) \cap \{r > \delta\}} |D(\overline{w}(X+t \zeta(X) (X-(0,z))))|^2 
		+ C \delta^{-2} \int_{B_{1/8}(0) \cap \{\delta/8 < r < 2\delta\}} |\overline{w}|^2. 
\end{equation*}
By \eqref{homogrep2_eqn2}, for every $t \in (-\epsilon,\epsilon)$, 
\begin{equation*}
	\int_{B_{1/16}(0) \cap \{r > \delta\}} |D\overline{w}|^2 
	\leq \int_{B_{1/16}(0) \cap \{r > \delta\}} |D(\overline{w}(X+t \zeta(X) (X-(0,z))))|^2 + C \delta^{2\alpha-\sigma}  
\end{equation*}
for some constant $C = C(n,m,q,\alpha,\varphi^{(0)}, \sigma) \in (0,\infty)$.  Recall that by \eqref{homogrep2_eqn2_1}, $\int_{B_{1/16}(0)} |D\overline{w}|^{2} < \infty$.  Hence we can let $\delta \downarrow 0$ and use the monotone convergence theorem to conclude that for every $t \in (-\epsilon,\epsilon)$, 
\begin{equation}\label{homogrep2_radial}
	\int_{B_{1/16}(0)} |D\overline{w}|^2 
	\leq \int_{B_{1/16}(0)} |D(\overline{w}(X+t \zeta(X) (X-(0,z))))|^2 
\end{equation}
whence $\left.\frac{d}{dt}\right|_{t=0}  \int_{B_{1/16}(0)} |D(\overline{w}(X+t \zeta(X) (X-(0,z))))|^2 = 0.$ By direct calculation, this gives \eqref{homogrep2_freqid2}.  
\end{proof}

By Lemma~\ref{blowup-freq}, the standard consequences of the monotonicity of frequency as in Section \ref{sec:frequency_sec} hold true with $\overline{w}$ in place of $u$ (and $(0,z)$ in place of $Y$).  Using this, we prove Theorem~\ref{homogrep_lemma} in the special case when $\lambda = 0$ as follows:  

\begin{proof}[Proof of Theorem~\ref{homogrep_lemma} when $\lambda = 0$]
Let $w \in \mathfrak{B}$ be homogeneous of degree $\alpha$ and suppose that \eqref{homogrep2_eqn1} holds.  The conclusion holds trivially if $n=2$, so suppose that $n \geq 3$. 

For each $z \in \mathbb{R}^{n-2}$, define $N_{\overline{w},(0,z)}$ by \eqref{homogrep2_freq} and let $\mathcal{N}_{\overline{w}}(0,z) = \lim_{\rho \downarrow 0} N_{\overline{w},(0,z)}(\rho)$.  By  \eqref{homogrep2_eqn2} and \eqref{consequence_eqn1} (with $\overline{w}$, $(0,z)$ in place of $u$, $Y$), $\mathcal{N}_{\overline{w}}(0,z) \geq \alpha$ for all $z \in \mathbb{R}^{n-2}$.  Since $\overline{w}$ is homogeneous of degree $\alpha$, it follows that $\mathcal{N}_{\overline{w}}(0,z) = \alpha$ for all $z \in \mathbb{R}^{n-2}$ and moreover that $\overline{w}(x,y+z) = \overline{w}(x,y)$ for all $(x,y) \in \mathbb{R}^n$ and $z \in \mathbb{R}^{n-2}$.  It readily follows that for each $j$ and $k$ 
\begin{equation} \label{homogrep2_eqn5}
	\sum_{l=1}^{q_{j,k}} \sum_{h=1}^{m_{j,k}} \llbracket w_{j,k,l,h}(re^{i\theta},y) \rrbracket = (m_{j,k} q_{j,k} - q_0 N_{j,k}) \llbracket 0 \rrbracket 
		+ \sum_{l=1}^{N_{j,k}} \sum_{h=0}^{q_0-1} \llbracket \op{Re}(a_{j,k,l} r^{\alpha} e^{i\alpha (\theta + 2\pi h)}) \rrbracket 
\end{equation}
on $B_{1/2}(0)$ for some integer $N_{j,k}$ with $1 \leq N_{j,k} \leq m_{j,k} q_{j,k}/q_0$ and some $a_{j,k,l} \in \mathbb{C}^m \setminus \{0\}$, where $w_{j,k,l,h}$ are as in \eqref{w_localized}.  
Let $\varphi^{(\infty)} = (m_{j, k}, \varphi^{(\infty)}) \in {\mathfrak D}$ be such that $w \in {\mathfrak B}(\varphi^{(\infty)}).$ When $\varphi^{(\infty)}_{j,k}$ is identically zero, \eqref{homogrep2_eqn5} immediately implies that $w_{j,k}$ takes the form of $\psi_{j,k}$ in \eqref{homogclass_eqn2}.  When $\varphi^{(\infty)}_{j,k}$ is not identically zero, $w_{j,k}$ is component-wise minimizing and thus $w_{j,k,l}$ are regular on each ball $B \subset \subset B_{1/2}(0) \setminus \{0\} \times \mathbb{R}^{n-2}$.  Hence we can use \eqref{homogrep2_eqn5} to represent $w_{j,k}$ in the form of $\psi_{j,k}$ in \eqref{homogclass_eqn1} and \eqref{homogclass_eqn3} with $b = 0$, so $w \in \mathfrak{L}$. 
\end{proof}

Before proceeding to prove Theorem~\ref{homogrep_lemma} in its full generality, let us point out the following interesting fact which illustrates the subtlety of the  variational property \eqref{homogrep2_freqid2} of $w$:

\begin{remark}\label{example} {\rm Although \eqref{homogrep2_radial} above  says that $\overline{w}$ is energy stationary (in fact minimizing) for domain deformations that are radial from any given point on $\{0\} \times {\mathbb R}^{n-2}$, $\overline{w}$ need \emph{not} be stationary for more general domain deformations of the type 
$\zeta_{t} \, : \, X \mapsto X + t(\zeta^{1}(X), \ldots, \zeta^{n}(X)),$ $X \in B_{1/2}(0)$, $t \in (-\e, \e),$ where $\zeta^{j} \in C^{1}_{c}(B_{1/2}(0))$ are arbitrary.  To see this, consider for instance 
$\overline{w}(x,y) = \op{Re}(c (x_1+ix_2)^{1/2})$ where $c = a+ib$ for $a,b \in \mathbb{R}^m$.  Let $c = (c_1,\ldots,c_m) \in \mathbb{C}^m$ with $c_k \in \mathbb{C}$, $k = 1,\ldots,m$, and let 
\begin{equation*}
	\gamma = c \cdot c = \sum_{k=1}^m c_k^2 = (|a|^2 - |b|^2) + 2i a \cdot b 
\end{equation*}
  Writing $\overline{w}_{t}(X) = \overline{w}(\zeta_{t}(X))$ for $\zeta_{t}$ as above, it is easy to see that $\left.\frac{d}{dt} \right|_{t=0} \int_{B_{1/2}(0)} |D\overline{w}_{t}|^{2} = 0$ if and only if 
\begin{equation} \label{eqn1} 
	\int_{B_1(0)} \left( |D\overline{w}|^2 D_i \zeta^i - 2 D_i \overline{w} \cdot D_j \overline{w} D_i \zeta^j \right) = 0.
\end{equation}
On the other hand, by direct calculation, \eqref{eqn1} holds if and only if  $\g = 0$, i.e.\ if and only if $|a| = |b|$ and $a \cdot b = 0$.  Indeed, since 
\begin{equation*}
	D_1 \overline{w} = \frac{1}{2} \op{Re}(c (x_1+ix_2)^{-1/2}) \hspace{5mm} 	
	D_2 \overline{w}= \frac{1}{2} \op{Re}(ic (x_1+ix_2)^{-1/2}) = \frac{-1}{2} \op{Im}(c (x_1+ix_2)^{-1/2})
\end{equation*}
we have 
\begin{align*}
	|D_1 \overline{w}|^2 - |D_2 \overline{w}|^2 &= \sum_{k=1}^m \frac{1}{4} \op{Re}((c_k r^{-1/2} e^{-i\theta/2})^2) 
		= \frac{1}{4} r^{-1} \op{Re}(\gamma e^{-i\theta}), \\
	2 D_1 \overline{w} \cdot D_2 \overline{w} &= \sum_{k=1}^m \frac{-1}{4} \op{Im}((c_k r^{-1/2} e^{-i\theta/2})^2) 
		= \frac{-1}{4} r^{-1} \op{Im}(\gamma e^{-i\theta}). 
\end{align*}
Hence 
\begin{align*} 
	&\int_{B_1(0)} \left( |D\overline{w}|^2 D_i \zeta^i - 2 D_i \overline{w} \cdot D_j \overline{w}D_i \zeta^j \right) 
	\\&= \lim_{\varepsilon \downarrow 0} \int_{\partial B_{\varepsilon}(0)} 
		\left( |D\overline{w}|^2 x_j - 2 D_i \overline{w} \cdot D_j \overline{w} \,x_i \right) \frac{\zeta^j}{r}
	\\&= \lim_{\varepsilon \downarrow 0} -\int_{\partial B_{\varepsilon}(0)} 
		\left( (|D_1 \overline{w}|^2 - |D_2 \overline{w}|^2) (x_1 \zeta^1 - x_2 \zeta^2) + 2 D_1 \overline{w} \cdot D_2 \overline{w} (x_1 \zeta^2 + x_2 \zeta^1) \right) \frac{1}{r}
	\\&=  -\int_0^{2\pi} 
		\left( \frac{1}{4} \op{Re}(\gamma e^{-i\theta}) (\cos(\theta) \zeta^1(0) - \sin(\theta) \zeta^2(0)) 
			- \frac{1}{4} \op{Im}(\gamma e^{-i\theta}) (\cos(\theta) \zeta^2 (0)+ \sin(\theta) \zeta^1(0)) \right) d\theta
	\\&=  -\int_0^{2\pi} \frac{1}{4} (\op{Re}(\gamma) \zeta^1(0) - \op{Im}(\gamma) \zeta^2(0)) 
	\\&= -\frac{\pi}{2} (\op{Re}(\gamma) \zeta^1(0) - \op{Im}(\gamma) \zeta^2(0))
\end{align*}
where $r = |x|$, so  \eqref{eqn1} holds for all $\zeta^1,\zeta^2 \in C_c^1(B_{1/2}(0))$ if and only if $\gamma = 0$. 

Now let $\varphi^{(0)} : \mathbb{R}^2 \rightarrow \mathcal{A}_2(\mathbb{R}^3)$ be given by 
\begin{equation*}
	\varphi^{(0)}(x) = \op{Re}(a(e_1+ie_2) (x_1+ix_2)^{1/2}) 
\end{equation*}
for all $x = (x_1,x_2) \in \mathbb{R}^2$, where $a > 0$ and $e_1,e_2,e_3$ is the standard basis for $\mathbb{R}^3$.  For each $t \in \mathbb{R}$, let $u_t : B_1(0) \rightarrow \mathcal{A}_2(\mathbb{R}^3)$ be given by
\begin{equation*}
	u_t(x) = \op{Re}(a (\cos(t) e_1 + i e_2 + \sin(t) e_3) (x_1+ix_2)^{1/2}) 
\end{equation*}
for all $x = (x_1,x_2) \in B_1(0)$.  Notice that $u_{t}  \to \varphi^{(0)}$ as $t \to 0$ and, since $\cos(t) e_1 + \sin(t) e_3$ and $e_2$ are an orthonormal pair of vectors, each of $u_{t}$ and $\varphi^{(0)}$ is energy minimizing.  The blow-up of $u_t$ relative to $\varphi^{(0)}$ as $t \downarrow 0$ is (up to a constant multiple) 
\begin{equation*}
	\overline{w} = \op{Re}(a e_3 (x_1+ix_2)^{1/2}) 
\end{equation*}
which by the discussion above does not satisfy \eqref{eqn1}.
} \end{remark}

\subsection{Reduction to the case $\lambda = 0$} \label{sec:homogblowup_subsec1}  Let $w = (w_{j, k}) \in {\mathfrak B}$ and suppose that $w$ is homogeneous of degree $\alpha.$ Our first goal is to show that the function $\lambda$ corresponding to $w$ as in \eqref{blowup_est4} is linear in $z$. To achieve this, we proceed as follows: 

For each $j = 1,2,\ldots,J$, $k = 1,2,\ldots,p_j$, and $\iota = 1,2$, define $\widetilde{w}_{j,k}^{\iota} : B^{n-1}_{1/2}(0) \cap \{r > 0\} \rightarrow \mathbb{R}$ by 
\begin{equation*}
	\widetilde{w}^{\iota}_{j,k}(r,y) = \frac{1}{\pi q_0} \int_0^{2\pi} \sum_{l=1}^{q_0} 
		r^{1-\alpha} \, w_{j,k,l;a}(re^{i\theta},y) \cdot D_{\iota} \varphi^{(0)}_{j,l}(e^{i\theta}) d\theta 
\end{equation*}
if $\varphi^{(0)}_j$ is not identically zero and $\widetilde{w}^{\iota}_{j,k}(r,y) = 0$ otherwise, where $\varphi^{(0)}_{j,l}(X) = \sum_{l=1}^{q_0} \llbracket \varphi^{(0)}_{j,l}(X) \rrbracket$ (as in \eqref{varphi_localized}) if $\varphi^{(0)}_j$ is nonzero and $w_{j,k,l;a}(X) = \frac{1}{m_{j,k}} \sum_{h=1}^{m_{j,k}} w_{j,k,l,h}(X)$ with $w_{j,k,l,h}$ as in \eqref{w_localized}.  Since $w$ is homogeneous of degree $\alpha$, $\widetilde{w}^{\iota}_{j,k}$ is homogeneous of degree one and thus we can extend $\widetilde{w}^{\iota}_{j,k}$ to a homogeneous degree one function on $\mathbb{R}^{n-1} \cap \{r > 0\}$.  We also extend $\widetilde{w}^{\iota}_{j,k}$ to a function on $\mathbb{R}^{n-1} \setminus \{r=0\}$ by $\widetilde{w}^{\iota}_{j,k}(r,y) = \widetilde{w}^{\iota}_{j,k}(-r,y)$ for all $r < 0$, $y \in \mathbb{R}^{n-2}$.  Since each $w_{j,k,l;a}$ is a locally defined harmonic function, each $\widetilde{w}^{\iota}_{j,k}$ is smooth on $\mathbb{R}^{n-1} \setminus \{r= 0\}$.  By \eqref{blowup_est4} with $\sigma = 1/(2q)$, 
\begin{equation*}
	\int_{B_{\rho}(0,z)} \sum_{j=1}^J \sum_{k=1}^{p_j} \sum_{l=1}^{q_{j,k}} |w_{j,k,l}|^2 \leq C |\lambda(z)|^2 \rho^{n+2\alpha-2} + C \rho^{n+2\alpha-\sigma} \leq C\rho^{n+2\alpha - 2}
\end{equation*}
for all $z \in B^{n-2}_{1/4}(0)$, $\rho \in (0,1/4]$ and some constant $C = C(n,m,q,\alpha,\varphi^{(0)}) \in (0,\infty)$.  Thus by the Schauder estimates, 
\begin{equation} \label{homogrep1_eqn1} 
	\sum_{j=1}^J \sum_{k=1}^{p_j} \sum_{l=1}^{q_{j,k}} |w_{j,k,l}(X)| \leq C |x|^{\alpha - 1}
\end{equation}
for all $X = (x,y) \in B_{1/8}(0) \setminus \{0\} \times {\mathbb R}^{n-2}$ where $C = C(n,m,q,\alpha,\varphi^{(0)}) \in (0,\infty)$.  By \eqref{homogrep1_eqn1}, $\widetilde{w}^{\iota}_{j,k} \in L^{\infty}(B^{n-1}_{1/8}(0))$ for each $j$, $k$, and $\iota$.  Also, by \eqref{blowup_est3}, 
\begin{equation} \label{homogrep1_eqn2} 
	\int_{B^{n-1}_{1/2}(0)} |r|^{2\alpha-1} |D_y \widetilde{w}^{\iota}_{j,k}|^2 \,dr\,dy < \infty. 
\end{equation}

For each $\iota = 1,2$, define $W^{\iota} : \mathbb{R}^{n-1} \rightarrow \mathbb{R}$ by 
\begin{equation*}
	W^{\iota}(r,y) = \sum_{j=1}^J \sum_{k=1}^{p_j} m_{j,k} \widetilde{w}^{\iota}_{j,k}(r,y) . 
\end{equation*}
Clearly, $W^{\iota} \in L^{\infty}(B^{n-1}_{1/8}(0)) \cap C^{\infty}(B^{n-1}_{1/8}(0) \setminus \{r=0\})$ and $W^{\iota}$ is homogeneous of degree one. By \eqref{blowup_est5}, for each function 
$\zeta \in C^2_c(\{(r, y) \, : \, r \geq 0, \, y \in {\mathbb R}^{n-2}, \, r^{2} + |y|^{2} < 1/4\})$ with $D_r \zeta(r,y) = 0$ in a neighborhood of $\{r = 0\}$, 
\begin{equation*}
	\int_{\{(r, y) \, : \, r \geq 0, \, y \in {\mathbb R}^{n-2},\,  r^{2} + |y|^{2} < 1/4\}} r^{2\alpha-1} \overline{D}W^{\iota} \cdot \overline{D}D_{y_h} \zeta \,dr \,dy = 0 
\end{equation*}
where $\overline{D}$ is the gradient operator on ${\mathbb R}^{n-1} = \{(r, y) \, : \, r \in {\mathbb R}, y \in {\mathbb R}^{n-1}\}.$ Since $\zeta$ is arbitrary, for each function  
$\zeta \in C^2_c(\{(r, y) \, : \, r \geq 0, y \in {\mathbb R}^{n-2}, r^{2} + |y|^{2} < 1/4\})$ with $D_r \zeta(r,y) = 0$ in a neighborhood of $\{r = 0\}$, 
\begin{equation*}
	\int_{\{(r, y) \, : \, r \geq 0, \, y \in {\mathbb R}^{n-2}, \,  r^{2} + |y|^{2} < 1/4\})} r^{2\alpha-1} \overline{D}W^{\iota} \cdot \overline{D}\delta_{h,\tau} \zeta \,dr \,dy = 0 
\end{equation*}
for each $\iota = 1,2$, $h = 1,\ldots,n-2$ and $\tau \neq 0$,  where $\delta_{h, \tau}$ is the difference quotient operator given by  
\begin{equation*}
	\delta_{h,\tau} f(r,y) = \frac{f(r,y+\tau e_h) - f(r,y)}{\tau}, 
\end{equation*}
where $e_1,e_2,\ldots,e_{n-2}$ denote the standard basis vectors for $\mathbb{R}^{n-2}$.  By integration by parts for difference quotients,  
\begin{equation*}
	\int_{\{(r, y) \, : \, r \geq 0, \, y \in {\mathbb R}^{n-2}, \, r^{2} + |y|^{2} < 1/4\}} r^{2\alpha-1} \overline{D}\delta_{h,\tau} W^{\iota} \cdot \overline{D}\zeta \,dr \,dy = 0 
\end{equation*}
and then by integration by parts, 
\begin{equation} \label{homogrep1_eqn3_1}
	\int_{\{(r, y) \, : \, r \geq 0, \, y \in {\mathbb R}^{n-2}, \, r^{2} + |y|^{2} < 1/4\}} \delta_{h,\tau} W^{\iota} \op{div}(r^{2\alpha-1} \overline{D}\zeta) \,dr \,dy = 0 
\end{equation}
for each $\iota = 1,2$, $h = 1,\ldots,n-2$, and $\tau \neq 0$.  Since $W^{\iota}$ is even in the $r$ variable, this implies that 
\begin{equation} \label{homogrep1_eqn3}
	\int_{B_{1/2}^{n-1}(0)} \delta_{h,\tau} W^{\iota}  \op{div}(|r|^{2\alpha-1} \overline{D}\zeta) \,dr \,dy = 0 
\end{equation}
for each $\zeta \in C^2_c(B^{n-1}_{1/2}(0))$ with $\zeta(r, y) = \zeta(-r, y)$ and $D_{r}\zeta(r, y)$ in a neighborhood of $\{r = 0\}$ and for each $\iota = 1,2$, $h = 1,\ldots,n-2$, and $\tau \neq 0$. Now take any function $\zeta \in C^2_c(B^{n-1}_{1/2}(0))$ with $\zeta(r, y) = \zeta(-r, y)$ and note that $D_r \zeta(0,y) = 0$ for all $y \in B^{n-2}_{1/2}(0)$.  There is a sequence of functions $\zeta_k \in C^2_c(B^{n-1}_{1/2}(0))$ such that for each $k$, $\zeta_{k}(r, y) = \zeta_{k}(-r, y)$, $D_r \zeta_k = 0$ in a neighborhood of $\{r=0\}$, $\overline{D}^j \zeta_k \rightarrow \overline{D}^j \zeta$ pointwise in $B^{n-1}_{1/2}(0) \setminus \{r=0\}$ as $k \rightarrow \infty$ for $j = 0,1,2$, and $\zeta_k$ and its derivatives up to order two are uniformly bounded on $B^{n-1}_{1/2}(0)$.  Hence, since $W^{\iota}$ is bounded, we can substitute $\zeta_k$ for $\zeta$ in \eqref{homogrep1_eqn3} and apply the dominated convergence theorem to deduce 
that \eqref{homogrep1_eqn3} holds true whenever $\zeta \in C^1_c(B^{n-1}_{1/2}(0))$ with $\zeta(r, y) = \zeta(-r, y)$.  By \eqref{homogrep1_eqn2} and the dominated convergence theorem again, we can let $\tau \rightarrow 0$ in \eqref{homogrep1_eqn3} to deduce that 
\begin{equation} \label{homogrep1_eqn4}
	\int_{B^{n-1}_{1/2}(0)} D_{y_h} W^{\iota} \op{div}(|r|^{2\alpha-1} \overline{D}\zeta) \,dr \,dy = 0 
\end{equation}
for all $\iota = 1,2$, $h = 1,\ldots,n-2$, and $\zeta \in C^2_c(B^{n-1}_{1/2}(0))$ with $\zeta(r, y) = \zeta(-r, y)$. We are thus in a position to apply the results of the next two lemmas with $W = D_{y_{h}} W^{\iota}$. 


\begin{lemma} \label{homogrep1_mvp_lemma}
Let $(r, y)$ denote a general point in ${\mathbb R}^{n-1}$ with $r \in {\mathbb R}$ and $y \in {\mathbb R}^{n-2}$.  Let $z \in {\mathbb R}^{n-2}$ and $\rho_{0} > 0$.  Suppose that $W \in C^{\infty}(B^{n-1}_{\rho_{0}}(0,z) \setminus \{r=0\})$ satisfies $W(r,y) = W(-r,y)$, 
\begin{equation*}
	\int_{B^{n-1}_{\rho_{0}}(0,z)} |r|^{2\alpha-1} |W(r,y)| \,dr\,dy < \infty 
\end{equation*}
and 

\begin{equation} \label{homogrep1_pde}
	\int_{B^{n-1}_{\rho_{0}}(0,z)} W \,\op{div}(|r|^{2\alpha-1} \overline{D}\zeta) \,dr \,dy = 0
\end{equation}
for all $\zeta \in C^2_c(B^{n-1}_{\rho_{0}}(0,z))$ with $\zeta(r, y) = \zeta(-r, y)$.
Then 

\begin{align}
	 \label{homogrep1_mvp1} &\frac{1}{\sigma^{n+2\alpha-3}} \int_{\partial B^{n-1}_{\sigma}(0,z)} |r|^{2\alpha-1} W(r,y) \,dr\,dy 
		= \frac{1}{\rho^{n+2\alpha-3}} \int_{\partial B^{n-1}_{\rho}(0,z)} |r|^{2\alpha-1} W(r,y) \,dr\,dy \;\; {\rm and}\\
	\label{homogrep1_mvp2} &\frac{1}{\sigma^{n+2\alpha-2}} \int_{B^{n-1}_{\sigma}(0,z)} |r|^{2\alpha-1} W(r,y) \,dr\,dy 
		= \frac{1}{\rho^{n+2\alpha-2}} \int_{B^{n-1}_{\rho}(0,z)} |r|^{2\alpha-1} W(r,y) \,dr\,dy
\end{align}
for all $0 < \sigma \leq \rho < \rho_{0}$. 
\end{lemma}
\begin{proof}
Fix $z \in {\mathbb R}^{n-2}$ and $\rho_{0} >0$.  For $\xi \in C^{1}_{c}({\mathbb R})$ with ${\rm spt} \, \xi \subset (0, \rho_{0})$, and $(r, y) \in {\mathbb R}^{n-1}$, let 
\begin{equation*}
	\zeta(r,y) = \int_R^{\rho_{0}} t^{3-n-2\alpha} \xi(t) \,dt
\end{equation*}
where $R = \sqrt{r^2+|y-z|^2}$. Then $\zeta \in C^{2}_{c}({\mathbb R}^{n-1})$, ${\rm spt} \, \zeta \subset B_{\rho_{0}}^{n-1}(0, z),$ $\zeta$ is constant near the point $(0, z)$  and $\zeta(r, y) = \zeta(-r, y)$ for all $r > 0$ and $y \in \mathbb{R}^{n-2}$.  Also, for $r > 0$, by direct calculation
\begin{equation*}
	\op{div}(r^{2\alpha-1} D \zeta) 
	= -\op{div}(r^{2\alpha-1} R^{2-n-2\alpha} \xi(R) \,(r,y)) 
	= -r^{2\alpha-1} R^{3-n-2\alpha} \,\xi'(R).
\end{equation*}
Substituting this into \eqref{homogrep1_pde} gives that
\begin{equation*}
	\int_{B^{n-1}_{\rho_{0}}(0,z)} |r|^{2\alpha-1} R^{3-n-2\alpha} W(r,y) \,\xi'(R) \,dr \,dy = 0
\end{equation*}
which, by the coarea formula, is equivalent to 
\begin{equation*}
	\int_{0}^{\rho_0} \left( \rho^{3-n-2\alpha} \int_{\partial B^{n-1}_{\rho}(0,z)} |r|^{2\alpha-1} W \right) \xi'(\rho) \,d\rho = 0.
\end{equation*}
This says that the function  $\rho \mapsto \rho^{3-n-2\alpha} \int_{\partial B^{n-1}_{\rho}(0,z)} |r|^{2\alpha-1} W$ has zero distributional derivative and therefore must be constant on $(0,\rho_0)$.  This gives us \eqref{homogrep1_mvp1}.  The identity \eqref{homogrep1_mvp2} follows from \eqref{homogrep1_mvp1} via integration.
\end{proof}

\begin{lemma} \label{homogrep1_liousville_lemma}
Let $n \geq 3$ and let $(r,y)$ denote a general point in ${\mathbb R}^{n-1}$ with $r \in {\mathbb R}$ and $y \in {\mathbb R}^{n-2}$. Suppose that $W \in C^{\infty}(\mathbb{R}^{n-1} \setminus \{r=0\})$ is a homogeneous degree zero function satisfying $W(r,y) = W(-r,y)$ on $\mathbb{R}^{n-1}$, \begin{equation} \label{homogrep1_L2}
	\int_{B^{n-1}_1(0)} |r|^{2\alpha-1} |W(r,y)|^2 \,dr\,dy < \infty, 
\end{equation}
and 
\begin{equation} \label{homogrep1_pde2}
	\int_{{\mathbb R}^{n-1}} W \,\op{div}(|r|^{2\alpha-1} \overline{D}\zeta) \,dr \,dy = 0
\end{equation}
for any $\zeta \in C^2_c({\mathbb R}^{n-1})$ with $\zeta(r, y) = \zeta(-r, y)$.  Then $W$ is a constant function on $\mathbb{R}^{n-1}$. 
\end{lemma}
\begin{proof}
Let us first suppose that $W \in C^0(\mathbb{R}^{n-1} \setminus \{0\})$.  Then by homogeneity, $W$ attains its maximum and minimum values at points in $B^{n-1}_1(0) \setminus \{0\}$.  By the strong maximum principle, $W$ must attain its maximum and minimum values on $\{r=0\} \setminus \{0\}$.  However, by \eqref{homogrep1_mvp2} and homogeneity and continuity of $W$, for each $z \in \mathbb{R}^{n-2} \setminus \{0\}$,  
\begin{align*}
	W(0,z) &= \lim_{\rho \downarrow 0} \frac{1}{c_{n,\alpha} \rho^{n+2\alpha-2}} \int_{B^{n-1}_{\rho}(0,z)} |r|^{2\alpha-1} W 
	= \lim_{\rho \rightarrow \infty} \frac{1}{c_{n,\alpha} \rho^{n+2\alpha-2}} \int_{B^{n-1}_{\rho}(0,z)} |r|^{2\alpha-1} W
	\\&= \lim_{\rho \rightarrow \infty} \frac{1}{c_{n,\alpha} \rho^{n+2\alpha-2}} \int_{B^{n-1}_{\rho}(0,0)} |r|^{2\alpha-1} W 
	= \frac{1}{c_{n,\alpha}} \int_{B^{n-1}_1(0,0)} |r|^{2\alpha-1} W, 
\end{align*}
where $c_{n,\alpha} = \int_{B^{n-1}_1(0)} |r|^{2\alpha-1}$. Thus $W$ is constant on $\mathbb{R}^{n-1}$. 

For the general case, observe that $W(r,y) = \psi(y/r)$ on ${\mathbb R}^{n-1} \cap \{r >0\}$ where $\psi \in C^{\infty}(\mathbb{R}^{n-2})$ is given by $\psi(z) = W(1, z)$, $z \in {\mathbb R}^{n-2}$.  By \eqref{homogrep1_pde2}, $\op{div}(r^{2\alpha-1} \,DW) = 0$ in $\mathbb{R}^{n-1} \cap \{r > 0\}$, so $\psi$ satisfies 
\begin{equation} \label{homogrep1_eqn5}
	\Delta \psi(z) + \sum_{i,j=1}^{n-2} z_i z_j D_{ij} \psi(z) - (2\alpha-3) \sum_{i=1}^{n-2} z_i D_i \psi(z) = 0
\end{equation}
for all $z \in \mathbb{R}^{n-2}$.  When $n = 3$, one can explicitly solve \eqref{homogrep1_eqn5} by integration to find that 
\begin{equation} \label{homogrep1_eqn6}
	\psi(z) = C_1 + C_2 \int_0^z (1+t^2)^{\frac{2\alpha-3}{2}} \,dt 
\end{equation}
for constants $C_1,C_2 \in \mathbb{R}$.  Hence either $\lim_{z \rightarrow \infty} \psi(z) = \pm\infty$ or $\lim_{z \rightarrow \infty} \psi(z)$ exists and is finite.  Equivalently (since $W$ is even in the $r$ variable) $\lim_{(r,y) \rightarrow (0,1)} W(r,y) = \pm\infty$ or $\lim_{(r,y) \rightarrow (0,1)} W(r,y)$ exists and is finite.  But by \eqref{homogrep1_mvp2} 
\begin{equation*}
	\lim_{\rho \downarrow 0} \frac{1}{\rho^{n+2\alpha-2}} \int_{B^{n-1}_{\rho}(0,1)} |r|^{2\alpha-1} W(r,y) \,dr\,dy 
		= 2^{n+2\alpha-2} \int_{B^{n-1}_{1/2}(0,1)} |r|^{2\alpha-1} W(r,y) \,dr\,dy < \infty 
\end{equation*}
so $\lim_{(r,y) \rightarrow (0,1)} W(r,y)$ exists and is finite.  By symmetry, $\lim_{(r,y) \rightarrow (0,-1)} W(r,y)$ exists and is finite. Hence $W$ extends to a continuous function on $\mathbb{R}^2 \setminus \{0\}$, and  therefore $W$ is constant on $\mathbb{R}^2$.  (Alternatively, one can use integration by parts and \eqref{homogrep1_eqn6} to show that $W$ satisfies \eqref{homogrep1_pde2} only if $C_2 = 0$.)

Suppose $n \geq 4$.  Since $\psi$ is smooth on $\mathbb{R}^{n-2}$, we can represent $\psi$ by the Fourier series expansion 
\begin{equation*}
	\psi(s\omega) = \sum_{k=1}^{\infty} \gamma_k(s) \phi_k(\omega)
\end{equation*}
for all $s > 0$ and $\omega \in \mathbb{S}^{n-3}$, where $\phi_k$ denote $L^{2}$ orthonormal eigenfunctions of the (negative) Laplacian on $\mathbb{S}^{n-3}$. Thus 
$\Delta_{S^{n-3}} \phi_k + \mu_k\phi_k = 0$ on $S^{n-3}$ for $0 = \mu_1 < \mu_2 \leq \mu_3 \leq \mu_4 \cdots$ and $\gamma_k(s) = \int_{\mathbb{S}^{n-3}} \psi(s \omega) \phi_k(\omega) d\omega$.  Notice that since $\psi \in C^{\infty}(\mathbb{R}^{n-2})$, $\gamma_k \in C^0([0,\infty)) \cap C^{\infty}((0,\infty))$ with 
\begin{equation*}
	\gamma_k(0) = \int_{\mathbb{S}^{n-3}} \psi(0) \phi_k(\omega) d\omega = \begin{cases} \psi(0) &\text{if } k = 1 \\ 0 &\text{if } k \geq 2 . \end{cases}
\end{equation*}
By \eqref{homogrep1_eqn5}, $\gamma_k$ satisfies 
\begin{equation} \label{homogrep1_eqn7}
	(1+s^2) \,\gamma''_k(s) + \left( \frac{n-3}{s} - (2\alpha-3) s \right) \gamma'_k(s) - \frac{\mu_k}{s^2} \,\gamma_k(s) = 0
\end{equation}
for all $s > 0$.  By applying the Frobenius method (see~\cite[Theorem 4.5]{Teschl}) to \eqref{homogrep1_eqn7} to obtain a series expansion at infinity 
\begin{gather*} 
	\gamma_k(s) = C_1 f_1(1/s) + C_2 s^{2\alpha-2} f_2(1/s) \quad \text{if } 2\alpha \not\in \mathbb{N}, \\
	\gamma_k(s) = (C_1 + C_2 c_{k,\alpha} \log(s)) f_1(1/s) + C_2 s^{2\alpha-2} f_2(1/s) \quad \text{if } 2\alpha \in \mathbb{N}, \, \alpha \geq 1, \\
	\gamma_k(s) = C_1 f_1(1/s) + (C_1 c_{k,1/2} \log(s) + C_2) s^{-1} f_2(1/s) \quad \text{if } \alpha = 1/2, 
\end{gather*}
where $C_1,C_2 \in \mathbb{R}$ are arbitrary constants, $c_{k,\alpha} \in \mathbb{R}$ are constants depending on $k$ and $\alpha$, and $f_1,f_2$ are real-analytic at the origin such that $f_j(0) = 1$ and $f_j(s) = f_j(-s)$ for $j = 1,2$.  In particular, either $\lim_{s \rightarrow \infty} \gamma_k(s) = \pm \infty$ or $\lim_{s \rightarrow \infty} \gamma_k(s)$ exists and is finite.  For each $k \geq 2$, by using the definition of $\gamma_k$ and $L^{2}$ orthogonality of $1$ and $\phi_k$, 
\begin{align} \label{homogrep1_eqn8}
	|\gamma_k(s)| &= \left| \int_{\mathbb{S}^{n-3}} \psi(s \omega) \phi_k(\omega) \,d\omega \right|
	\leq \left| \int_{\mathbb{S}^{n-3}} \psi(0) \phi_k(\omega) d\omega \right| + \int_{\mathbb{S}^{n-3}} |\psi(s \omega) - \psi(0)| \,|\phi_k(\omega)| \,d\omega 
	\\&\leq 0 + C |s| = C |s| \nonumber 
\end{align}
for all $0 < |s| \leq 1$, where $C \in (0,\infty)$ is a constant independent of $s$.  

For each $k$, let $\psi_k : \mathbb{R}^{n-2} \rightarrow \mathbb{R}$ be the function given by $\psi_k(z) = \gamma_k(|z|) \phi_k(z/|z|)$ for $z \in \mathbb{R}^{n-2} \setminus \{0\}$, $\psi_1(0) = \psi(0)$ and $\psi_k(0) = 0$ if $k > 1$.  Let $W_k : \mathbb{R}^{n-1} \setminus \{r=0\} \rightarrow \mathbb{R}$ be the function given by $W_k(r,y) = \psi_k(y/r)$ for $r \neq 0$ and $y \in \mathbb{R}^{n-2}$ (and in particular for $r \neq 0$, $W_1(r,0) = \psi(0)$ and $W_k(r,0) = 0$ if $k > 1$).  We want to show that $W_k$ satisfies \eqref{homogrep1_pde2} for any $\zeta \in C^2_c({\mathbb R}^{n-1})$ with $\zeta(r, y) = \zeta(-r, y)$ and that $W_k$ extends to a continuous function on $\mathbb{R}^{n-1} \setminus \{0\}$.  To see this, note first that it readily follows from \eqref{homogrep1_eqn7} that $\psi_k$ is a smooth solution to \eqref{homogrep1_eqn5} in ${\mathbb R}^{n-2} \setminus \{0\}$.  In particular, by integration by parts $\psi_k$ satisfies 
\begin{equation} \label{homogrep1_eqn9}
	\int_{{\mathbb R}^{n-2}} \psi_k \left( \Delta \zeta + \sum_{i,j=1}^{n-2} D_{ij} (z_i z_j \zeta) + (2\alpha-3) \sum_{i=1}^{n-2} D_i (z_i \zeta) \right) = 0
\end{equation}
for all $\zeta \in C^2_c({\mathbb R}^{n-2} \setminus \{0\})$.  To extend $\psi_k$ to a solution to \eqref{homogrep1_eqn9} in ${\mathbb R}^{n-2}$, first observe that when $k = 1$ we have $\psi_k(z) = \gamma_1(|z|)$.  Since $\gamma_1$ is bounded at $0$, $\gamma_1$ extends to a smooth solution to \eqref{homogrep1_eqn7} on $\mathbb{R}$ such that $\gamma_1(s) = \gamma_1(-s)$.  Thus $\psi_k(z) = \gamma_1(|z|)$ is a smooth radial solution to \eqref{homogrep1_eqn5} in ${\mathbb R}^{n-2}$.  Next suppose $k \geq 2$.  For each $\delta > 0$ let $\eta_{\delta} \in C^2([0,\infty))$ be a cutoff function such that $\eta_{\delta}(t) = 0$ for $t \in [0,\delta/2]$, $\eta_{\delta}(t) = 1$ for $t \in [\delta,\infty)$, and $|D\eta_{\delta}| \leq C\delta^{-1}$ and $|D^2 \eta_{\delta}| \leq C\delta^{-2}$ for some constant $C = C(n) \in (0,\infty)$.  By replacing $\zeta(z)$ with $\eta_{\delta}(|z|) \,\zeta(z)$ in \eqref{homogrep1_eqn9} and using \eqref{homogrep1_eqn8}, for each $\zeta \in C^2_c({\mathbb R}^{n-2})$ and sufficient small $\delta > 0$ 
\begin{align*}
	&\left| \int_{{\mathbb R}^{n-2}} \psi_k \,\eta_{\delta}(|z|) \left( \Delta \zeta + \sum_{i,j=1}^{n-2} D_{ij} (z_i z_j \zeta) 
		+ (2\alpha-3) \sum_{i=1}^{n-2} D_i (z_i \zeta) \right) \right| 
	\\&\hspace{10mm} \leq C(n) \,\delta^{-2} \,\|\zeta\|_{C^1({\mathbb R}^{n-2})} \int_{B^{n-2}_{\delta}(0)} |\psi_k| \leq C \delta^{n-3} 
\end{align*}
where $C \in (0,\infty)$ is a constant independent of $\delta$.  Letting $\delta \downarrow 0$, we conclude that $\psi_k$ satisfies \eqref{homogrep1_eqn9} for all $\zeta \in C^2_c({\mathbb R}^{n-2})$.  By elliptic regularity, $\psi_k$ is a smooth solution to \eqref{homogrep1_eqn5} in ${\mathbb R}^{n-2}$.  It follows that $W_k$ is a smooth solution to $\op{div}(|r|^{2\alpha-1} \overline{D}W_k) = 0$ in ${\mathbb R}^{n-1} \setminus \{r = 0\}$.  

Take any $\zeta \in C^2_c({\mathbb R}^{n-1} \setminus \{y=0\})$ with $\zeta(r, y) = \zeta(-r, y)$ and $D_r \zeta(0,y) = 0$ for all $y \in {\mathbb R}^{n-2}$.  We represent $\zeta$ by the Fourier series expansion 
\begin{equation*}
	\zeta(r,s\omega) = \sum_{k=1}^{\infty} \xi_{k}(r,s) \phi_{k}(\omega) \quad\text{where}\quad 
	\xi_{k}(r,s) = \int_{S^{n-3}} \zeta(r,s\omega) \phi_{k}(\omega) \,d\omega  
\end{equation*}
for $r,s > 0$ and $\omega \in \mathbb{S}^{n-3}$.  Define $\zeta_{k} \in C^2_c({\mathbb R}^{n-1})$ by $\zeta_{k}(r,s\omega) = \xi_{k}(r,s) \phi_{k}(\omega)$ for each $r,s > 0$ and $\omega \in \mathbb{S}^{n-3}$.  Since $\zeta \in C^2_c(\mathbb{R}^{n-1} \setminus \{y=0\})$, $\xi_{k} \in C^2_c(\mathbb{R}^2 \setminus \mathbb{R} \times \{0\})$.  Moreover, using the fact that $D_r \zeta(0,y) = 0$ for all $y \in {\mathbb R}^{n-2}$, we see that $|r|^{1-2\alpha} \op{div}(|r|^{2\alpha-1}D \zeta) = \Delta \zeta + (2\alpha-1) r^{-1} D_r \zeta$ is bounded and similarly $|r|^{1-2\alpha} \op{div}(|r|^{2\alpha-1}D \zeta_{k})$ is bounded.  We have the Fourier series expansion  
\begin{align*}
	|r|^{1-2\alpha} \op{div}(|r|^{2\alpha-1} \overline{D}\zeta) 
	&= \sum_{k=1}^{\infty} \left( |r|^{1-2\alpha} \op{div}(|r|^{2\alpha-1} \overline{D}\xi_{k}(r,s)) - \mu_{k} s^{-2} \xi_{k}(r,s) \right) \phi_{k}(\omega) 
	\\&= \sum_{k=1}^{\infty} |r|^{1-2\alpha} \op{div}(|r|^{2\alpha-1} \overline{D}\zeta_{k}) ,
\end{align*}
where $s = |y|$ and $\omega = y/|y|$.  Since $W$ satisfies \eqref{homogrep1_pde2} for all $\zeta \in C^2_c({\mathbb R}^{n-1})$ with $\zeta(r, y) = \zeta(-r, y)$, we can replace $\zeta$ with $\zeta_{k}$ in \eqref{homogrep1_pde} and use the fact that $\phi_{k}$ is an $L^{2}$ orthonormal basis to obtain 
\begin{align*}
	\int_{{\mathbb R}^{n-1}} W_k \,\op{div}(|r|^{2\alpha-1} \overline{D}\zeta) \,dr \,dy
	&= \int_{{\mathbb R}^{n-1}} W_k \,\op{div}(|r|^{2\alpha-1} \overline{D}\zeta_k) \,dr \,dy 
	\\&= \int_{{\mathbb R}^{n-1}} W \,\op{div}(|r|^{2\alpha-1} \overline{D}\zeta_k) \,dr \,dy = 0. 
\end{align*}
Hence $W_k$ satisfies \eqref{homogrep1_pde2} for $\zeta \in C^2_c({\mathbb R}^{n-1} \setminus \{y=0\})$ with $\zeta(r, y) = \zeta(-r, y)$ and $D_r \zeta(0,y) = 0$ for all $y \in {\mathbb R}^{n-2}.$ It follows that $W_k$  satisfies \eqref{homogrep1_pde2} for all $\zeta \in C^2_c({\mathbb R}^{n-1} \setminus \{0\})$ with $\zeta(r, y) = \zeta(-r, y)$. 

Recall that by the series expansion $\psi_k$ at infinity, either $\lim_{s \rightarrow \infty} \gamma_k(s) = \infty$ or $\lim_{s \rightarrow \infty} \gamma_k(s)$ exists and is finite.  Thus either $\lim_{(r,y) \rightarrow (0,z)} W_k(r,y) = \infty$ for all $z \in {\mathbb S}^{n-3}$ such that $\phi_k(z) \neq 0$ or $\lim_{(r,y) \rightarrow (0,z)} W_k(r,y)$ exists and is finite for all $z \in {\mathbb S}^{n-3}$.  By \eqref{homogrep1_mvp2} with $W_k$ in place of $W$ 
\begin{equation*}
	\lim_{\rho \downarrow 0} \frac{1}{\rho^{n+2\alpha-2}} \int_{B^{n-1}_{\rho}(0,z)} |r|^{2\alpha-1} W_k(r,y) \,dr\,dy 
		= 2^{n+2\alpha-2} \int_{B^{n-1}_{1/2}(0,z)} |r|^{2\alpha-1} W_k(r,y) \,dr\,dy < \infty 
\end{equation*}
whenever $z \in \mathbb{S}^{n-3}$ with $\phi_k(z) \neq 0$.  Hence $\lim_{(r,y) \rightarrow (0,z)} W_k(r,y)$ exists and is finite for all $z \in \mathbb{S}^{n-3}$, and in particular $W_k$ extends to a continuous function on $\mathbb{R}^{n-1} \setminus \{0\}$. 

For each $\delta > 0$, let $\eta_{\delta} \in C^1([0,\infty))$ is a cutoff function such that $\eta_{\delta}(t) = 0$ for $t \in [0,\delta/2]$, $\eta_{\delta}(t) = 1$ for $t \in [\delta,\infty)$, and $|D\eta_{\delta}| \leq C\delta^{-1}$ and $|D^2 \eta_{\delta}| \leq C\delta^{-2}$ for some constant $C = C(n) \in (0,\infty)$.  By replacing $\zeta(r,y)$ in \eqref{homogrep1_pde2} with $\eta_{\delta}(|(r,y)|) \,\zeta(r,y)$, 
\begin{equation*}
	\int_{{\mathbb R}^{n-1}} W_k \,\eta_{\delta}(|(r,y)|) \,\op{div}(|r|^{2\alpha-1} \overline{D}\zeta) \,dr \,dy 
	\leq C(n) \,\delta^{n-4+2\alpha} \|W_k\|_{L^{\infty}(B_1(0))} \|\zeta\|_{C^1({\mathbb R}^{n-1})} 
\end{equation*}
for any $\zeta \in C^2_c({\mathbb R}^{n-1})$ with $\zeta(r, y) = \zeta(-r, y)$.  By letting $\delta \downarrow 0$, we conclude that $W_k$ satisfies \eqref{homogrep1_pde2} for each 
$\zeta \in C^2_c({\mathbb R}^{n-1})$ with $\zeta(r, y) = \zeta(-r, y)$.  

Now for each integer $k \geq 2$, $W_k$ is a homogeneous degree zero function that is continuous on $\mathbb{R}^{n-1} \setminus \{0\}$ and that satisfies \eqref{homogrep1_pde2} for $\zeta \in C^2_c({\mathbb R}^{n-1})$ with $\zeta(r, y) = \zeta(-r, y)$.  Consequently, $W_k(r,y) = \gamma_k(|y|/r) \phi_k(y/|y|)$ is constant.  Hence for each $k \geq 2$, since $\phi_k$ in nonconstant, $W_{k} = 0$. 
Since $W_1$ is constant, it follows that $W$ is constant on $\mathbb{R}^{n-1}$. 
\end{proof}

In view of \eqref{homogrep1_eqn2} and \eqref{homogrep1_eqn4}, we may apply Lemma~\ref{homogrep1_liousville_lemma} with $D_{y_h} W^{\iota}$  in place of $W$ to conclude that $D_{y_h} W^{\iota}$ is constant 
on $\mathbb{R}^{n-1}$ for $\iota = 1, 2$ and $h= 1, \ldots, n-2$.  Thus we may express $W = (W^{1}, W^{2})^{T}$ as 
\begin{equation} \label{homogrep1_eqn10}
	W(r,y) = \sum_{j=1}^J \sum_{k=1}^{p_j} m_{j,k} \widetilde{w}_{j,k}(r,y) = W(r,0) + c Ay, 
\end{equation}
where $\widetilde{w}_{j, k} = (\widetilde{w}^{1}_{j, k}, \widetilde{w}^{2}_{j, k})^{T}$ and $A$ is a constant $2 \times (n-2)$ matrix and for convenience we set  
\begin{equation*}
	c = \sum_{j=1}^{J} \sum_{k=1}^{p_j} \alpha^2 m_{j,k} |c^{(0)}_j|^2, 
\end{equation*}
where $c^{(0)}_j$ is as in \eqref{cylindrical-2} if $\varphi^{(0)}_{j}$ is nonzero and $c^{(0)}_{j} = 0$ otherwise.  Moreover, since $W^{\iota}$ is homogeneous degree one, $W(r,0) = ar$ where $a = W(1,0) \in \mathbb{R}^{2}$.  (While the particular value of $a$ is not important, one can use the fact that each $w_{j,k,l;a}$ is locally harmonic to show that $\op{div}(r^{2\alpha-1} \overline{D}W^{\iota}) = 0$ weakly in $\mathbb{R}^n \cap \{r > 0\}$ and thus $a = 0$ if $\alpha \neq 1/2$.  If $\alpha = 1/2$, 
$a$ may be non-zero and this is consistent with the definition of $\mathfrak{L}$.)

Let $z \in B^{n-2}_{1/4}(0)$.  By \eqref{blowup_est4}, the $L^2$ orthogonality of $D_1 \varphi^{(0)}_{j,k}(e^{i\theta})$ and $D_2 \varphi^{(0)}_{j,k}(e^{i\theta})$, and the fact that
\begin{equation*}
	\frac{1}{\pi q_0} \int_0^{2\pi q_0} |D_{\iota} \varphi^{(0)}_{j,k}(e^{i\theta})|^2 \,d\theta 
	= \alpha^2 |c^{(0)}_j|^2 \cdot \frac{1}{\pi q_0} \int_0^{2\pi q_0} \cos^2(\alpha \theta) \,d\theta 
	= \alpha^2 |c^{(0)}_j|^2
\end{equation*} 
for $\iota = 1,2$, we have 
\begin{equation*}
	\int_{B^{n-1}_{1/2}(0)} \frac{r^{2\alpha-1}}{|(r,y-z)|^{n+2\alpha-\sigma}} \left| \widetilde{w}_{j,k}(r,y) - \alpha^2 |c^{(0)}_j|^2 \lambda(z) \right|^2 dr dy \leq C 
\end{equation*}
for all $j$ and $k$ for which $\varphi^{(0)}_j$ is not identically zero, where $C = C(n,m,q,\alpha,\varphi^{(0)},\sigma) \in (0,\infty)$ is a constant.  By the triangle inequality, 
\begin{equation*}
	\int_{B^{n-1}_{1/2}(0)} \frac{r^{2\alpha-1}}{|(r,y-z)|^{n+2\alpha-\sigma}} \left| \sum_{j=1}^{J} \sum_{k=1}^{p_j} m_{j,k} (\widetilde{w}_{j,k}(r,y) - \alpha^2 |c^{(0)}_j|^2 \lambda(z)) \right|^2 dr dy \leq C  
\end{equation*}
for some constant $C = C(n,m,q,\alpha,\varphi^{(0)},\sigma) \in (0,\infty)$ (recall that $\widetilde{w}^1_{j,k} = \widetilde{w}^2_{j,k} = 0$ if $\varphi^{(0)}_j$ is identically zero).  In other words, by \eqref{homogrep1_eqn10}, 
\begin{equation*}
	\int_{B^{n-1}_{1/2}(0)} \frac{r^{2\alpha-1}}{|(r,y-z)|^{n+2\alpha-\sigma}} \left| Ay - \lambda(z) \right|^2 dr dy \leq C < \infty, 
\end{equation*}
which implies that $\lambda(z) = Az$. 

Thus, taking $\lambda(z) = Az$ in \eqref{blowup_est4}, we conclude the following: \emph{If $w \in {\mathfrak B}$ is homogeneous of degree $\alpha$, then there exists a constant   $2 \times (n-2)$ real matrix $A$ such that 
\begin{equation}\label{linear_lambda}
	\int_{B_{1/2}(0)} \sum_{j=1}^J \sum_{k=1}^{p_j} \sum_{l=1}^{q_{j,k}} \sum_{h=1}^{m_{j, k}}
		\frac{|w_{j,k,l}(X) - D_x \varphi^{(0)}_{j, l}(X) \cdot Az|^2}{|X-(0,z)|^{n+2\alpha-\sigma}} dX \leq C 
\end{equation}
for all $\sigma \in (0,1/q)$ and some constant $C = C(n,m,q,\alpha,\varphi^{(0)},\sigma) \in (0,\infty)$.}  

\begin{proof}[Proof of Theorem~\ref{homogrep_lemma} (for general $\lambda$)] Let  $w  = (w_{j, k}) \in {\mathfrak B}$ and suppose that $w$ is homogeneous of degree $\alpha$.   Let $A$ be the matrix as in \eqref{linear_lambda}. Write $A = (\op{Re}(b) \hspace{2mm} \op{Im}(b))^{T}$ for some $b \in {\mathbb C}^{n-2}$, and let $B$ be the $n \times n$ matrix defined by $b$ as in \eqref{B_matrix}. For $\nu =1, 2, 3, \ldots,$ let 
$u^{(\nu)}$, $\varphi^{(\nu)}$, $\varepsilon_{\nu}$, $\beta_{\nu}$ and $\delta_{\nu}$ correspond to $w$ as in Defintion~\ref{blowupclass_defn}, and let 
$E_{\nu} = \left( \int_{B_1(0)} \mathcal{G}(u^{(\nu)},\varphi^{(\nu)})^2 \right)^{1/2}.$ We have that for any 
$\varphi \in \Phi_{\varepsilon}(\varphi^{(0)})$ (where $\varepsilon$ is any small number as in Definition~\ref{varphi_defn}), 
\begin{align*} 
	&\int_{B_1(0)} \mathcal{G}(u^{(\nu)} \circ e^{-E_{\nu} B},\varphi)^2 
	\\&\hspace{10mm} \leq 2 \int_{B_1(0)} \mathcal{G}(u^{(\nu)} \circ e^{-E_{\nu} B},\varphi \circ e^{-E_{\nu} B})^2 
		+ 2 \int_{B_1(0)} \mathcal{G}(\varphi \circ e^{-E_{\nu} B},\varphi)^2 \nonumber 
	\\&\hspace{10mm} \leq 2 \int_{B_1(0)} \mathcal{G}(u^{(\nu)},\varphi)^2 
		+ 2 E_{\nu}^2 |B|^2 \int_{B_1(0)} \int_0^1 |\nabla \varphi(e^{-t E_{\nu} B} X)|^2 \,dt \,dX \nonumber 
	\\&\hspace{10mm} \leq 2 \int_{B_1(0)} \mathcal{G}(u^{(\nu)},\varphi)^2 
		+ 2 E_{\nu}^2 |B|^2 \int_{B_1(0)} |\nabla \varphi(X)|^2 \,dX \nonumber 
\end{align*}
so in particular 
\begin{equation}\label{homogrep1_eqn11}
\int_{B_1(0)} \mathcal{G}(u^{(\nu)} \circ e^{-E_{\nu} B},\varphi^{(0)})^2 \leq C_{1}\epsilon_{\nu}^{2} 
\end{equation}
and 
\begin{equation}\label{homogrep1_eqn12}
\int_{B_1(0)} \mathcal{G}(u^{(\nu)} \circ e^{-E_{\nu} B},\varphi^{(\nu)})^2 \leq CE_{\nu}^{2} 
\end{equation}
where $C_{1}$, $C \in (0,\infty)$ are  constants depending only on $n$, $m$, $q$, $\alpha$, $\varphi^{(0)}$, and $B$.  Moreover, in case $p^{(\nu)} > p_{0}$ (notation as in Definition~\ref{blowupclass_defn} and Definition~\ref{varphi_defn}), we have by \eqref{blowupclass_eqn3}, \eqref{homogrep1_eqn11} and Remark~\ref{graphical_rmk}(a) that 
\begin{align*} 
	\int_{B_1(0)} \mathcal{G}(u^{(\nu)},\varphi^{(\nu)})^2 &\leq
	\beta_{\nu} \inf_{\varphi' \in \bigcup_{p'=p_{0}}^{p^{(\nu)}-1} \Phi_{3\varepsilon_{\nu},p'}(\varphi^{(0)})} \int_{B_1(0)} \mathcal{G}(u^{(\nu)},\varphi')^2
	\\&\leq \beta_{\nu} \inf_{\varphi' \in \bigcup_{p'=p_{0}}^{p^{(\nu)}-1} \Phi_{3\sqrt{C_{1}}\varepsilon_{\nu},p'}(\varphi^{(0)})} \int_{B_1(0)} \mathcal{G}(u^{(\nu)},\varphi')^2
\end{align*}
with $\beta_{\nu} \to 0$ as $\nu \to \infty$. It follows from this and the triangle inequality that for each $\nu$ with $p^{(\nu)} > p_{0}$ and each 
$\varphi' \in \bigcup_{p'=p_{0}}^{p^{(\nu)}-1} \Phi_{3\sqrt{C_{1}}\varepsilon_{\nu},p'}(\varphi^{(0)})$, 
\begin{align*}
&\int_{B_1(0)} \mathcal{G}(u^{(\nu)},\varphi^{(\nu)})^2 
\\&\hspace{10mm}\leq 2\beta_{\nu}\left(\int_{B_1(0)} \mathcal{G}(u^{(\nu)},\varphi^{\prime}\circ e^{E_{\nu}B})^2 + \int_{B_1(0)} \mathcal{G}(\varphi^{\prime} \circ e^{E_{\nu}B},\varphi^{\prime})^2\right)
\\&\hspace{10mm}\leq 2\beta_{\nu} \left(\int_{B_1(0)} \mathcal{G}(u^{(\nu)} \circ e^{-E_{\nu}B},\varphi^{\prime})^2 + CE_{\nu}^{2}\right)
\end{align*}
where $C = C(n, m,  q, \alpha, \varphi^{(0)}, B) \in (0, \infty)$, whence $E_{\nu}^{2} \leq 4\beta_{\nu}   \int_{B_1(0)} \mathcal{G}(u^{(\nu)} \circ e^{-E_{\nu}B},\varphi')^2$ and therefore, in view of \eqref{homogrep1_eqn12}, 
 \begin{equation}\label{homogrep1_eqn13}
	\int_{B_1(0)} \mathcal{G}(u^{(\nu)} \circ e^{-E_{\nu}B},\varphi^{(\nu)})^2 
	\leq C\beta_{\nu} \inf_{\varphi' \in \bigcup_{p'=p_{0}}^{p^{(\nu)}-1} \Phi_{3\sqrt{C_{1}}\varepsilon_{\nu},p'}(\varphi^{(0)})} \int_{B_1(0)} \mathcal{G}(u^{(\nu)}\circ e^{-E_{\nu} B},\varphi')^2.
\end{equation}

The inequalities \eqref{homogrep1_eqn11} and \eqref{homogrep1_eqn13}  and the definition of ${\mathfrak B}$ imply that  there exists $w^{\star} \in {\mathfrak B}$ such that, passing to a subsequence of $(\nu)$, $w^{\star}$ is the blow-up of $u^{(\nu)} \circ e^{-E_{\nu} B}$ relative to $\varphi^{(\nu)}$ and excess 
$$E_{\nu}^{\star} = \sqrt{\int_{B_1(0)} \mathcal{G}(u^{(\nu)} \circ e^{-E_{\nu} B},\varphi^{(\nu)})^2}.$$  
Let $\tau > 0$ arbitrary.  
Let $v^{(\nu)}_{j,k,l}$ be as in \eqref{v_localized} corresponding to $v^{(\nu)}_{j,k}$ as in Corollary~\ref{graphical_cor} with $u = u^{(\nu)}$ and $\varphi = \varphi^{(\nu)}$.  
Taylor's theorem \eqref{blowup2L_eqn4} (applied locally away from the axis $\{0\} \times \mathbb{R}^{n-2}$ with $\varphi^{(\nu)}$ in place of $f$ and $-B$ in place of $B$) gives us that for sufficiently large $\nu$ and all $X \in B_{1-\tau}(0) \cap \{r > \tau\}$,
\begin{align*}
	u^{(\nu)}(e^{-E_{\nu} B} X) 
	&= \sum_{j=1}^{J} \sum_{k=1}^{p_j} \sum_{l=1}^{q_{j,k}} \sum_{h=1}^{m_{j,k}} 
		\llbracket \varphi^{(\nu)}_{j,k,l}(e^{-E_{\nu} B} X) + v^{(\nu)}_{j,k,l,h}(e^{-E_{\nu} B} X) \rrbracket
	\\&= \sum_{j=1}^{J} \sum_{k=1}^{p_j} \sum_{l=1}^{q_{j,k}} \sum_{h=1}^{m_{j,k}} 
		\llbracket \varphi^{(\nu)}_{j,k,l}(X) - E_{\nu} D_x \varphi^{(\nu)}_{j,k,l}(X) \cdot Ay + v^{(\nu)}_{j,k,l,h}(e^{-E_{\nu} B}X) + \mathcal{R}^{(\nu)}_{j,k,l,h}(X) \rrbracket
\end{align*}
where $E_{\nu}^{-1} \mathcal{R}^{(\nu)}_{j,k,l,h}(X) \rightarrow 0$ as $\nu \rightarrow \infty$.  Therefore, by the definition of blow-up, the fact that the convergence $E_{\nu}^{-1}\llbracket v_{j, k, l, h}^{(\nu)}(\cdot)\rrbracket \to \llbracket w_{j, k, l, h}(\cdot)\rrbracket$ is uniform in $B_{1-\t}(0) \cap \{r > \tau\}$ and the arbitrariness of $\t$, we have that 
\begin{equation}\label{homogrep1_eqn15}
c^{\star}w^{\star} = (w_{j,k} - D\varphi^{(0)}_j(X) \cdot Ay) 
\end{equation}
in $B_{1}(0) \setminus \{0\} \times {\mathbb R}^{n-2}$ for some constant $c^{\star} \in [0, C]$ with $C$ as in \eqref{homogrep1_eqn12}; in fact $c^{\star}= \lim_{\nu \to \infty} \, E_{\nu}^{-1}E^{\star}_{\nu}$ where, by \eqref{homogrep1_eqn12}, the limit exists after passing to a subsequence. (The notation in \eqref{homogrep1_eqn15} means the following: if $\varphi^{(\infty)} = (m_{j, k}, \varphi^{(\infty)}_{j, k})_{1 \leq j \leq J, 1\leq k \leq p_{j}}) \in {\mathfrak D}$ is determined by the sequence 
$(\varphi^{(\nu)})$ in the manner described in Remark~\ref{blow-up-conditions} (so that $w, w^{\star} \in {\mathfrak B}(\varphi^{(\infty)}$)), and if we write $w^{\star} = (w^{\star}_{j, k})$, then for  $j \in \{1, \ldots, J\}$, $k \in \{1, \ldots, p_{j}\},$ 
$$c^{\star}w^{\star}_{j, k}(X) = \sum_{l=1}^{q_{0}}\sum_{h=1}^{m_{j, k}}\llbracket w_{j, k, l, h}(X) - D\varphi^{(0)}_{j, l}(X) \cdot Ay\rrbracket$$
if $\varphi^{(0)}_{j} \neq 0,$ where  $\varphi^{(0)}_{j}(X) = \sum_{l=1}^{q_{0}} \llbracket \varphi^{(0)}_{j, l}(X)\rrbracket$ and  $w_{j, k, l, h}$ is as in \eqref{w_localized}, and 
$$c^{\star}w^{\star}_{j, k}= w_{j, k}$$
if $\varphi^{(0)}_{j} = 0$.)

If $c^{\star} = 0$ in \eqref{homogrep1_eqn15} there is nothing further to prove, so assume that $c^{\star} >0.$ It follows from \eqref{homogrep1_eqn15} that $w^{\star}$ is homogeneous of degree $\alpha$. Furthermore, since $w^{\star} \in {\mathfrak B}$, using again \eqref{homogrep1_eqn15} together with Lemma~\ref{blowup_est_lemma}(c) (asserting uniqueness of $\lambda$ associated with $w$) and \eqref{linear_lambda}, we see  
that \eqref{blowup_est4} holds with $w = w^{\star}$ and $\lambda = 0$.  Thus the conclusion of Theorem~\ref{homogrep_lemma} follows from its validity in the special case $\lambda = 0$ (established above) and \eqref{homogrep1_eqn15}. 
\end{proof}

\section{Asymptotic decay for the blow-ups} \label{sec:asymptotics_sec}

The main result of this section is a decay estimate for the blow-ups, Theorem~\ref{blowupdecay_lemma}, which will be the basis for the proof (given in Section~\ref{sec:main_lemma_sec}) of the main excess improvement result Lemma~\ref{main_lemma}.  The general idea of the proof of Theorem~\ref{blowupdecay_lemma} is similar to that of the corresponding result in Section 4 of~\cite{Sim93}, and is based on a hole-filling argument that relies on the estimate \eqref{blowup_est1} and the key estimate in Lemma~\ref{holefilling_lemma} below. In contrast to \cite{Sim93} however, the proof of Lemma~\ref{holefilling_lemma} is complicated by  issues arising from higher multiplicity; unlike the corresponding argument in~\cite{Sim93}, Lemma~\ref{holefilling_lemma} requires a two step argument where the first step is Lemma~\ref{holefilling0_lemma} below giving the same conclusions as Lemma~\ref{holefilling_lemma} subject to additional hypotheses.    

\begin{lemma} \label{holefilling_lemma}
Let $w \in \mathfrak{B}$ and $\psi$ be a projection (as in Definition~\ref{homogprojection_defn}) of $w$ onto $\mathfrak{L}$ in $B_{1/2}(0)$.  Then 
\begin{equation} \label{holefilling_est}
	\int_{B_{1/2}(0)} \sum_{j=1}^J \sum_{k=1}^{p_j} \sum_{l=1}^{q_{j,k}} \mathcal{G}(w_{j,k,l},\psi_{j,k,l})^2  
	\leq C \int_{B_{1/2}(0) \setminus B_{1/8}(0)} \sum_{j=1}^J \sum_{k=1}^{p_j} \sum_{l=1}^{q_{j,k}} \left| \frac{\partial (w_{j,k,l}/R^{\alpha})}{\partial R} \right|^2 
\end{equation}
for some constant $C  = C(n, m, q, \alpha, \varphi^{(0)}) \in (0,\infty),$  where $w_{j,k,l}$ is as in \eqref{w_localized} and $\psi_{j,k,l}$ is as in \eqref{psi_localized}. 
\end{lemma}

For Lemma~\ref{holefilling0_lemma} we need the following additional notation:




\begin{definition} 
For each $\varphi^{(\infty)} \in {\mathfrak D}$, $\mathfrak{L}_0(\varphi^{(\infty)})$ is the set of all $\psi \in \mathfrak{L}(\varphi^{(\infty)})$ such that $\psi$ is as in Definition~\ref{homogclass_defn} with $b = 0$.
\end{definition}

\begin{definition} \label{L0s defn}
For each $\psi  = (\psi_{j, k})\in \mathfrak{L}_0(\varphi^{(\infty)})$, we call the functions $0$ and $\op{Re}(a_{j,k,l} r^{\alpha} e^{i\alpha \theta})$ as in Definition~\ref{homogclass_defn} the components of $\psi_{j,k}$.  

Let $\varphi^{(\infty)} = (m_{j, k}, \varphi^{(\infty)}_{j, k}) \in {\mathfrak D}$ and let $s_0$ be the number of (not necessarily distinct) nonzero functions $\varphi^{(\infty)}_{j,k}$ for $j \in \{1, \ldots, J\}$ and $k \in \{1, \ldots, p_{j}\}$.  For each $s \in \{s_0,s_0+1,\ldots,\lceil q/q_0 \rceil\}$, we let $\mathfrak{L}_{0,s}(\varphi^{(\infty)})$ denote the set of all $\psi  = (\psi_{j, k}) \in \mathfrak{L}_0(\varphi^{(\infty)})$ such that $s = \sum_{j=1}^J \sum_{k=1}^{p_{j}} s_{j,k}$ where: \begin{enumerate}
	\item[(i)] if $\varphi^{(\infty)}_{j,k}$ is nonzero, $s_{j,k}$ is the number of distinct components of $\psi_{j,k}$ (not counting multiplicity);
	\item[(ii)] if $\varphi^{(\infty)}_{j,k}$ is identically zero, $s_{j,k}$ is the number of distinct nonzero components of $\psi_{j,k}$ (not counting multiplicity). 
\end{enumerate}
\end{definition}

\begin{remark} \label{L0 comp rmk}  
Let $\varphi^{(\infty)} = (m_{j,k}, \varphi^{(\infty)}_{j,k}) \in \mathfrak{D}$, $s_{0}$ be as in Definition~\ref{L0s defn}, $s \in \{s_0,s_0+1,\ldots,\lceil q/q_0 \rceil\}$, and $\psi \in \mathfrak{L}_{0,s}(\varphi^{(\infty)}).$  For each $\nu = 1,2,3,\ldots$, let $u^{(\nu)}$, $\varphi^{(\nu)}$, and $E_{\nu}$ be as in Definition~\ref{blowupclass_defn}.  Assume $\varphi^{(\nu)}$ satisfies the requirements of Remark~\ref{blow-up-conditions}(A)(B) taken with $\varphi^{(\infty)} \in \mathfrak{D}$ for which $\psi \in \mathfrak{L}(\varphi^{(\infty)})$.  Let $\widetilde{\varphi}^{(\nu)}$ be as in Lemma~\ref{blowup2L_lemma} corresponding to $\varphi^{(\nu)}$, $E_{\nu}$ and $\psi$.  By Remark~\ref{blow-up-conditions}(B), $\varphi^{(\infty)}_{j,k} \equiv 0$ if and only if $\varphi^{(\nu)}_{j,k} \equiv 0$ for all $\nu$ and thus $\varphi^{(\nu)}$ has precisely $s_0$ nonzero components, where $s_0$ is as in Definition~\ref{L0s defn}.  Moreover, we constructed $\widetilde{\varphi}^{(\nu)}$ in Lemma~\ref{blowup2L_lemma} so that $\widetilde{\varphi}^{(\nu)} = \sum_{j=1}^J \sum_{k=1}^{p_j} \widetilde{\varphi}^{(\nu)}_{j,k}$ where $\widetilde{\varphi}^{(\nu)}_{j,k}$ is given by \eqref{blowup2L_eqn2} with $B = 0$ if $\varphi^{(\infty)}_{j,k}$ is nonzero and $\widetilde{\varphi}^{(\nu)}_{j,k} = E_{\nu} \psi_{j,k}(X,0)$ if $\varphi^{(\infty)}_{j,k} \equiv 0$.  It follows, using Remark~\ref{graphical_rmk}(b), that $\widetilde{\varphi}^{(\nu)}_{j,k}$ has a zero component if and only if $\varphi^{(\infty)}_{j,k} \equiv 0$ and $\psi_{j,k}$ has a zero component.  It follows that each $\widetilde{\varphi}^{(\nu)}_{j,k}$ has precisely $s_{j,k}$ distinct nonzero components, where $s_{j,k}$ is as in Definition~\ref{L0s defn}.  Therefore, recalling Remark~\ref{graphical_rmk}(b), $\widetilde{\varphi}^{(\nu)}$ has precisely $s = \sum_{j=1}^J \sum_{k=1}^{p_{j}} s_{j,k}$ distinct nonzero components.
\end{remark}

\begin{lemma} \label{holefilling0_lemma}
For every $M \in [1,\infty)$ there exists $\overline{\beta} \in (0,1)$ and $\overline{C} \in (0,\infty)$ depending only on $n$, $m$, $q$, $\alpha$, $\varphi^{(0)}$ and $M$ such that the following holds true:  Suppose that $\varphi^{(\infty)}  = (m_{j, k}, \varphi^{(\infty)}_{j, k})\in {\mathfrak D}$, $w \in \mathfrak{B}(\varphi^{(\infty)})$ and $\psi \in \mathfrak{L}_{0,s}(\varphi^{(\infty)})$ (where $s \in \{s_{0}, \ldots, s_{0}+ 1, \ldots, \lceil q/q_0 \rceil\}$ with $s_{0}$ corresponding to $\varphi^{(\infty)}$ as in Definition~\ref{L0s defn}) are such that  
\begin{equation} \label{holefilling0_hyp1}
	\int_{B_{1/2}(0)} \sum_{j=1}^J \sum_{k=1}^{p_j} \sum_{l=1}^{q_{j,k}} \mathcal{G}(w_{j,k,l},\psi_{j,k,l})^2 
		\leq M \inf_{\psi' \in \mathfrak{L}(\varphi^{(\infty)})} \int_{B_{1/2}(0)} \sum_{j=1}^J \sum_{k=1}^{p_j} \sum_{l=1}^{q_{j,k}} \mathcal{G}(w_{j,k,l},\psi'_{j,k,l})^2,
\end{equation}
and that either (i) $s = s_0$, or (ii) $s > s_0$ and 
\begin{equation} \label{holefilling0_hyp2}
	\int_{B_{1/2}(0)} \sum_{j=1}^J \sum_{k=1}^{p_j} \sum_{l=1}^{q_{j,k}} \mathcal{G}(w_{j,k,l},\psi_{j,k,l})^2 
		\leq \overline{\beta} \inf_{\psi' \in \bigcup_{s'=s_0}^{s-1} \mathfrak{L}_{0,s'}(\varphi^{(\infty)})} 
			\int_{B_{1/2}(0)} \sum_{j=1}^J \sum_{k=1}^{p_j} \sum_{l=1}^{q_{j,k}} \mathcal{G}(w_{j,k,l},\psi'_{j,k,l})^2,  
\end{equation}
where $w_{j,k,l}$ is as in \eqref{w_localized}, $\psi_{j,k,l}$ is as in \eqref{psi_localized} and $\psi'_{j,k,l}$ is as in \eqref{psi_localized} with $\psi'$ in place of $\psi$.  Then \eqref{holefilling_est} holds true with $C = \overline{C}$. 
\end{lemma}

\begin{proof}
Observe that it suffices to fix $s \in \{s_0,s_0+1,\ldots,\lceil q/q_0 \rceil\}$ and prove Lemma~\ref{holefilling0_lemma} with $\overline{\beta}$ and $\overline{C}$ depending on $s$.  Suppose that for some fixed $M \in [1,\infty)$ and $s \in \{s_0,s_0+1,\ldots,\lceil q/q_0 \rceil\}$ and each $\nu =1, 2, 3, \ldots,$ there exist $\overline{\beta}_{\nu}>0$, $\varphi^{(\nu,\infty)} \in {\mathfrak D}$, $w^{(\nu)} = (w^{(\nu)}_{j,k}) \in \mathfrak{B}(\varphi^{(\nu,\infty)})$, $\psi^{(\nu)} = (\psi^{(\nu)}_{j,k}) \in \mathfrak{L}_{0,s}(\varphi^{(\nu,\infty)})$ such that $\overline{\beta}_{\nu} \downarrow 0$ and \eqref{holefilling0_hyp1}, \eqref{holefilling0_hyp2} hold true with $\varphi^{(\nu,\infty)}$, $w^{(\nu)}$, $\psi^{(\nu)}$, $\overline{\beta}_{\nu}$ in place of $\varphi^{(\infty)}$, $w$, $\psi$, $\overline{\beta}$ respectively, and yet  
\begin{equation} \label{holefilling0_eqn1}
	\int_{B_{1/2}(0) \setminus B_{1/8}(0)} \sum_{j=1}^J \sum_{k=1}^{p_j} \sum_{l=1}^{q_{j,k}} 
		\left| \frac{\partial (w^{(\nu)}_{j,k,l}/R^{\alpha})}{\partial R} \right|^2 
	\leq \frac{1}{\nu} \int_{B_{1/2}(0)} \sum_{j=1}^J \sum_{k=1}^{p_j} \sum_{l=1}^{q_{j,k}} \mathcal{G}(w^{(\nu)}_{j,k,l},\psi^{(\nu)}_{j,k,l})^2 
\end{equation}
for all $\nu$, where $w^{(\nu)}_{j,k,l}$ is as in \eqref{w_localized} with $w = w^{(\nu)}$ and $\psi_{j,k,l}^{(\nu)}$ is as in \eqref{psi_localized} with $\psi = \psi^{(\nu)}$.  After passing to a subsequence, we can take $m_{j,k}$ (corresponding to $\varphi^{(\nu,\infty)}$) to be independent of $\nu$.  Moreover, we may assume for each $j$ and $k$ that either $\varphi^{(\nu,\infty)}_{j,k}$ is nonzero for all $\nu$ or $\varphi^{(\nu,\infty)}_{j,k}$ is identically zero for all $\nu$. For each $\nu$, let 
\begin{equation*}
	F_{\nu} = \left( \int_{B_{1/2}(0)} \sum_{j=1}^J \sum_{k=1}^{p_j} \sum_{l=1}^{q_{j,k}} \mathcal{G}(w^{(\nu)}_{j,k,l},\psi^{(\nu)}_{j,k,l})^2 \right)^{1/2}. 
\end{equation*}

Since $w^{(\nu)} \in \mathfrak{B}$, there exist a sequence of average-free locally energy minimizing $q$-valued functions $(u^{(\nu,\kappa)})_{\kappa = 1}^{\infty}$ and cylindrical functions $\varphi^{(\nu,\kappa)} \in \Phi_{1/\kappa,p}(\varphi^{(0)})$ as in Definition~\ref{blowupclass_defn} such that $w^{(\nu)}$ is the blow-up of $u^{(\nu,\kappa)}$ relative to $\varphi^{(\nu,\kappa)}$ by the excess $E_{\nu,\kappa} = \left( \int_{B_1(0)} \mathcal{G}(u^{(\nu,\kappa)},\varphi^{(\nu,\kappa)})^2 \right)^{1/2}.$   By Remark \ref{blow-up-conditions}(A)(B) and the fact that $\varphi^{(\nu,\kappa)} \in \Phi_{1/\kappa,p}(\varphi^{(0)})$, we may assume that $\varphi^{(\nu,\kappa)}$ have components $\varphi^{(\nu,\kappa)}_{j,k}$ with multiplicity $m_{j,k}$ and for each $j \in \{1,\ldots,J\}$ and $k \in \{1,\ldots,p_j\}$ one of the following hold true: 
\begin{enumerate}
	\item[(a)]  $\varphi^{(0)}_j$ is nonzero, $\varphi^{(\nu,\infty)}_{j,k}(X) = \varphi^{(0)}_j(X) = \op{Re}(c^{(0)}_j (x_1+ix_2)^{\alpha})$ for all $\nu$, and $\varphi^{(\nu,\kappa)}_{j,k}(X) = \op{Re}(c^{(\nu,\kappa)}_{j,k} (x_1+ix_2)^{\alpha})$ for each $\nu,\kappa$ and some $c^{(\nu,\kappa)}_{j,k} \in \mathbb{C}^m \setminus \{0\}$ with $|c^{(\nu,\kappa)}_{j,k} - c^{(0)}_j| < C(n,q,\alpha) \,\nu^{-1}$; 
	\item[(b)] $\varphi^{(0)}_j(X) = \varphi^{(\nu,\infty)}_{j,k}(X) = \varphi^{(\nu,\kappa)}_{j,k}(X) = 0$ for all $\nu,\kappa$; 
	\item[(c)] $\varphi^{(0)}_j$ is identically zero but $\varphi^{(\nu,\infty)}_{j,k}(X) = \op{Re}(c^{(\nu,\infty)}_{j,k} (x_1+ix_2)^{\alpha})$ for some $c^{(\nu,\infty)}_{j,k} \in \mathbb{C}^m \setminus \{0\}$ and $\varphi^{(\nu,\kappa)}_{j,k}(X) = \op{Re}(c^{(\nu,\kappa)}_{j,k} (x_1+ix_2)^{\alpha})$ for all $\nu,\kappa$ and some $c^{(\nu,\kappa)}_{j,k} \in \mathbb{C}^m \setminus \{0\}$ with $\left| c^{(\nu,\kappa)}_{j,k}/|c^{(\nu,\kappa)}_{j,k}| - c^{(\nu,\infty)}_j \right| < 1/\nu$.
\end{enumerate}
Let $\tau_{\nu,\kappa} \downarrow 0$ as $\kappa \rightarrow \infty$ and $v^{(\nu,\kappa)}_{j,k} : {\rm graph}\,\varphi^{(\nu,\kappa)}_{j,k} |_{B_{1/2}(0) \cap \{r > \tau_{\nu,\kappa}\}} \rightarrow \mathcal{A}_{m_{j,k}}(\mathbb{R}^m)$ be as in Corollary \ref{graphical_cor} with $\gamma = 1/2$, $\tau = \tau_{\nu,\kappa}$, $u = u^{(\nu,\kappa)}$, and $\varphi = \varphi^{(\nu,\kappa)}$.  Let $v^{(\nu,\kappa)}_{j,k,l}$ be as in \eqref{v_localized} with $\varphi^{(\nu,\kappa)}$ and $v^{(\nu,\kappa)}$ in place of $\varphi$ and $v$.  For each $\nu$ and $\kappa$, let $\widetilde{\varphi}^{(\nu,\kappa)}$ be as constructed in Lemma~\ref{blowup2L_lemma} with $\varphi^{(\nu,\kappa)}$, $\psi^{(\nu)}$, and $E_{\nu,\kappa}$ in place of $\varphi^{(\nu)}$, $\psi$, and $E_{\nu}$.  Observe that by the definition of blow-up, 
\begin{equation} \label{holefilling0_eqn2}
	\lim_{\kappa \rightarrow \infty} \sum_{j=1}^J \sum_{k=1}^{p_j} \sum_{l=1}^{q_{j,k}} 
		\mathcal{G}\left( \frac{v^{(\nu,\kappa)}_{j,k,l}(X)}{E_{\nu,\kappa}}, w^{(\nu)}_{j,k,l}(X) \right) = 0
\end{equation}
uniformly on $B_{1/2}(0) \cap \{r > \tau\}$ for each $\tau > 0$.  Moreover, by Lemma~\ref{blowup_norm_conv_lemma}, 
\begin{equation} \label{holefilling0_eqn3}
	F_{\nu}^2 = \int_{B_{1/2}(0)} \sum_{j=1}^J \sum_{k=1}^{p_j} \sum_{l=1}^{q_{j,k}} \mathcal{G}(w^{(\nu)}_{j,k,l},\psi^{(\nu)}_{j,k,l})^2 
		= \lim_{\kappa \rightarrow \infty} \frac{1}{E_{\nu,\kappa}^2} \int_{B_{1/2}(0)} \mathcal{G}(u^{(\nu,\kappa)},\widetilde{\varphi}^{(\nu,\kappa)})^2. 
\end{equation}

Select diagonal sequences $u^{(\nu,\kappa(\nu))}$, $\varphi^{(\nu,\kappa(\nu))}$, $\widetilde{\varphi}^{(\nu,\kappa(\nu))}$ and $E_{(\nu,\kappa(\nu))}$ by choosing $\kappa = \kappa(\nu)$ large enough such that, with $u^{(\nu,\kappa(\nu))}$, $\varphi^{(\nu,\kappa(\nu))}$ in place of $u^{(\nu)}$, $\varphi^{(\nu)}$, conditions (a)--(e) of Definition~\ref{blowupclass_defn} are satisfied for $\delta_{\nu}, \varepsilon_{\nu}, \beta_{\nu} \downarrow 0$ and, in view of  \eqref{holefilling0_eqn2} and \eqref{holefilling0_eqn3}, also such that
\begin{gather}
	\sup_{B_{1/2}(0) \cap \{r > 1/\nu\}} \sum_{j=1}^J \sum_{k=1}^{p_j} \sum_{l=1}^{q_{j,k}} 
		\mathcal{G}\left( \frac{v^{(\nu,\kappa(\nu))}_{j,k,l}(X)}{E_{\nu,\kappa(\nu)}}, w^{(\nu)}_{j,k,l}(X) \right) 
		< \frac{1}{\nu} F_{\nu} \label{holefilling0_eqn4}\;\;\; {\rm and} \\
	\lim_{\nu \rightarrow \infty} \frac{1}{E_{\nu,\kappa(\nu)}^2 F_{\nu}^2} \int_{B_{1/2}(0)} 
		\mathcal{G}(u^{(\nu,\kappa(\nu))},\widetilde{\varphi}^{(\nu,\kappa(\nu))}) = 1. \label{holefilling0_eqn5}
\end{gather}
Set $u^{(\nu)} = u^{(\nu,\kappa(\nu))}$, $\varphi^{(\nu)} = \varphi^{(\nu,\kappa(\nu))}$, $\widetilde{\varphi}^{(\nu)} = \widetilde{\varphi}^{(\nu,\kappa(\nu))}$, and $E_{\nu} = E_{(\nu,\kappa(\nu))}$. 

Observe that by taking $\psi' = 0$ in \eqref{holefilling0_hyp1} and applying \eqref{blowup_norm_conv_eqn2} gives us  
\begin{equation*}
	F_{\nu}^2 = \int_{B_{1/2}(0)} \sum_{j=1}^J \sum_{k=1}^{p_j} \sum_{l=1}^{q_{j,k}} \mathcal{G}(w_{j,k,l},\psi_{j,k,l})^2 
		\leq M \int_{B_{1/2}(0)} \sum_{j=1}^J \sum_{k=1}^{p_j} \sum_{l=1}^{q_{j,k}} |w_{j,k,l}|^2 \leq M 
\end{equation*}
and thus by \eqref{holefilling0_eqn5}
\begin{equation} \label{holefilling0_eqn6}
	\int_{B_{1/2}(0)} \mathcal{G}(u^{(\nu)},\widetilde{\varphi}^{(\nu)})^2 \leq 2 E_{\nu}^2 F_{\nu}^2 \leq 2 M E_{\nu}^2. 
\end{equation}
for all sufficiently large $\nu$. 

By the triangle inequality, \eqref{blowupclass_eqn2} and \eqref{holefilling0_eqn6}, 
\begin{equation*}
	\int_{B_{1/2}(0)} \mathcal{G}(\widetilde{\varphi}^{(\nu)},\varphi^{(0)})^2 
	\leq 2 \int_{B_{1/2}(0)} \mathcal{G}(u^{(\nu)},\widetilde{\varphi}^{(\nu)})^2 + 2 \int_{B_{1/2}(0)} \mathcal{G}(u^{(\nu)},\varphi^{(0)})^2 
	\leq 2 (2M+1) \varepsilon_{\nu}^2.
\end{equation*}
Hence by Lemma~\ref{blowup2L_lemma}, Remark~\ref{L0 comp rmk} and $\psi^{(\nu)} \in \mathfrak{L}_{0,s}(\varphi^{(\nu,\infty)})$, we have that $\widetilde{\varphi}^{(\nu)} \in \Phi_{\sqrt{2(2M+1)} \varepsilon_{\nu},s}(\varphi^{(0)})$.  We claim that either (i) $s = p_0$ or (ii) $s > p_0$ and 
\begin{equation} \label{holefilling0_eqn7}
	\int_{B_1(0)} \mathcal{G}(u^{(\nu)},\widetilde{\varphi}^{(\nu)})^2 
		\leq \max\{2 M \beta_{\nu},64 \overline{\beta}_{\nu}\} \inf_{\varphi' \in \bigcup_{p'=1}^{p-1} \Phi_{3 \sqrt{2(2M+1)} \varepsilon_{\nu},p'}(\varphi^{(0)})} 
		\int_{B_1(0)} \mathcal{G}(u^{(\nu)},\varphi')^2. 
\end{equation}
To verify this, suppose that $s > p_{0}.$ It suffices to separately consider the cases $p_0 < s_0 = s$ and $p_0 \leq s_0 < s$ and show that \eqref{holefilling0_eqn7} holds true in each case.  If $p_0 < s_0 = s$, \eqref{holefilling0_eqn6} and \eqref{blowupclass_eqn3} give us \eqref{holefilling0_eqn7}.  On the other hand, if $p_0 \leq s_0 < s$, by the triangle inequality, \eqref{holefilling0_hyp2} implies that 
\begin{equation} \label{holefilling0_eqn8}
	\int_{B_{1/2}(0)} \sum_{j=1}^J \sum_{k=1}^{p_j} \sum_{l=1}^{q_{j,k}} \mathcal{G}(w^{(\nu)}_{j,k,l},\psi^{(\nu)}_{j,k,l})^2 
		\leq 4 \overline{\beta}_{\nu} \inf_{\psi' \in \bigcup_{s'=s_0}^{s-1} \mathfrak{L}_{0,s'}} 
			\int_{B_{1/2}(0)} \sum_{j=1}^J \sum_{k=1}^{p_j} \sum_{l=1}^{q_{j,k}} \mathcal{G}(\psi^{(\nu)}_{j,k,l},\psi'_{j,k,l})^2 
\end{equation}
for $\nu$ sufficiently large.  By \eqref{holefilling0_eqn3} and \eqref{holefilling0_eqn8}, 
\begin{align} \label{holefilling0_eqn9}
	\int_{B_{1/2}(0)} \mathcal{G}(u^{(\nu)},\widetilde{\varphi}^{(\nu)})^2 
	&\leq 2 E_{\nu}^2 \int_{B_{1/2}(0)} \sum_{j=1}^J \sum_{k=1}^{p_j} \sum_{l=1}^{q_{j,k}} \mathcal{G}(w^{(\nu)}_{j,k,l},\psi^{(\nu)}_{j,k,l})^2  
	\\&\leq 8 \overline{\beta}_{\nu} E_{\nu}^2 \inf_{\psi' \in \bigcup_{s'=s_0}^{s-1} \mathfrak{L}_{0,s'}} 
		\int_{B_{1/2}(0)} \sum_{j=1}^J \sum_{k=1}^{p_j} \sum_{l=1}^{q_{j,k}} \mathcal{G}(\psi^{(\nu)}_{j,k,l},\psi'_{j,k,l})^2 \nonumber 
\end{align}
for $\nu$ sufficiently large.  Let $\widehat{\varphi}^{(\nu)} \in \bigcup_{p'=p_0}^{s-1} \Phi_{3 \sqrt{2(2M+1)} \varepsilon_{\nu},p'}(\varphi^{(0)})$ be such that 
\begin{equation} \label{holefilling0_eqn10}
	\int_{B_{1/2}(0)} \mathcal{G}(\widetilde{\varphi}^{(\nu)},\widehat{\varphi}^{(\nu)})^2 
		< 2 \inf_{\varphi' \in \bigcup_{p'=p_0}^{s-1} \Phi_{3 \sqrt{2(2M+1)} \varepsilon_{\nu},p'}(\varphi^{(0)})} 
		\int_{B_{1/2}(0)} \mathcal{G}(\widetilde{\varphi}^{(\nu)},\varphi')^2. 
\end{equation}
By Remark~\ref{L0 comp rmk} we have $\varphi^{(\nu)} \in \Phi_{\varepsilon_{\nu},s_0}(\varphi^{(0)})$ and by assumption $s_0 < s$.  Thus we may take $\varphi' = \varphi^{(\nu)}$ in \eqref{holefilling0_eqn10} to get 
\begin{equation} \label{holefilling0_eqn11}
	\int_{B_{1/2}(0)} \mathcal{G}(\widetilde{\varphi}^{(\nu)},\widehat{\varphi}^{(\nu)})^2 
		< 2 \int_{B_{1/2}(0)} \mathcal{G}(\widetilde{\varphi}^{(\nu)},\varphi^{(\nu)})^2. 
\end{equation}
By the triangle inequality, \eqref{holefilling0_eqn11} and \eqref{holefilling0_eqn6}, 
\begin{align} \label{holefilling0_eqn12}
	\int_{B_{1/2}(0)} \mathcal{G}(\widehat{\varphi}^{(\nu)},\varphi^{(\nu)})^2 
	&\leq 2 \int_{B_{1/2}(0)} \mathcal{G}(\widehat{\varphi}^{(\nu)},\widetilde{\varphi}^{(\nu)})^2  
		+ 2 \int_{B_{1/2}(0)} \mathcal{G}(\widetilde{\varphi}^{(\nu)},\varphi^{(\nu)})^2 
	\\&\leq 6 \int_{B_{1/2}(0)} \mathcal{G}(\widetilde{\varphi}^{(\nu)},\varphi^{(\nu)})^2  \nonumber
	\\&\leq 12 \int_{B_{1/2}(0)} \mathcal{G}(u^{(\nu)},\widetilde{\varphi}^{(\nu)})^2  
		+ 12 \int_{B_{1/2}(0)} \mathcal{G}(u^{(\nu)},\varphi^{(\nu)})^2 \nonumber
	\\&\leq 12 (2M+1) E_{\nu}^2. \nonumber 
\end{align}
By \eqref{blowupclass_eqn2}, \eqref{blowupclass_eqn3} and Corollary \ref{graphical_cor}(a) we know that the separation between the distinct values of $\varphi^{(\nu)}(x,y)$ is $\geq C \beta_{\nu}^{-1} |x|^{\alpha} E_{\nu}$ for all $(x,y) \in B_{1/2}(0)$, where $C = C(m,n,q,\alpha,\varphi^{(0)}) \in (0,\infty)$ is a constant, and thus by \eqref{holefilling0_eqn12} there exists $\widehat{\psi}^{(\nu)} = (\widehat{\psi}^{(\nu)}_{j,k}) \in \bigcup_{s'=s_0}^{s-1} \mathfrak{L}_{0,s'}(\varphi^{(\infty)})$ such that 
\begin{gather}
	\widehat{\varphi}^{(\nu)}(X) = \sum_{j=1}^{J} \sum_{k=1}^{p_j} \sum_{l=1}^{q_{j,k}} \sum_{h=1}^{m_{j,k}}
		\llbracket \varphi^{(\nu)}_{j,k,l}(X) + E_{\nu} \,\widehat{\psi}^{(\nu)}_{j,k,l,h}(X) \rrbracket \text{ for all } X \in B_{1/2}(0), \nonumber \\
	\label{holefilling0_eqn13} \int_{B_{1/2}(0)} \mathcal{G}(\widetilde{\varphi}^{(\nu)},\widehat{\varphi}^{(\nu)})^2
		= E_{\nu}^2 \int_{B_{1/2}(0)} \sum_{j=1}^J \sum_{k=1}^{p_j} \sum_{l=1}^{q_{j,k}} \mathcal{G}(\psi^{(\nu)}_{j,k,l},\widehat{\psi}^{(\nu)}_{j,k,l})^2, 
\end{gather}
where $\varphi^{(\nu)}_{j,k,l}$ is as in \eqref{varphi_localized} with $\varphi^{(\nu)}$ in place of $\varphi$ and $\widehat{\psi}^{(\nu)}_{j,k,l,h}$ is as in \eqref{psi_localized} with $\widehat{\psi}^{(\nu)}$ in place of $\psi$.  Hence by \eqref{holefilling0_eqn9} taking $\psi' = \widehat{\psi}^{(\nu)}$, \eqref{holefilling0_eqn13} and \eqref{holefilling0_eqn10},  
\begin{equation} \label{holefilling0_eqn14}
	\int_{B_{1/2}(0)} \mathcal{G}(u^{(\nu)},\widetilde{\varphi}^{(\nu)})^2 
	\leq 16 \overline{\beta}_{\nu} \inf_{\varphi' \in \bigcup_{p'=p_0}^{s-1} \Phi_{3 \sqrt{2(2M+1)} \varepsilon_{\nu},p'}(\varphi^{(0)})} 
		\int_{B_{1/2}(0)} \mathcal{G}(\widetilde{\varphi}^{(\nu)},\varphi')^2  
\end{equation}
for $\nu$ sufficiently large.  By the triangle inequality, \eqref{holefilling0_eqn14} implies that \eqref{holefilling0_eqn7} holds true. 

Therefore, since $u^{(\nu)}$ and $\varphi^{(\nu)}$ satisfy conditions (a)--(e) of Definition~\ref{blowupclass_defn}, $\widetilde{\varphi}^{(\nu)} \in \Phi_{\sqrt{2(2M+1)} \varepsilon_{\nu},s}(\varphi^{(0)})$, and \eqref{holefilling0_eqn5} and \eqref{holefilling0_eqn7} hold true, we can let $\widetilde{w}$ be the blow-up of $u^{(\nu)}$ relative to $\widetilde{\varphi}^{(\nu)}$ and excess $E_{\nu} F_{\nu}$ in $B_{1/2}(0)$.  

To understand $\widetilde{w}$ more concretely, let us consider $j,k$ such that $\varphi^{(0)}_j$ is nonzero.  Recall that $\varphi^{(\nu,\infty)}_{j,k}(X) = \varphi^{(0)}_j(X) = \op{Re}(c^{(0)}_j (x_1+ix_2)^{\alpha})$.  Recalling Definition~\ref{homogclass_defn}, we can express $\psi^{(\nu)}_{j,k}$ as 
\begin{equation*} 
	\psi^{(\nu)}_{j, k}(re^{i\theta},y,\op{Re}(c^{(0)}_j r^{\alpha}e^{i\alpha \theta})) 
	= \sum_{l=1}^{L_{j,k}} m_{j,k,l} \llbracket \op{Re}(a^{(\nu)}_{j,k,l} r^{\alpha} e^{i\alpha \theta}) \rrbracket 
\end{equation*}
for a positive integer $L_{j, k}$, distinct $a^{(\nu)}_{j,k,l} \in \mathbb{C}^m$ and positive integer multiplicities $m_{j,k,l}$ such that $\sum_{l=1}^{L_{j,k}} m_{j,k,l} = m_{j,k}$ (with $L_{j, k}$, $m_{j, k, l}$ taken to be independent of $\nu$ by passing to a subsequence).  Thus by the construction of $\widetilde{\varphi}^{(\nu)}$ in Lemma~\ref{blowup2L_lemma}, $\widetilde{\varphi}^{(\nu)}$ has distinct components 
\begin{equation} \label{holefilling0_eqn15}
	\widetilde{\varphi}^{(\nu)}_{j,k,l}(X) = \llbracket \op{Re}((c^{(\nu)}_{j,k} + E_{\nu} a^{(\nu)}_{j,k,l}) (x_1+ix_2)^{\alpha}) \rrbracket 
\end{equation}
with multiplicity $m_{j,k,l}$  near $\varphi^{(\nu)}_{j,k}$ (not to be confused with the notation of \eqref{varphi_localized}), where $\varphi^{(\nu)}_{j,k}(X) = \op{Re}(c^{(\nu)}_{j,k} (x_1+ix_2)^{\alpha})$ for $c^{(\nu)}_{j,k} \in \mathbb{C}^m$ with $|c^{(\nu)}_{j,k} - c^{(0)}_j| \leq C(n,q,\alpha) \varepsilon_{\nu}$.  Then we can use Corollary \ref{graphical_cor} to express $u^{(\nu)}$ as the graph of $\widetilde{v}^{(\nu)}_{j,k,l} : {\rm graph}\, \widetilde{\varphi}^{(\nu)}_{j,k,l} |_{\Omega_{\nu}} \rightarrow \mathcal{A}_{m_{j,k,l}}(\mathbb{R}^m)$, where $\Omega_{\nu} = B_{1/2-\tau_{\nu}}(0) \cap \{r > \tau_{\nu}\}$ for $\tau_{\nu} \downarrow 0$, and then use the blow-up procedure of Subsection \ref{blow-up-procedure} to take a limit of $\widetilde{v}^{(\nu)}_{j,k,l}/(E_{\nu} F_{\nu})$.  Let us denote the resulting limit function as $\widetilde{w}_{j,k,l} : {\rm graph}\, \widetilde{\varphi}^{(\infty)}_{j,k,l} |_{B_{1/2}(0)} \rightarrow \mathcal{A}_{m_{j,k,l}}(\mathbb{R}^m)$ where $\widetilde{\varphi}^{(\infty)}_{j,k,l} = \varphi^{(0)}_j$.  Notice that by \eqref{holefilling0_eqn15}, 
\begin{equation*}
	v^{(\nu)}_{j,k}(re^{i\theta},y, \op{Re}(c^{(\nu)}_{j,k} r^{\alpha} e^{i\alpha \theta})) 
	= \sum_{l=1}^{L_{j,k}} \sum_{h=1}^{m_{j,k,l}} \llbracket \op{Re}(E_{\nu} a^{(\nu)}_{j,k,l} r^{\alpha} e^{i\alpha \theta}) 
		+ \widetilde{v}^{(\nu)}_{j,k,l,h}(re^{i\theta},y, \op{Re}(\widetilde{c}^{(\nu)}_{j,k,l} r^{\alpha} e^{i\alpha \theta})) \rrbracket 
\end{equation*}
whenever $(re^{i\theta},y) \in \Omega_{\nu}$, where $\widetilde{c}^{(\nu)}_{j,k,l} = c^{(\nu)}_{j,k} + E_{\nu} a^{(\nu)}_{j,k,l}$ and $\widetilde{v}^{(\nu)}_{j,k,l} = \sum_{h=1}^{m_{j,k,l}} \llbracket \widetilde{v}^{(\nu)}_{j,k,l,h} \rrbracket$ on ${\rm graph}\, \widetilde{\varphi}^{(\nu)}_{j,k} |_{ \Omega_{\nu}}$ for a Lebesgue measurable function $\widetilde{v}^{(\nu)}_{j,k,l,h} : {\rm graph}\, \widetilde{\varphi}^{(\nu)}_{j,k} |_{ \Omega_{\nu}} \rightarrow \mathbb{R}^m$.  Thus by \eqref{holefilling0_eqn4}, 
\begin{equation} \label{holefilling0_eqn16}
	\lim_{\nu \rightarrow \infty} \frac{1}{F_{\nu}} \,\mathcal{G}\left( w^{(\nu)}_{j,k},\, 
		\sum_{l=1}^{L_{j,k}} \sum_{h=1}^{m_{j,k,l}} \llbracket \op{Re}(a^{(\nu)}_{j,k,l} r^{\alpha} e^{i\alpha \theta}) + F_{\nu} \widetilde{w}_{j,k,l,h} \rrbracket 
		\right) = 0 
\end{equation}
uniformly on compact subsets of $B_{1/2}(0) \setminus \{0\} \times \mathbb{R}^{n-2}$, where $w^{(\nu)}_{j,k}$ and $\widetilde{w}_{j,k,l,h}$ are evaluated at $(re^{i\theta},y, \op{Re}(c^{(0)}_j r^{\alpha} e^{i\alpha \theta}))$. 

The justification of \eqref{holefilling0_eqn16} in case $\varphi^{(0)}_j$ is identically zero but $\varphi^{(\infty)}_{j,k}$ is not identically zero is essentially the same (in which case $\varphi^{(\nu)}_{j,k}(X) = \op{Re}(c^{(\nu)}_{j,k} (x_1+ix_2)^{\alpha})$, $\varphi^{(\nu,\infty)}_{j,k}(X) = \op{Re}(c^{(\nu,\infty)}_{j,k} (x_1+ix_2)^{\alpha})$ and $\widetilde{\varphi}^{(\infty)}_{j,k,l}(X) = \op{Re}(\widetilde{c}^{(\infty)}_{j,k,l} (x_1+ix_2)^{\alpha})$ where $c^{(\nu)}_{j,k}, c^{(\nu,\infty)}_{j,k}, \widetilde{c}^{(\infty)}_{j,k,l} \in \mathbb{C}^m \setminus \{0\}$ with $\lim_{\nu \rightarrow \infty} c^{(\nu)}_{j,k}/|c^{(\nu)}_{j,k}| = \lim_{\nu \rightarrow \infty} c^{(\nu,\infty)}_{j,k} = \widetilde{c}^{(\infty)}_{j,k,l}$), and similarly it requires only obvious modifications in case $\varphi^{(0)}_j$ and $\varphi^{(\nu,\infty)}_{j,k}$ are both identically zero (in which case $\psi^{(\nu)}$ might have both a zero component $\psi^{(\nu)}_{j,k,l}(X,0) = 0$ and nonzero components $\psi^{(\nu)}_{j,k,l}(X,0) = \op{Re}(a^{(\nu)}_{j,k,l} (x_1+ix_2)^{\alpha})$ for $a^{(\nu)}_{j,k,l} \in \mathbb{C}^m \setminus \{0\}$; hence if $\psi^{(\nu)}_{j,k,l}$ is identically zero then $\widetilde{\varphi}^{(\infty)}_{j,k,l}(X) = 0$ and otherwise $\widetilde{\varphi}^{(\infty)}_{j,k,l}(X) = \op{Re}(\widetilde{c}^{(\infty)}_{j,k,l} (x_1+ix_2)^{\alpha})$ where $\widetilde{c}^{(\infty)}_{j,k,l} = \lim_{\nu \rightarrow \infty} a^{(\nu)}_{j,k,l}/|a^{(\nu)}_{j,k,l}|$).

Now divide both sides of \eqref{holefilling0_eqn1} by $F_{\nu}^2$ and let $\nu \rightarrow \infty$ using Fatou's lemma, \eqref{holefilling0_eqn16}, and Lemma~\ref{gradient_convergence} of the appendix to conclude that $\widetilde{w}$ is homogeneous degree $\alpha$ in $B_{1/2}(0) \setminus B_{1/8}(0)$.  Hence the homogeneous degree $\alpha$ extension of $\left.\widetilde{w}\right|_{B_{1/2}(0) \setminus B_{1/8}(0)}$ is componentwise locally energy minimizing in ${\mathbb R}^{n} \setminus \{0\} \times {\mathbb R}^{n-2}$, so by unique continuation (Lemma~\ref{unique continuation lemma} of the appendix), $\widetilde{w}$ is homogeneous of degree $\alpha$ in $B_{1/2}(0)$.  Since $\widetilde{w} \in \mathfrak{B}$, we conclude from Theorem~\ref{homogrep_lemma} that $\widetilde{w} \in \mathfrak{L}$.  

We define $\widetilde{\psi}^{(\nu)} = (\widetilde{\psi}^{(\nu)}_{j,k}) \in \mathfrak{L}(\varphi^{(\nu,\infty)})$ as follows.  Recalling \eqref{holefilling0_eqn15}, if $\varphi^{(\nu,\infty)}_{j,k}(X) = \op{Re}(c^{(\nu,\infty)}_{j,k} (x_1+ix_2)^{\alpha})$ is not identically zero let 
\begin{equation*}
	\widetilde{\psi}^{(\nu)}_{j,k}(re^{i\theta},y,\op{Re}(c^{(\nu,\infty)}_{j,k} r^{\alpha} e^{i\alpha\theta})) 
	= \sum_{l=1}^{L_{j,k}} \sum_{h=1}^{m_{j,k,l}} \llbracket \op{Re}(a^{(\nu)}_{j,k,l} r^{\alpha} e^{i\alpha \theta}) 
		+ F_{\nu} \widetilde{w}_{j,k,l,h}(re^{i\theta},y,\op{Re}(\widetilde{c}^{(\infty)}_{j,k,l} r^{\alpha} e^{i\alpha\theta})) \rrbracket , 
\end{equation*}
noting that $\widetilde{\varphi}^{(\infty)}_{j,k,l} = \varphi^{(\nu,\infty)}_{j,k}$ for all $\nu$.  Similarly if $\varphi^{(\nu,\infty)}_{j,k}$ is identically zero, let 
\begin{align*}
	\widetilde{\psi}^{(\nu)}_{j,k}(re^{i\theta},y,0) 
	&= \sum_{l \,:\, \widetilde{\varphi}^{(\infty)}_{j,k,l} \equiv 0} \sum_{h=1}^{m_{j,k,l}} \llbracket F_{\nu} \widetilde{w}_{j,k,l,h}(re^{i\theta},y,0) \rrbracket
		\\&\hspace{5mm} + \sum_{l \,:\, \widetilde{\varphi}^{(\infty)}_{j,k,l} \not\equiv 0} \sum_{h=1}^{m_{j,k,l}} 
		\llbracket \op{Re}(a^{(\nu)}_{j,k,l} r^{\alpha} e^{i\alpha \theta}) + F_{\nu} \widetilde{w}_{j,k,l,h}(re^{i\theta},y,
		\op{Re}(\widetilde{c}^{(\infty)}_{j,k,l} r^{\alpha} e^{i\alpha\theta})) \rrbracket . 
\end{align*}
where we recall that $\widetilde{c}^{(\infty)}_{j,k,l} = \lim_{\nu \rightarrow \infty} a^{(\nu)}_{j,k,l}/|a^{(\nu)}_{j,k,l}|$.  


Now observe that for $\rho_1,\rho_2 \in (1/8,1/2)$ and $\omega \in \mathbb{S}^{n-1}$, 
\begin{equation*}
	\mathcal{G}\left( \frac{w^{(\nu)}_{j,k,l}(\rho_1 \omega)}{\rho_1^{\alpha}}, \frac{w^{(\nu)}_{j,k,l}(\rho_2 \omega)}{\rho_2^{\alpha}} \right) 
	\leq \int_{1/8}^{1/2} \left| \frac{\partial}{\partial R} \left( \frac{w^{(\nu)}_{j,k,l}(R \omega)}{R^{\alpha}} \right) \right| \,dR .
\end{equation*}
for each $j$, $k$, and $l$, where we let $\varphi^{(\infty)}_{j,k}(X) = \sum_{l=1}^{q_{j,k}} \llbracket \varphi^{(\infty)}_{j,k,l}(X) \rrbracket$ on the conical domain $K = \{ R\omega : R > 0, \, \omega \in B_1(\omega) \cap \mathbb{S}^{n-1} \}$, $w^{(\nu)}_{j,k,l}(X) = w^{(\nu)}_{j,k}(X,\varphi^{(\infty)}_{j,k,l}(X))$ on $K$, and $\psi^{(\nu)}_{j,k,l}(X) = \psi^{(\nu)}_{j,k}(X,\varphi^{(\infty)}_{j,k,l}(X))$ on $K$ (like in \eqref{varphi_infty_localized}, \eqref{w_localized}, and \eqref{psi_localized}).  Thus by the triangle inequality, homogeneity of $\psi^{(\nu)}$, and the Cauchy-Schwarz inequality, 
\begin{align*}
	\sum_{j=1}^J \sum_{k=1}^{p_j} \sum_{l=1}^{q_{j,k}} \mathcal{G}\left( \frac{w^{(\nu)}_{j,k,l}(\rho_1 \omega)}{\rho_1^{\alpha}}, 
		\frac{\widetilde{\psi}^{(\nu)}_{j,k,l}(\rho_1 \omega)}{\rho_1^{\alpha}} \right)^2 
	\leq{}& 2 \sum_{j=1}^J \sum_{k=1}^{p_j} \sum_{l=1}^{q_{j,k}} \mathcal{G}\left( \frac{w^{(\nu)}_{j,k,l}(\rho_2 \omega)}{\rho_2^{\alpha}}, 
		\frac{\widetilde{\psi}^{(\nu)}_{j,k,l}(\rho_2 \omega)}{\rho_2^{\alpha}} \right)^2 
	\\&+ 2 \int_{1/8}^{1/2} \sum_{j=1}^J \sum_{k=1}^{p_j} \sum_{l=1}^{q_{j,k}} 
		\left| \frac{\partial}{\partial R} \left( \frac{w^{(\nu)}_{j,k,l}(R \omega)}{R^{\alpha}} \right) \right|^2 \,dR ,
\end{align*}
Multiply both sides by $\rho_1^{n-1} \rho_2^{n-1}$, integrate over $\omega \in S^{n-1}$, $\rho_1 \in (1/8,1/2)$, and $\rho_2 \in (1/8,1/4)$, and sum over $j$ to get 
\begin{align*}
	\int_{B_{1/2}(0) \setminus B_{1/8}(0)} \sum_{j=1}^J \sum_{k=1}^{p_j} \sum_{l=1}^{q_{j,k}} \mathcal{G}(w^{(\nu)}_{j,k,l},\widetilde{\psi}^{(\nu)}_{j,k,l})^2 
	\leq{}& C \int_{B_{1/4}(0) \setminus B_{1/8}(0)} \sum_{j=1}^J \sum_{k=1}^{p_j} \sum_{l=1}^{q_{j,k}} 
		\mathcal{G}(w^{(\nu)}_{j,k,l},\widetilde{\psi}^{(\nu)}_{j,k,l})^2 
	\\&+ C \int_{B_{1/2}(0) \setminus B_{1/8}(0)} \sum_{j=1}^J \sum_{k=1}^{p_j} \sum_{l=1}^{q_{j,k}} 
		\left| \frac{\partial (w^{(\nu)}_{j,k,l}/R^{\alpha})}{\partial R} \right|^2 \nonumber 
\end{align*}
for some constant $C = C(n,\alpha) \in [1,\infty)$.  By adding $\int_{B_{1/8}(0)} \sum_{j=1}^J \sum_{k=1}^{p_j} \sum_{l=1}^{q_{j,k}} \mathcal{G}(w^{(\nu)}_{j,k,l},\widetilde{\psi}^{(\nu)}_{j,k,l})^2$ to both sides, 
\begin{align} \label{holefilling0_eqn17}
	\int_{B_{1/2}(0)} \sum_{j=1}^J \sum_{k=1}^{p_j} \sum_{l=1}^{q_{j,k}} \mathcal{G}(w^{(\nu)}_{j,k,l},\widetilde{\psi}^{(\nu)}_{j,k,l})^2 
	\leq{}& C \int_{B_{1/4}(0)} \sum_{j=1}^J \sum_{k=1}^{p_j} \sum_{l=1}^{q_{j,k}} 
		\mathcal{G}(w^{(\nu)}_{j,k,l},\widetilde{\psi}^{(\nu)}_{j,k,l})^2 
	\\&+ C \int_{B_{1/2}(0) \setminus B_{1/8}(0)} \sum_{j=1}^J \sum_{k=1}^{p_j} \sum_{l=1}^{q_{j,k}} 
		\left| \frac{\partial (w^{(\nu)}_{j,k,l}/R^{\alpha})}{\partial R} \right|^2 \nonumber 
\end{align}
By \eqref{holefilling0_hyp1}, the left-hand side of \eqref{holefilling0_eqn17} is bounded below by $F_{\nu}^2/M$.  By \eqref{holefilling0_eqn16}, the nonconcentration estimate \eqref{blowup_est2} with $w = w^{(\nu)}$ and $\psi = \widetilde{\psi}^{(\nu)}$, and \eqref{holefilling0_eqn1}, after dividing by $F_{\nu}^2$ the right-hand side of \eqref{holefilling0_eqn17} converges to 0 as $\nu \rightarrow \infty$. This contradiction proves the lemma. 
\end{proof}

\begin{proof}[Proof of Lemma~\ref{holefilling_lemma}] 
Let $w \in \mathfrak{B}$ and let $\psi$ be a projection of $w$ onto $\mathfrak{L}$ in $B_{1/2}(0)$.  Let $\varphi^{(\infty)} \in {\mathfrak D}$ such that $w \in \mathfrak{B}(\varphi^{(\infty)})$.  For $\nu =1, 2, 3, \ldots,$ there are $u^{(\nu)}$, $\varphi^{(\nu)}$ satisfying the requirements of Defintion~\ref{blowupclass_defn} such that $w$ is the blow-up of $u^{(\nu)}$ relative to $\varphi^{(\nu)}$ and excess $E_{\nu} = \left( \int_{B_1(0)} \mathcal{G}(u^{(\nu)},\varphi^{(\nu)})^2 \right)^{1/2}.$
Let $B$ be given by \eqref{B_matrix} where $b$ is as in Definition~\ref{homogclass_defn}.  
By the argument at the end of Section~\ref{sec:homogblowup_subsec1}, there exist $w^{\star} \in {\mathfrak B}$ and a number $c^{\star} \in [0, \infty)$ such that  $w^{\star}$ is the blow-up of $u^{(\nu)} \circ e^{-E_{\nu} B}$ relative to $\varphi^{(\nu)}$ and excess $\sqrt{\int_{B_{1}(0)} {\mathcal G}(u^{(\nu)} \circ e^{-E_{\nu}B}, \varphi^{(\nu)})^{2}},$ and 
$c^{\star}w^{\star} = w_{j,k,l}(X) - D\varphi^{(0)}_j(X) \cdot By.$ If $c^{\star} = 0$, Lemma~\ref{holefilling_lemma} trivially holds, so assume $c^{\star} > 0$. Since $\psi$ is a projection of $w$ onto ${\mathfrak L}$ in $B_{1/2}(0)$, it follows that $(c^{\star})^{-1}(\psi_{j,k,l}(X) - D\varphi^{(0)}_j(X) \cdot By) \in {\mathfrak L}_{0}(\varphi^{(\infty)})$ is a projection of $w^{\star}$ onto ${\mathfrak L}$ in $B_{1/2}(0).$ 
Hence in order to prove Lemma~\ref{holefilling_lemma}, we may, and shall, assume without loss of generality that 
$\psi \in \mathfrak{L}_0(\varphi^{(\infty)})$.  

Let $s \in \{s_0,\ldots,\lceil q/q_0 \rceil\}$ such that $\psi \in \mathfrak{L}_{0,s}(\varphi^{(\infty)})$.  For $i \in \{1,2,\ldots,\lceil q/q_0 \rceil - s_0+1\}$, inductively define $\beta_i$ and $C_i$ by $\beta_1 = \overline{\beta}$ and $C_1 = \overline{C}$ where $\overline{\beta}$ and $\overline{C}$ are as in Lemma~\ref{holefilling0_lemma} with $M = 1$ and for each $i \geq 2$, $\beta_i = \overline{\beta}$ and $C_i = \overline{C}$ where $\overline{\beta}$ and $\overline{C}$ are as in Lemma~\ref{holefilling0_lemma} with $M = 2^{i-1} (\beta_1 \beta_2 \cdots \beta_{i-1})^{-1}$.  

Observe that \eqref{holefilling0_hyp1} with $M = 1$ holds true since $\psi$ is a projection of $w$ onto $\mathfrak{L}$ in $B_{1/2}(0)$.  If either $s = s_0,$ or if $s > s_0$ and $w$ and $\psi$ satisfy \eqref{holefilling0_hyp2} with $\overline{\beta} = \beta_1$, then by Lemma~\ref{holefilling0_lemma}, $w$ and $\psi$ satisfy \eqref{holefilling_est} with $C = C_1$.  If instead $s > s_0$ and $w$ and $\psi$ do not satisfy \eqref{holefilling0_hyp2} with $\overline{\beta} = \beta_1$, for $i = 1,2,\ldots,i_0$ inductively select $\psi^{(i)} \in \mathfrak{L}_{0,s_i}(\varphi^{(\infty)})$ such that when $i = 1$ we have $s_0 \leq s_1 < s$ and 
\begin{equation*}
	\int_{B_{1/2}(0)} \sum_{j=1}^J \sum_{k=1}^{p_j} \sum_{l=1}^{q_{j,k}} \mathcal{G}(w_{j,k,l},\psi^{(1)}_{j,k,l})^2 
		< 2 \inf_{\psi' \in \bigcup_{s'=s_0}^{s-1} \mathfrak{L}_{0,s'}(\varphi^{(\infty)})} 
			\int_{B_{1/2}(0)} \sum_{j=1}^J \sum_{k=1}^{p_j} \sum_{l=1}^{q_{j,k}} \mathcal{G}(w_{j,k,l},\psi'_{j,k,l})^2 
\end{equation*}
and for each $i \geq 2$ we have $s_0 \leq s_i < s_{i-1}$ and 
\begin{equation*}
	\int_{B_{1/2}(0)} \sum_{j=1}^J \sum_{k=1}^{p_j} \sum_{l=1}^{q_{j,k}} \mathcal{G}(w_{j,k,l},\psi^{(i)}_{j,k,l})^2 
		< 2 \inf_{\psi' \in \bigcup_{s'=s_0}^{s_{i-1}-1} \mathfrak{L}_{0,s'}(\varphi^{(\infty)})} 
			\int_{B_{1/2}(0)} \sum_{j=1}^J \sum_{k=1}^{p_j} \sum_{l=1}^{q_{j,k}} \mathcal{G}(w_{j,k,l},\psi'_{j,k,l})^2 
\end{equation*}
and terminate either when $i$ equals the smallest $i_0$ for which $s_{i_0} > s_0$ and 
\begin{align} \label{holefilling_eqn1}
	&\int_{B_{1/2}(0)} \sum_{j=1}^J \sum_{k=1}^{p_j} \sum_{l=1}^{q_{j,k}} \mathcal{G}(w_{j,k,l},\psi^{(i_0)}_{j,k,l})^2 
		\\&\hspace{10mm} \leq \beta_{i_0+1} \inf_{\psi' \in \bigcup_{s'=s_0}^{s_{i_0}-1} \mathfrak{L}_{0,s'}(\varphi^{(\infty)})} 
			\int_{B_{1/2}(0)} \sum_{j=1}^J \sum_{k=1}^{p_j} \sum_{l=1}^{q_{j,k}} \mathcal{G}(w_{j,k,l},\psi'_{j,k,l})^2 \nonumber 
\end{align}
or (if no such $i_{0}$ exists) when $i$ equals the value $i_{0}$ for which $s_{i_{0}} = s_{0}.$   
By choice of $\psi^{(i)}$, one readily checks that $\psi^{(i_0)}$ satisfies 
\begin{align} \label{holefilling_eqn2}
	&\int_{B_{1/2}(0)} \sum_{j=1}^J \sum_{k=1}^{p_j} \sum_{l=1}^{q_{j,k}} \mathcal{G}(w_{j,k,l},\psi_{j,k,l})^2 
	\leq \int_{B_{1/2}(0)} \sum_{j=1}^J \sum_{k=1}^{p_j} \sum_{l=1}^{q_{j,k}} \mathcal{G}(w_{j,k,l},\psi^{(i_0)}_{j,k,l})^2 
	\\&\hspace{10mm} < \frac{2^{i_0}}{\beta_1 \beta_2 \cdots \beta_{i_0}} \int_{B_{1/2}(0)} 
		\sum_{j=1}^J \sum_{k=1}^{p_j} \sum_{l=1}^{q_{j,k}} \mathcal{G}(w_{j,k,l},\psi_{j,k,l})^2. \nonumber 
\end{align}
In view of \eqref{holefilling_eqn1} and the second inequality of \eqref{holefilling_eqn2}, we may apply Lemma~\ref{holefilling0_lemma} and  to conclude that \eqref{holefilling_est} holds true with $w$, $\psi^{(i_0)}$, and $C_{i_0+1}$ in place of $w$, $\psi$, and $C$.  Hence, by the first inequality of \eqref{holefilling_eqn2}, $w$ and $\psi$ satisfy \eqref{holefilling_est} with $C = C_{i_0+1}$.  Since $i_0 \leq \lceil q/q_0 \rceil - s_0$,  we conclude that $w$ and $\psi$ satisfy \eqref{holefilling_est} with $C = \max\{C_1,C_2,\ldots,C_{\lceil q/q_0 \rceil - s_0+1}\}\}$. 
\end{proof}

\begin{theorem} \label{blowupdecay_lemma}
Let $\vartheta \in (0,1/8]$ and $w \in \mathfrak{B}$ and for each $\rho \in (0,1/2]$ let $\psi^{(\rho)} = (\psi^{(\rho)}_{j,k})$ be a projection of $w$ onto $\mathfrak{L}$ in $B_{\rho}(0)$.  Then, 
\begin{align} \label{blowupdecay_eqn1}
	\vartheta^{-n-2\alpha} \int_{B_{\vartheta}(0)} \sum_{j=1}^J \sum_{k=1}^{p_j} \sum_{l=1}^{q_{j,k}} \mathcal{G}(w_{j,k,l},\psi^{(\vartheta)}_{j,k,l})^2 
	\leq C \vartheta^{2\mu} \int_{B_{1/2}(0)} \sum_{j=1}^J \sum_{k=1}^{p_j} \sum_{l=1}^{q_{j,k}} \mathcal{G}(w_{j,k,l},\psi^{(1/2)}_{j,k,l})^2 
\end{align}
for some constants $\mu \in (0,1)$ and $C \in (0,\infty)$ depending on $n$, $m$, $q$, $\alpha$ and $\varphi^{(0)}$, where $w_{j,k,l}(X) = w_{j,k}(X,\varphi^{(\infty)}_{j,k,l}(X))$ and $\psi^{(\rho)}_{j,k,l}(X) = \psi^{(\rho)}_{j,k}(X,\varphi^{(\infty)}_{j,k,l}(X))$ for $\varphi^{(\infty)}_{j,k,l}$ as in \eqref{varphi_infty_localized}.
\end{theorem}
\begin{proof}
Let $\rho \in [\vartheta, 1/2]$. By \eqref{blowup_est1} with $\psi^{(\rho)}$ in place of $\psi$, 
\begin{equation} \label{blowupdecay_eqn2}
	\int_{B_{\rho/4}(0)} \sum_{j=1}^J \sum_{k=1}^{p_j} \sum_{l=1}^{q_{j,k}} R^{2-n} \left| \frac{\partial (w_{j,k,l}/R^{\alpha})}{\partial R} \right|^2 
	\leq C \rho^{-n-2\alpha} \int_{B_{\rho}(0)} \sum_{j=1}^J \sum_{k=1}^{p_j} \sum_{l=1}^{q_{j,k}} \mathcal{G}(w_{j,k,l},\psi^{(\rho)}_{j,k,l})^2 
\end{equation}
where $C = C(n,m,q,\alpha,\varphi^{(0)}) \in (0,\infty)$. Let $u^{(\nu)}$, $\varphi^{(\nu)}$ be as in Definition~\ref{blowupclass_defn} such that $w$ is the blow-up of $u^{(\nu)}$ relative to 
$\varphi^{(\nu)}$ by the excess $E_{\nu} = \sqrt{\int_{B_{1}(0)} {\mathcal G}(u^{(\nu)}, \varphi^{(\nu)})^{2}}.$ With the help of the argument of the proof of Corollary~\ref{graphical_cor}, we may readily verify that there exists $w_{1} \in {\mathfrak B}$ such that $w_{1}$ is the blow-up of the sequence ${u}^{(\rho, \nu)} \equiv \rho^{-\alpha}u^{(\nu)}(\rho(\cdot))$ relative to $\varphi^{(\nu)}$  by the excess $E_{\nu}^{(\rho)} = \sqrt{\int_{B_{1}(0)} {\mathcal G}(u^{(\rho, \nu)}, \varphi^{(\nu)})^{2}},$ and moreover that $c_{1}w_{1} = \rho^{-\alpha}w(\rho(\cdot))$ where $c_{1} = \lim_{\nu \to \infty} \, E_{\nu}^{-1}E_{\nu}^{(\rho)} \in [0, \rho^{-\alpha}]$. Since it suffices to prove the present theorem assuming that $\left.w\right|_{B_{\vartheta}} \neq 0$, we may assume that $c_{1}>0$. Hence by applying Lemma~\ref{holefilling_lemma} to $w_{1}$, we see that 
\begin{align} \label{blowupdecay_eqn3}
	&\rho^{-n-2\alpha} \int_{B_{\rho}(0)} \sum_{j=1}^J \sum_{k=1}^{p_j} \sum_{l=1}^{q_{j,k}} \mathcal{G}(w_{j,k,l},\psi^{(\rho)}_{j,k,l})^2 
	\\&\hspace{15mm} \leq C \int_{B_{\rho}(0) \setminus B_{\rho/4}(0)} \sum_{j=1}^J \sum_{k=1}^{p_j} \sum_{l=1}^{q_{j,k}} 
		R^{2-n} \left| \frac{\partial (w_{j,k,l}/R^{\alpha})}{\partial R} \right|^2 \nonumber 
\end{align}
where $C = C(n,m,q,\alpha,\varphi^{(0)}) \in (0,\infty)$.  Thus by \eqref{blowupdecay_eqn2} and \eqref{blowupdecay_eqn3}, 
\begin{align} \label{blowupdecay_eqn4}
	&\int_{B_{\rho/4}(0)} \sum_{j=1}^J \sum_{k=1}^{p_j} \sum_{l=1}^{q_{j,k}} R^{2-n} \left| \frac{\partial (w_{j,k,l}/R^{\alpha})}{\partial R} \right|^2	
	\\&\hspace{15mm} \leq C_0 \int_{B_{\rho}(0) \setminus B_{\rho/4}(0)} \sum_{j=1}^J \sum_{k=1}^{p_j} \sum_{l=1}^{q_{j,k}} 
		R^{2-n} \left| \frac{\partial (w_{j,k,l}/R^{\alpha})}{\partial R} \right|^2 \nonumber 
\end{align}
for all $\rho \in [\vartheta,1/2]$ and some constant $C_0 = C_0(n,m,q,\alpha,\varphi^{(0)}) \in (0,\infty)$.  By adding $C_0$ times the left-hand side of \eqref{blowupdecay_eqn4} to both sides of \eqref{blowupdecay_eqn4}, 
\begin{equation} \label{blowupdecay_eqn5}
	\int_{B_{\rho/4}(0)} \sum_{j=1}^J \sum_{k=1}^{p_j} \sum_{l=1}^{q_{j,k}} R^{2-n} \left| \frac{\partial (w_{j,k,l}/R^{\alpha})}{\partial R} \right|^2
	\leq \gamma \int_{B_{\rho}(0)} \sum_{j=1}^J \sum_{k=1}^{p_j} \sum_{l=1}^{q_{j,k}} R^{2-n} \left| \frac{\partial (w_{j,k,l}/R^{\alpha})}{\partial R} \right|^2
\end{equation}
for all $\rho \in [\vartheta,1/2]$, where $\gamma = C_0/(1 + C_0) \in (0,1)$.  Iteratively applying \eqref{blowupdecay_eqn5} with $\rho = 2^{-2i-1}$ for $i = 1,2,\ldots,N-1$, where $N$ is the positive integer such that $2^{-2N-3} < \vartheta \leq 2^{-2N-1}$, 
\begin{equation} \label{blowupdecay_eqn6}
	\int_{B_{\vartheta}(0)} \sum_{j=1}^J \sum_{k=1}^{p_j} \sum_{l=1}^{q_{j,k}} R^{2-n} \left| \frac{\partial (w_{j,k,l}/R^{\alpha})}{\partial R} \right|^2
	\leq C \vartheta^{2\mu} \int_{B_{1/8}(0)} \sum_{j=1}^J \sum_{k=1}^{p_j} \sum_{l=1}^{q_{j,k}} 
		R^{2-n} \left| \frac{\partial (w_{j,k,l}/R^{\alpha})}{\partial R} \right|^2, 
\end{equation}
where $\mu = -\log \gamma/\log 16$.  By combining \eqref{blowupdecay_eqn2} with $\rho = 1/2$, \eqref{blowupdecay_eqn3} with $\rho = \vartheta$, and \eqref{blowupdecay_eqn6}, we obtain \eqref{blowupdecay_eqn1}.
\end{proof}

\section{Excess decay lemmas} \label{sec:main_lemma_sec} 

We will first prove the following preliminary excess decay lemma.

\begin{lemma} \label{premain_lemma} 
Let $\varphi^{(0)}$ be as in Definition~\ref{varphi0_defn} and let $p \in \{p_0,p_0+1,\ldots, \lceil q/q_0 \rceil\}$.  Given $\vartheta \in (0,1/8)$, there exists $\overline{\delta} \in (0,1/4)$ and $\overline{\varepsilon}, \overline{\beta} \in (0,1)$ depending only on $n$, $m$, $q$, $\alpha$, $\varphi^{(0)}$, $p$, and $\vartheta$ such that if $u \in W^{1,2}(B_1(0);\mathcal{A}_q(\mathbb{R}^m))$ is an average-free, energy minimizing function and $\varphi \in \Phi_{\overline{\varepsilon},p}(\varphi^{(0)})$ such that $\mathcal{N}_u(0) \geq \alpha$, 
\begin{equation} \label{premain_hyp1} 
	\int_{B_1(0)} \mathcal{G}(u,\varphi)^2 < \overline{\varepsilon}^2
\end{equation}
and either (i) $p = p_0$ or (ii) $p > p_0$ and 
\begin{equation} \label{premain_hyp2}
	\int_{B_1(0)} \mathcal{G}(u,\varphi)^2 \leq \overline{\beta} \inf_{\varphi' \in \bigcup_{p'=p_0}^{p-1} \Phi_{3\overline{\varepsilon},p'}(\varphi^{(0)})} \int_{B_1(0)} \mathcal{G}(u,\varphi')^2, 
\end{equation}
then either 
\begin{enumerate}
\item[(i)] $B_{\overline{\delta}}(0,y_0) \cap \{ X \in B_{1/2}(0) \cap \Sigma_{u,q} : \mathcal{N}_u(X) \geq \alpha \} = \emptyset$ for some $y_0 \in B^{n-2}_{1/2}(0)$ or 

\item[(ii)] there is a $\widetilde{\varphi} \in \widetilde{\Phi}_{\overline{\gamma} \overline{\varepsilon}}(\varphi^{(0)})$ such that 
\begin{equation*}
	\vartheta^{-n-2\alpha} \int_{B_{\vartheta}(0)} \mathcal{G}(u,\widetilde{\varphi})^2 
	\leq \overline{C} \vartheta^{2\overline{\mu}} \int_{B_1(0)} \mathcal{G}(u,\varphi)^2, 
\end{equation*}
where $\overline{\gamma} \in [1,\infty)$ is a constant depending only on $n$, $m$, $q$, $\alpha$, $\varphi^{(0)}$, $p$, and $\vartheta$ and $\overline{\mu} \in (0,1)$ and $\overline{C} \in (0,\infty)$ are constants depending only on $n$, $m$, $q$, $\alpha$, $\varphi^{(0)}$, and $p$ (independent of $\vartheta$). 
\end{enumerate}
\end{lemma}
\begin{proof}
Let $\vartheta \in (0,1/8)$ and $\varphi^{(0)}$ be fixed as in the lemma.  For $\nu = 1,2,3,\ldots$, let $0 < \varepsilon_{\nu} \leq \delta_{\nu} \downarrow 0$, $\beta_{\nu} \downarrow 0$, $u^{(\nu)} \in W^{1,2}(B_1(0);\mathcal{A}_q(\mathbb{R}^m))$ be an average-free, energy minimizing function and $\varphi^{(\nu)} \in \Phi_{\varepsilon_{\nu},p}(\varphi^{(0)})$ such that \eqref{premain_hyp1} and \eqref{premain_hyp2} hold true with $\overline{\varepsilon} = \varepsilon_{\nu}$, $\overline{\beta} = \beta_{\nu}$, $u = u^{(\nu)}$, and $\varphi = \varphi^{(\nu)}$ and option (i) of Lemma~\ref{premain_lemma} does not hold true with $\overline{\delta} = \delta_{\nu}$ and $u = u^{(\nu)}$.  We want to show that for some constants $\overline{\gamma} \in [1,\infty)$ depending only on $n$, $m$, $q$, $\alpha$, $\varphi^{(0)}$, $p$, and $\vartheta$ and $\overline{\mu} \in (0,1)$ and $\overline{C} \in (0,\infty)$ depending only on $n$, $m$, $q$, $\alpha$, $\varphi^{(0)}$, and $p$ and for infinitely many $\nu$, there exists $\widetilde{\varphi}^{(\nu)} \in \widetilde{\Phi}_{\overline{\gamma} \overline{\varepsilon}_k}(\varphi^{(0)})$ such that 
\begin{equation*}
	\vartheta^{-n-2\alpha} \int_{B_{\vartheta}(0)} \mathcal{G}(u^{(\nu)},\widetilde{\varphi}^{(\nu)})^2 
	\leq \overline{C} \vartheta^{2\overline{\mu}} \int_{B_1(0)} \mathcal{G}(u^{(\nu)},\varphi^{(\nu)})^2. 
\end{equation*}
By the arbitrariness of the sequences this will complete the proof of Lemma~\ref{premain_lemma}. 

By \eqref{premain_hyp1}, \eqref{premain_hyp2}, and the failure of option (i), let $w \in \mathfrak{B}$ be a blow-up of $u^{(\nu)}$ relative to $\varphi^{(\nu)}$ in $B_{3/4}(0)$ obtained via the blow-up procedure in Section \ref{sec:graphical_sec}.  By Lemma~\ref{blowupdecay_lemma}, 
\begin{equation} \label{premain_eqn1}
	\vartheta^{-n-2\alpha} \int_{B_{\vartheta}(0)} \sum_{j=1}^J \sum_{k=1}^{p_j} \sum_{l=1}^{q_{j,k}} \mathcal{G}(w_{j,k,l},\psi^{(\vartheta)}_{j,k,l})^2 
	\leq C \vartheta^{2\overline{\mu}} \int_{B_{1/2}(0)} \sum_{j=1}^J \sum_{k=1}^{p_j} \sum_{l=1}^{q_{j,k}} |w_{j,k,l}|^2 
	\leq C \vartheta^{2\overline{\mu}} 
\end{equation}
for some constants $\overline{\mu} \in (0,1)$ and $C \in (0,\infty)$ depending on $n$, $m$, $q$, $\alpha$, and $\varphi^{(0)}$ (and independent of $\vartheta$), where $w_{j,k,l}$ is as in \eqref{w_localized} and $\psi^{(\vartheta)}$ is as in \eqref{psi_localized} with $\psi = \psi^{(\vartheta)}$.  Since $\psi^{(\vartheta)}$ is homogeneous degree $\alpha$ and is a projection of $w$ onto $\mathfrak{L}$ in $B_{\vartheta}(0)$, 
\begin{align} \label{premain_eqn2}
	\int_{B_{1/2}(0)} \sum_{j=1}^J \sum_{k=1}^{p_j} \sum_{l=1}^{q_{j,k}} |\psi^{(\vartheta)}_{j,k,l}|^2 
	&= (2\vartheta)^{-n-2\alpha} \int_{B_{\vartheta}(0)} \sum_{j=1}^J \sum_{k=1}^{p_j} \sum_{l=1}^{q_{j,k}} |\psi^{(\vartheta)}_{j,k,l}|^2 
	\\&\leq 2 (2\vartheta)^{-n-2\alpha} \int_{B_{\vartheta}(0)} \sum_{j=1}^J \sum_{k=1}^{p_j} \sum_{l=1}^{q_{j,k}} 
		(|w_{j,k,l}|^2 + \mathcal{G}(w_{j,k,l},\psi^{(\vartheta)}_{j,k,l})^2) \nonumber 
	\\&\leq 4 (2\vartheta)^{-n-2\alpha} \int_{B_{\vartheta}(0)} \sum_{j=1}^J \sum_{k=1}^{p_j} \sum_{l=1}^{q_{j,k}} |w_{j,k,l}|^2 \nonumber 
	\\&\leq 4 (2\vartheta)^{-n-2\alpha}. \nonumber 
\end{align}
Define $\widetilde{\varphi}^{(\nu)}$ by Lemma~\ref{blowup2L_lemma} with $\psi = \psi^{(\vartheta)}$.  Note that by Lemma~\ref{blowup2L_lemma} and \eqref{premain_eqn2} we have $\widetilde{\psi}^{(\nu)} \in \Phi_{\overline{\gamma}\,\varepsilon_{\nu}}(\varphi^{(0)})$ for some constant $\overline{\gamma} = \overline{\gamma}(n,m,q,\alpha,\varphi^{(0)},\vartheta) \in [1,\infty)$.  By Lemma~\ref{blowup_norm_conv_lemma}, 
\begin{equation*}
	\lim_{\nu \rightarrow \infty} E_{\nu}^{-2} \int_{B_{\vartheta}(0)} \mathcal{G}(u^{(\nu)},\widetilde{\varphi}^{(\nu)})^2 
	= \int_{B_{\vartheta}(0)} \sum_{j=1}^J \sum_{k=1}^{p_j} \sum_{l=1}^{q_{j,k}} \mathcal{G}(w_{j,k,l},\psi^{(\vartheta)}_{j,k,l})^2 
\end{equation*}
and thus by \eqref{premain_eqn1}, 
\begin{equation*}
	\vartheta^{-n-2\alpha} E_{\nu}^{-2} \int_{B_{\vartheta}(0)} \mathcal{G}(u^{(\nu)},\widetilde{\varphi}^{(\nu)})^2 
	\leq 2C \vartheta^{2\overline{\mu}} 
\end{equation*}
for $\nu$ sufficiently large, completing the proof. 
\end{proof}

Notice that the hypothesis \eqref{premain_hyp2} plays an important role in the blow-up method used to prove Lemma~\ref{premain_lemma} above.  We will now  deduce Lemma~\ref{main_lemma} from Lemma~\ref{premain_lemma}, i.e.\  show that hypothesis \eqref{premain_hyp2} can be removed in favour of weakening the conclusion to allow excess improvement to occur at one of finitely many fixed scales. The argument is the same as that in Section 13 of~\cite{Wic14} and is included here for completion.

\begin{proof}[Proof of Lemma~\ref{main_lemma}]
Let us first consider the special case $\varphi \in \Phi_{\varepsilon_0}(\varphi^{(0)})$.  For each $p \in \{p_0,p_0+1,\ldots,\lceil q/q_0 \rceil\}$ and $\vartheta \in (0,1/8)$, let $\overline{\varepsilon} = \overline{\varepsilon}(p,\vartheta)$, $\overline{\beta} = \overline{\beta}(p,\vartheta)$, $\overline{\delta} = \overline{\delta}(p,\vartheta)$, $\overline{\gamma} = \overline{\gamma}(p)$, $\overline{\mu} = \overline{\mu}(p)$, and $\overline{C} = \overline{C}(p)$ be as in Lemma~\ref{premain_lemma}, where we omit the dependence on $n$, $m$, $q$, $\alpha$, and $\varphi^{(0)}$ to simplify notation.  For each $j = 1,2,\ldots, \lceil q/q_0 \rceil - p_0 + 1$, set 
\begin{equation*}
	\beta_j = \min \{ \overline{\beta}(p,\vartheta_j) : p = p_0,p_0+1,\ldots, \lceil q/q_0 \rceil \} .
\end{equation*}
Set  
\begin{align*}
	\delta_0 &= \min \{\overline{\delta}(p,\vartheta_j) : p = p_0,p_0+1,\ldots, \lceil q/q_0 \rceil, \, j = 1,2,\ldots, \lceil q/q_0 \rceil - p_0 + 1 \} , \\
	\varepsilon_0 &= \min \left\{ 3^{-j} \sqrt{\beta_1 \beta_2 \cdots \beta_{j-1}} \,\overline{\varepsilon}(p,\vartheta_j)
		 : p = p_0,p_0+1,\ldots, \lceil q/q_0 \rceil, \, j = 1,2,\ldots, \lceil q/q_0 \rceil - p_0 + 1 \right\} , \\
	\gamma &= \max \{ 3^j \overline{\gamma}(p,\vartheta_j) : p = p_0,p_0+1,\ldots, \lceil q/q_0 \rceil, \, j = 1,2,\ldots, \lceil q/q_0 \rceil - p_0 + 1 \} .
\end{align*}
Additionally, set 
\begin{align*}
	\mu &= \min \{\overline{\mu}(p) : p = p_0,p_0+1,\ldots, \lceil q/q_0 \rceil \} , \\
	C_1 &= \max \{ \overline{C}(p) : p = p_0,p_0+1,\ldots, \lceil q/q_0 \rceil \} , \\
	C_j &= \frac{2^{j-1} C_1}{\beta_1 \beta_2 \cdots \beta_{j-1}} \text{ for } j = 2,3,\ldots, \lceil q/q_0 \rceil - p_0 + 1, 
\end{align*}
and notice that $\mu$ is independent of $\vartheta_1, \vartheta_2, \ldots, \vartheta_{\lceil q/q_0 \rceil - p_0 + 1}$ and each $C_j$ is independent of 
$$\vartheta_j, \vartheta_{j+1}, \ldots, \vartheta_{\lceil q/q_0 \rceil - p_0 + 1}.$$

Suppose that $\varphi \in \Phi_{\varepsilon_0,p}(\varphi^{(0)})$ for some $p \in \{p_0,p_0+1,\ldots, \lceil q/q_0 \rceil\}$ and that option (i) of Lemma~\ref{main_lemma} does not hold true.  If 
\begin{equation} \label{multiscale_eqn2}
	\int_{B_1(0)} \mathcal{G}(u,\varphi)^2 \leq \beta_1 \inf_{\varphi' \in \bigcup_{p'=p_0}^{p-1} \Phi_{3\varepsilon_0,p'}(\varphi^{(0)})} \int_{B_1(0)} \mathcal{G}(u,\varphi')^2 
\end{equation}
then by Lemma~\ref{premain_lemma} with $\vartheta = \vartheta_1$, there exists $\widetilde{\varphi} \in \widetilde{\Phi}_{\gamma \varepsilon_0}(\varphi^{(0)})$ such that 
\begin{equation*}
	\vartheta_1^{-n-2\alpha} \int_{B_{\vartheta_1}(0)} \mathcal{G}(u,\widetilde{\varphi})^2 
	\leq C_1 \vartheta_1^{2\mu} \int_{B_1(0)} \mathcal{G}(u,\varphi)^2 . 
\end{equation*}
If instead \eqref{multiscale_eqn2} does not hold true, inductively select $s_k \in \{p_0,p_0+1,\ldots, \lceil q/q_0 \rceil \}$ and $\varphi^{(k)} \in \Phi_{3^k \varepsilon_0,s_k}(\varphi^{(0)})$ for $k = 1,2,3,\ldots,j$ as follows.  Set $s_1 = p$ and $\varphi^{(1)} = \varphi$.  For each $k \geq 1$, if 
\begin{equation} \label{multiscale_eqn3}
	\int_{B_1(0)} \mathcal{G}(u,\varphi^{(k)})^2 > \beta_k 
		\inf_{\varphi' \in \bigcup_{p'=p_0}^{s_k-1} \Phi_{3^k \varepsilon_0,p'}(\varphi^{(0)})} \int_{B_1(0)} \mathcal{G}(u,\varphi')^2
\end{equation}
select, $p_0 \leq s_{k+1} < s_k$ and $\varphi^{(k+1)} \in \Phi_{3^k \varepsilon_0,s_k}(\varphi^{(0)})$ such that 
\begin{equation} \label{multiscale_eqn4}
	\int_{B_1(0)} \mathcal{G}(u,\varphi^{(k+1)})^2 < 2 \inf_{\varphi' \in \bigcup_{p'=p_0}^{s_k-1} \Phi_{3^k \varepsilon_0,p'}(\varphi^{(0)})} \int_{B_1(0)} \mathcal{G}(u,\varphi')^2. 
\end{equation}
Otherwise, stop and set $j = k$ so that 
\begin{equation} \label{multiscale_eqn5}
	\int_{B_1(0)} \mathcal{G}(u,\varphi^{(j)})^2 \leq \beta_j 
		\inf_{\varphi' \in \bigcup_{p'=p_0}^{s_j-1} \Phi_{3^j \varepsilon_0,p'}(\varphi^{(0)})} \int_{B_1(0)} \mathcal{G}(u,\varphi')^2. 
\end{equation}
Notice that $\varphi^{(j)} \in \Phi_{3^j \varepsilon_0}(\varphi^{(0)})$ where $3^j \varepsilon_0 \leq \overline{\varepsilon}(s_j,\theta_j)$.  By \eqref{multiscale_eqn3} and \eqref{multiscale_eqn4}
\begin{equation} \label{multiscale_eqn6}
	\int_{B_1(0)} \mathcal{G}(u,\varphi^{(j)})^2 
	< \frac{2^{j-1}}{\beta_1 \beta_2 \cdots \beta_{j-1}} \int_{B_1(0)} \mathcal{G}(u,\varphi)^2 
	< \frac{2^{j-1}}{\beta_1 \beta_2 \cdots \beta_{j-1}} \varepsilon_0^2 \leq \overline{\varepsilon}(s_j,\theta_j)^2 . 
\end{equation}
By \eqref{multiscale_eqn5} and \eqref{multiscale_eqn6}, we can apply Lemma~\ref{premain_lemma} with $\vartheta = \vartheta_j$ to obtain $\widetilde{\varphi} \in \widetilde{\Phi}_{\gamma \varepsilon_0}(\varphi^{(0)})$ such that 
\begin{align*}
	\vartheta_j^{-n-2\alpha} \int_{B_{\vartheta_j}(0)} \mathcal{G}(u,\widetilde{\varphi})^2 
	&\leq C_1 \vartheta_j^{2\mu} \int_{B_1(0)} \mathcal{G}(u,\varphi^{(j)})^2 
	\leq \frac{2^{j-1} C_1}{\beta_1 \beta_2 \cdots \beta_{j-1}} \vartheta_j^{2\mu} \int_{B_1(0)} \mathcal{G}(u,\varphi)^2 
	\\&= C_j \vartheta_j^{2\mu} \int_{B_1(0)} \mathcal{G}(u,\varphi)^2 . 
\end{align*}
(Notice that $C_j$ depends on $\vartheta_1,\vartheta_2,\ldots,\vartheta_{j-1}$ but is independent of $\vartheta_j$.) 

In the general case of $\varphi \in \widetilde{\Phi}_{\varepsilon_0}(\varphi^{(0)})$, let $A$ be a skew-symmetric $n \times n$ matrix with $A_{ij} = 0$ for $i = 1,2$ and $j = 3,4,\ldots,n$ and $|A| < \varepsilon_0$ such that $\varphi \circ e^{-A} \in \Phi_{\varepsilon_0}(\varphi^{(0)})$.  Letting $\widehat{\varepsilon}_0$ and $\widehat{\delta}_0$ represent $\varepsilon_0$ and $\delta_0$ from the discussion above, let $\varepsilon_0 = \widehat{\varepsilon}_0/2$ and $\delta_0 = \widehat{\delta}_0 + \widehat{\varepsilon}_0/2$.  Obviously one can apply Lemma~\ref{main_lemma} to $u \circ e^{-A}$ and $\varphi \circ e^{-A}$ to conclude that either option (i) or option (ii) holds true with $\widehat{\delta}_0, u \circ e^{-A}, \varphi \circ e^{-A}$ in place of $\delta_0, u, \varphi$.  Suppose $u \circ e^{-A}$ satisfies option (i), that is for some $y_0 \in B^{n-2}_{1/2}(0)$ there exists $Z \in B_{\widehat{\delta}_0}(0,y_0) \cap B_{1/2}(0) \cap \Sigma_{u \circ e^{-A}}$ with $\mathcal{N}_{u \circ e^{-A}}(Z) \geq \alpha$.   Since $|A| < \varepsilon_0 = \widehat{\varepsilon}_0/2$, 
$$|e^{-A} Z - (0,y_0)| \leq |Z - (0,y_0)| + |e^{-A} Z - Z| \leq \widehat{\delta} + \widehat{\varepsilon}_0/2 = \delta_0.$$  
Thus $e^{-A} Z \in B_{\delta_0}(0,y_0) \cap B_{1/2}(0) \cap \Sigma_{u}$ with $\mathcal{N}_{u}(e^{-A} Z) \geq \alpha$.  Therefore, $u$ satisfies option (i).  If instead $u \circ e^{-A}$ and $\varphi \circ e^{-A}$ satisfy option (ii), it readily follows that $u$ and $\varphi$ satisfy option (ii). 
\end{proof}

\section{Proofs of the main results} \label{sec:finale_sec}

\begin{lemma} \label{decomposition_lemma}
Let $\varphi^{(0)}$ be as in Definition~\ref{varphi0_defn}.  There exist $\varepsilon, \delta, \overline{\mu} \in (0,1)$ depending only on $n$, $m$, $q$, $\alpha$, and $\varphi^{(0)}$ such that if $u \in W^{1,2}(B_2(0);\mathcal{A}_q(\mathbb{R}^m))$ is an average-free, energy minimizing function such that  
\begin{equation*}
	\int_{B_2(0)} \mathcal{G}(u,\varphi^{(0)})^2 < \varepsilon^2, 
\end{equation*}
then 
\begin{equation*}
	\{ X \in \Sigma_{u,q} \cap B_1(0) : \mathcal{N}_u(X) \geq \alpha \} = S \cup T
\end{equation*}
where $S$ is contained in a properly embedded $(n-2)$-dimensional $C^{1,\overline{\mu}}$ submanifold $\Gamma$ of $B_1(0)$ with $\mathcal{H}^n(\Gamma \cap B_1(0)) \leq \omega_{n-2}$ and $T \subseteq  \bigcup_{j=1}^{\infty} B_{\rho_j}(X_j)$ for a countable family of balls $B_{\rho_j}(X_j)$ with $\sum_{j=1}^{\infty} \rho_j^{n-2} \leq 1-\delta$.  Moreover, for $\mathcal{H}^{n-2}$-a.e.~$X \in S$, there exists a unique nonzero, average-free, homogeneous degree $\alpha$, cylindrical, energy minimizing function $\varphi^{(Z)} : \mathbb{R}^n \rightarrow \mathcal{A}_q(\mathbb{R}^m)$ such that 
\begin{equation*}
	\rho^{-n} \int_{B_{\rho}(0)} \mathcal{G}(u(Z+X),\varphi^{(Z)}(X))^2 \leq C \rho^{2\alpha+2\overline{\mu}}
\end{equation*}
for some constant $C \in (0,\infty)$ depending only on $n$, $m$, $q$, $\alpha$, and $\varphi^{(0)}$.
\end{lemma}
\begin{proof}
Inductively choose $\vartheta_j \in (0,1/8)$ for $j = 1,2,\ldots, \lceil q/q_0 \rceil - p_0 + 1$ such that $\vartheta_j < \vartheta_{j-1}/8$ for all $j > 1$ and $C_j \vartheta_j^{\mu} \leq 1$ for all $j$, where $\mu \in (0,1)$ and $C_j = C_j(\vartheta_1,\ldots,\vartheta_{j-1}) \in (0,\infty)$ are as in Lemma~\ref{main_lemma}.  Let $\varepsilon_0$ and $\delta_0$ be as in Lemma~\ref{main_lemma}.  Define 
\begin{equation*}
	\Sigma_{u,q}^* = \{ X \in B_{1/2}(0) \cap \Sigma_{u,q} : \mathcal{N}_u(X) \geq \alpha \} .
\end{equation*}

If $u(Z+\rho X)$ satisfies option (i) of Lemma~\ref{main_lemma} for some $Z \in \Sigma_{u,q}^*$ and $\rho \in [\vartheta_{\lceil q/q_0 \rceil - p_0 + 1},1]$, then by Corollary \ref{graphical_cor}, $\Sigma_{u,q} \cap B_1(0) \subseteq B^2_{\tau(\varepsilon)}(0) \times \mathbb{R}^{n-2}$ for some $\tau(\varepsilon)$ such that $\tau(\varepsilon) \rightarrow 0$ as $\tau \downarrow 0$, hence we trivially have Lemma~\ref{decomposition_lemma} with $S = \emptyset$ and $T = \Sigma_{u,q} \cap B_1(0)$.  Thus we may assume that $u(Z+\rho X)$ does not satisfy option (i) of Lemma~\ref{main_lemma} for all $Z \in \Sigma_{u,q}^*$ and $\rho \in [\vartheta_{\lceil q/q_0 \rceil - p_0 + 1},1]$. 

For each $k \in \{1,2,3,\ldots\} \cup \{\infty\}$, we define the set $\Upsilon_k$ to be the set of points $Z \in \Sigma_{u,q}^*$ such that, letting $s_0 = 1$ and $\varphi_0 = \varphi^{(0)}$, for each integer $1 \leq i \leq k$ there exists $j(i) \in \{1,2,\ldots, \lceil q/q_0 \rceil - p_0 + 1\}$, a radius $s_i$ given by $s_i = \vartheta_{j(i)} s_{i-1}$, and $\varphi_i \in \widetilde{\Phi}_{\varepsilon_0}(\varphi^{(0)})$ such that $u(Z+ s_{i-1} X)$ does not satisfy option (i) of Lemma~\ref{main_lemma}, 
\begin{equation} \label{decomposition_eqn1}
	s_i^{-n-2\alpha} \int_{B_{s_i}(0)} \mathcal{G}(u(Z+X),\varphi_i(X))^2 dX \leq \vartheta_i^{\mu} s_{i-1}^{-n-2\alpha} \int_{B_{s_{i-1}}(0)} \mathcal{G}(u(Z+X),\varphi_{i-1}(X))^2 dX, 
\end{equation}
and either $k = \infty$ or $k < \infty$ and $u(Z+ s_k X)$ satisfies option (i) of Lemma~\ref{main_lemma}.  For every point $Z \in \Sigma_{u,q}^*$, we can inductively apply Lemma~\ref{main_lemma} to find $j(i)$ and $\varphi_i$ while $u(Z+s_{i-1} X)$ does not satisfy option (i) of Lemma~\ref{main_lemma} and thereby conclude that $Z \in \Upsilon_k$ for some $k \in \{1,2,3,\ldots\} \cup \{\infty\}$.  Note that for each integer $i \geq 1$ having found $j(l)$ and $\varphi_l$ for $l = 1,2,\ldots,i$, by \eqref{decomposition_eqn1} and Lemma~\ref{branchdist_lemma},
\begin{align} \label{decomposition_eqn2}
	s_i^{-n-2\alpha} \int_{B_{s_i}(0)} \mathcal{G}(u(Z+X),\varphi_i(X))^2 dX 
	&\leq s_i^{\mu} \int_{B_1(0)} \mathcal{G}(u(Z+X),\varphi^{(0)}(X))^2 dX \\
	&\leq C s_i^{\mu} \varepsilon^2 \nonumber 
\end{align}
for some constant $C \in (0,\infty)$ depending only on $n$, $m$, $q$, $\alpha$, and $\varphi^{(0)}$.  Hence, 
\begin{equation*}
	\int_{B_1(0)} \mathcal{G}(\varphi_i,\varphi_{i-1})^2 \leq C s_i^{\mu} \varepsilon^2
\end{equation*}
for all $i = 2,3,4,\ldots$ and some constant $C \in (0,\infty)$ depending only on $n$, $m$, $q$, $\alpha$, and $\varphi^{(0)}$.  By the triangle inequality, for $1 \leq i < j \leq k$, 
\begin{equation} \label{decomposition_eqn3}
	\left( \int_{B_1(0)} \mathcal{G}(\varphi_i,\varphi_j)^2 \right)^{1/2}
	\leq C \sum_{l=i+1}^j s_l^{\mu/2} \varepsilon 
	\leq C \sum_{l=i+1}^j \vartheta_1^{\mu (l-i)/2} s_i^{\mu/2} \varepsilon
	\leq C s_i^{\mu/2} \varepsilon 
\end{equation}
and in particular since $\varphi_0 = \varphi^{(0)}$, 
\begin{equation} \label{decomposition_eqn4}
	\int_{B_1(0)} \mathcal{G}(\varphi_i,\varphi^{(0)})^2 \leq C \varepsilon^2 
\end{equation}
for all $i = 1,2,3,\ldots$, where $C \in (0,\infty)$ are constants depending only on $n$, $m$, $q$, $\alpha$, and $\varphi^{(0)}$.  Therefore, provided $\varepsilon$ is small enough that $C^{1/2} \varepsilon < \varepsilon_0$ and $u(Z + s_i X)$ does not satisfy option (i) of Lemma~\ref{main_lemma}, we can apply Lemma~\ref{main_lemma} to find $j(i+1)$ and $\varphi_{i+1}$. 

We will now show that the conclusion of the lemma holds true with $S = \Upsilon_{\infty}$ and $T = (\Sigma_{u,q}^* \setminus \Upsilon_{\infty}) \cup (\Sigma_{u,q} \cap B_{1}(0) \setminus B_{1/2}(0))$.  Suppose $Z \in \Upsilon_{\infty}$.  By \eqref{decomposition_eqn3}, $\varphi_i$ is a Cauchy sequence in $L^2(B_1(0);\mathcal{A}_q(\mathbb{R}^m))$ and so $\varphi_i$ converges in $L^2(B_1(0);\mathcal{A}_q(\mathbb{R}^m))$ to some $\varphi^{(Z)} \in \widetilde{\Phi}_{\varepsilon_0}(\varphi^{(0)})$.  By letting $j \rightarrow \infty$ in \eqref{decomposition_eqn3},
\begin{equation*}
	\int_{B_1(0)} \mathcal{G}(\varphi_i,\varphi^{(Z)})^2 \leq C \varepsilon^2 s_i^{\mu} 
\end{equation*}
for all $i = 1,2,3,\ldots$ and some constant $C \in (0,\infty)$ depending only on $n$, $m$, $q$, $\alpha$, and $\varphi^{(0)}$ and so by \eqref{decomposition_eqn2} and the triangle inequality
\begin{equation*}
	s_i^{-n-2\alpha} \int_{B_{s_i}(0)} \mathcal{G}(u(Z+X),\varphi^{(Z)}(X))^2 dX \leq C \varepsilon^2 s_i^{\mu}  
\end{equation*}
for all $i = 0,1,2,3,\ldots$ and some constant $C \in (0,\infty)$ depending only on $n$, $m$, $q$, $\alpha$, and $\varphi^{(0)}$.  Given $\rho \in (0,1]$, choose an integer $i \geq 0$ such that $s_{i+1} < \rho \leq s_i$ to get 
\begin{equation} \label{decomposition_eqn5}
	\rho^{-n-2\alpha} \int_{B_{\rho}(0)} \mathcal{G}(u(Z+X),\varphi^{(Z)}(X))^2 dX \leq C \varepsilon^2 \rho^{\mu} 
\end{equation}
for some constant $C \in (0,\infty)$ depending only on $n$, $m$, $q$, $\alpha$, and $\varphi^{(0)}$.  Clearly $\varphi^{(Z)}$ is unique for each $Z \in \Upsilon_{\infty}$.  Since $\varphi^{(Z)}$ is a constant multiple of a blow-up of $u$ at $Z$, $\varphi^{(Z)}$ is energy minimizing.  By letting $i \rightarrow \infty$ in \eqref{decomposition_eqn4}, $\varphi^{(Z)} \in \widetilde{\Phi}_{C\varepsilon}(\varphi^{(0)})$ and thus there is a rotation $Q_Z$ of $\mathbb{R}^n$ such that $\varphi^{(Z)}(Q_Z X) \in \Phi_{C\varepsilon}(\varphi^{(0)})$ and $|Q_Z - I| \leq C \varepsilon$.  By the estimate on $|\xi|^2$ in Lemma~\ref{branchdist_lemma} and by \eqref{decomposition_eqn5}, 
\begin{equation*}
	\op{dist}(Q_Z^{-1} (\Sigma_{u,q}^* - Z) \cap B_{\rho}(0), \{0\} \times \mathbb{R}^{n-2}) \leq C \varepsilon \rho^{1+\mu/2}
\end{equation*}
for all $\rho \in (0,1/2]$ and some constant $C \in (0,\infty)$ depending only on $n$, $m$, $q$, $\alpha$, and $\varphi^{(0)}$.  Thus it follows from \eqref{decomposition_eqn5} that 
\begin{equation*}
	\int_{B_1(0)} \mathcal{G}(\varphi^{(Y)},\varphi^{(Z)})^2 \leq C \varepsilon^2 |Y - Z|^{\mu} 
\end{equation*}
for all $Y, Z \in \Sigma_{u,q}^*$ and so 
\begin{equation*}
	|Q_Y - Q_Z| \leq C \varepsilon |Z_1 - Z_2|^{\mu/2} 
\end{equation*}
for all $Y, Z \in \Sigma_{u,q}^*$, where $C \in (0,\infty)$ are constants depending only on $n$, $m$, $q$, $\alpha$, and $\varphi^{(0)}$.  Thus $\Upsilon_{\infty} \subseteq \op{graph} f \cap B_{1/2}(0)$ is contained in the graph of a function $f \in C^{1,\mu/2}(B^{n-2}_{1/2}(0); \mathbb{R}^{n-2})$ such that $\|f\|_{C^{1,\mu/2}(B^{n-2}_{1/2}(0))} \leq C\varepsilon$. 

Now suppose $Z \in \Upsilon_k$ for some integer $1 \leq k < \infty$.  Take $\varphi^{(Z)} = \varphi_k$.  Note that $\varphi^{(Z)}$ is no longer unique.  By the argument above, there is a rotation $Q_Z$ of $\mathbb{R}^n$ such that $\varphi^{(Z)}(Q_Z X) \in \Phi_{C\varepsilon}(\varphi^{(0)})$ and $|Q_Z - I| \leq C \varepsilon$ and \begin{equation*}
	\op{dist}(Q_Z^{-1} (\Sigma_{u,q}^* - Z) \cap B_{\rho}(0), \{0\} \times \mathbb{R}^{n-2}) \leq C \varepsilon\rho^{1+\mu/2}
\end{equation*}
for all $\rho \in [s_k,1/2]$, where $C \in (0,\infty)$ are constants depending only on $n$, $m$, $q$, $\alpha$, and $\varphi^{(0)}$.  Hence 
\begin{equation} \label{decomposition_eqn6}
	\op{dist}(\Sigma_{u,q}^* \cap B_{\rho}(Z), Z + \{0\} \times \mathbb{R}^{n-2}) \leq C \varepsilon\rho
\end{equation}
for some constant $C \in (0,\infty)$ depending only on $n$, $m$, $q$, $\alpha$, and $\varphi^{(0)}$.  By the definition of $\Upsilon_k$, $u(Y + s_k X)$ satisfies option (i) of Lemma~\ref{main_lemma}, i.e. 
\begin{equation} \label{decomposition_eqn7}
	\forall Z \in \Upsilon_k \text{ } \exists Y \in Z + \{0\} \times B^{n-2}_{s_k/2}(0) 
		\text{ such that } \Sigma_{u,q}^* \cap B_{\delta_0 s_k}(Y) \cap B_{s_k/2}(Z) = \emptyset
\end{equation}
Now arguing exactly as in pages 642-643 of~\cite{Sim93} (using \eqref{decomposition_eqn6}, \eqref{decomposition_eqn7} in place of (12), (13) on page 642 of~\cite{Sim93}), we obtain a covering of $\left( \bigcup_{1 \leq k < \infty} \Upsilon_k \right) \cup (\Sigma_{u,q} \cap B_{1}(0) \setminus B_{1/2}(0))$ by balls $B_{\rho_j}(X_j)$, $j = 1,2,3,\ldots$, such that $\sum_j \rho_j^{n-2} \leq 1-\delta$. 
\end{proof}

\begin{proof}[Proof of Theorem~\ref{main_thm}] 
The argument is similar to the proof of Theorem $2^{\prime}$ of~\cite{Sim93}, so we will only sketch it here.  Let $u \in W^{1,2}(\Omega;\mathcal{A}_q(\mathbb{R}^m))$ be a nonzero, average-free, energy minimizing $q$-valued function.  Since $\op{dim}_{\mathcal{H}} \Sigma_{u,q}^{(n-3)} \leq n-3$, it suffices to consider the set $\Sigma_{u,q,\alpha}$ of all points of $\Sigma_{u,q}$ at which $u$ has a homogeneous degree $\alpha$ cylindrical blow-up.  Let $Y_0 \in \Sigma_{u,q,\alpha}$ and $\varphi^{(0)}$ be a cylindrical blow-up of $u$ at $Y_0$.  By the definition of blow-ups and monotonicity of frequency functions, for every $\varepsilon > 0$ there exists $\sigma > 0$ such that $B_{\sigma R(\varepsilon)}(Y_0) \subset \Omega$ and 
\begin{equation*}
	\int_{B_1(0)} \mathcal{G}(u_{Y_0,\sigma},\varphi^{(0)})^2 < \varepsilon^2, \quad 
	N_{u_{Y_0,\sigma}}(R(\varepsilon)) - \alpha < \delta(\varepsilon), 
\end{equation*}
where $R(\varepsilon), \delta(\varepsilon)$ are as in Lemma~\ref{lemma2_4}.  Let $\overline{u} = u_{Y_0,\sigma}$.  For each $\rho_0 \in (0,1/2]$, define the outer measure $\mu_{\rho_0}$ on $B_1(0)$ by 
\begin{equation*}
	\mu_{\rho_0}(A) = \inf \sum_{i=1}^N \omega_{n-2} \sigma_i^{n-2}
\end{equation*}
for every set $A \subseteq B_1(0)$, where the infimum is taken over all finite covers of $A$ by open balls $B_{\sigma_i}(Y_i)$, $i = 1,2,\ldots,N$, with $\sigma_i \leq \rho_0$.  Choose a cover of $\Sigma^+_{\alpha} \cap \overline{B_1(0)}$ by a finite collection of open balls $B_{\sigma_i}(Y_i)$ such that 
\begin{equation*}
	\sum_{i=1}^N \omega_{n-2} \sigma_i^{n-2} \leq \mu_{\rho_0}(\Sigma^+_{\alpha}) + 1, 
\end{equation*}
where 
\begin{equation*}
	\Sigma^+_{\alpha} = \{ X \in \Sigma_{\overline{u},q} \cap B_1(0) : \mathcal{N}_{\overline{u}}(X) \geq \alpha \} .
\end{equation*}
Remove the balls $B_{\sigma_i}(Y_i)$ that do not intersect $\Sigma^+_{\alpha}$ from the collection.  For each $i$, let $Z_i \in B_{\sigma_i}(Y_i) \cap \Sigma^+_{\alpha}$.  By Lemma~\ref{lemma2_4}, either there exists a nonzero, cylindrical, homogeneous degree $\alpha$, energy minimizing $q$-valued function $\varphi_i : \mathbb{R}^n \rightarrow \mathcal{A}_q(\mathbb{R}^m)$ such that 
\begin{equation} \label{main_theorem_eqn1}
	\int_{B_1(0)} \mathcal{G}(\overline{u}_{Z_i,2\sigma_i},\varphi_i)^2 < \varepsilon^2 
\end{equation}
or there exists an $(n-3)$-dimensional subspace $L$ of $\mathbb{R}^n$ such that 
\begin{equation} \label{main_theorem_eqn2}	
	\{ X \in \Sigma_{\overline{u}} \cap \overline{B_{2\sigma_i}(Z_i)} : \mathcal{N}_{\overline{u}}(X) \geq \alpha \} \subset \{ X \in \mathbb{R}^n : \op{dist}(X,Z_i + L) < \varepsilon \} . 
\end{equation}
Note that we use the fact that the degree of homogeneity $\alpha$ of a cylindrical multivalued function must equal $\ell_0/q_0$ for some relatively prime positive integers $\ell_0, q_0$ with $q_0 \leq q$ and thus the set of all such $\alpha$ is discrete.  The conclusion of the theorem now follows by arguing exactly like in~\cite{Sim93}, iteratively applying Lemma~\ref{decomposition_lemma} using the fact that either \eqref{main_theorem_eqn1} or \eqref{main_theorem_eqn2} holds true.
\end{proof}

\begin{proof}[Proof of Theorem A of the Introduction]
Let $u \in W^{1,2}(\Omega;\mathcal{A}_q(\mathbb{R}^m))$ be a $q$-valued energy minimizing function.  Set $h = u_a$, the average of the values of $u$, which by~\cite[Theorem~2.6]{Almgren} is a single-valued harmonic function, and let $v = u-h$.  By Theorem~\ref{main_thm}(b), for $\mathcal{H}^{n-2}$-a.e. $Z \in \mathcal{B}_u$, there exists an average-free, homogeneous, cylindrical, energy minimizing $q$-valued function $\varphi^{(Z)}$ and $\rho_Z > 0$ such that 
\begin{equation*}
	\rho^{-n} \int_{B_{\rho}(0)} \mathcal{G}(v(Z+X),\varphi^{(Z)}(X))^2 \leq C_Z \rho_Z^{2\alpha+2\mu_Z}
\end{equation*}
for some constants $\mu_Z \in (0,\infty)$ and $C_Z \in (0,\infty)$.  Consequently, the desired conclusion of the theorem, including the $L^2$ estimate on the error term $\epsilon^{(Z)}_j$, holds true.  
\end{proof}

\begin{proof}[Proof of Theorem B of the Introduction] 
First observe that if for some $q$-valued energy minimizing function $u \in W^{1,2}(\Omega;\mathcal{A}_q(\mathbb{R}^m))$ and some closed ball $B \subset \Omega$, $\mathcal{H}^{n-2}(\mathcal{B}_u \cap B) = 0$, then $B \setminus \mathcal{B}_u$ is simply-connected (see the Appendix of~\cite{SimWic11}).  Since locally in $B \setminus \mathcal{B}_u$, $u$ decomposes into $q$ single-valued harmonic functions, it follows that $u$ decomposes into $q$ single-valued harmonic functions on $B$.  Hence $\mathcal{B}_u \cap B = \emptyset$.  

Let $B$ be a closed ball in $\Omega$.  By Theorem~\ref{main_thm}, there is a finite set $\{ \alpha_1, \alpha_2, \ldots, \alpha_k \}$ such that $B \cap \mathcal{B}_{u,q} \cap \Sigma_{u,q,\alpha_j}$ is nonempty for all $j = 1,2,\ldots,k$ and for each $j = 1,2,\ldots,k$ there exists an open set $V_{\alpha_j} \supset B \cap \Sigma_{u,q,\alpha_j}$ such that $V_{\alpha_j} \cap \{ X : \mathcal{N}_u(X) = \alpha_j \}$ has locally finite $\mathcal{H}^{n-2}$ measure in $V_{\alpha_j}$.  Set $\alpha_0 = 0$ and $\alpha_{k+1} = \infty$.  For $j = 0,1,2,\ldots,k$, let 
\begin{equation*}
	\Gamma_j = \{ X \in B \cap \Sigma_{u,q} : \alpha_j \leq \mathcal{N}_u(X) < \alpha_{j+1} \} \cap V_{\alpha_j}
\end{equation*}
so that $\Gamma_j$ has locally finite measure (in $V_{\alpha_j}$) and let 
\begin{equation*}
	\widetilde{\Gamma}_j = \{ X \in B \cap \Sigma_{u,q} : \alpha_j \leq \mathcal{N}_u(X) < \alpha_{j+1} \} \setminus V_{\alpha_j}.
\end{equation*}
Since $\widetilde{\Gamma}_j \subset \Sigma_{u,q}^{(n-3)}$, by Lemma~\ref{alm_fed_lemma}, $\widetilde{\Gamma}_j$ has Hausdorff dimension at most $n-3$.  Moreover, by upper semi-continuity of $\mathcal{N}_u$, each of $\Gamma_j, \widetilde{\Gamma}_j$ is the intersection of an open set and a closed set and hence is locally compact.  Of course, $\mathcal{B}_{u,q} \cap B = \bigcup_{j=0}^k \Gamma_j \cup \widetilde{\Gamma}_j$. 

Now let $k \in \{1,2,3,\ldots,q-1\}$ and $B$ be a closed ball in $\Omega \setminus \bigcup_{0 \leq l < k} \mathcal{S}_{u,l}$.  For each $Y \in \mathcal{S}_{u,k}$, there exists a $\rho \in (0,\op{dist}(Y,\partial \Omega))$ such that $u(X) = \sum_{j=1}^k u_j(X)$ on $B_{\rho}(Y)$ for $q_j$-valued energy minimizing functions $u_j$ and $Y \in \bigcup_{j=1}^k \mathcal{B}_{u_j,q_j}$.  Observe that $\mathcal{S}_{u,k} \cap B_{\rho}(Y) = \bigcup_{j=1}^k \mathcal{B}_{u_j,q_j}$.  By the above discussion applied to $u_j$, $\mathcal{B}_{u_j,q_j}$ is a union of finitely many pairwise disjoint, locally compact sets each of which is locally $(n-2)$-rectifiable.  By the compactness of $B$, it follows that $\mathcal{S}_{u,k}$ is a union of finitely many pairwise disjoint, locally compact sets each of which is locally $(n-2)$-rectifiable.
\end{proof}

\begin{proof}[Proof of Theorem C of the Introduction]
See \cite{KrumWic1}.
\end{proof}

\appendix
\section*{Appendix} 
\renewcommand{\thetheorem}{A\arabic{theorem}}
\renewcommand{\theequation}{A\arabic{equation}}

Here we will collect some elementary properties of Dirichlet energy minimizing functions and cylindrical functions that we have used above.  The first is a well-known compactness property. 

\begin{lemma} \label{compactness_lemma} 
Let $\Omega \subset \mathbb{R}^n$ be an open set and for $k = 1,2,3,\ldots$ let $u^{(k)} \in W^{1,2}(\Omega;\mathcal{A}_q(\mathbb{R}^m))$ be locally Dirichlet energy minimizing functions such that 
\begin{equation*}
	\sup_k \int_{\Omega'} |u^{(k)}|^2 < \infty \quad \text{for all } \Omega' \subset\subset \Omega .
\end{equation*}
Then there exists a subsequence $\{u^{(k')}\}$ of $\{u^{(k)}\}$ and a locally Dirichlet energy minimizing function $u \in W^{1,2}(\Omega;\mathcal{A}_q(\mathbb{R}^m))$ such that $u^{(k')} \rightarrow u$ uniformly on compact subsets of $\Omega$ and 
\begin{equation} \label{compactness_eqn1}
	\int_{\Omega'} |Du|^2 = \lim_{k' \rightarrow \infty} \int_{\Omega'} |Du^{(k')}|^2 \quad \textit{for all } \Omega' \subset\subset \Omega .
\end{equation}
\end{lemma} 
\begin{proof}
Immediate consequence of \eqref{regestimate}, the Arzela-Ascoli theorem, and~\cite[Propositions 2.11 and 3.20]{DeLSpa11}.
\end{proof}

In the case of single-valued (harmonic) functions, in addition to the conclusions of Lemma~\ref{compactness_lemma} we also have that $Du^{(k')} \rightarrow Du$ pointwise in $\Omega$.  In the case of multi-valued locally Dirichlet energy minimizers, the presence of singularties makes it more difficult to interpret and prove the statement $Du^{(k')} \rightarrow Du$ pointwise in $\Omega$.  Nonetheless, away from the singular set $\Sigma_u$ of the limit function $u$, we can show that $Du^{(k')} \rightarrow Du$ pointwise a.e.~in $\Omega$ in the precise sense of Lemma~\ref{gradient_convergence} below. 

\begin{lemma} \label{gradient_convergence} 
Let $\Omega \subset \mathbb{R}^n$ be an open set and suppose $u^{(k)}, u \in W^{1,2}(\Omega;\mathcal{A}_q(\mathbb{R}^m))$ are Dirichlet energy minimizing functions such that $u^{(k)} \rightarrow u$ uniformly on each compact subset of $\Omega$.  For $\mathcal{L}^n$ a.e.~$Y \in \Omega$ there exists $\rho > 0$ such that 
\begin{equation} \label{grad_conv_eqn1}
	u^{(k)}(X) = \sum_{i=1}^N u^{(k)}_i(X), \quad u(X) = \sum_{i=1}^N q_i \llbracket u_i(X) \rrbracket 
		\quad \text{for all } X \in B_{\rho}(Y) 
\end{equation}
and
\begin{gather} 
	u^{(k)}_i \rightarrow q_i \llbracket u_i \rrbracket \text{ uniformly on } B_{\rho}(Y) \nonumber \\
	Du^{(k)}_i \rightarrow q_i \llbracket Du_i \rrbracket \text{ in } L^2(B_{\rho}(Y);\mathcal{A}_{q_i}(\mathbb{R}^m)), \nonumber 
\end{gather}
for some positive integers $N$ and $q_i$ with $\sum_{i=1}^N q_i = q$, $q_i$-valued Dirichlet energy minimizing functions $u^{(k)}_i \in W^{1,2}(B_{\rho}(Y);\mathcal{A}_{q_i}(\mathbb{R}^m))$, and single-valued harmonic functions $u_i \in C^{\infty}(B_{\rho}(Y);\mathbb{R}^m)$ (with multiplicity $q_i$) such that $u_i(X) \neq u_j(X)$ for all $X \in B_{\rho}(Y)$ and $i \neq j$.  In particular, $Du^{(k)}_i \rightarrow q_i \llbracket Du_i \rrbracket$ pointwise a.e.~in $B_{\rho}(Y)$. 
\end{lemma}
\begin{proof}
By~\cite[Theorem 2.14]{Almgren}, the singular set $\Sigma_u$ of $u$ is a relatively closed subset of $\Omega$ of Hausdorff dimension at most $n-2$.  Take any point $Y \in \Omega \setminus \Sigma_u$.  Since $Y$ is a regular point of $u$ and $u^{(k)} \rightarrow u$ uniformly, there exists $\rho > 0$ such that we can represent $u^{(k)}$ and $u$ as in \eqref{grad_conv_eqn1} for some single-valued functions harmonic $u_i \in C^{\infty}(B_{\rho}(Y);\mathbb{R}^m)$ with multiplicity $q_i$ such that $u_i(X) \neq u_j(X)$ for all $X \in B_{\rho}(Y)$ and $i \neq j$ and some energy minimizing functions $u^{(k)}_i \in W^{1,2}(B_{\rho}(Y);\mathcal{A}_{q_i}(\mathbb{R}^m))$ converging to $q_i \llbracket u_i \rrbracket$ uniformly in $B_{\rho}(Y)$.  Let $u^{(k)}_{i;a}(X) = \frac{1}{q_j} \sum_{l=1}^{q_i} u^{(k)}_{i,l}(X)$ denote the average of $u^{(k)}_i$ and $u^{(k)}_{i;f}(X) = \sum_{l=1}^{q_i} \llbracket u^{(k)}_{i,l}(X) - u^{(k)}_{i;a}(X) \rrbracket$ denote the average-free part of $u^{(k)}_i$, where $u^{(k)}_i(X) = \sum_{l=1}^{q_i} \llbracket u^{(k)}_{i,l}(X) \rrbracket$.  By the compactness of single-valued harmonic functions, $u^{(k)}_{i;a} \rightarrow u_i$ in $C^1(B_{\rho/2}(Y))$ as $k \rightarrow \infty$.  By Lemma~\ref{compactness_lemma}, $\|Du^{(k)}_{i;f}\|_{L^2(B_{\rho/2}(Y))} \rightarrow 0$ as $k \rightarrow \infty$.  Therefore, for each $i = 1,2,\ldots,N$, $Du^{(k)}_i \rightarrow q_i \llbracket Du_i \rrbracket$ in $L^2(B_{\rho/2}(Y);\mathcal{A}_{q_i}(\mathbb{R}^m))$.
\end{proof}

Next, we have the following unique continuation of property of Dirichlet minimizing multi-valued functions. 

\begin{lemma} \label{unique continuation lemma} 
Let $\Omega \subset \mathbb{R}^n$ be a connected open set and $u,v \in W^{1,2}(\Omega;\mathcal{A}_q(\mathbb{R}^m))$ be Dirichlet energy minimizing functions.  Suppose there exists an open set $U \subset \Omega$ such that $u(X) = v(X)$ for every $X \in U$.  Then $u (X)= v(X)$ for every $X \in \Omega$. 
\end{lemma}

\begin{proof}
By~\cite[Theorem 2.14]{Almgren}, the singular sets $\Sigma_u$ and $\Sigma_v$ are both relatively closed subsets of $\Omega$ of Hausdorff dimension at most $n-2$.  Thus $\Omega^* = \Omega \setminus (\Sigma_u \cup \Sigma_v)$ is a connected open set and, by the continuity of $u$ and $v$, it suffices to show that $u = v$ in $\Omega^*$.  Let $\Xi$ be the set of all points $Y \in \Omega^*$ such that there exists a $\delta > 0$ such that $B_{\delta}(Y) \subset \Omega^*$ and $u = v$ in $\Omega^*$.  Clearly $\Xi$ is open.  We want to show that $\Xi$ is relatively closed.  Then it will follow that since $\Omega^*$ is connected and $\Xi \neq \emptyset$ by assumption, we must have $\Xi = \Omega^*$, i.e. $u = v$ on $\Omega^*$.

Suppose $Y_k \in \Xi$ and $Y \in \Omega^*$ such that $Y_k \rightarrow Y$.  We want to show that $Y \in \Xi$.  Observe that $Y$ is a regular point of $u$ and $v$ and, since $Y \in \Xi$, $u(Y) = v(Y)$.  Thus, setting $\varepsilon = \frac{1}{3} \op{sep} u(Y)$, there exists $\rho > 0$ such that $B_{\rho}(Y) \subset \Omega$ and for each $X \in B_{\rho}(Y)$ 
\begin{equation*}
	u(X) = \sum_{i=1}^N m_i \llbracket u_i(X) \rrbracket , \quad v(X) = \sum_{i=1}^N m_i \llbracket v_i(X) \rrbracket
\end{equation*}
for some positive integers $N$ and $m_i$ with $\sum_{i=1}^N m_i =q$ and some single-valued harmonic functions $u_i,v_i : B_{\rho}(Y) \rightarrow \mathbb{R}^m$ such that $u_i(Y) = v_i(Y)$, $u_i(Y) \neq u_j(Y)$ for all $i \neq j$, and $|u_i(X) - u_i(Y)| < \varepsilon$ and $|v_i(X) - v_i(Y)| < \varepsilon$ for all $X \in B_{\rho}(Y)$.  For $k$ sufficiently large, $Y_k \in \Xi \cap B_{\rho/2}(Y)$ and thus there exists $\delta_k \in (0,\rho/2)$ such that $u = v$ on $B_{\delta_k}(Y_k)$, hence $u_i = v_i$ on $B_{\delta_k}(Y_k)$ for all $i$.  By the unique continuation of single-valued harmonic functions, $u_i = v_i$ on $B_{\rho}(Y)$ for all $i$ and therefore $u = v$ on $B_{\rho}(Y)$.  In other words, $Y \in \Xi$. 
\end{proof}

Finally, we address the following fact concerning the $L^2$-metric of cylindrical functions.  

\begin{lemma} \label{cylindrical norm lemma} 
Let $q \geq 1$ be an integers and $\alpha = k_0/q_0$ for some relatively prime positive integers $k_0,q_0$ with $q_0 \leq q$.  Let $\varphi, \psi : [0,2\pi] \rightarrow \mathcal{A}_q(\mathbb{R}^m)$ such that for each $\theta \in [0,2\pi]$
\begin{align} \label{cylindrical norm eqn1} 
	\varphi(\theta) &= (q - N q_0) \llbracket 0 \rrbracket + \sum_{j=1}^N \sum_{l=0}^{q_0-1} \llbracket \op{Re}(a_j e^{i\alpha (\theta + 2\pi l)}) \rrbracket, \\ 
	\psi(\theta) &= (q - N q_0) \llbracket 0 \rrbracket + \sum_{j=1}^N \sum_{l=0}^{q_0-1} \llbracket \op{Re}(b_j e^{i\alpha (\theta + 2\pi l)}) \rrbracket 
		\nonumber 
\end{align} 
for some integer $1 \leq N \leq \lceil q/q_0 \rceil$ and $a_j,b_j \in \mathbb{C}^m$.  Assume 
\begin{equation} \label{cylindrical norm eqn2} 
	\sum_{j=1}^N |a_j - b_j|^2 \leq \sum_{j=1}^N |a_j - e^{i2\pi l_j/q_0} b_{\sigma(j)}|^2
\end{equation}
for every integer $0 \leq l_j < q_0$ and permutation $\sigma$ of $\{1,2,\ldots,N\}$.  Then  
\begin{equation*} 
	\frac{1}{C} \sum_{j=1}^N |a_j - b_j|^2 \leq \int_0^{2\pi} \mathcal{G}(\varphi(\theta),\psi(\theta))^2 \,d\theta \leq C \sum_{j=1}^N |a_j - b_j|^2 
\end{equation*}
for some constant $C = C(n,m,q,\alpha) \in (0,\infty)$. 
\end{lemma}

\begin{proof}
We will in fact show that if either 
\begin{enumerate}
	\item[(a)] $q_0 = 1$ and $\alpha \in (0,\infty)$ is any positive real number or 
	\item[(b)] $\alpha = k_0/q_0$ for relatively prime positive integers $k_0,q_0$ with $q_0 \leq q$, 
\end{enumerate}
then for every $L \in (0,2\pi]$ and every function $\varphi, \psi : [0,L] \rightarrow \mathcal{A}_q(\mathbb{R}^m)$ given by \eqref{cylindrical norm eqn1} for some integer $1 \leq N \leq \lceil q/q_0 \rceil$ and some $a_j,b_j \in \mathbb{C}^m$ satisfying \eqref{cylindrical norm eqn2}, 
\begin{equation} \label{cylindrical norm eqn3} 
	\frac{1}{C} \sum_{j=1}^N |a_j - b_j|^2 \leq \frac{1}{L} \int_0^L \mathcal{G}(\varphi(\theta),\psi(\theta))^2 \,d\theta \leq C \sum_{j=1}^N |a_j - b_j|^2 
\end{equation}
for some constant $C = C(n,m,q,\alpha) \in (0,\infty)$.  Lemma~\ref{cylindrical norm lemma} then follows from case (b) with $L = 2\pi$.  The second inequality in \eqref{cylindrical norm eqn3} obviously holds true, so let us focus on proving the second inequality in \eqref{cylindrical norm eqn3}.  Without loss of generality we may assume $q = N q_0$.

First we will consider case (a) where $q_0 = 1$ and $\alpha \in (0,\infty)$.  By scaling, we may let $L = 2\pi$.  Suppose $N \geq 2$ is is the smallest integer such that there exists $\varphi_k, \psi_k : [0,2\pi] \rightarrow \mathcal{A}_N(\mathbb{R}^m)$ such that for each $\theta \in [0,2\pi]$ 
\begin{equation*} 
	\varphi_k(\theta) = \sum_{j=1}^N \llbracket \op{Re}(a_{j,k} e^{i\alpha \theta}) \rrbracket, \quad 
	\psi_k(\theta) = \sum_{j=1}^N \llbracket \op{Re}(b_{j,k} e^{i\alpha \theta}) \rrbracket
\end{equation*}
for some $a_{j,k},b_{j,k} \in \mathbb{C}^m$ such that \eqref{cylindrical norm eqn2} holds true with $a_j = a_{j,k}$ and $b_j = b_{j,k}$ and 
\begin{equation} \label{cylindrical norm eqn4} 
	\int_0^{2\pi} \mathcal{G}(\varphi_k(\theta),\psi_k(\theta))^2 \,d\theta \leq \frac{1}{k} \sum_{j=1}^N |a_{j,k} - b_{j,k}|^2.  
\end{equation}
By replacing $a_{j,k}$ with $a_{j,k} - \frac{1}{N} \sum_{l=1}^N a_{j,l}$ and $b_{j,k}$ with $b_{j,k} - \frac{1}{N} \sum_{l=1}^N b_{j,l}$ and scaling, we may assume that 
\begin{equation} \label{cylindrical norm eqn5} 
	\sum_{j=1}^N a_{j,k} = \sum_{j=1}^N b_{j,k} = 0, \quad \max\left\{\sum_{j=1}^N |a_{j,k}|^2, \sum_{j=1}^N |b_{j,k}|^2\right\} = 1. 
\end{equation}
After passing to a subsequence, let $a_{j,k} \rightarrow a_j$ and $b_{j,k} \rightarrow b_j$ as $k \rightarrow \infty$.  By \eqref{cylindrical norm eqn4}, $a_j = b_j$ for each $j$.  By \eqref{cylindrical norm eqn5}, there exists $j \neq j'$ such that $a_{j} \neq a_{j'}$.  Notice that there are only finitely many $\theta \in [0,1]$ such that $\op{Re}((a_j - a_{j'}) e^{i\alpha \theta}) = 0$ for some integers $1 \leq j,j' \leq N$ such that $a_{j} \neq a_{j'}$.  Let $I \subset [0,1]$ be a closed interval of positive length such that $\op{Re}((a_j - a_{j'}) e^{i\alpha \theta}) \neq 0$ for all $\theta \in I$ and all integers $1 \leq j,j' \leq N$ such that $a_{j} \neq a_{j'}$.  For $k$ sufficiently large, $\varphi_k |_I$ and $\psi_k |_I$ decompose into $N_j$-valued functions uniformly close to $N_j \llbracket \op{Re}(a_j e^{i\alpha \theta}) \rrbracket$ for each distinct $a_j$ with multiplicity $N_j < N$.  By the minimality of $N$, 
\begin{equation} \label{cylindrical norm eqn6} 
	\sum_{j=1}^N |a_{j,k} - b_{j,k}|^2 \leq C \int_I \mathcal{G}(\varphi_k(\theta),\psi_k(\theta))^2 \,d\theta 
		\leq C \int_0^{2\pi} \mathcal{G}(\varphi_k(\theta),\psi_k(\theta))^2 \,d\theta
\end{equation}
for all $k$ sufficiently large and some constant $C = C(n,m,N,\alpha,|I|) \in (0,\infty)$, contradicting \eqref{cylindrical norm eqn4}. 

Next we consider case (b) $q_0 \leq q$ and $\alpha = k_0/q_0$ where $k_0,q_0$ are relatively prime positive integers.  The cases where $q_0 = 1$ or $0 < L \leq 2\pi/q_0$ are already covered by case (a), so we may assume $q_0 \geq 2$ and $L \in [2\pi/q_0,2\pi]$.  Suppose $N \geq 1$ is is the smallest integer such that there exists $L_k \in [2\pi/q_0,2\pi]$ and $\varphi_k, \psi_k : [0,L_k] \rightarrow \mathcal{A}_N(\mathbb{R}^m)$ such that for each $\theta \in [0,L_k]$ 
\begin{equation*} 
	\varphi_k(\theta) = \sum_{j=1}^N \sum_{l=0}^{q_0-1} \llbracket \op{Re}(a_{j,k} e^{i\alpha (\theta + 2\pi l)}) \rrbracket, \quad 
	\psi_k(\theta) = \sum_{j=1}^N \sum_{l=0}^{q_0-1} \llbracket \op{Re}(b_{j,k} e^{i\alpha (\theta + 2\pi l)}) \rrbracket
\end{equation*}
for some $a_{j,k},b_{j,k} \in \mathbb{C}^m$ such that \eqref{cylindrical norm eqn2} holds true with $a_j = a_{j,k}$ and $b_j = b_{j,k}$ and 
\begin{equation} \label{cylindrical norm eqn7} 
	\int_0^{L_k} \mathcal{G}(\varphi_k(\theta),\psi_k(\theta))^2 \,d\theta \leq \frac{1}{k} \sum_{j=1}^N |a_{j,k} - b_{j,k}|^2.  
\end{equation}
By scaling, we may assume that 
\begin{equation} \label{cylindrical norm eqn8} 
	\max\left\{\sum_{j=1}^N |a_{j,k}|^2, \sum_{j=1}^N |b_{j,k}|^2\right\} = 1. 
\end{equation}
After passing to a subsequence, let $L_k \rightarrow L \in [2\pi/q_0,2\pi]$, $a_{j,k} \rightarrow a_j$, and $b_{j,k} \rightarrow b_j$ as $k \rightarrow \infty$.  By \eqref{cylindrical norm eqn7}, $a_j = b_j$ for each $j$.  By \eqref{cylindrical norm eqn8}, there exists $j$ such that $a_{j} \neq 0$.  Notice that there are only finitely many $\theta \in [0,1]$ such that $\op{Re}((a_j - e^{i2\pi l/q_0} a_{j'}) e^{i\alpha \theta}) = 0$ for some integers $1 \leq j,j' \leq N$ and $0 \leq l < q_0$ such that either $a_{j} \neq a_{j'}$ or $l \neq 0$.  Let $I \subset [0,1]$ be a closed interval of positive length such that $\op{Re}((a_j - e^{i2\pi l/q_0} a_{j'}) e^{i\alpha \theta}) \neq 0$ for all $\theta \in I$ and all integers $1 \leq j,j' \leq N$ and $0 \leq l < q_0$ such that either $a_{j} \neq a_{j'}$ or $l \neq 0$.  For $k$ sufficiently large, $\varphi_k |_I$ and $\psi_k |_I$ decompose into an $N_{j_1}$-valued function uniformly close to $N_{j_1} \llbracket 0 \rrbracket$ with multiplicity $N_{j_1} < N$ if $a_{j_1} = 0$ for some $j_1$ and $N_j$-valued functions close to $N_j \llbracket \op{Re}(a_j e^{i\alpha (\theta + 2\pi l)}) \rrbracket$ for each distinct nonzero $a_j$ with multiplicity $N_j$ and each $l = 0,1,2,\ldots,q_0-1$.  By case (a) and the minimality of $N$, \eqref{cylindrical norm eqn6} holds true for all $k$ sufficiently large and some constant $C = C(n,m,N,\alpha,|I|) \in (0,\infty)$, contradicting \eqref{cylindrical norm eqn7}. 
\end{proof}

\bigskip
\hskip-.2in\vbox{\hsize3in\obeylines\parskip -1pt 
  \small 
Brian Krummel
DPMMS
University of Cambridge
Cambridge CB3 0WB, United Kingdom
\vspace{4pt}
\emph{Current address}: 
Department of Mathematics
University of California, Berkeley
970 Evans Hall
Berkeley CA 94720-3840, USA
\vspace{4pt}
{\tt bkrummel@math.berkeley.edu}} 
\vbox{\hsize3in
\obeylines 
\parskip-1pt 
\small 
Neshan Wickramasekera
DPMMS 
University of Cambridge 
Cambridge CB3 0WB, United Kingdom
\vspace{4pt}
{\tt N.Wickramasekera@dpmms.cam.ac.uk}
\vspace{4pt}
\hfill
\hfill
\hfill
\hfill
\hfill}

\end{document}